\documentclass[a4paper]{article}

\usepackage[top=1in, bottom=1in, left=1.1in, right=1.8in]{geometry}
\usepackage{amsfonts}
\usepackage{amsmath}
\usepackage{amsthm,amssymb}
\usepackage{CJK,graphicx}
\usepackage{amscd}
\usepackage{amssymb}
\usepackage{mathrsfs}
\usepackage[all,cmtip]{xy}
\usepackage{lmodern}
\usepackage[Symbol]{upgreek}
\usepackage{bm}
\usepackage{eucal}
\usepackage[nopar]{lipsum}
\usepackage{rotating}
\usepackage{tikz}
\usepackage{calc}
\usepackage{tikz-cd}
\usepackage{mathabx}
\usepackage{tikz-3dplot}
\usepackage{hyperref}
\usetikzlibrary{calc}
\usetikzlibrary{decorations.markings,arrows}
\usetikzlibrary{shapes.geometric}
\newsavebox\CBox
\newcommand\hcancel[2][0.5pt]{%
	\ifmmode\sbox\CBox{$#2$}\else\sbox\CBox{#2}\fi%
	\makebox[0pt][l]{\usebox\CBox}%
	\rule[0.5\ht\CBox-#1/2]{\wd\CBox}{#1}}

\newtheorem{theorem}{Theorem}
\newtheorem{corollary}[theorem]{Corollary}
\newtheorem{definition}[theorem]{Definition}
\newtheorem{lemma}[theorem]{Lemma}
\newtheorem{proposition}[theorem]{Proposition}
\newtheorem{assumption}[theorem]{Assumption}
\newtheorem{conjecture}[theorem]{Conjecture}
\newtheorem{question}[theorem]{Question}
\newtheorem{remark}[theorem]{Remark}
\newtheorem{example}[theorem]{Example}

\numberwithin{equation}{section}

\begin{document}
	
\title{\textbf{Exact Calabi-Yau categories and odd-dimensional Lagrangian spheres}}\author{Yin Li}
\newcommand{\Addresses}{{
		\bigskip
		\footnotesize
		\textsc{Department of Mathematics, Columbia University,
			NY 10027, United States}\par\nopagebreak
		\textit{E-mail address}: \texttt{yinlee@math.columbia.edu}
}}
\date{}\maketitle

\begin{abstract}
An exact Calabi-Yau structure, originally introduced by Keller, is a special kind of smooth Calabi-Yau structure in the sense of Kontsevich-Vlassopoulos $\cite{kv}$. For a Weinstein manifold $M$, the existence of an exact Calabi-Yau structure on the wrapped Fukaya category $\mathcal{W}(M)$ imposes strong restrictions on its symplectic topology. Under the cyclic open-closed map constructed by Ganatra $\cite{sg1}$, an exact Calabi-Yau structure on $\mathcal{W}(M)$ induces a class $\tilde{b}$ in the degree one equivariant symplectic cohomology $\mathit{SH}_{S^1}^1(M)$. Any Weinstein manifold admitting a quasi-dilation in the sense of Seidel-Solomon $\cite{ss}$ has an exact Calabi-Yau structure on $\mathcal{W}(M)$. We prove that there are many Weinstein manifolds whose wrapped Fukaya categories are exact Calabi-Yau despite the fact the fact there is no quasi-dilation in $\mathit{SH}^1(M)$, a typical example is given by the affine hypersurface $\{x^3+y^3+z^3+w^3=1\}\subset\mathbb{C}^4$. As an application, we prove the homological essentiality of Lagrangian spheres in many odd-dimensional smooth affine varieties with exact Calabi-Yau wrapped Fukaya categories. 
\end{abstract}

\section{Introduction}

Given a closed Lagrangian submanifold $L$ in some symplectic manifold $M$, one of the most important questions in symplectic topology is to determine whether it represents a non-trivial homology class $[L]\in H_n(M;\mathbb{Z})$. In particular, for Weinstein manifolds, we have the following conjecture (see, for example, Section 5 of $\cite{ye}$):

\begin{conjecture}[folklore]\label{conjecture:primitivity}
Let $M$ be any Weinstein manifold, and $L\subset M$ a closed, oriented, exact Lagrangian submanifold with vanishing Maslov class. Then its homology class $[L]\in H_n(M;\mathbb{Z})$ is primitive.
\end{conjecture}

It is, however, already not obvious to see whether the homology class of an odd-dimensional Lagrangian sphere $L$ is necessarily non-trivial, since the topological intersection number $[L]\cdot[L]$ vanishes. 

This is one of the motivations for Seidel and Solomon to introduce in $\cite{ss}$ a refined version of the usual intersection number, called the \textit{$q$-intersection number}, between exact Lagrangian submanifolds which are unobstructed (e.g. when they are simply-connected). More precisely, Seidel-Solomon's theory relies on the existence of a distinguished cohomology class $b\in\mathit{SH}^1(M)$ in the first degree symplectic cohomology, which satisfies
\begin{equation}\label{eq:dilation}
\Delta(b)=h\in\mathit{SH}^0(M)^\times
\end{equation}
under the BV (Batalin-Vilkovisky) operator $\Delta:\mathit{SH}^\ast(M)\rightarrow\mathit{SH}^{\ast-1}(M)$, where $\mathit{SH}^0(M)^\times$ is the set of invertible elements elements in $\mathit{SH}^0(M)$. The class $b$ will be called a \textit{quasi-dilation} (it is defined in Lecture 19 of $\cite{ps5}$ as a class satisfying $\Delta(hb)=h$, which differs slightly from the convention used here), and in the special case when $h=1$, it is called a \textit{dilation}. Taking an algebraic viewpoint, a quasi-dilation $b$ can be regarded as a noncommutative vector field (i.e. a degree one Hochschild cocycle) over the Fukaya category $\mathcal{F}(M)$ of compact Lagrangians, so one can deform the objects of $\mathcal{F}(M)$ along $b$. For $b$-equivariant objects $L_0,L_1\subset M$, one then gets an infinitesimal $\mathbb{C}^\ast$-action on their endomorphism space $\mathit{CF}^\ast(L_0,L_1)$, which induces a derivation $\Phi_{\widetilde{L}_0,\widetilde{L}_1}$ on the Floer cohomology algebra $\mathit{HF}^\ast(L_0,L_1)$, whose non-trivialness is ensured by the condition (\ref{eq:dilation}). The $q$-intersection number between $L_0$ and $L_1$ is defined by considering the generalized eigenspaces of $\Phi_{\widetilde{L}_0,\widetilde{L}_1}$:
\begin{equation}\label{eq:q}
L_0\bullet_q L_1:=\mathrm{Str}\left(e^{\log(q)\Phi_{\widetilde{L}_0,\widetilde{L}_1}}\right),
\end{equation}
where $\mathrm{Str}$ is the supertrace.

Note that the $q$-intersection number, as opposed to the usual topological intersection number, detects odd-dimensional Lagrangian homology spheres as $L\bullet_qL=1-q$. This, together with some other algebraic properties, enables Seidel to prove that a (finite type complete) Liouville manifold $M$ cannot contain infinitely many disjoint Lagrangian spheres if it admits a dilation ($\cite{ps4}$, Theorem 1.4). Moreover, in the cases where we have a good understanding of the geometry of the dilation, it is possible to make the bound on the number of pairwise disjoint Lagrangian spheres explicit by relating it to the ordinary topology of $M$, in this way the non-triviality of the homology classes of Lagrangian spheres in certain Milnor fibers are proved, which provides evidences for Conjecture \ref{conjecture:primitivity}. However, as is already observed in $\cite{ps4}$, the dilation condition $\Delta(b)=1$ imposes very strong restrictions on the Liouville manifold $M$. For example, the only known examples of simply-connected 3-dimensional smooth affine varieties admitting dilations are Milnor fibers associated to $A_m$ singularities.
\bigskip

It is the purpose of this paper to generalize Seidel's results on the homological essentiality of odd-dimensional Lagrangian spheres to a more general class of Weinstein manifolds $M$. More precisely, the condition we need is the existence of a cohomology class $\tilde{b}\in\mathit{SH}_{S^1}^1(M)$ in the degree one $S^1$-equivariant symplectic cohomology, satisfying
\begin{equation}\label{eq:cyclic}
\mathbf{B}(\tilde{b})=h\in\mathit{SH}^0(M)^\times,
\end{equation}
where the map $\mathbf{B}:\mathit{SH}^\ast_{S^1}(M)\rightarrow\mathit{SH}^{\ast-1}(M)$ is the connecting map in Gysin's long exact sequence (\ref{eq:Gysin-SH}) relating the ordinary and equivariant symplectic cohomologies, see Section \ref{section:SH}. The class $\tilde{b}$ will be called a \textit{cyclic dilation}. Our main result proves that if the geometry of the class $\tilde{b}$ is sufficiently simple, which means it appears as a class in the $S^1$-equivariant Floer cohomology group $\mathit{HF}_{S^1}^1(\lambda)$ of a Hamiltonian with relatively small slope $\lambda>0$ at infinity, then the homology class of any Lagrangian sphere in $M$ is non-trivial, see Theorem \ref{theorem:main} below for the precise statement. Conjecturally, the class of Weinstein manifolds admitting cyclic dilations contains all the smooth affine varieties with log Kodaira dimension $-\infty$ as a subclass. Some examples and evidences are provided in Sections \ref{section:trichotomy} and \ref{section:conjecture}.
\bigskip

The notion of a cyclic dilation has an algebraic counterpart, which is known as an \textit{exact Calabi-Yau structure} on a homologically smooth $A_\infty$-category (cf. Definition \ref{definition:key}). To understand when does the wrapped Fukaya category $\mathcal{W}(M)$ of a Liouville manifold $M$ admit the structure of an exact Calabi-Yau category is actually the original motivation of this work. As will be explained in Section \ref{section:OC}, an exact Calabi-Yau structure on $\mathcal{W}(M)$ is related to a cyclic dilation $\tilde{b}\in\mathit{SH}_{S^1}^1(M)$ via the cyclic open-closed string map defined by Ganatra $\cite{sg1}$. This algebraic interpretation enables us to implement the Koszul duality between $A_\infty$-algebras to deduce the existence of cyclic dilations for many interesting examples of Liouville manifolds, including those without dilations or quasi-dilations, see Section \ref{section:trichotomy}.
\bigskip

The paper is organized as follows. Section \ref{section:results} is essentially an overview of the contents of this paper, where the motivations for considering the cyclic dilation condition (\ref{eq:cyclic}) are discussed and many of our results are summarized. Section \ref{section:preliminary} is a brief sketch of some basic algebraic notions and facts which already exist in the literature, and they are included purely for self-containedness. Section \ref{section:PF} contains our main geometric inputs, where various moduli spaces arising from the parametrized Floer theory are considered. In Section \ref{section:Lag}, we then apply the Floer theoretical techniques collected in Section \ref{section:PF} to study the Lagrangian submanifolds in Liouville manifolds with cyclic dilations. We generalize Seidel-Solomon's construction of $q$-intersection numbers in the case of a single Lagrangian sphere, and Theorem \ref{theorem:main} is proved there. Finally, in Section \ref{section:existence}, we apply $A_\infty$-Koszul duality and Lefschetz fibration techniques to produce examples of Liouville manifolds which carry cyclic dilations.

\section*{Acknowledgements}
I would like to thank my PhD supervisor Yank{\i} Lekili for his persistent encouragement and various useful suggestions during the preparation of this paper. I am also grateful to Tobias Ekholm, Daniel Pomerleano, Travis Schedler, and Jingyu Zhao, who patiently answered many of my questions concerning various related topics, and Sheel Ganatra, Paul Seidel, Nick Sheridan, and Zhengyi Zhou for pointing out many mistakes and misattributions in earlier versions of this paper. Conversations with Mark McLean during the British Isles Graduate Workshop (\textit{Singularities and Symplectic Topology}, 14th-20th July, 2018) at Jersey provides useful inspirations for the proof of Theorem \ref{theorem:unique}. I would also like to thank Uppsala University, where part of this work was done, for providing me with excellent research environment during my visit from September to November, 2018. Finally, I thank the anonymous referee for carefully reviewing this paper and providing many useful suggestions.
\bigskip

This work was supported by the Engineering and Physical Sciences Research Council [EP/L015234/1], the EPSRC Centre for Doctoral Training in Geometry and Number Theory (The London School of Geometry and Number Theory), University College London. The author is also funded by King's College London for his PhD studies.

\section{Background and results}\label{section:results}

For simplicity, we will work throughout this paper with a field $\mathbb{K}$ with $\mathrm{char}(\mathbb{K})=0$, whose algebraic closure is denoted by $\overline{\mathbb{K}}$. All the dg or $A_\infty$-categories in this paper will be defined over $\mathbb{K}$, so do the corresponding homotopy or homology theories. When the categories are split-generated by finitely many objects in the sense of $\cite{ps1}$, it is convenient to use an equivalent language, namely dg or $A_\infty$-algebras over the semisimple ring $\Bbbk:=\bigoplus_{i\in I}\mathbb{K}e_i$, where $I$ is a finite set and $\{e_i\}_{i\in I}$ is a set of idempotents indexed by $I$. In this way, we are not going to distinguish below between an $A_\infty$-category $\mathcal{A}$ split-generated by finitely many objects $\{S_i\}_{i\in I}$ and its endomorphism algebra of the object $\bigoplus_{i\in I}S_i$ in the formal enlargement $\mathcal{A}^\mathit{tw}$, which is an $A_\infty$-algebra over $\Bbbk$.
\bigskip

All the dg or $A_\infty$-algebras in this paper will be $\mathbb{Z}$-graded. The Hochschild chain complex of an $A_\infty$-algebra $\mathcal{A}$ will be denoted by $\mathit{CH}_\ast(\mathcal{A})$, and we use $\mathit{HH}_\ast(\mathcal{A})$ to denote its homology. The more familiar notation $\mathit{CC}_\ast(\mathcal{A})$ will be reserved for the cyclic chain complex $\mathit{CH}_\ast(\mathcal{A})\otimes_\mathbb{K}\mathbb{K}((u))/u\mathbb{K}[[u]]$, which computes the (positive) cyclic homology $\mathit{HC}_\ast(\mathcal{A})$. The definitions of these complexes will be briefly recalled in Section \ref{section:S1cpx}.

\subsection{Exact Calabi-Yau structures}\label{section:eCY}

Let $\mathcal{A}$ be a homologically smooth $A_\infty$-algebra over some semisimple ring $\Bbbk$. It can be regarded as a bimodule over itself, which is known as the diagonal bimodule and by slight abuse of notation, we will still denote it by $\mathcal{A}$. By our assumption, $\mathcal{A}$ is a perfect bimodule. Its dual bimodule, $\mathcal{A}^\vee$, is defined as
\begin{equation}
\mathcal{A}^\vee:=R\mathrm{Hom}_{\mathcal{A}^e}(\mathcal{A},\mathcal{A}^e),
\end{equation}
where $\mathcal{A}^e=\mathcal{A}\otimes\mathcal{A}^\mathit{op}$. $\mathcal{A}$ is a \textit{weak smooth $n$-Calabi-Yau algebra} if there exists a non-degenerate Hochschild cycle $\eta\in\mathit{CH}_{-n}(\mathcal{A})$, i.e. a cocycle which induces an isomorphism
\begin{equation}
\mathcal{A}^\vee[n]\cong\mathcal{A}
\end{equation}
between $\mathcal{A}$-bimodules. Recall that $\mathit{CH}_\ast(\mathcal{A})\cong\mathcal{A}\otimes^L_{\mathcal{A}^e}\mathcal{A}\cong R\mathrm{Hom}_{\mathcal{A}^e}(\mathcal{A},\mathcal{A}^e)$. The reader may refer to \cite{mv}, Section 8 for basic definitions and properties related to Calabi-Yau algebras.

Associated to $\mathcal{A}$ there is a long exact sequence (\cite{jll}, Theorem 2.2.1)
\begin{equation}\label{eq:Connes-LES}
\cdots\rightarrow\mathit{HC}_{-\ast-1}(\mathcal{A})\xrightarrow{S}\mathit{HC}_{-\ast+1}(\mathcal{A})\xrightarrow{B}\mathit{HH}_{-\ast}(\mathcal{A})\xrightarrow{I}\mathit{HC}_{-\ast}(\mathcal{A})\rightarrow\cdots
\end{equation}
relating Hochschild and cyclic homologies of $\mathcal{A}$, which is known as \textit{Connes' long exact sequence} $\cite{jll}$. The following definition is our main subject of study in this paper.

\begin{definition}[\cite{bd}, Definition 2.3.6]\label{definition:key}
A weak smooth $n$-Calabi-Yau structure on $\mathcal{A}$ is said to be exact if the Hochschild homology class $[\eta]$ lies in the image of Connes' map $B:\mathit{HC}_{-n+1}(\mathcal{A})\rightarrow\mathit{HH}_{-n}(\mathcal{A})$.
\end{definition}
\bigskip

Notice that the notion of an exact Calabi-Yau structure is strictly more restrictive than a \textit{smooth Calabi-Yau structure} in the sense of Kontsevich-Vlassopoulos $\cite{cg,kv}$, which is defined as a negative cyclic cycle $\tilde{\eta}\in\mathit{CC}_{-n}^-(\mathcal{A})$ whose induced Hochschild cycle in $\mathit{CH}_{-n}(\mathcal{A})$ under the inclusion map of homotopy fixed points $\iota:\mathit{CC}_\ast^-(\mathcal{A})\rightarrow\mathit{CH}_\ast(\mathcal{A})$ defines a weak smooth $n$-Calabi-Yau structure on $\mathcal{A}$. This can be easily seen from the following commutative diagram ($\cite{jll}$, Proposition 5.1.5):
\begin{equation}
\begin{tikzcd}
\mathit{HC}_{-n+1}(\mathcal{A}) \arrow[rd, "B"] \arrow[r] & \mathit{HC}_{-n}^-(\mathcal{A}) \arrow[d,"{[\iota]}"] \\
& \mathit{HH}_{-n}(\mathcal{A})
\end{tikzcd}
\end{equation}
Just as a Calabi-Yau structure $[\eta]\in\mathit{HH}_{-n}(\mathcal{A})$ is the noncommutative analogue of a holomorphic volume form $\Omega$, the existence of a lift $[\tilde{\eta}]$ in $\mathit{HC}_{-n}^-(\mathcal{A})$ corresponds to the (trivial) fact that $\Omega$ is necessarily closed. Since Connes' differential $B$ is the noncommutative analogue of the de Rham differential, the exact Calabi-Yau condition imposed on $\mathcal{A}$ is analogous to the exactness of $\Omega$ as a differential form. This explains the terminology.
\bigskip

We remark that an important class of examples of exact Calabi-Yau $A_\infty$-algebras is the so called \textit{superpotential algebras} introduced by Ginzburg $\cite{vg}$, which is roughly a dg algebra whose underlying associative algebra is modelled on some localization of the path algebra, and whose differential is specified by a superpotential lying in the commutator quotient $\mathcal{A}/[\mathcal{A},\mathcal{A}]$, see Section \ref{section:superpotential} for details. As a special case, we have the Ginzburg dg algebra $\mathcal{G}(Q,w)$ associated to a quiver with potential $(Q,w)$, see $\cite{vg}$.

\subsection{Wrapped Fukaya categories}\label{section:wrap}

Let $M$ be a $2n$-dimensional Liouville manifold with $c_1(M)=0$, which is obtained by completing a Liouville domain $\overline{M}$ with the cylindrical end $[1,\infty)\times\partial\overline{M}$. Associated to $M$ is a $\mathbb{Z}$-graded $A_\infty$-category $\mathcal{W}(M)$, well-defined up to quasi-isomorphism, known as the \textit{wrapped Fukaya category} $\cite{as}$. The objects of $\mathcal{W}(M)$ are closed, exact, oriented, $\mathit{Spin}$ Lagrangian submanifolds with vanishing Maslov class, together with certain non-compact exact Lagrangian submanifolds which are modelled at infinity as cones over Legendrian submanifolds in the contact boundary $\partial\overline{M}$.
\bigskip

In the case when $M$ is Weinstein, for any handlebody decomposition of $M$, there is a set of distinguished objects in $\mathcal{W}(M)$, namely the Lagrangian cocores $L_1,\dots, L_k$ of the $n$-handles. It is proved in $\cite{cggr}$ and $\cite{gps2}$ that $\mathcal{W}(M)$ is generated by these cocores. In particular, denote by $\mathit{CW}^\ast(L_i,L_j)$ the wrapped Floer cochain complex of two cocores $L_i,L_j\subset M$, we have an equivalence
\begin{equation}
D^\mathit{perf}(\mathcal{W}(M))\cong D^\mathit{perf}(\mathcal{W}_M)
\end{equation}
between the derived wrapped Fukaya category and the derived category of perfect modules over the wrapped Fukaya $A_\infty$-algebra
\begin{equation}
\mathcal{W}_M:=\bigoplus_{1\leq i,j\leq k}\mathit{CW}^\ast(L_i,L_j),
\end{equation}
which can be regarded as an $A_\infty$-algebra over the semisimple ring $\Bbbk=\bigoplus_{1\leq i\leq k}\mathbb{K}e_i$.
\bigskip

Combining the generation result in $\cite{cggr}$ and $\cite{gps2}$ with $\cite{sg1}$, Theorem 3 implies that the wrapped Fukaya category $\mathcal{W}(M)$ of any Weinstein manifold $M$ carries a smooth Calabi-Yau structure. Also, in favourable situations it is possible to upgrade the smooth Calabi-Yau structure on $\mathcal{W}(M)$ to an exact one. For example, the author studied in $\cite{yl}$ the wrapped Fukaya $A_\infty$-algebras of the 6-dimensional Milnor fibers associated to the isolated singularities
\begin{equation}
x^p+y^q+z^r+\lambda xyz+w^2=0
\end{equation}
in $\mathbb{C}^4$, where $\lambda\neq0,1$ and $\frac{1}{p}+\frac{1}{q}+\frac{1}{r}\leq1$, and identified them with \textit{Calabi-Yau completions} in the sense of Keller $\cite{bk,bke}$ of certain directed quiver algebras, which in particular shows that their wrapped Fukaya categories all admit exact Calabi-Yau structures. Earlier results in this direction include the work of Etg\"{u}-Lekili $\cite{etl1}$, which proves the existence of an exact Calabi-Yau structure in the case when $M$ is a 4-dimensional Milnor fiber of type $A_m$ or $D_m$ ($\mathrm{char}(\mathbb{K})\neq2$ in the latter case), and Ekholm-Lekili $\cite{ekl}$, which shows the same to be true when $M$ is a plumbing of $T^\ast S^n$'s according to any tree, where $n\geq3$. See also $\cite{is}$, where a relation between the wrapped Fukaya category $\mathcal{W}(Y_\phi;b_0)$ of some quasi-projective 3-folds $Y_\phi$ arising from meromorphic quadratic differentials $\phi$, twisted by some non-trivial bulk term $b_0\in H^2(Y_\phi;\mathbb{Z}/2)$, and the completed Ginzburg algebras arising from ideal triangulations of marked bordered surfaces is conjectured.
\bigskip

However, it is in general not true that the wrapped Fukaya category of any Weinstein manifold carries an exact Calabi-Yau structure. The first set of such counterexamples is found by Davison $\cite{bd}$, who studied the fundamental group algebra $\mathbb{K}[\pi_1(Q)]$ of a $K(\pi,1)$ space $Q$, and showed that when $Q$ is a hyperbolic manifold, $\mathbb{K}[\pi_1(Q)]$ is not exact Calabi-Yau. Note that for a closed manifold which is topologically $K(\pi,1)$, we have a quasi-isomorphism $\mathcal{W}_{T^\ast Q}:=\mathit{CW}^\ast(T_q^\ast Q,T_q^\ast Q)\cong\mathbb{K}[\pi_1(Q)]$ between (formal) $A_\infty$-algebras $\cite{ma3}$.

\subsection{Symplectic cohomologies}\label{section:SH}

There is a closed string counterpart of our discussions in Section \ref{section:wrap}. Recall that for a Liouville manifold $M$ with $c_1(M)=0$, one can define, using a Hamiltonian function which is quadratic at infinity, the symplectic cohomology $\mathit{SH}^\ast(M)$, which carries the structure of a $\mathbb{Z}$-graded unital algebra over $\mathbb{K}$. There is also an $S^1$-equivariant theory, denoted as $\mathit{SH}_{S^1}^\ast(M)$, whose construction will be recalled in Section \ref{section:equi-symp}. Analogous to Connes' long exact sequence (\ref{eq:Connes-LES}), $\mathit{SH}^\ast(M)$ and $\mathit{SH}_{S^1}^\ast(M)$ fit into the following Gysin type long exact sequence (\cite{bo2}, Theorem 1.3):
\begin{equation}\label{eq:Gysin-SH}
\cdots\rightarrow\mathit{SH}^{\ast-1}(M)\xrightarrow{\mathbf{I}}\mathit{SH}^{\ast-1}_{S^1}(M)\xrightarrow{\mathbf{S}}\mathit{SH}_{S^1}^{\ast+1}(M)\xrightarrow{\mathbf{B}}\mathit{SH}^\ast(M)\rightarrow\cdots
\end{equation}
where the composition $\mathbf{B}\circ\mathbf{I}$ gives the BV operator $\Delta$.

\begin{remark}\label{remark:notation}
Note that in the above, we have used the bold letters $\mathbf{I}$, $\mathbf{B}$ and $\mathbf{S}$ to denote the maps corresponding to $I$, $B$ and $S$ in Connes' long exact sequence (\ref{eq:Connes-LES}). This is to emphasize that we are dealing with closed string invariants. As a convention, we will use the notations $\mathbb{I}$, $\mathbb{B}$ and $\mathbb{S}$ for the open string counterparts of the maps $\mathbf{I}$, $\mathbf{B}$ and $\mathbf{S}$. In particular, there is a long exact sequence
\begin{equation}\label{eq:Gysin-wrap}
\cdots\rightarrow\mathit{HC}_{-\ast-1}(\mathcal{W}(M))\xrightarrow{\mathbb{S}}\mathit{HC}_{-\ast+1}(\mathcal{W}(M))\xrightarrow{\mathbb{B}}\mathit{HH}_{-\ast}(\mathcal{W}(M))\xrightarrow{\mathbb{I}}\mathit{HC}_{-\ast}(\mathcal{W}(M))\rightarrow\cdots,
\end{equation}
which is simply (\ref{eq:Connes-LES}) applied to the wrapped Fukaya category $\mathcal{W}(M)$.
\end{remark}

To relate the two long exact sequences (\ref{eq:Gysin-SH}) and (\ref{eq:Gysin-wrap}), we implement the \textit{cyclic open-closed string map} constructed by Ganatra $\cite{sg1}$, which on the cohomology level descends to a map
\begin{equation}
[\widetilde{\mathit{OC}}]:\mathit{HC}_\ast(\mathcal{W}(M))\rightarrow\mathit{SH}_{S^1}^{\ast+n}(M),
\end{equation}
from which one obtains the following geometric interpretation of an exact Calabi-Yau structure on $\mathcal{W}(M)$.

\begin{proposition}\label{proposition:geom-inter}
Let $M$ be a non-degenerate Liouville manifold, its wrapped Fukaya category $\mathcal{W}(M)$ carries an exact Calabi-Yau structure if and only if the connecting map $\mathbf{B}:\mathit{SH}_{S^1}^1(M)\rightarrow\mathit{SH}^0(M)$ in (\ref{eq:Gysin-SH}) hits an invertible element $h\in\mathit{SH}^0(M)^\times$.
\end{proposition}

Proposition \ref{proposition:geom-inter} will be proved in Corollary \ref{corollary:non-degenerate}. In the above, the non-degeneracy condition on a Liouville manifold is introduced by Ganatra in $\cite{sg2}$, which ensures that the open-closed map $[\mathit{OC}]:\mathit{HH}_\ast(\mathcal{W}(M))\rightarrow\mathit{SH}^{\ast+n}(M)$ is an isomorphism. A Liouville manifold $M$ is said to be \textit{non-degenerate} if there is a finite collection of Lagrangians $\{L_i\}$ in $M$ such that $\mathit{OC}$ restricted to the full $A_\infty$-subcategory $\mathcal{L}(M)\subset\mathcal{W}(M)$ formed by $\{L_i\}$ hits the identity $1\in\mathit{SH}^0(M)$. As we have seen in Section \ref{section:wrap}, any Weinstein manifold is non-degenerate since one can take $\mathcal{L}(M)$ to be the full $A_\infty$-subcategory of cocores.
\bigskip

In other words, an exact Calabi-Yau structure on $\mathcal{W}(M)$ of a non-degenerate Liouville manifold $M$ induces a cyclic dilation $\tilde{b}\in\mathit{SH}_{S^1}^1(M)$ mentioned in the introduction. Note that if $M$ admits a quasi-dilation in the sense of (\ref{eq:dilation}), then it also admits a cyclic dilation $\tilde{b}$ which arises as the image of $b$ under the map $\mathbf{I}:\mathit{SH}^1(M)\rightarrow\mathit{SH}_{S^1}^1(M)$. It is natural to ask whether the converse is true. We postpone the discussions about whether the cyclic dilation condition (\ref{eq:cyclic}) is strictly weaker than the quasi-dilation condition (\ref{eq:dilation}) to Section \ref{section:trichotomy}, and look here an immediate geometric implication by assuming the existence of a cyclic dilation.
\bigskip

Let $L\subset M$ be a closed exact Lagrangian submanifold, equipped with a rank 1 local system $\nu$ so that the isomorphism $\mathit{SH}^\ast(T^\ast L)\cong H_{n-\ast}(\mathcal{L}L;\nu)$ holds $\cite{ma2}$, where $\mathcal{L}L$ denotes the free loop space of $L$. There is an $S^1$-equivariant version of Viterbo functoriality, namely the ($S^1$-equivariant lift of) the \textit{Cieliebak-Latschev map} constructed by Cohen-Ganatra $\cite{cg}$
\begin{equation}\label{eq:CL}
[\widetilde{\mathit{CL}}]:\mathit{SH}^\ast_{S^1}(M)\rightarrow H^{S^1}_{n-\ast}(\mathcal{L}L;\nu),
\end{equation}
which is compatible with the Viterbo functoriality and the Gysin sequence, see Section \ref{section:CL}. Combined with Proposition \ref{proposition:geom-inter}, we can reinterpret Davison's non-existence result mentioned in Section \ref{section:eCY} in the following slightly more general form.

\begin{proposition}\label{proposition:Davison}
Let $M$ be a Liouville manifold which admits a cyclic dilation, then it does not contain any closed, orientable, exact Lagrangian submanifold $L\subset M$ which is hyperbolic.
\end{proposition}

Proposition \ref{proposition:Davison} will be proved in Section \ref{section:CL}. In particular, when $M=T^\ast Q$, it follows from Proposition \ref{proposition:geom-inter} and the formality result mentioned at the end of Section \ref{section:wrap} that the fundamental group algebra $\mathbb{K}[\pi_1(Q)]$ cannot be exact Calabi-Yau if $Q$ is hyperbolic, which recovers \cite{bd}, Corollary 6.2.4.

Analogous to what Seidel and Solomon have done in the case of dilations and quasi-dilations $\cite{ps5,ss}$, one can use Lefschetz fibrations to produce more examples of Liouville manifolds which admit cyclic dilations starting from the known ones. More precisely, we prove in Section \ref{section:Lefschetz} the following:
\begin{theorem}\label{theorem:Lefschetz}
Let $M$ be a $2n$-dimensional Liouville manifold, with $n\geq3$. Suppose that $\pi:M\rightarrow\mathbb{C}$ is an exact symplectic Lefschetz fibration with smooth fiber $F$. If $F$ admits a cyclic dilation, then the same is true for the total space $M$.
\end{theorem}

\subsection{Trichotomy of affine varieties}\label{section:trichotomy}

The well-known trichotomy of Riemannian manifolds says that positively curved, flat, and negatively curved manifolds have distinct geometric behaviours. In symplectic topology, there is an analogy of this trichotomy for Liouville manifolds. Geometrically, this can be understood by studying the existence and abundance of $J$-holomorphic maps $u:S\rightarrow M^\mathit{in}$ with finite energy in the interior $M^\mathit{in}$ of the associated Liouville domain $\overline{M}$, where $S$ is a punctured sphere, see $\cite{mm}$.

For simplicity, we restrict our attention to the case when $M\subset\mathbb{C}^N$ is an $n$-dimensional smooth affine variety, equipped with the restriction of the constant symplectic form on the ambient affine space, in which case the aforementioned trichotomy has a numerical description in terms of the \textit{log Kodaira dimension}
\begin{equation}\label{eq:Kodaira}
\kappa(M)\in\{-\infty\}\cup\{0,\cdots,n\}.
\end{equation}
This is defined by choosing a compactification $X$ of $M$ so that $X$ is a smooth projective variety, and the divisor $D=X\setminus M$ has simple normal crossing. $\kappa(M)$ is defined as the Kodaira-Iitaka dimension of the line bundle $K_X+D$ over $X$. We shall be interested here in the cases when $\kappa(M)=-\infty$, $\kappa(M)=0$ (in which case $M$ is known as \textit{log Calabi-Yau}); and $\kappa(M)=n$ (in which case $M$ is \textit{log general type}). These should be thought of as analogues of positively curved, flat and negatively curved Riemannian manifolds, respectively. The existence question of a cyclic dilation (or equivalently, an exact Calabi-Yau structure on $\mathcal{W}(M)$) will be considered separately in these three cases.
\bigskip

First, let $M$ be a smooth affine variety with $\kappa(M)=-\infty$. An important class of such manifolds is given by the Milnor fibers $M_{a_1,\cdots,a_{n+1}}\subset\mathbb{C}^{n+1}$ associated to the Brieskorn singularities
\begin{equation}\label{eq:Brieskorn}
z_1^{a_1}+z_2^{a_2}+\cdots+z_{n+1}^{a_{n+1}}=0,
\end{equation}
where $\sum_{i=1}^{n+1}\frac{1}{a_i}>1$. In this paper, we will study the simplest non-trivial case, namely a Fermat affine cubic 3-fold $M_{3,3,3,3}\subset\mathbb{C}^4$.

\begin{theorem}\label{theorem:Fano}
The manifold $M_{3,3,3,3}$ admits a cyclic dilation. 
\end{theorem}

This will be proved in Section \ref{section:Koszul} using essentially algebraic arguments. Abstractly, one should think of Theorem \ref{theorem:Fano} as a consequence of the Koszul duality between the compact and the wrapped Fukaya categories of $M_{3,3,3,3}$.

Another key point of the proof is to show that up to quasi-isomorphism, the wrapped Fukaya $A_\infty$-algebra of $M_{3,3,3,3}$ is concentrated in non-positive degrees, which is expected to be true for any $M_{a_1,\cdots,a_{n+1}}$ with $\sum_{i=1}^{n+1}\frac{1}{a_i}>1$, although the verification is more involved in general. 

\begin{remark}
Since our proof of Theorem \ref{theorem:Fano} relies on the results of $\cite{ekl}$, it is also dependent on the Legendrian surgery description of the wrapped Fukaya category due to Bourgeois-Ekholm-Eliashberg $\cite{bee}$. Details of the proofs of the results sketched in $\cite{bee,ekl}$ can be found in the recent work $\cite{te}$.
\end{remark}

Note that if $a_i\geq3$ for all $i$, then $M_{a_1,\cdots,a_{n+1}}$ does not admit a quasi-dilation. This is argued in $\cite{ps4}$, Example 2.7 for dilations, and the argument there extends trivially to the more general case of quasi-dilations (a sketch is given in the proof of Corollary \ref{corollary:non-formality}). In particular, Theorem \ref{theorem:Fano} shows that the existence of a cyclic dilation is strictly weaker than having a quasi-dilation.

Combining Theorems \ref{theorem:Lefschetz} and \ref{theorem:Fano} we have the following:
\begin{corollary}
Take the affine hypersurface $\{p(z_1,\cdots,z_{n+1})=0\}\subset\mathbb{C}^{n+1}$, such that
\begin{equation}\label{eq:cubic}
p(z)=z_1^3+z_2^3+z_3^3+z_4^3+\tilde{p}(z_5,\cdots,z_{n+1})
\end{equation}
has an isolated singularity at the origin. Let $M$ be the Milnor fiber associated to $p$, then $M$ admits a cyclic dilation.
\end{corollary}

\begin{remark}
More interesting examples of Liouville manifolds admitting cyclic dilations are established in the recent work of Zhou $\cite{zz2}$ based on the machinery of Diogo-Lisi $\cite{dl}$. In particular, his result implies that $M_{a,\cdots,a}$ admits a cyclic dilation as long as $n\geq a$, therefore generalizes Theorem \ref{theorem:Fano} above. Our method has the advantage that it is applicable to examples beyond complements of smooth divisors in projective varieties. As an example, see Proposition \ref{proposition:plumbing}.
\end{remark}
\bigskip

Second, we consider the case when $M$ is a smooth log Calabi-Yau variety. These manifolds provide important local examples for testing the validity of mirror symmetry and have been studied extensively in the existing literature $\cite{aak,gp1,gp2,ghk,jp}$. As an illustration to the general situation, we consider here the simplest case when $\dim_\mathbb{C}(M)=2$, and make the following observation.\footnote{The author thanks Daniel Pomerleano for suggesting this approach to prove Proposition \ref{proposition:log-CY}, which greatly simplifies the original argument.} Recall that a log Calabi-Yau surface with \textit{maximal boundary} is the complement $X\setminus D$, where $X$ is a smooth projective surface, and $D\subset X$ is a singular anticanonical divisor with nodal singularities.

\begin{proposition}\label{proposition:log-CY}
Let $M$ be an affine log Calabi-Yau surface with maximal boundary, then $M$ admits a cyclic dilation if and only if it admits a quasi-dilation.
\end{proposition}
\begin{proof}
With our assumptions, one can arrange so that the Conley-Zehnder indices of the periodic orbits are 0, 1 and 2, see for example $\cite{gp2,jp}$. In particular, the cochain complex $\mathit{SC}^\ast(M)$ defining the symplectic cohomology $\mathit{SH}^\ast(M)$ is supported in these three degrees. Thus any cyclic dilation can only come from a cocycle in $\mathit{SC}^1(M)$, see our discussions in Section \ref{section:Gysin} for details. 
\end{proof}

We expect the same to be true in higher dimensions, although no insights can be drawn from the argument above.
\bigskip

Finally, let us take a look at the case when $M$ is a smooth affine variety of log general type. To get some concrete examples, one can take any Milnor fiber $M_{a_1,\cdots,a_{n+1}}$ as above, but now with $\sum_{i=1}^{n+1}\frac{1}{a_i}<1$. In complex dimension 2, the Milnor fibers associated to Arnold's 14 exceptional unimodal singularities are affine surfaces of log general type, since they are complements of ample divisors in K3 surfaces, see $\cite{lu}$.

Via the Abel-Jacobi map, we can embed a genus two curve $\Sigma_2$ in its Jacobian variety $J(\Sigma_2)$, let $M$ be the complement in $J(\Sigma_2)$ of the image of $\Sigma_2$. Clearly, $M$ is log general type. On the other hand, since there is an embedding $D^\ast T^2\# D^\ast T^2\hookrightarrow M$ from the plumbing of two copies of the disc cotangent bundles over $T^2$ into $M$ as a Liouville subdomain, Lagrangian surgery produces a genus two exact Lagrangian surface in $M$. One can therefore use Proposition \ref{proposition:Davison} to conclude that there is no cyclic dilation in $\mathit{SH}_{S^1}^1(M)$. This example can be generalized to the case when $M$ is the complement of a nearly tropical hypersurface in the abelian variety $(\mathbb{C}^\ast)^n/\Gamma$, where $\Gamma\subset\mathbb{R}^n$ is a lattice, see $\cite{aak}$, Section 10.

It seems that genus two exact Lagrangian surfaces can also be established in the 4-dimensional Milnor fibers $M_{a_1,a_2,a_3}$ with $\frac{1}{a_1}+\frac{1}{a_2}+\frac{1}{a_3}<1$, by imitating the strategy of Keating $\cite{ak1}$. However, it is not true that hyperbolic exact Lagrangian submanifolds can always be constructed in varieties of log general type. For instance, this is the case of the complement $M$ of $n+2$ generic hyperplanes in $\mathbb{CP}^n$, with $n\geq2$. These manifolds are known as higher dimensional pair-of-pants, and are studied extensively in the context of mirror symmetry, see for example $\cite{lp}$. Since $M$ is uniruled by $(n+2)$-punctured holomorphic spheres, similar argument as in the proof of $\cite{egh}$, Theorem 1.7.5 excludes the existence of hyperbolic Lagrangians in $M$.

We expect that a smooth affine variety of log general type can never admit a cyclic dilation, and will prove the following general statement in Section \ref{section:general type}. 

\begin{theorem}\label{theorem:unique}
Let $M$ be a smooth affine variety of log general type which contains an exact Lagrangian torus, then it does not admit a cyclic dilation.
\end{theorem}

Theorem \ref{theorem:unique} shows that the higher-dimensional pair-of-pants mentioned above do not admit cyclic dilations, but it is not helpful in general as there are many contractible affine varieties of log general type, which conjecturally do not contain any closed exact Lagrangian submanifold. One example is the Ramanujam surface studied in $\cite{ssm}$.

\subsection{Categorical dynamics}\label{section:dynamics}

We start with a brief overview of the theory of categorical dynamics, which is developed by Seidel in a series of works $\cite{ps3,ps4,ps5,ps7,ss}$. Various assumptions will be imposed here to keep the exposition simple enough. Given a quiver with potential $(Q,w)$, one can associate to it two $A_\infty$-algebras. One of them is the (completed) Ginzburg dg algebra $\widehat{\mathcal{G}}(Q,w)$ mentioned at the end of Section \ref{section:eCY} (see also Section \ref{section:superpotential} below for related backgrounds), while the other one, denoted as $\mathcal{B}(Q,w)$, is introduced by Kontsevich-Soibelman $\cite{ks}$. $\mathcal{B}(Q,w)$ is related to $\widehat{\mathcal{G}}(Q,w)$ via Koszul duality, which means there are quasi-isomorphisms
\begin{equation}\label{eq:BQW}
\mathcal{B}(Q,w)\cong R\mathrm{Hom}_{\mathcal{G}(Q,w)}(\Bbbk,\Bbbk), \widehat{\mathcal{G}}(Q,w)\cong R\mathrm{Hom}_{\mathcal{B}(Q,w)}(\Bbbk,\Bbbk),
\end{equation}
where $\Bbbk:=\bigoplus_{i\in Q_0}\mathbb{K}e_i$ is the semisimple ring consisting of copies of $\mathbb{K}$ indexed by the set of vertices $Q_0$ of $Q$. Koszul duality between $\mathcal{B}(Q,w)$ and $\mathcal{G}(Q,w)$ induces an isomorphism between their Hochschild cohomologies \cite{bk1}:
\begin{equation}\label{eq:HH-Koszul}
\mathit{HH}^\ast(\mathcal{B}(Q,w))\cong\mathit{HH}^\ast(\mathcal{G}(Q,w)).
\end{equation}
In general this is only an isomorphism of Gerstenhaber algebras, but since $\mathcal{B}(Q,w)$ is a cyclic $A_\infty$-algebra, and $\mathcal{G}(Q,w)$ is exact Calabi-Yau, their Hochschild cohomologies carry naturally induced BV structures. For cyclic $A_\infty$-algebras, this is proved by Tradler (\cite{tt}, Theorem 1, and for smooth Calabi-Yau algebras, this is due to Ginzburg (\cite{vg}, Theorem 3.4.3). Thus one expects that (\ref{eq:HH-Koszul}) is a BV algebra isomorphism. See \cite{cyz}, Theorem A, where this is confirmed for Koszul Calabi-Yau algebras.
\bigskip

Let us assume temporarily that the superpotential $w$ is homogeneous and consists only of cubic terms, which in particular implies the formality of the $A_\infty$-structure on $\mathcal{B}(Q,w)$ (since it is a dg algebra with vanishing differential), and work over the ground field $\mathbb{K}=\mathbb{C}$. The graded associative algebra $\mathcal{B}(Q,w)$ then carries a rational $\mathbb{C}^\ast$-action, which has weight $i$ on the degree $i$ part. According to $\cite{ps7}$, this $\mathbb{C}^\ast$-action enables us to define a bigraded refinement $\mathcal{B}(Q,w)^\mathbf{perf}$ of the $A_\infty$-category $\mathcal{B}(Q,w)^\mathit{perf}$ of perfect $A_\infty$-modules over $\mathcal{B}(Q,w)$. The $\mathbb{C}^\ast$-action on $\mathcal{B}(Q,w)$ induces at the infinitesimal level a Hochschild cocycle $\mathit{eu}_\mathcal{B}\in\mathit{CH}^1(\mathcal{B}(Q,w))$. This particular noncommutative vector field is known as the \textit{Euler vector field}. Under the BV operator, $\mathit{eu}_\mathcal{B}$ goes to a non-zero scalar multiple of the identity. Any object $\mathcal{E}$ of $\mathcal{B}(Q,w)^\mathit{perf}$ which is \textit{rigid} and \textit{simple}, meaning that
\begin{equation}
H^0(\hom_{\mathcal{B}(Q,w)^\mathit{perf}}(\mathcal{E},\mathcal{E}))\cong\mathbb{C}, H^1(\hom_{\mathcal{B}(Q,w)^\mathit{perf}}(\mathcal{E},\mathcal{E}))\cong 0,
\end{equation}
is $\mathbb{C}^\ast$-equivariant, and therefore defines an object in the category $\mathcal{B}(Q,w)^\mathbf{perf}$. In particular, it is infinitesimally equivariant with respect to $\mathit{eu}_\mathcal{B}$. For any two such objects $\mathcal{E}_0$ and $\mathcal{E}_1$, the derivation $[\mathit{eu}_\mathcal{B}]$ defines an endomorphism of $H^\ast(\hom_{\mathcal{B}(Q,w)^\mathit{perf}}(\mathcal{E}_0,\mathcal{E}_1))$, from which one recovers the weight grading on $\mathcal{B}(Q,w)^\mathbf{perf}$.
\bigskip

Geometrically, let $M$ be a $2n$-dimensional Weinstein manifold, and assume that its wrapped Fukaya $A_\infty$-algebra $\mathcal{W}_M$ is quasi-isomorphic to some Ginzburg dg algebra $\mathcal{G}(Q,w)$. In view of Proposition \ref{proposition:geom-inter}, $M$ admits a cyclic dilation. For simplicity, assume further that the Fukaya categories $\mathcal{F}(M)$ and $\mathcal{W}(M)$ are Koszul dual, so that the endomorphism algebra $\mathcal{F}_M$ of a set of split-generators in $\mathcal{F}(M)$ can be identified with $\mathcal{B}(Q,w)$. Our assumptions therefore ensure that the closed-open string map
\begin{equation}\label{eq:CO}
[\mathit{CO}]:\mathit{SH}^\ast(M)\rightarrow\mathit{HH}^\ast(\mathcal{F}(M))
\end{equation}
is an isomorphism.
\bigskip

Via the quasi-isomorphism $\mathcal{F}_M\cong\mathcal{B}(Q,w)$ and the inverse of (\ref{eq:CO}), the Euler vector field $\mathit{eu}_\mathcal{B}$ gives rise to a quasi-dilation $b\in\mathit{SH}^1(M)$. In other words, the infinitesimal symmetry $b$ integrates to a \textit{dilating} $\mathbb{C}^\ast$-action on the Fukaya category $\mathcal{F}(M)$ in the sense of $\cite{ps3}$. For any two $\mathbb{C}^\ast$-equivariant objects $L_0,L_1$ of $\mathcal{F}(M)$, this enables us to define an endomorphism of their Floer cohomology algebra $\mathit{HF}^\ast(L_0,L_1)$, which equips $\mathit{HF}^\ast(L_0,L_1)$ with an additional $\mathbb{C}$-grading by generalized eigenspaces. 

More generally, as mentioned in the introduction, one can start directly with a quasi-dilation $b\in\mathit{SH}^1(M)$, and consider the infinitesimal deformation of the objects in $\mathcal{F}(M)$ along $b$. For $b$-equivariant Lagrangian submanifolds $L_0,L_1\subset M$, the infinitesimal action of $b$ still defines a derivation $\Phi_{\widetilde{L}_0,\widetilde{L}_1}$ on the $\mathbb{C}$-algebra $\mathit{HF}^\ast(L_0,L_1)$, which appears in the definition of the $q$-intersection number (\ref{eq:q}). This construction is due to Seidel-Solomon $\cite{ss}$. As an application, Seidel proves in $\cite{ps4}$ that for Liouville manifolds with dilations, there is an upper bound on the number of disjoint Lagrangian spheres.
\bigskip

It is natural to ask whether there is a similar theory after removing the assumption that $\mathcal{B}(Q,w)$ (and thus $\mathcal{F}_M$) is formal. Note that as a consequence of non-formality, the aforementioned $\mathbb{C}^\ast$-action on $\mathcal{B}(Q,w)$ does not preserve the $A_\infty$-structure. 

In general, given any Liouville manifold $M$ with a cyclic dilation $\tilde{b}\in\mathit{SH}_{S^1}^1(M)$, one can try to imitate Seidel-Solomon's construction by making use of the higher order closed-open string maps introduced by Ganatra $\cite{sg1}$, see Section \ref{section:CO}. Due to the existence of certain obstruction terms (corresponding to some unwanted boundary strata in the relevant moduli spaces), it is in general not possible to obtain endomorphisms on Floer cohomology groups $\mathit{HF}^\ast(L_0,L_1)$ of any simply-connected Lagrangian submanifolds. However, for an odd-dimensional Lagrangian sphere $L\subset M$ with dimension $n\geq3$, the obstruction vanishes and we obtain a non-trivial derivation $\Phi_{\widetilde{L},\widetilde{L}}$ on $\mathit{HF}^\ast(L,L)$, see Section \ref{section:q}. This indicates that one can use cyclic dilations to detect the homological non-triviality of odd-dimensional Lagrangian spheres. In fact, we will prove the following result in Section \ref{section:disjoint}.

\begin{theorem}\label{theorem:main}
Let $M$ be a $2n$-dimensional Weinstein manifold, where $n\geq3$ is odd, and $c_1(M)=0$. Assume that $M$ admits a cyclic dilation $\tilde{b}\in\mathit{SH}_{S^1}^1(M)$, and satisfies an additional property ($\widetilde{\textrm{H}}$) (cf. Definition \ref{definition:H}). Then for any Lagrangian sphere $L\subset M$, its homology class $[L]\in H_n(M;\mathbb{K})$ is nonzero.
\end{theorem}

In the above, if $\mathbb{K}$ can be taken to be an arbitrary field, then one would be able to arrive at the conclusion that $[L]\in H_n(M;\mathbb{Z})$ is primitive. However, proving Theorem \ref{theorem:main} may involve essential difficulties when $\mathrm{char}(\mathbb{K})=2$, see Remark \ref{remark:characteristic}. Here, property ($\widetilde{\textrm{H}}$) imposes additional restrictions on the cohomology class $\tilde{b}$. More precisely, there is a continuation map $\mathit{HF}^1_{S^1}(\lambda)\rightarrow\mathit{SH}_{S^1}^1(M)$ from the $S^1$-equivariant Floer cohomology of a Hamiltonian with slope $\lambda$ on the cylindrical end $M\setminus\overline{M}$ to the $S^1$-equivariant symplectic cohomology, see (\ref{eq:continuation}). Roughly speaking, property ($\widetilde{\textrm{H}}$) says that the class $\tilde{b}$ appears already in $\mathit{HF}^1_{S^1}(\lambda)$, with $\lambda>0$ being relatively small with respect to the minimal period of the Reeb orbits on the contact boundary $\partial\overline{M}$. See Definition \ref{definition:H} for details.

This property is expected to be true for many Weinstein manifolds with cyclic dilations, including the Milnor fibers $M_{a_1,\cdots,a_{n+1}}$ with $\sum_{i=1}^{n+1}\frac{1}{a_i}>1$. In fact, it follows from the argument in $\cite{zz2}$, Section 5.2 that it holds for the Milnor fibers $M_{a,\cdots,a}$ with $n\geq a$. When the cyclic dilation satisfies $\mathbf{B}(\tilde{b})=1$, the assumption of Theorem \ref{theorem:main} can be equivalently expressed in terms of the first Gutt-Hutchings capacity of the Weinstein domain $\overline{M}$, see Corollary \ref{corollary:GH}.

In algebraic terms, a related question is the following:

\begin{question}\label{question:phantom}
Let $X$ be any spherical object in the triangulated $A_\infty$-category $\mathcal{B}(Q,w)^\mathit{perf}$, then is it always true that its class $[X]$ in the Grothendieck group $K_0\left(\mathcal{B}(Q,w)^\mathit{perf}\right)$ is nonzero?
\end{question}

For Weinstein manifolds whose wrapped Fukaya $A_\infty$-algebra $\mathcal{W}_M$ can be identified with the Ginzburg dg algebra $\mathcal{G}(Q,w)$, the work of Lazarev ($\cite{ol}$, Theorem 1.8) combined with Koszul duality gives rise to an injective map
\begin{equation}
K_0(\mathcal{F}(M))\rightarrow H_n(M;\mathbb{Z}).
\end{equation}
Thus an affirmative answer to Question \ref{question:phantom} would imply the homological essentiality of Lagrangian spheres in $M$. It is in general unknown how to answer Question \ref{question:phantom} in odd dimensions, unless $Q$ is a finite oriented tree and $w$ vanishes ($\cite{ps5}$, Remark 15.16). 

Theorem \ref{theorem:main} will be proved in Section \ref{section:disjoint}.

\section{Algebraic preliminaries}\label{section:preliminary}

We summarize in this section some basic algebraic notions and facts that will be used in this paper. The expositions in Sections \ref{section:S1cpx} and \ref{section:non-unital} follow essentially from Sections 2 and 3 of $\cite{sg1}$.

\subsection{Superpotential algebras}\label{section:superpotential}

As we have mentioned in Section \ref{section:eCY}, an important class of exact Calabi-Yau algebras is given by superpotential algebras defined by Ginzburg $\cite{vg}$, whose definition we shall briefly recall here.
\bigskip 

Let $\mathcal{A}$ be a finitely generated unital dg algebra over $\Bbbk$, denote by $\Omega_\mathcal{A}^1$ the $\mathcal{A}$-bimodule of non-commutative differentials on $\mathcal{A}$, which is the kernel of the multiplication map $\mathcal{A}\otimes\mathcal{A}\rightarrow\mathcal{A}$. $\mathcal{A}$ is said to be \textit{quasi-free} if $\Omega_\mathcal{A}^1$ is projective. For example, any path algebra of a quiver is quasi-free, so is its localization. Assume from now on that $\mathcal{A}$ is quasi-free, define the space of de Rham differential forms
\begin{equation}
\mathit{DR}_{\Bbbk}(\mathcal{A}):=T_\mathcal{A}(\Omega_\mathcal{A}^1)/[T_\mathcal{A}(\Omega_\mathcal{A}^1),T_\mathcal{A}(\Omega_\mathcal{A}^1)],
\end{equation}
where $T_\mathcal{A}(\Omega_\mathcal{A}^1)$ is the tensor algebra of $\Omega_\mathcal{A}^1$ over $\mathcal{A}$, so in particular $\mathit{DR}_{\Bbbk}(\mathcal{A})$ carries the structure of a dg algebra. Denote by $\mathit{Der}_\Bbbk(\mathcal{A},\mathcal{A})$ the dg vector space of $\Bbbk$-linear (super)derivations on $\mathcal{A}$. For a closed 2-form $\omega\in\mathit{DR}_{\Bbbk}^2(\mathcal{A})$, we have a map
\begin{equation}\label{eq:contra}
i_\omega:\mathit{Der}_\Bbbk(\mathcal{A},\mathcal{A})\rightarrow\mathit{DR}_\Bbbk^1(\mathcal{A})
\end{equation}
defined by contracting $\omega$ with every derivation in $\mathit{Der}_\Bbbk(\mathcal{A},\mathcal{A})$. We say that $\omega$ is \textit{symplectic} if $i_\omega$ is an isomorphism. Let $\mathbb{D}\mathit{er}_\Bbbk(\mathcal{A}):=\mathit{Der}_\Bbbk(\mathcal{A},\mathcal{A}\otimes\mathcal{A})=(\Omega_\mathcal{A}^1)^\vee$ be the bimodule of \textit{double derivations} of $\mathcal{A}$, analogous to (\ref{eq:contra}) we have a map
\begin{equation}
\iota_\omega:\mathbb{D}\mathit{er}_\Bbbk(\mathcal{A})\rightarrow\Omega_\mathcal{A}^1.
\end{equation}
A symplectic form $\omega$ is \textit{bisymplectic} if $\iota_\omega$ is an isomorphism.

Given any $a\in\mathcal{A}$ and let $\omega\in\mathit{DR}_{\Bbbk}^2(\mathcal{A})$ be bisymplectic, consider the double derivation $H_a\in\mathbb{D}\mathit{er}_\Bbbk(\mathcal{A})$ defined by $\iota_{H_a}\omega=Da:=a\otimes1-1\otimes a$. Using $H_a$ one can define a bracket $\{\cdot,\cdot\}$ on $\mathcal{A}$ by
\begin{equation}
\{a_1,a_2\}=\circ(H_{a_1}(a_2)),
\end{equation}
where $\circ:\mathcal{A}\otimes\mathcal{A}\rightarrow\mathcal{A}$ is the multiplication on $\mathcal{A}$. It can be checked that $\{\cdot,\cdot\}$ descends to a Lie bracket on $\mathcal{A}/[\mathcal{A},\mathcal{A}]$, and it defines an action of $\mathcal{A}/[\mathcal{A},\mathcal{A}]$ on $\mathcal{A}$ by derivations. See \cite{mv1}, Appendix A.

On the other hand, given $\theta\in\mathit{Der}_\Bbbk(\mathcal{A},\mathcal{A})$, we can define the Lie derivative $L_\theta$ on $T_\mathcal{A}(\Omega_\mathcal{A}^1)$ by
\begin{equation}
L_\theta(a)=\theta(a), L_\theta(Da)=D(\theta(a))
\end{equation}
for any $a\in\mathcal{A}$ and $Da\in\Omega_\mathcal{A}^1$. Clearly, $L_\theta$ descends to a map on the quotient $\mathit{DR}_{\Bbbk}(\mathcal{A})$.

Consider the triple $(\mathcal{A},\omega,\theta)$, where $\mathcal{A}=T_\mathcal{B}(\mathcal{M})$ is connected (\cite{bd}, Definition 4.3.2) and non-positively graded, with $\mathcal{B}$ being a quasi-free associative algebra concentrated in degree 0, and $\mathcal{M}$ a $\mathcal{B}$-bimodule. The 2-form $\omega\in\mathit{DR}_{\Bbbk}^2(\mathcal{A})$ is a bisymplectic form which is homogeneous with respect to the grading induced from $\mathcal{A}$, and $\theta\in\mathit{Der}_\Bbbk(\mathcal{A},\mathcal{A})$ has cohomological degree 1, which satisfies $L_\theta\omega=0$ and $\theta^2=0$. One can associate to this data a dg algebra $\mathcal{G}(\omega,\theta)$, called \textit{Ginzburg dg algebra}, whose underlying graded algebra is given by the free product $\mathcal{A}\ast\Bbbk[t]$, where $|t|=|\omega|-1$. The differential $d$ on $\mathcal{G}(\omega,\theta)$ is defined by considering the non-commutative moment map
\begin{equation}
\mu_\mathit{nc}:\mathit{DR}_{\Bbbk}^2(\mathcal{A})_\mathit{cl}\rightarrow\overline{\mathcal{A}}:=\mathcal{A}/\Bbbk
\end{equation}
on the space of closed cyclic 2-forms $\mathit{DR}_{\Bbbk}^2(\mathcal{A})_\mathit{cl}\subset\mathit{DR}_{\Bbbk}^2(\mathcal{A})$, which satisfies $D(\mu_\mathit{nc}(\omega))=\iota_\Delta\omega$, where $\Delta\in\mathbb{D}\mathit{er}(\mathcal{A})$ is the double derivation
\begin{equation}
\Delta(a)=a\otimes1-1\otimes a.
\end{equation}
Under the assumption that $\mathcal{A}$ is connected, $\mu_\mathit{nc}$ lifts to a map
\begin{equation}
\tilde{\mu}_\mathit{nc}:\mathit{DR}_{\Bbbk}^2(\mathcal{A})_\mathit{cl}\rightarrow[\mathcal{A},\mathcal{A}].
\end{equation}
With the above notations, we define
\begin{equation}
da=\theta(a),dt=\tilde{\mu}_\mathit{nc}(\omega).
\end{equation}
Note that our assumptions on the derivation $\theta$ ensures that $d^2=0$, so the dg algebra $\mathcal{G}(\omega,\theta)$ is well-defined.

\begin{definition}
The Ginzburg dg algebra $\mathcal{G}(\omega,\theta)$ defined above is a superpotential algebra if $\theta=\{w,\cdot\}$ for some $w\in\mathcal{A}/[\mathcal{A},\mathcal{A}]$. $w$ called the superpotential.
\end{definition}

If one takes $\mathcal{A}$ above to be the path algebra $\mathbb{K}\widetilde{Q}$ of the double $\widetilde{Q}$ of some quiver $Q=(Q_0,Q_1)$ obtained by adding a reverse $a^\ast$ to all the arrows $a\in Q_1$ (if $a$ is a cycle of odd degree, then $a^\ast=a$), the above construction recovers the Ginzburg dg algebra (or the \textit{dg preprojective algebra}, in the terminology of $\cite{mv}$) $\mathcal{G}(Q,w)$ associated to the quiver with potential $(Q,w)$, which can be defined in a more concrete and elementary way, see $\cite{mv}$, Section 9.3.

However, there are also many superpotential algebras which are not of the form $\mathcal{G}(Q,w)$, a typical example is the fundamental group algebra $\mathbb{K}\left[\pi_1(T^3)\right]$, see $\cite{bd}$, Example 4.3.4. Another counterexample is the following:

\begin{example}\label{example:110}
Consider the associative algebra $\mathbb{K}[x,y][(xy-1)^{-1}]$, regarded as a trivially graded dg algebra with vanishing differential. One can show that this algebra has a superpotential description. In fact, consider the path algebra $\mathbb{K}\widetilde{Q}_{1,0,0}$ of the double of the quiver $Q_{1,0,0}$ which consists of a single vertex and a unique cycle $x$, see Figure \ref{fig:quiver_{1,0,0}}. Let
\begin{equation}
\mathcal{A}=\mathbb{K}\langle x,y\rangle[(xy-1)^{-1}]
\end{equation}
be the localization of $\mathbb{K}\widetilde{Q}_{1,0,0}$ at $xy-1$, where the generator $y$ represents the reverse of the cycle $x$ in the quiver $\widetilde{Q}_{1,0,0}$. Since $\mathbb{K}\widetilde{Q}_{1,0,0}$ is quasi-free, so is $\mathcal{A}$. In this way, we have identified $\mathbb{K}[x,y][(xy-1)^{-1}]$ with the dg algebra $\left(\mathcal{A}\ast\mathbb{K}[t],d\right)$, so that $dx=dy=0$ and $dt=xy-yx$. The superpotential vanishes for dimension reasons.

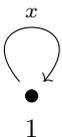
\begin{figure}[h!]
	\centering
	\begin{tikzpicture}
	\foreach \ang\lab\anch in {90/1/north}{
		\draw[fill=black] ($(0,0)+(\ang:3)$) circle (.08);
		\node[anchor=\anch] at ($(0,0)+(\ang:2.8)$) {$\lab$};
		\draw[->,shorten <=7pt, shorten >=7pt] ($(0,0)+(\ang:3)$).. controls +(\ang+40:1.5) and +(\ang-40:1.5) .. ($(0,0)+(\ang:3)$);
		\node at (0,4.1) {\footnotesize $x$};
	}
	\end{tikzpicture}
	\caption{The quiver $Q_{1,0,0}$}
	\label{fig:quiver_{1,0,0}}
\end{figure}
\end{example}

\begin{theorem}[$\cite{bd}$, Theorem 4.3.8]
Let $\mathcal{G}(\omega,\theta)$ be a Ginzburg dg algebra so that $\omega$ has cohomological degree $-n+2$, then $\mathcal{G}(\omega,\theta)$ has a smooth $n$-Calabi-Yau structure which is exact.
\end{theorem}

\subsection{Complexes with $S^1$-actions}\label{section:S1cpx}

This section follows closely $\cite{sg1}$, Section 2. Let $C_{-\ast}(S^1)$ be the dg algebra of chains on $S^1$, there is a quasi-equivalence of dg algebras $C_{-\ast}(S^1)\cong\mathbb{K}[t]/(t^2),|t|=-1$. Note that the algebra structure on $C_{-\ast}(S^1)$ comes from the group structure of $S^1$. The following definition is introduced in $\cite{bo2}$ and $\cite{sg1}$.

\begin{definition}[$\cite{sg1}$, Definitions 1 and 2]\label{definition:S1cpx}
An $S^1$-complex, or a chain complex with an $A_\infty$ $S^1$-action, is a strictly unital $A_\infty$-module $\mathcal{P}$ over $\mathbb{K}[t]/(t^2)$. Equivalently, it is a graded $\mathbb{K}$-vector space equipped with operations $\delta_j^\mathcal{P}:\mathcal{P}\rightarrow\mathcal{P}[1-2j]$, with $\delta_0^\mathcal{P}:=d^\mathcal{P}$ being the differential on $\mathcal{P}$,  and $\delta_j^\mathcal{P}:=\mu_\mathcal{P}^{j|1}(t,\cdots,t,\cdot),j>0$ being the structure maps, such that for each $k\geq0$, the equation
\begin{equation}
\sum_{j=0}^k\delta_j^\mathcal{P}\delta_{k-j}^\mathcal{P}=0
\end{equation}
holds. If the $A_\infty$-module $\mathcal{P}$ is a dg module, i.e. $\delta_j^\mathcal{P}$ for $j\geq2$, it is called a strict $S^1$-complex.
\end{definition}
We will abbreviate $\delta_j^\mathcal{P}$ to $\delta_j$ as long as there is no confusion.
\bigskip

$S^1$-complexes form a dg category $C_{-\ast}(S^1)^{\mathit{umod}}$, whose morphisms we will now recall. Let $A:=C_{-\ast}(S^1)$ be the quadratic algebra with a degree $-1$ generator, denote by $\varepsilon:A\rightarrow\mathbb{K}$ the trivial augmentation, and by $\overline{A}:=\ker(\varepsilon)$ the augmentation ideal. Let $\mathcal{P}$ and $\mathcal{Q}$ be strictly unital $A_\infty$-modules over $A$. A \textit{unital pre-morphism} of degree $k$ from $\mathcal{P}$ to $\mathcal{Q}$ is a collection of maps
\begin{equation}
F^{d|1}:\overline{A}^{\otimes d}\otimes\mathcal{P}\rightarrow\mathcal{Q}[k-d]
\end{equation}
for each $d\geq0$. Or equivalently, it can be expressed as a set of maps $\left\{F^d\right\}$, with
\begin{equation}
F^d:=F^{d|1}(t,\cdots,t,\cdot):\mathcal{P}\rightarrow\mathcal{Q}[k-2d].
\end{equation}
The space of pre-morphisms in each degree form the graded vector space of morphisms between the objects $\mathcal{P}$ and $\mathcal{Q}$ in the dg category $A^\mathit{umod}$, which will be denoted by $R\mathrm{Hom}_{S^1}(\mathcal{P},\mathcal{Q})$. Since $A^\mathit{umod}$ is a dg category, there is a differential $\partial$ on $R\mathrm{Hom}_{S^1}(\mathcal{P},\mathcal{Q})$, which is defined by
\begin{equation}
(\partial F)^s:=\sum_{i=0}^s F^i\circ\delta_{s-i}^\mathcal{P}-(-1)^{\deg(F)}\sum_{j=0}^s\delta_{s-j}^\mathcal{Q}\circ F^j.
\end{equation}
An $S^1$-\textit{complex homomorphism} is a pre-morphism which is closed under $\partial$. A homomorphism $F:\mathcal{P}\rightarrow\mathcal{Q}$ between $S^1$-complexes is a \textit{quasi-isomorphism} if the induced map $\left[F^0\right]:H^\ast(\mathcal{P})\rightarrow H^{\ast+\deg(F)}(\mathcal{Q})$ on cohomologies is an isomorphism.

Let $\mathcal{P}$ and $\mathcal{Q}$ be $S^1$-complexes, one can define their (derived) tensor product, which is another $S^1$-complex $\mathcal{Q}\otimes_{S^1}^\mathbb{L}\mathcal{P}$. To do this, note that we can view $\mathcal{Q}$ as a right $A_\infty$-module over $A$ since $A$ is commutative. The chain complex $\mathcal{Q}\otimes_{S^1}^\mathbb{L}\mathcal{P}$ has underlying vector space
\begin{equation}
\bigoplus_{d\geq0}\mathcal{Q}\otimes\overline{A}[1]^{\otimes d}\otimes\mathcal{P},
\end{equation}
and the differential acts as
\begin{equation}
\partial(q\otimes\underbrace{t\otimes\cdots\otimes t}_d\otimes p)=\sum_{i=0}^d\left((-1)^{|m|}\delta_i^\mathcal{Q}(q)\otimes\underbrace{t\otimes\cdots\otimes t}_{d-i}\otimes p+q\otimes \underbrace{t\otimes\cdots\otimes t}_{d-i}\otimes\delta_i^\mathcal{P}(p)\right).
\end{equation}
The tensor product $\mathcal{Q}\otimes_{S^1}^\mathbb{L}\mathcal{P}$ is functorial in the sense that if $F=\left\{F^d\right\}:\mathcal{P}_0\rightarrow\mathcal{P}_1$ is a pre-morphism of $S^1$-complexes, then there are induced maps
\begin{equation}
F_\#:\mathcal{Q}\otimes_{S^1}^\mathbb{L}\mathcal{P}_0\rightarrow\mathcal{Q}\otimes_{S^1}^\mathbb{L}\mathcal{P}_1,\textrm{ }{}_\#F:\mathcal{P}_0\otimes_{S^1}^\mathbb{L}\mathcal{Q}\rightarrow\mathcal{P}_1\otimes_{S^1}^\mathbb{L}\mathcal{Q}
\end{equation}
given by
\begin{equation}
F_\#(q\otimes\underbrace{t\otimes\cdots\otimes t}_{d}\otimes p)=\sum_{j=0}^d q\otimes\underbrace{t\otimes\cdots\otimes t}_{d-j}\otimes F^j(p)
\end{equation}
and
\begin{equation}
{}_\#F(p\otimes\underbrace{t\otimes\cdots\otimes t}_d\otimes q)=\sum_{j=0}^d(-1)^{\deg(F)\cdot|q|}F^j(p)\otimes\underbrace{t\otimes\cdots\otimes t}_{d-j}\otimes q,
\end{equation}
which are chain maps if $F$ is closed.

\begin{proposition}[$\cite{sg1}$, Proposition 1]
Let $F:\mathcal{P}\rightarrow\mathcal{P}'$ be a quasi-isomorphism of $S^1$-complexes, then we have induced quasi-isomorphisms between $\hom$ spaces and tensor products, i.e. we have quasi-isomorphisms
\begin{equation}
F\circ:R\mathrm{Hom}_{S^1}(\mathcal{P}',\mathcal{Q})\cong R\mathrm{Hom}_{S^1}(\mathcal{P},\mathcal{Q}),\textrm{ }\circ F:R\mathrm{Hom}_{S^1}(\mathcal{Q},\mathcal{P})\cong R\mathrm{Hom}_{S^1}(\mathcal{Q},\mathcal{P}'),
\end{equation}
\begin{equation}
{}_\#F:\mathcal{Q}\otimes_{S^1}^\mathbb{L}\mathcal{P}\cong\mathcal{Q}\otimes_{S^1}^\mathbb{L}\mathcal{P}',\textrm{ }F_\#:\mathcal{P}\otimes_{S^1}^\mathbb{L}\mathcal{Q}\cong\mathcal{P}'\otimes_{S^1}^\mathbb{L}\mathcal{Q}.
\end{equation}
\end{proposition}

Let $\mathcal{P}$ and $\mathcal{Q}$ be two $S^1$-complexes, and $f:\mathcal{P}\rightarrow\mathcal{Q}$ a chain map. We say that a homomorphism $F=\left\{F^d\right\}$ from $\mathcal{P}$ to $\mathcal{Q}$ an $S^1$-\textit{equivariant enhancement} of $f$ if $\left[F^0\right]=[f]$. In particular, $F$ has degree $\deg(f)$.

Finally, we notice that if $\mathcal{P}$ and $\mathcal{Q}$ are $S^1$-complexes, then their (linear) tensor product $\mathcal{P}\otimes\mathcal{Q}$ can also be equipped with an $S^1$-complex structure with
\begin{equation}\label{eq:diagonal}
\delta_k^{\mathcal{P}\otimes\mathcal{Q}}(p\otimes q):=(-1)^{|q|}\delta_k^\mathcal{P}(p)\otimes q+p\otimes\delta_k^\mathcal{Q}(q).
\end{equation}
We call this $S^1$-action on $\mathcal{P}\otimes\mathcal{Q}$ the \textit{diagonal $S^1$-action}.

\begin{definition}[$\cite{sg1}$, Definitions 4 and 5]
The homotopy orbit complex of $\mathcal{P}$ is the derived tensor product
\begin{equation} 
\mathcal{P}_{hS^1}:=\mathbb{K}\otimes_{S^1}^\mathbb{L}\mathcal{P}.
\end{equation}
The homotopy fixed point complex of $\mathcal{P}$ is the chain complex of morphisms
\begin{equation}
\mathcal{P}^{hS^1}:=R\mathrm{Hom}_{S^1}(\mathbb{K},\mathcal{P}).
\end{equation}
Since $R\mathrm{Hom}_{S^1}(\mathbb{K},\mathbb{K})\cong\mathbb{K}[u]$, $\mathcal{P}^{hS^1}$ carries the structure of a $\mathbb{K}[u]$-module. The Tate complex of $\mathcal{P}$ is defined as the localization of $\mathcal{P}^{hS^1}$ at $u$
\begin{equation}
\mathcal{P}^\mathit{Tate}:=\mathcal{P}^{hS^1}\otimes_{\mathbb{K}[u]}\mathbb{K}[u,u^{-1}],
\end{equation}
where $u$ is a formal variable with $|u|=2$.
\end{definition}

Let $F:\mathcal{P}\rightarrow\mathcal{Q}$ be a homomorphism of $S^1$-complexes, it induces chain maps
\begin{equation}\label{eq:induced}
F^{hS^1}:\mathcal{P}^{hS^1}\rightarrow\mathcal{Q}^{hS^1},\textrm{ }F_{hS^1}:\mathcal{P}_{hS^1}\rightarrow\mathcal{Q}_{hS^1},\textrm{ }F^\mathit{Tate}:\mathcal{P}^\mathit{Tate}\rightarrow\mathcal{Q}^\mathit{Tate}.
\end{equation}
When $F$ is a quasi-isomorphism (of $S^1$-complexes), $F^{hS^1}$, $F_{hS^1}$ and $F^\mathit{Tate}$ are quasi-isomorphisms (of chain complexes).

\begin{proposition}[$\cite{sg1}$, Proposition 2]\label{proposition:intert}
If $F:\mathcal{P}\rightarrow\mathcal{Q}$ is a homomorphism between $S^1$-complexes, then the various induced maps in (\ref{eq:induced}) intertwine all of the long exact sequences for equivariant homology groups of $\mathcal{P}$ with those for $\mathcal{Q}$. 
\end{proposition}

For instance, we have the Gysin exact triangle (\cite{sg1}, Remark 23)
\begin{equation}
\mathcal{P}\rightarrow\mathcal{P}_{hS^1}\rightarrow\mathcal{P}_{hS^1}[2]\xrightarrow{[1]}.
\end{equation}
Taking $\mathcal{P}=\mathit{CH}_\ast(\mathcal{A})$ to be the Hochschild chain complex of some strictly unital $A_\infty$-algebra $\mathcal{A}$ recovers Connes' long exact sequence (\ref{eq:Connes-LES}). If $F:\mathcal{P}\rightarrow\mathcal{Q}$ is an $S^1$-complex homomorphism, there is a commutative diagram
\begin{equation}
\begin{tikzcd}
\mathcal{P} \arrow[d] \arrow[r] & \mathcal{P}_{hS^1} \arrow[d] \arrow[r] & \mathcal{P}_{hS^1}[2] \arrow[d] \arrow[r,"{[1]}"] & \textrm{} \\
\mathcal{Q} \arrow[r] & \mathcal{Q}_{hS^1} \arrow[r] &  \mathcal{Q}_{hS^1}[2] \arrow[r,,"{[1]}"] & \textrm{}
\end{tikzcd}
\end{equation}

\bigskip

The structure of an $S^1$-complex $\left(\mathcal{P},\{\delta_j\}_{j\geq0}\right)$ admits an alternative description by implementing the $u$-linear model. Let $u$ be a formal variable of degree 2, consider the $u$-adically completed tensor product
\begin{equation}\label{eq:positive_eq}
\mathcal{P}[[u]]:=\mathcal{P}\widehat{\otimes}_\mathbb{K}\mathbb{K}[[u]]
\end{equation}
in the category of graded vector spaces. An $S^1$-complex can be equivalently formulated as a graded $\mathbb{K}$-vector space $\mathcal{P}$ equipped with a map $\delta_\mathit{eq}:\mathcal{P}\rightarrow\mathcal{P}[[u]]$ of total degree 1 defined by
\begin{equation}\label{eq:Tate}
\delta_\mathit{eq}:=\sum_{j=0}^\infty\delta_iu^i,
\end{equation}
which satisfies $\delta_\mathit{eq}^2=0$. The map $\delta_\mathit{eq}$ is called an \textit{equivariant differential}, note that it extends $u$-linearly to a map $\mathcal{P}[[u]]\rightarrow\mathcal{P}[[u]]$, which we will still denote by $\delta_\mathit{eq}$.

The \textit{(positive) $S^1$-equivariant homology}, \textit{negative $S^1$-equivariant homology} and \textit{periodic $S^1$-equivariant homology} are defined respectively as homologies of the following complexes:
\begin{equation}\label{eq:po-S1}
\mathcal{P}_{hS^1}:=\left(\mathcal{P}((u))/u\mathcal{P}[[u]],\delta_\mathit{eq}\right),
\end{equation}
\begin{equation}
\mathcal{P}^{hS^1}:=\left(\mathcal{P}[[u]],\delta_\mathit{eq}\right),
\end{equation}
\begin{equation}
\mathcal{P}^\mathit{Tate}:=\left(\mathcal{P}((u)),\delta_\mathit{eq}\right).
\end{equation}

\subsection{Non-unital Hochschild chain complex}\label{section:non-unital}

Our exposition here follows $\cite{sg1}$, Section 3.1. Let $\mathcal{A}$ be an $A_\infty$-algebra over the semisimple ring $\Bbbk:=\bigoplus_{i=1}^r\mathbb{K}e_i$. For the usual Hochschild chain complex $\mathit{CH}_\ast(\mathcal{A})$ to be a strict $S^1$-complex in the sense of Definition \ref{definition:S1cpx}, one needs to assume that $\mathcal{A}$ is strictly unital. However, in the geometric context, the Fukaya category is in general only cohomologically unital, we therefore need a replacement of the usual Hochschild chain complex, so that it possess the structure of a strict $S^1$-complex and is quasi-isomorphic to the usual Hochschild complex $\mathit{CH}_\ast(\mathcal{A})$. This construction, known as the \textit{non-unital Hochschild chain complex}, will be recalled below.
\bigskip

Let $\mathcal{A}$ be a cohomologically unital $A_\infty$-algebra over $\Bbbk$. As a graded vector space, the non-unital Hochschild chain complex consists of two copies of the ordinary Hochschild chain complex, with the grading of the second copy shifted down by 1, i.e.
\begin{equation}\label{eq:CHnu}
\mathit{CH}^\mathit{nu}_\ast(\mathcal{A}):=\mathit{CH}_\ast(\mathcal{A})\oplus\mathit{CH}_\ast(\mathcal{A})[1].
\end{equation}
With respect to the decomposition (\ref{eq:CHnu}), elements in the complex $\mathit{CH}^\mathit{nu}_\ast(\mathcal{A})$ can be written as $\check{\alpha}+\hat{\beta}$, where $\check{\alpha}\in\mathit{CH}_\ast(\mathcal{A})$ and $\hat{\beta}\in\mathit{CH}_\ast(\mathcal{A})[1]$. As a convention, we will refer to the left factor in $\mathit{CH}^\mathit{nu}_\ast(\mathcal{A})$ as the \textit{check factor} and the right factor in $\mathit{CH}^\mathit{nu}_\ast(\mathcal{A})$ as the \textit{hat factor}.

The differential $b^\mathit{nu}$ on the complex $\mathit{CH}^\mathit{nu}_\ast(\mathcal{A})$ can therefore be expressed as a block matrix
\begin{equation}
b^\mathit{nu}:=\left[
\begin{array}{ll}
b & d_{\wedge\vee} \\ 0 & b'
\end{array}
\right]
\end{equation}
where 
\begin{equation}
\begin{split}
b(\check{\alpha})=\sum&(-1)^{\maltese_1^k\cdot(1+\maltese_{k+1}^d)+\maltese_{k+1}^{d-1}+1}\mu^{d-i}(x_k\otimes\cdots\otimes x_1\otimes x_d\otimes x_{d-1}\otimes\cdots\otimes x_{k+i+1})\otimes \\
&x_{k+i}\otimes\cdots\otimes x_{k+1}+\sigma(-1)^{\maltese_1^s}x_d\otimes\cdots\otimes\mu^j(x_{s+j+1}\otimes\cdots\otimes x_{s+1})\otimes x_s\otimes\cdots\otimes x_1
\end{split}
\end{equation}
is the usual Hochschild differential on the check factor $\mathit{CH}_\ast(\mathcal{A})$, with $\check{\alpha}=x_d\otimes\cdots\otimes x_1$, $b'$ is the \textit{bar differential} on the hat factor defined by
\begin{equation}
\begin{split}
b'(\hat{\beta})&=\sum(-1)^{\maltese_1^s}x_d\otimes\cdots\otimes x_{s+j+1}\otimes\mu_\mathcal{A}^j(x_{s+j}\otimes\cdots x_{s+1})\otimes x_s\otimes\cdots\otimes x_1 \\
&+\sum(-1)^{\maltese_1^{d-j}}\mu_\mathcal{A}^j(x_d\otimes\cdots\otimes x_{d-j+1})\otimes x_{d-j}\otimes\cdots\otimes x_1
\end{split}
\end{equation}
where $\hat{\beta}=x_d\otimes\cdots\otimes x_1$, and $d_{\wedge\vee}:\mathit{CH}_\ast(\mathcal{A})[1]\rightarrow\mathit{CH}_\ast(\mathcal{A})[1]$ is defined by
\begin{equation}
d_{\wedge\vee}(\hat{\beta}):=(-1)^{\maltese_2^d+||x_1||\cdot\maltese_2^d+1}x_1\otimes x_d\otimes\cdots\otimes x_2+(-1)^{\maltese_1^{d-1}}x_d\otimes\cdots\otimes x_1.
\end{equation}
In the above, we have followed the convention of $\cite{ps1}$, in particular the symbol
\begin{equation}
\maltese_i^j:=\sum_{k=i}^j||x_k||
\end{equation}
is used to abbreviate the signs, where $||x_k||:=|x_k|-1$ is the reduced grading.

\begin{remark}
The idea of non-unital Hochschild complex also appears in the geometric context, for example, in the Legendrian surgery description of the symplectic cohomology $\cite{bee}$.
\end{remark}

The natural inclusion $\mathit{CH}_\ast(\mathcal{A})\hookrightarrow\mathit{CH}_\ast^\mathit{nu}(\mathcal{A})$ is a quasi-isomorphism, since the quotient complex is acyclic, see \cite{sg1}, Lemma 2. Associated to $\mathit{CH}^\mathit{nu}_\ast(\mathcal{A})$ there is a corresponding non-unital version of the (chain level representative of) Connes' operator $B^\mathit{nu}:\mathit{CH}_\ast^\mathit{nu}(\mathcal{A})\rightarrow\mathit{CH}_{\ast}^\mathit{nu}(\mathcal{A})[1]$, which is defined explicitly by
\begin{equation}\label{eq:Bnu}
B^\mathit{nu}(x_k\otimes\cdots\otimes x_1,y_l\otimes\cdots\otimes y_1):=\sum_i(-1)^{\maltese_1^i\maltese_{i+1}^k+||x_k||+\maltese_1^k+1}(0,x_i\otimes\cdots\otimes x_1\otimes x_k\otimes\cdots\otimes x_{i+1}).
\end{equation}
Note that we can write $B^\mathit{nu}=s^\mathit{nu}N$, where
\begin{equation}
s^\mathit{nu}(x_k\otimes\cdots\otimes x_1,y_l\otimes\cdots\otimes y_1):=(-1)^{\maltese_1^k+||x_k||+1}(0,x_k\otimes\cdots\otimes x_1),
\end{equation}
and
\begin{equation}
N(x_k\otimes\cdots\otimes x_1):=(1+\lambda+\cdots+\lambda^{k-1})(x_k\otimes\cdots\otimes x_1)
\end{equation}
is the \textit{norm} of the \textit{cyclic permutation operator}
\begin{equation}\label{eq:permu}
\lambda(x_k\otimes\cdots\otimes x_1):=(-1)^{||x_1||\cdot\maltese_2^k+||x_1||+||x_k||}x_1\otimes x_k\otimes\cdots\otimes x_2.
\end{equation}
One can verify that $(B^\mathit{nu})^2=0$ and $b^\mathit{nu}B^\mathit{nu}+B^\mathit{nu}b^\mathit{nu}=0$, which shows that
\begin{lemma}[$\cite{sg1}$, Lemma 3]
$\mathit{CH}_\ast^\mathit{nu}(\mathcal{A})$ is a strict $S^1$-complex.
\end{lemma}
When $\mathcal{A}$ is strictly unital, there is an $S^1$-equivariant enhancement of the natural inclusion $\mathit{CH}_\ast(\mathcal{A})\hookrightarrow\mathit{CH}_\ast^\mathit{nu}(\mathcal{A})$, which is a quasi-isomorphism of $S^1$-complexes. Again, one can package everything in the $u$-linear model, and define the equivariant differential on $\mathit{CH}_\ast^\mathit{nu}(\mathcal{A})$ as
\begin{equation}
b_\mathit{eq}:=b^\mathit{nu}+uB^\mathit{nu}.
\end{equation}
The positive, negative and periodic cyclic homologies of a cohomoloically unital $A_\infty$-algebra $\mathcal{A}$ are then defined respectively as the homologies of the following complexes:
\begin{equation}
\mathit{CC}_\ast(\mathcal{A}):=\left(\mathit{CH}_\ast^\mathit{nu}(\mathcal{A})\otimes_\mathbb{K}\mathbb{K}((u))/u\mathbb{K}[[u]],b_\mathit{eq}\right),
\end{equation} 
\begin{equation}
\mathit{CC}_\ast^-(\mathcal{A}):=\left(\mathit{CH}_\ast^\mathit{nu}(\mathcal{A})\otimes_\mathbb{K}\mathbb{K}[[u]],b_\mathit{eq}\right),
\end{equation}
\begin{equation}
\mathit{CC}_\ast^\mathit{per}:=\left(\mathit{CH}_\ast^\mathit{nu}(\mathcal{A})\otimes_\mathbb{K}\mathbb{K}((u)),b_\mathit{eq}\right).
\end{equation}

\section{Parametrized Floer theory}\label{section:PF}

As the general geometric set up of this paper, $(M,\theta_M)$ will be a $2n$-dimensional Liouville manifold. In order to have $\mathbb{Z}$-gradings on various Floer cochain complexes, we assume that $c_1(M)=0$, and will in fact fix the choice of a trivialization of the canonical bundle $K_M$.

\subsection{Equivariant symplectic cohomology}\label{section:equi-symp}

We recall the definition of the $S^1$-equivariant symplectic cohomology $\mathit{SH}_{S^1}^\ast(M)$, whose construction is sketched in $\cite{ps2}$ and later carried out in detail in $\cite{bo2}$. We will actually follow closely the general framework of $\cite{sg1}$, which has the advantage of being coordinate-free. The idea behind the construction will also be used in Appendix \ref{section:product} to define higher order operations involving the pair-of-pants surface.

The construction of $\mathit{SH}_{S^1}^\ast(M)$ is more involved than its non-equivariant version $\mathit{SH}^\ast(M)$ in the sense that besides the Floer differential $\delta_0=d$, there is now a sequence of higher order corrections $\delta_k$, $k\geq1$, whose definitions make use of parametrized moduli spaces. We now recall the definition of the domains of these parametrized maps.

\begin{definition}[$\cite{sg1}$, Definition 7]\label{definition:p-cyl}
A $k$-point angle-decorated cylinder consists of a cylinder $Z=\mathbb{R}\times S^1$, together with a collection of auxiliary marked points $p_1,\cdots,p_k\in Z$, such that their $s\in\mathbb{R}$ coordinates $(p_i)_s,1\leq i\leq k$ satisfy
\begin{equation}\label{eq:radial1}
(p_1)_s\leq\cdots\leq(p_k)_s.
\end{equation}
We call these coordinates the heights of the marked points, and denote them by
\begin{equation}
h_i:=(p_i)_s,i=1,\cdots,k.
\end{equation}
Similarly, the $t\in S^1$ coordinates $(p_i)_t,1\leq i\leq k$ of the marked points are called angles, and we introduce the notations
\begin{equation}
\theta_i:=(p_i)_t,i=1,\cdots,k.
\end{equation}
\end{definition}

Denote by $\mathcal{M}_k$ the moduli space of $k$-point angle-decorated cylinders, modulo translation in the $s$-direction. Given a marked cylinder $(Z,p_1,\cdots,p_k)$ representing an element of $\mathcal{M}_k$. For a fixed constant $\xi>0$, define positive and negative cylindrical ends
\begin{equation}\label{eq:cyl-ends}
\varepsilon^+:[0,\infty)\times S^1\rightarrow Z \textrm{ and } \varepsilon^-:(-\infty,0]\times S^1\rightarrow Z
\end{equation}
by
\begin{equation}\label{eq:cyl-ends1}
\varepsilon^+(s,t)=(s+h_k+\xi,t) \textrm{ and } \varepsilon^-(s,t)=(s-h_1+\xi,t+\theta_1)
\end{equation}
respectively. Note that the negative cylindrical end involves a twist by $\theta_1$. If we regard $Z$ as a sphere $S^2$ with two punctures $\zeta_{\mathit{in}}$, which is the input, and $\zeta_{\mathit{out}}$, which is the output, the parametrizations of $\varepsilon^+$ and $\varepsilon^-$ in the $t$-component can be equivalently described by putting \textit{asymptotic markers} $\ell_{\mathit{in}}$ and $\ell_\mathit{out}$ at $\zeta_{\mathit{in}}$ and $\zeta_{\mathit{out}}$ respectively. These are half-lines in the real projectivized tangent spaces $\mathbb{P}T_{\zeta_\mathit{in}}S^2$ and $\mathbb{P}T_{\zeta_\mathit{out}}S^2$. In the present situation, $\ell_{\mathit{in}}$ is pointing along the arc $\{\varepsilon^+(s,0)\}$, while $\ell_\mathit{out}$ is pointing towards $p_1$, the closest auxiliary marked point to $\zeta_{\mathit{out}}$. Note that there is a non-canonical isomorphism $\mathcal{M}_k\cong(S^1)^k\times[0,1)^{k-1}$.

$\mathcal{M}_k$ can be compactified to a manifold with corners $\overline{\mathcal{M}}_k$ by adding broken $k$-point angle-decorated cylinders, by which we mean
\begin{equation}
\bigsqcup_s\bigsqcup_{j_1,\cdots,j_s;j_i>0,\sum{j_i}=k}\mathcal{M}_{j_1}\times\cdots\times\mathcal{M}_{j_s}.
\end{equation}
In particular, the codimension 1 boundary of $\overline{\mathcal{M}}_k$ is covered by the images of the natural inclusions
\begin{equation}\label{eq:break1}
\overline{\mathcal{M}}_{k-j}\times\overline{\mathcal{M}}_j\hookrightarrow\partial\overline{\mathcal{M}}_k,0<j<k;
\end{equation}
\begin{equation}\label{eq:coin1}
\overline{\mathcal{M}}_k^{i,i+1}\hookrightarrow\partial\overline{\mathcal{M}}_k,
\end{equation}
where $\overline{\mathcal{M}}_k^{i,i+1}$ is the compactification of the locus where $h_i=h_{i+1}$. On $\mathcal{M}_k^{i,i+1}$ there is a forgetful map
\begin{equation}\label{eq:forget1}
\pi_i:\mathcal{M}_k^{i,i+1}\rightarrow\mathcal{M}_{k-1}
\end{equation}
which remembers only the first one of the angles of the interior marked points with coincident heights. Since $\pi_i$ is compatible with the choices of cylindrical ends $\varepsilon^\pm$ specified by (\ref{eq:cyl-ends1}), it extends to a map $\bar{\pi}_i:\overline{\mathcal{M}}_k^{i,i+1}\rightarrow\overline{\mathcal{M}}_{k-1}$ defined on the compactifications.
\bigskip

In order to write down the Floer equations for the parametrized maps, we need to introduce Floer data on the domain cylinders. We start by specifying the sets of Hamiltonian functions and almost complex structures to work with. We say that a time-dependent Hamiltonian $H_t:S^1\times M\rightarrow\mathbb{R}$ is \textit{admissible} if $H_t=H+F_t$ is the sum of an autonomous Hamiltonian $H:M\rightarrow\mathbb{R}$ which is quadratic at infinity, namely
\begin{equation}
H(r,y)=r^2
\end{equation}
on the cylindrical end $[r_0,\infty)\times\partial\overline{M}$, where $r\in\mathbb(1,\infty)$ is the radial coordinate and $r_0\gg1$, and a time-dependent perturbation $F_t:S^1\times M\rightarrow\mathbb{R}$. We require that on $M\setminus\overline{M}$, we have that for any $r_1\gg1$, there exists an $r>r_1$ such that $F_t$ vanishes in a neighborhood of the hypersurface $\{r\}\times\partial\overline{M}\subset M$. For instance, one can take $F_t$ to be a function supported near non-constant orbits of the Hamiltonian vector field $X_H$, where it is modelled on a Morse function on $S^1$.

Denote by $\mathcal{H}(M)$ the set of admissible Hamiltonians $H_t$ on $M$ such that all 1-periodic orbits of the Hamiltonian vector field $X_{H_t}$ are non-degenerate, and write $\mathcal{O}_M$ for the set of 1-periodic orbits of $X_{H_t}$. For an orbit $y\in\mathcal{O}_M$, we define its degree to be $\deg(y):=n-\mathit{CZ}(y)$, where $\mathit{CZ}(y)$ is the Conley-Zehnder index of $y$. With these data one can define a $\mathbb{Z}$-graded Floer cochain complex of $H_t$, which is called the \textit{symplectic cochain complex}. It will be denoted by $\mathit{SC}^\ast(M)$, with its degree $i$ piece given by
\begin{equation}
\mathit{SC}^i(M):=\bigoplus_{y\in\mathcal{O}_M,\deg(x)=i}|o_y|_\mathbb{K},
\end{equation}
where $|o_y|_\mathbb{K}$ is the $\mathbb{K}$-normalization of the orientation line $o_y$ defined via index theory.

\begin{remark}\label{remark:o-line}
For convenience, we will often choose generators of the orientation lines $o_y$ associated to each Hamiltonian orbit and denote them by $y$ in a slight abuse of notation. The same convention applies to orientation lines associated Hamiltonian chords when dealing with open-string invariants.
\end{remark}

Let $J_t$ be a time-dependent almost complex structure on $M$, we say that it is \textit{of weak contact type} on the conical end if there is a sequence $\{r_i\}$ with $\lim_{i\rightarrow\infty}r_i=\infty$ so that near $\{r_i\}\times\partial\overline{M}$ we have $dr\circ J_t=-\theta_M$. Denote by $\mathcal{J}(M)$ the set of $d\theta_M$-compatible alsmost complex structures on $M$ which are of weak contact type on the conical end. Recall that the usual Floer differential $d:\mathit{SC}^\ast(M)\rightarrow\mathit{SC}^{\ast+1}(M)$ is defined by counting rigid $J_t$-holomorphic cylinders $u:Z\rightarrow M$ with asymptotics at some Hamiltonian orbits $y^+,y^-\in\mathcal{O}_M$. This can be regarded as the special case $k=0$ of the operations $\delta_k$ on the complex $\mathit{SC}^\ast(M)$ defined below.

\begin{definition}[$\cite{sg1}$]\label{definition:data-cylinder}
A Floer datum for a $k$-point angle-decorated cylinder $(Z,p_1,\cdots,p_k)$ consists of the following:
\begin{itemize}
	\item choices of positive and negative cylindrical ends for $Z$, as in (\ref{eq:cyl-ends});
	\item a 1-form $\alpha_Z=dt$ on $Z$;
	\item a surface-dependent Hamiltonian function $H_Z:Z\rightarrow\mathcal{H}(M)$ which satisfies
	\begin{equation}
	(\varepsilon^\pm)^\ast H_Z=\lambda_\pm H_t
	\end{equation}
	for some $\lambda_->\lambda_+>0$, where $H_t\in\mathcal{H}(M)$ is some fixed choice of an admissible Hamiltonian;
	\item a surface-dependent almost complex structure $J_Z:Z\rightarrow\mathcal{J}(M)$ such that
	\begin{equation}
	(\varepsilon^\pm)^\ast J_Z=J_t
	\end{equation}
	for some fixed choice of $J_t\in\mathcal{J}(M)$.
\end{itemize}
\end{definition}

\begin{remark}
In the definition of Floer datum, we require the domain-dependent Hamiltonian $H_Z$ to be conformally equivalent, instead of equivalent on the two cylindrical ends. This is due to the fact that in order for the maximum principle to hold in the definition of the equivariant symplectic cohomology, we need an extra condition $\partial_sH_Z\leq0$. This fact was not taken into account in $\cite{ma1}$, Appendix B, as the energy computations there leaves out the derivatives of the Hamiltonian with respect to the domain.
\end{remark}

\textit{Universal and consistent} choices of Floer data over the compactified moduli spaces $\overline{\mathcal{M}}_k$ for all $k\geq1$ can be constructed in an inductive way. In our specific situation, this means that
\begin{itemize}
	\item At the boundary strata (\ref{eq:break1}), the choices of Floer data should coincide with the product of the Floer data chosen on lower dimensional moduli spaces. Moreover, the choices vary smoothly with respect to the gluing charts.
	\item At a boundary stratum of the form (\ref{eq:coin1}), the Floer datum chosen for a representative $(Z,p_1,\cdots,p_k)$ of an element of $\overline{\mathcal{M}}^{i,i+1}_k$ coincides with the one pulled back from the corresponding element of $\overline{\mathcal{M}}_{k-1}$ via the forgetful map $\bar{\pi}_i$.
\end{itemize}

\bigskip

Fix a universal and consistent choice of Floer data, whose existence is ensured by an inductive argument. For any pair of orbits $y^+,y^-\in\mathcal{O}_M$ and any integer $k\geq1$, we introduce the moduli space $\mathcal{M}_k(y^+;y^-)$ of pairs
\begin{equation}
\left((Z,p_1,\cdots,p_k),u\right),
\end{equation}
where $(Z,p_1,\cdots,p_k)\in\mathcal{M}_k$ and $u:Z\rightarrow M$ is a map which satisfies Floer's equation
\begin{equation}\label{eq:Floer}
\left(du-X_{H_Z}\otimes dt\right)^{0,1}=0,
\end{equation}
where the $(0,1)$-part is taken with respect to the domain-dependent almost complex structure $J_Z$ chosen above as part of our Floer data. Additionally, $u$ is required to satisfy the asymptotic conditions
\begin{equation}
\lim_{s\rightarrow\pm\infty}(\varepsilon^\pm)^\ast u(s,\cdot)=y^\pm.
\end{equation}
The boundary of the Gromov compactification $\overline{\mathcal{M}}_k(y^+;y^-)$ is covered by the images of the natural inclusions
\begin{equation}
\overline{\mathcal{M}}_j(y;y^-)\times\overline{\mathcal{M}}_{k-j}(y^+;y)\hookrightarrow\partial\overline{\mathcal{M}}_k(y^+;y^-),
\end{equation}
\begin{equation}
\overline{\mathcal{M}}_k^{i,i+1}(y^+;y^-)\hookrightarrow\partial\overline{\mathcal{M}}_k(y^+;y^-),
\end{equation}
which come from the boundary strata of $\overline{\mathcal{M}}_k$, along with the strata coming from the usual semi-stable strip breaking
\begin{equation}
\overline{\mathcal{M}}_k(y;y^-)\times\overline{\mathcal{M}}(y^+;y)\hookrightarrow\partial\overline{\mathcal{M}}_k(y^+;y^-),
\end{equation}
\begin{equation}
\overline{\mathcal{M}}(y;y^-)\times\overline{\mathcal{M}}_k(y^+;y)\hookrightarrow\partial\overline{\mathcal{M}}_k(y^+;y^-).
\end{equation}
For generic choices of Floer data, the moduli spaces $\overline{\mathcal{M}}_k(y^+;y^-)$ are compact manifolds-with-corners of dimension
\begin{equation}
\deg(y^+)-\deg(y^-)+2k-1.
\end{equation}
Let $((Z,p_1,\cdots,p_k),u)$ be a rigid element of $\overline{\mathcal{M}}_k(y^+;y^-)$, so we have $\deg(y^+)=\deg(y^-)-2k+1$. In this case, there are natural isomorphisms
\begin{equation}
\mu_u:o_{y^-}\rightarrow o_{y^+}
\end{equation}
between the orientation lines defined via index theory. One defines the operation
\begin{equation}\label{eq:del_k}
\delta_k:\mathit{SC}^\ast(M)\rightarrow\mathit{SC}^{\ast-2k+1}(M)
\end{equation}
by a signed count of rigid elements of the moduli spaces $\overline{\mathcal{M}}_k(y^+;y^-)$ for varying asymptotics $y^+$ and $y^-$. Our choices of Floer data ensures that the elements of $\overline{\mathcal{M}}_k^{i,i+1}(y^+,y^-)$ are never rigid, from which the identity
\begin{equation}
\sum_{i=0}^k\delta_i\delta_{k-i}=0
\end{equation}
follows, see $\cite{sg1}$, Lemma 9 for details. This shows that $\left(\mathit{SC}^\ast(M),\{\delta_k\}_{k\geq0}\right)$ is an $S^1$-complex (usually not strict) in the sense of Definition \ref{definition:S1cpx}.

As in (\ref{eq:positive_eq}), the $S^1$-\textit{equivariant symplectic cohomology} $\mathit{SH}_{S^1}^\ast(M)$ is defined to be the cohomology of the complex
\begin{equation}
\left(\mathit{SC}^\ast(M)\otimes_\mathbb{K}\mathbb{K}((u))/u\mathbb{K}[[u]],\delta_\mathit{eq}\right),
\end{equation}
where $\delta_\mathit{eq}:=\sum_{k=0}^\infty\delta_ku^k$ and $u$ is a formal variable of degree 2. Note that this is the $\mathcal{P}_{hS^1}$ defined in (\ref{eq:po-S1}).
\bigskip

One can replace the autonomous Hamiltonian $H$ in the above construction (which is quadratic on the conical end) with a Hamiltonian which is linear at infinity, i.e.
\begin{equation}\label{eq:linear-Ham}
H_\lambda(r,y)=\lambda r+C
\end{equation}
on $[r_0,\infty)\times\partial\overline{M}$, where $\lambda\in\mathbb{R}$ and $C$ is some constant, and consider its time-dependent perturbation $H_{\lambda,t}=H_\lambda+F_t$ as above, so that all the time-1 orbits of $X_{H_{\lambda,t}}$ are non-degenerate. The space of such Hamiltonians $H_{\lambda,t}:S^1\times M\rightarrow\mathbb{R}$ with varying $\lambda$ will be denoted by $\mathcal{H}_\ell(M)$, and we have
\begin{equation}
\mathcal{H}_\ell(M)=\bigcup_{\lambda\in\mathbb{R}}\mathcal{H}_\lambda(M),
\end{equation}
where $\mathcal{H}_\lambda(M)$ is the space of Hamiltonians of the form $H_{\lambda,t}$ with fixed slope $\lambda$ at infinity. We also require that $\lambda\notin\mathcal{P}_M$, where $\mathcal{P}_M\subset\mathbb{R}$ is the collection of those $\lambda$ such that the set of 1-periodic orbits $\mathcal{O}_{M,\lambda}$ of $X_{H_{\lambda,t}}$ is not contained in any compact subset of $M$. The same construction as above gives rise to the $S^1$-equivariant Floer cochain complex
\begin{equation}\label{eq:equi-cpx}
\mathit{CF}_{S^1}^\ast(M,\lambda):=\left(\mathit{CF}^\ast(M,\lambda)\otimes_\mathbb{K}\mathbb{K}((u))/u\mathbb{K}[[u]],\delta_\mathit{eq}\right),
\end{equation} 
where $\mathit{CF}^\ast(M,\lambda)$ is the Floer cochain complex of $H_{\lambda,t}$. When it is clear from the context which manifold we are working with, we shall simply write these Floer complexes as $\mathit{CF}_{S^1}^\ast(\lambda)$ and $\mathit{CF}^\ast(\lambda)$. The cohomology of (\ref{eq:equi-cpx}) will be denoted by $\mathit{HF}_{S^1}^\ast(\lambda)$. Just as in the non-equivariant case, for $\lambda_1<\lambda_2$ one can build equivariant continuation maps
\begin{equation}\label{eq:continuation}
\kappa_\mathit{eq}^{\lambda_1,\lambda_2}:=\sum_{k=0}^\infty\kappa_k^{\lambda_1,\lambda_2}u^k:\mathit{CF}_{S^1}^\ast(\lambda_1)\rightarrow\mathit{CF}_{S^1}^\ast(\lambda_2)
\end{equation}
via counting 1-parameter families of $k$-point angle decorated cylinders equipped with monotone homotopies connecting elements of $\mathcal{H}_{\lambda_1}(M)$ and $\mathcal{H}_{\lambda_2}(M)$. The maps $\left\{\kappa_k^{\lambda_1,\lambda_2}\right\}$ constitute an $S^1$-complex homomorphism between $\mathit{CF}^\ast(\lambda_1)$ and $\mathit{CF}^\ast(\lambda_2)$. Passing to direct limit yields an alternative definition of the equivariant symplectic cohomology
\begin{equation}\label{eq:limit-def}
\mathit{SH}_{S^1}^\ast(M):={\varinjlim}_\lambda\mathit{HF}^\ast_{S^1}(\lambda).
\end{equation}
This approach of defining $\mathit{SH}_{S^1}^\ast(M)$ is carried out in detail in $\cite{bo2}$. Note that for the purpose of defining $\mathit{SH}_{S^1}^\ast(M)$, one can assume that $\lambda>0$. However, the flexibility of allowing negative $\lambda$ will be needed for our later purposes, see Section \ref{section:disjoint}.
\bigskip

Return to our original set up. One can arrange so that $H_t\in\mathcal{H}(M)$ is a $C^2$-small and time-independent Morse function in the interior $M^\mathit{in}$ of $\overline{M}$, this produces a copy of the Morse complex $\mathit{CM}^\ast(H_t)$, which sits in $\mathit{SC}^\ast(M)$ as a subcomplex. When passing to cohomology, the inclusion $\mathit{CM}^\ast(H_t)\hookrightarrow\mathit{SC}^\ast(M)$ induces the classical PSS map
\begin{equation}\label{eq:PSS}
\mathit{PSS}:H^\ast(M;\mathbb{K})\rightarrow\mathit{SH}^\ast(M).
\end{equation}
Since the $S^1$-action on the symplectic cochain complex $\mathit{SC}^\ast(M)$ comes from reparametrizing the Hamiltonian orbits, one may expect that it becomes trivial when restricted to the subcomplex $\mathit{CM}^\ast(H_t)$, i.e. all the operations $\delta_k,k\geq1$ should vanish on it. This is indeed the case with appropriate choices of Floer data over the moduli spaces $\overline{\mathcal{M}}_k$. As a consequence, the natural inclusion $\mathit{CM}^\ast(H_t)\hookrightarrow\mathit{SC}^\ast(M)$ can be regarded as an $S^1$-complex homomorphism, which induces the $S^1$-\textit{equivariant PSS map}
\begin{equation}\label{eq:eq-PSS}
\widetilde{\mathit{PSS}}:H^\ast(M;\mathbb{K}((u))/u\mathbb{K}[[u]])\rightarrow\mathit{SH}_{S^1}^\ast(M)
\end{equation}
after passing to cohomologies.

As a variant, one can also use a $H_{\lambda,t}\in\mathcal{H}_\lambda(M)$ which is $C^2$-small and Morse in $M^\mathit{in}$, which will then give rise to maps
\begin{equation}\label{eq:PSS1}
\mathit{PSS}_\lambda:H^\ast(M;\mathbb{K})\rightarrow\mathit{HF}^\ast(\lambda),
\end{equation}
\begin{equation}\label{eq:eq-PSS1}
\widetilde{\mathit{PSS}}_\lambda:H^\ast(M;\mathbb{K}((u))/u\mathbb{K}[[u]])\rightarrow\mathit{HF}^\ast_{S^1}(\lambda).
\end{equation}
The maps (\ref{eq:PSS}) and (\ref{eq:eq-PSS}) can then be thought of as (\ref{eq:PSS1}) and (\ref{eq:eq-PSS1}) composed with the continuation maps and equivariant continuation maps respectively.

\subsection{Cyclic dilations}\label{section:Gysin}

We consider in this subsection the geometric counterpart of an exact Calabi-Yau on a homologically smooth $A_\infty$-category. The precise relationship between these two notions will be established in the next subsection, using Ganatra's construction of the cyclic open-closed string map.
\bigskip

Take any $\lambda\in[0,\infty]$ with $\lambda\notin\mathcal{P}_M$ if $\lambda<\infty$. Recall that the $S^1$-equivariant Hamiltonian Floer cohomology $\mathit{HF}_{S^1}^\ast(\lambda)$ fits into the following Gysin type long exact sequence:
\begin{equation}\label{eq:Gysin}
\cdots\rightarrow\mathit{HF}^{\ast-1}(\lambda)\xrightarrow{\mathbf{I}}\mathit{HF}^{\ast-1}_{S^1}(\lambda)\xrightarrow{\mathbf{S}}\mathit{HF}_{S^1}^{\ast+1}(\lambda)\xrightarrow{\mathbf{B}}\mathit{HF}^\ast(\lambda)\rightarrow\cdots,
\end{equation}
see $\cite{bo1,bo2,bo3}$ for a detailed discussion. The BV operator $\Delta=[\delta_1]:\mathit{HF}^\ast(\lambda)\rightarrow\mathit{HF}^{\ast-1}(\lambda)$ coincides on the cohomology level with the composition $\mathbf{B}\circ\mathbf{I}$ (in the non-trivial order). When $M=T^\ast Q$ is the cotangent bundle of a compact $\mathit{Spin}$ manifold $Q$, and $\lambda=\infty$, the above Gysin exact sequence (\ref{eq:Gysin}) reduces to the well-known long exact sequence in string topology $\cite{cs}$:
\begin{equation}\label{eq:string}
\cdots\rightarrow H_{\ast+1}(\mathcal{L}Q;\mathbb{K})\xrightarrow{\mathbf{I}}H_{\ast+1}^{S^1}(\mathcal{L}Q;\mathbb{K})\xrightarrow{\mathbf{S}}H^{S^1}_{\ast-1}(\mathcal{L}Q;\mathbb{K})\xrightarrow{\mathbf{B}}H_{\ast}(\mathcal{L}Q;\mathbb{K})\rightarrow\cdots,
\end{equation}
where the maps $\mathbf{I}$ and $\mathbf{B}$ are known as the \textit{erasing map} and the \textit{marking map} respectively. Because of this, the same terminology will be used to refer to $\mathbf{I}$ and $\mathbf{B}$ in the general case. 

Recall that in string topology, the map $\mathbf{B}$ is defined as follows: by definition, an equivariant homology class $\gamma\in H_{\ast-1}^{S^1}(\mathcal{L}Q;\mathbb{K})$ can be equivalently realized as a homology class of $H_{\ast-1}(S^\infty\times_{S^1}\mathcal{L}Q;\mathbb{K})$, whose lift in $H_\ast(S^\infty\times\mathcal{L}Q;\mathbb{K})$ will be denoted by $\tilde{\gamma}$. Let $\Pi$ be the trivial projection $S^\infty\times\mathcal{L}Q\rightarrow\mathcal{L}Q$, then $\mathbf{B}(\gamma)=\Pi_\ast(\tilde{\gamma})$.

Since the natural inclusion $\mathit{CM}^\ast(H_t)\hookrightarrow\mathit{SC}^\ast(M)$ is an $S^1$-complex homomorphism, by Proposition \ref{proposition:intert} the long exact sequence (\ref{eq:Gysin}) is compatible with the ordinary and equivariant PSS maps (cf. (\ref{eq:PSS1}) and (\ref{eq:eq-PSS1})), which leads to the following commutative diagram:
\begin{equation}\label{eq:Gysin-diagram}
\begin{tikzcd}
\cdots \arrow[r] &H^{\ast-1}(M;\mathbb{K}) \arrow[d, "\mathit{PSS}_\lambda"] \arrow[r] &H^{\ast-1}_{S^1}(M;\mathbb{K}) \arrow[d, "\widetilde{\mathit{PSS}}_\lambda"] \arrow[r] &H_{S^1}^{\ast+1}(M;\mathbb{K}) \arrow[d, "\widetilde{\mathit{PSS}}_\lambda"] \arrow[r] &H^\ast(M;\mathbb{K}) \arrow[r] \arrow[d, "\mathit{PSS}_\lambda"] &\cdots \\
\cdots \arrow[r] &\mathit{HF}^{\ast-1}(\lambda) \arrow[r, "\mathbf{I}"] &\mathit{HF}^{\ast-1}_{S^1}(\lambda) \arrow[r, "\mathbf{S}"] &\mathit{HF}_{S^1}^{\ast+1}(\lambda) \arrow[r, "\mathbf{B}"] &\mathit{HF}^\ast(\lambda) \arrow[r] &\cdots
\end{tikzcd}
\end{equation}
where the upper row is the usual Gysin long exact sequence for the trivial $S^1$-action on $M$, so in particular there is an isomorphism $H^\ast_{S^1}(M;\mathbb{K})\cong H^\ast(M;\mathbb{K}((u))/u\mathbb{K}[[u]])$.
\bigskip

We proceed to give a chain level interpretation of the condition (\ref{eq:cyclic}) appeared in the introduction. To do this, we need to find the chain level expression of the coboundary map $\mathbf{B}:\mathit{HF}_{S^1}^1(\lambda)\rightarrow\mathit{HF}^0(\lambda)$ in the Gysin sequence (\ref{eq:Gysin}). Write a degree 1 cocycle $\tilde{\beta}$ in the $S^1$-equivariant Hamiltonian Floer cochain complex $\mathit{CF}^\ast(\lambda)\otimes_\mathbb{K}\mathbb{K}((u))/u\mathbb{K}[[u]]$ as $\sum_{k=0}^\infty\beta_k\otimes u^{-k}$, where $\beta_k\in\mathit{CF}^{2k+1}(\lambda)$ and only finitely many terms in the infinite sum are non-zero. If we use $\mathbf{B}_c$ to denote the underlying chain level map of the marking map $\mathbf{B}$, standard diagram chasing argument enables us to find that (see $\cite{bo2}$, Proposition 2.9)
\begin{equation}\label{eq:Connesmap}
\mathbf{B}_c\left(\sum_{k=0}^\infty\beta_k\otimes u^{-k}\right)=\sum_{k=0}^\infty\delta_{k+1}(\beta_k).
\end{equation}

\begin{definition}\label{definition:cyclic-dilation}
A cyclic dilation is a cohomology class $\tilde{b}\in\mathit{SH}^1_{S^1}(M)$ whose representing cocycle $\tilde{\beta}\in\mathit{SC}^1_{S^1}(M)$ consists of a sequence of odd degree Floer cochains $\{\beta_k\}_{k\geq0}$ with $\beta_k\in\mathit{SC}^{2k+1}(M)$, and $\beta_k\neq0$ for only finitely many $k$, so that the cocycle $\sum_{k=0}^\infty\delta_{k+1}(\beta_k)$, after passing to cohomology, defines an invertible element $h\in\mathit{SH}^0(M)^\times$.
\end{definition}
By definition, there exists a $\lambda\in\mathbb{R}_{>0}\setminus\mathcal{P}_M$ so that $\tilde{\beta}$ lies in the image of the equivariant continuation map (cf. (\ref{eq:continuation}))
\begin{equation}
\kappa^{\lambda,\infty}_\mathit{eq}:\mathit{CF}^1_{S^1}(\lambda)\rightarrow\mathit{CF}^1_{S^1}(\infty):=\mathit{SC}^1_{S^1}(M).
\end{equation}
By slight abuse of notation, the preimage $\left(\kappa^{\lambda,\infty}_\mathit{eq}\right)^{-1}(\tilde{\beta})\in\mathit{CF}_{S^1}^1(\lambda)$ will still be denoted by $\tilde{\beta}$, and we will often refer to its cohomology class $\tilde{b}\in\mathit{HF}_{S^1}^1(\lambda)$ as a cyclic dilation.
\bigskip

Let $\lambda\gg0$ be sufficiently large. We consider an important special case of the above definition, namely when $h=1$ is the identity of $\mathit{SH}^0(M)$. It follows from the exactness of (\ref{eq:Gysin}) that there is a cohomology class $\tilde{b}\in\mathit{HF}^1_{S^1}(\lambda)$ satisfying $\mathbf{B}(\tilde{b})=1$ if and only if $1\in\mathit{HF}^0(\lambda)$ vanishes under the erasing map $\mathbf{I}:\mathit{HF}^0(\lambda)\rightarrow\mathit{HF}_{S^1}^0(\lambda)$.

In view of the commutative diagram (\ref{eq:Gysin-diagram}), this is precisely the case when the image of the (locally finite) fundamental class $1\in H^0(M;\mathbb{K})$ vanishes under the composition
\begin{equation}\label{eq:compo}
H^\ast(M;\mathbb{K})\hookrightarrow H^\ast(M;\mathbb{K}((u))/u\mathbb{K}[[u]])\xrightarrow{\widetilde{\mathit{PSS}}_\lambda}\mathit{HF}_{S^1}^\ast(\lambda).
\end{equation}
On the other hand, it follows from $\cite{jz}$, Lemma 4.2.4 that if $M$ admits a dilation in $\mathit{HF}^1(\lambda)$, i.e. a class $b\in\mathit{HF}^1(\lambda)$ which becomes a dilation in $\mathit{SH}^1(M)$ under the continuation map
\begin{equation}
\kappa^{\lambda,\infty}:\mathit{CF}^1(\lambda)\rightarrow\mathit{CF}^1(\infty):=\mathit{SC}^1(M),
\end{equation}
then $1\in H^0(M;\mathbb{K})$ lies in the kernel of (\ref{eq:compo}). This observation enables us to relate cyclic dilations to the following notion introduced by Zhao $\cite{jz}$.

\begin{definition}[$\cite{jz}$, Definition 4.2.1]
We say that a Liouville manifold $M$ admits a higher dilation if the identity $1\in H^\ast(M;\mathbb{K})$ lies in the kernel of the localized equivariant PSS map
\begin{equation}
\widehat{\mathit{PSS}}:H^\ast(M;\mathbb{K}((u)))\rightarrow\widehat{\mathit{PSH}}^\ast(M),
\end{equation}
where $\widehat{\mathit{PSH}}^\ast(M)$ is the completed periodic symplectic cohomology, which is the cohomology of the Tate complex (cf. (\ref{eq:Tate}))
\begin{equation}
\left(\mathit{SC}^\ast(M)\otimes_\mathbb{K}\mathbb{K}((u)),\delta_\mathit{eq}\right).
\end{equation}
\end{definition}

By $\cite{jz}$, Lemma 4.2.5, a higher dilation can be equivalently interpreted using the equivariant PSS map (\ref{eq:eq-PSS}). Precisely, $M$ admits a higher dilation if and only if $1\otimes u^{-k}\in H^\ast(M;\mathbb{K}((u))/u\mathbb{K}[[u]])$ lies in the kernel of the map $\widetilde{\mathit{PSS}}$ (cf. (\ref{eq:eq-PSS})) for every $k\geq0$. In view of the above discussions, we get the following:

\begin{proposition}\label{proposition:comparison}
If $M$ admits a higher dilation, then it admits a cyclic dilation with $h=1$.
\end{proposition}

\begin{remark}\label{remark:zig-zag}
In fact, it is an observation made in $\cite{ss}$, Remark 6.5 that the existence of a dilation in $\mathit{HF}^1(\lambda)$ is equivalent to the existence of a cocycle $\beta_0\in\mathit{CF}^1(\lambda)$ and a cochain $\beta_{-1}\in\mathit{CF}^{-1}(\lambda)$, so that $\delta_\mathit{eq}(\beta_{-1}+\beta_0\otimes u^{-1})=e$, where $e\in\mathit{CF}^0(\lambda)$ is the chain level representative of the identity. It is therefore natural to consider a sequence of cochains $\{\beta_j\}_{j\geq-1}$ with arbitrary length, where $\beta_j\in\mathit{CF}^{2j+1}(\lambda)$, so that
\begin{equation}\label{eq:zig-zag}
\delta_\mathit{eq}\left(\sum_{j=-1}^\infty\beta_j\otimes u^{-j-1}\right)=e,
\end{equation}
where only finitely many $\beta_j$ can be nonzero. Our discussions above shows that (\ref{eq:zig-zag}) holds for $\lambda\gg0$ if and only if $M$ admits a cyclic dilation with $h=1$. On the other hand, the notion of a higher dilation imposes the additional restriction that $e\otimes u^{-k}$ need to be coboundaries in the complexes $\mathit{CF}_{S^1}^{-2k}(\lambda)$. 
\end{remark}

This observation enables us to get some first examples of Liouville manifolds with cyclic dilations. For any closed manifold $Q$, one can consider the classifying map $f:Q\rightarrow B\pi_1(Q)$ for its universal cover. $Q$ is called \textit{rationally inessential} if the fundamental class $[Q]\in H_n(Q;\mathbb{Q})$ vanishes under the pushforward
\begin{equation}
f_\ast:H_n(Q;\mathbb{Q})\rightarrow H_n(B\pi_1(Q);\mathbb{Q}).
\end{equation}
In particular, every simply connected closed manifold is rationally inessential. It follows from $\cite{jz}$, Corollary 1.1.6 and Proposition \ref{proposition:comparison} stated above that for any rationally inessential manifold $Q$, $M=T^\ast Q$ admits a cyclic dilation over $\mathbb{Q}$. It is, however, not clear whether such a cotangent bundle admits a quasi-dilation. In fact, it is even unknown whether $T^\ast Q$ admits a dilation over $\mathbb{Q}$ for any simply connected \textit{formal} manifold $Q$, see $\cite{ps5}$, Lecture 18. More interesting examples of Liouville manifolds which admit cyclic dilations will be established in Section \ref{section:existence}.

\begin{remark}\label{remark:k-dilation}
Related notions are introduced in $\cite{zz2}$, where the author considers the spectral sequence associated to the $u$-adic filtration on the equivariant Floer cochain complex $\mathit{CF}_{S^1}^\ast(\lambda)$, and $M$ is said to admit a $k$-dilation if for sufficiently large $\lambda\notin\mathcal{P}_M$, the identity $e\in\mathit{CF}^0(\lambda)$ is killed in the $(k+1)$-th page of the spectral sequence. In particular, any flexible Weinstein manifold admits a 0-dilation, and a dilation in the sense of Seidel-Solomon $\cite{ss}$ is a 1-dilation. In general, having a $k$-dilation for $k\geq1$ is equivalent to requiring that $\sum_{j=0}^{k-1}\beta_j\otimes u^{-j}\in\mathit{CF}_{S^1}^1(\lambda)$ defines a cyclic dilation with $h=1$.
\end{remark}

\subsection{Cyclic open-closed map}\label{section:OC}

We briefly summarize the construction of the cyclic open-closed string map due to Ganatra $\cite{sg1}$. Details can be found in $\cite{sg1}$, Section 5. See also $\cite{gps1}$ for applications of the cyclic open-closed string map in the study of mirror symmetry of closed symplectic manifolds.
\bigskip

Roughly speaking, the cyclic open-closed map is a parametrized version of the usual open-closed map
\begin{equation}\label{eq:OC}
\mathit{OC}:\mathit{CH}_\ast(\mathcal{W}(M))\rightarrow\mathit{SC}^{\ast+n}(M)
\end{equation}
considered in $\cite{ma1}$ and $\cite{ps6}$, which keeps track of the $S^1$-complex structures on both sides.

However, as we have already noticed in Section \ref{section:non-unital}, in order to keep track of the $S^1$-action on the open-string side, one needs to consider the non-unital Hochschild complex $\mathit{CH}_\ast^\mathit{nu}(\mathcal{W}(M))$ instead of the usual Hochschild complex $\mathit{CH}_\ast(\mathcal{W}(M))$. Thus the first step towards the construction of an ``$S^1$-equivariant enhancement" of the usual open-closed map $\mathit{OC}$ would be to replace $\mathit{OC}$ by a map
\begin{equation}
\mathit{OC}^\mathit{nu}:\mathit{CH}_\ast^\mathit{nu}(\mathcal{W}(M))\rightarrow\mathit{SC}^{\ast+n}(M)
\end{equation}
defined on the non-unital Hochschild complex. Following Ganatra, we will call $\mathit{OC}^\mathit{nu}$ the \textit{non-unital open-closed string map}. In view of the definition of $\mathit{CH}_\ast^\mathit{nu}(\mathcal{W}(M))$ recalled in Section \ref{section:non-unital}, the map $\mathit{OC}^\mathit{nu}$ should consist of the check component $\widecheck{\mathit{OC}}:\mathit{CH}_\ast(\mathcal{W}(M))\rightarrow\mathit{SC}^{\ast+n}(M)$ and the hat component $\widehat{\mathit{OC}}:\mathit{CH}_\ast(\mathcal{W}(M))[1]\rightarrow\mathit{SC}^{\ast+n}(M)$, which act respectively on the check and hat factors of the non-unital Hochschild complex, and
\begin{equation}
\mathit{OC}^\mathit{nu}(\check{\alpha},\hat{\beta})=\widecheck{\mathit{OC}}(\check{\alpha})+\widehat{\mathit{OC}}(\hat{\beta}),
\end{equation}
where $\check{\alpha}\in\mathit{CH}_\ast(\mathcal{W}(M))$ and $\hat{\beta}\in\mathit{CH}_\ast(\mathcal{W}(M))[1]$. 

The map $\widecheck{\mathit{OC}}$ is defined on the ordinary Hochschild chain complex, and its definition in fact coincides with the ordinary open-closed string map (\ref{eq:OC}). Recall that $\mathit{OC}$ is defined by considering closed discs $\overline{S}$ equipped with boundary marked points $\zeta_1,\cdots,\zeta_d\in\partial\overline{S}$ which serve as inputs, and an interior marked point $\zeta_\mathit{out}$, which is an output. There is also an asymptotic marker $\ell_{\mathit{out}}$ at $\zeta_\mathit{out}$ pointing towards $\zeta_d$. One can assign Floer data to such discs $\left(\overline{S};\zeta_1,\cdots,\zeta_d;\zeta_\mathit{out},\ell_{\mathit{out}}\right)$ in the usual way, and when forming the moduli space of Floer trajectories, the boundary components of $\partial\overline{S}\setminus\{\zeta_1,\cdots,\zeta_d\}$ will be labelled with Lagrangian submanifolds $L_1,\cdots,L_d$ which are objects of the wrapped Fukaya category $\mathcal{W}(M)$, so that $L_i$ is the label of the arc along the boundary between $\zeta_i$ and $\zeta_{i+1\textrm{ mod }d}$, and the marked points $\zeta_1,\cdots,\zeta_d$ are associated with asymptotics $x_1,\cdots,x_d$, which are time-1 chords of the Hamiltonian vector field $X_{H_t}$ from $L_{i-1}$ to $L_{i\textrm{ mod }d}$, for some $H_t\in\mathcal{H}(M)$. For any $y_\mathit{out}\in\mathcal{O}_M$, the coefficient before $|o_{y_\mathit{out}}|_\mathbb{K}$ in $\widecheck{\mathit{OC}}\left(|o_{x_d}|_\mathbb{K},\cdots,|o_{x_1}|_\mathbb{K}\right)$ is determined by a signed count of rigid Floer trajectories $u:S\rightarrow M$ which satisfy the relevant Floer equation, with boundary conditions determined by the Lagrangian labellings $(L_1,\cdots,L_d)$ and asymptotic conditions specified by $(\vec{x}:=(x_d,\cdots,x_1);y_\mathit{out})$.

The definition of the map $\widehat{\mathit{OC}}$ differs from $\widecheck{\mathit{OC}}$ in the sense that one now considers the same closed discs $\left(\overline{S};\zeta_1,\cdots,\zeta_d;\zeta_\mathit{out},\ell_{\mathit{out}}\right)$ with $d$ boundary marked points, and an interior marked point as the domains, but the asymptotic marker $\ell_{\mathit{out}}$ is required to point between $\zeta_1$ and $\zeta_d$. The collection $(\vec{x};y_\mathit{out})$ of Hamiltonian chords and orbits as above still determines the asymptotic conditions for the corresponding Floer equation. Since the direction of $\ell_{\mathit{out}}$ remembers the position of $\zeta_f$, its freedom to vary increases the degree of the map by 1, which explains why $\widehat{\mathit{OC}}$ is a map defined on the shifted Hochschild chain complex $\mathit{CH}_\ast(\mathcal{W}(M))[1]$.

\begin{proposition}[$\cite{sg1}$, Lemma 11]
The non-unital open-closed map $\mathit{OC}^\mathit{nu}=\widecheck{\mathit{OC}}\oplus\widehat{\mathit{OC}}$ is a chain map.
\end{proposition}

Consider the natural inclusion $\iota:\mathit{CH}_\ast(\mathcal{W}(M))\hookrightarrow\mathit{CH}_\ast^\mathit{nu}(\mathcal{W}(M))$, whose composition with the non-unital open-closed map gives rise to a chain map $\mathit{OC}^\mathit{nu}\circ\iota:\mathit{CH}_\ast(\mathcal{W}(M))\rightarrow\mathit{SC}^{\ast+n}(M)$, which coincides on the chain level with the usual open-closed string map $\mathit{OC}$. Since we have learned from Section \ref{section:non-unital} that $\iota$ is a quasi-isomorphism, it follows that as homology level maps, $[\mathit{OC}^\mathit{nu}]=[\mathit{OC}]$.
\bigskip

The cyclic open-closed string map $\widetilde{\mathit{OC}}$ will be defined as an $S^1$-equivariant enhancement of $\mathit{OC}^\mathit{nu}$, by including higher cyclic chain homotopies. More precisely, it consists of a sequence of maps
\begin{equation}
\mathit{OC}^k=\widecheck{\mathit{OC}}^k\oplus\widehat{\mathit{OC}}^k:\mathit{CH}_\ast^\mathit{nu}(\mathcal{W}(M))\rightarrow\mathit{SC}^{\ast+n-2k}(M)
\end{equation}
for each $k\geq0$, such that $\widecheck{\mathit{OC}}^0=\widecheck{\mathit{OC}}$,  $\widehat{\mathit{OC}}^0=\widehat{\mathit{OC}}$, and for any $k\geq1$, we have
\begin{equation}\label{eq:chain3}
(-1)^n\sum_{i=0}^k\delta_i\circ\widecheck{\mathit{OC}}^{k-i}=\widehat{\mathit{OC}}^{k-1}\circ\mathbb{B}^\mathit{nu}+\widecheck{\mathit{OC}}^k\circ b,
\end{equation}
\begin{equation}\label{eq:chain4}
(-1)^n\sum_{i=0}^k\delta_i\circ\widehat{\mathit{OC}}^{k-i}=\widehat{\mathit{OC}}^{k}\circ b'+\widecheck{\mathit{OC}}^k\circ(1-\lambda),
\end{equation}
where $\mathbb{B}^\mathit{nu}:\mathit{CH}_\ast^\mathit{nu}(\mathcal{W}(M))\rightarrow\mathit{CH}_{\ast-1}^\mathit{nu}(\mathcal{W}(M))$ is the map (\ref{eq:Bnu}) applied to the wrapped Fukaya category, and $\lambda$ is the cyclic permutation operator defined in (\ref{eq:permu}). Roughly speaking, the maps $\widecheck{\mathit{OC}}^k$ and $\widehat{\mathit{OC}}^k$ are defined in the same way as $\widecheck{\mathit{OC}}$ and $\widehat{\mathit{OC}}$, but with additional interior marked points $p_1,\cdots,p_k$ included in the respective domains, which are located near $\zeta_\mathit{out}$ and are \textit{strictly radially ordered} in the sense that
\begin{equation}\label{eq:radial}
0<|p_1|<\cdots<|p_k|<\frac{1}{2}
\end{equation}
with respect to the standard complex coordinate near $\zeta_\mathit{out}$. Note that now the cylinders in the moduli spaces $\mathcal{M}_i$, where $0\leq i\leq k$, break ``below" the surface $S$ at $\zeta_\mathit{out}$ instead of ``above" it. As before, $\ell_{\mathit{out}}$ is required to point towards $p_1$ if $k>0$.

Define
\begin{equation}
\widetilde{\mathit{OC}}:=\sum_{k=0}^\infty\left(\widecheck{\mathit{OC}}^k\oplus\widehat{\mathit{OC}}^k\right)u^k,
\end{equation}
it follows from (\ref{eq:chain3}) and (\ref{eq:chain4}) that:

\begin{theorem}[$\cite{sg1}$, Theorem 1]\label{theorem:cyclic-OC}
The non-unital open-closed map $\mathit{OC}^\mathit{nu}$ admits a geometrically defined $S^1$-equivariant enhancement $\widetilde{\mathit{OC}}\in R\mathrm{Hom}_{S^1}\left(\mathit{CH}_\ast^\mathit{nu}(\mathcal{W}(M))[n],\mathit{SC}^\ast(M)\right)$.
\end{theorem}

Combining $\cite{sg2}$, Theorem 1.1 with $\cite{sg1}$, Corollary 1, we have the following:

\begin{theorem}[Ganatra]\label{theorem:iso}
Let $M$ be a Liouville manifold which is non-degenerate, then the homology level maps
\begin{equation}
[\mathit{OC}]:\mathit{HH}_\ast(\mathcal{W}(M))\rightarrow\mathit{SH}^{\ast+n}(M),[\widetilde{\mathit{OC}}]:\mathit{HC}_\ast(\mathcal{W}(M))\rightarrow\mathit{SH}_{S^1}^{\ast+n}(M)
\end{equation}
are isomorphisms.
\end{theorem}

We can now fulfil our promise at the beginning of Section \ref{section:Gysin}, namely to explain the relationship between exact Calabi-Yau structures on $\mathcal{W}(M)$ and cyclic dilations.

\begin{proposition}
For any Liouville manifold $M$, there is a commutative diagram
\begin{equation}\label{eq:comm1}
\xymatrix{
	\mathit{HC}_{\ast+1}(\mathcal{W}(M)) \ar[d]_{[\widetilde{\mathit{OC}}]} \ar[r]^-{\mathbb{B}}
	& \mathit{HH}_{\ast}(\mathcal{W}(M)) \ar[d]^{[\mathit{OC}]} \\
	\mathit{SH}_{S^1}^{\ast+n+1}(M) \ar[r]^-{\mathbf{B}}
	& \mathit{SH}^{\ast+n}(M)
}
\end{equation}
where $\mathbb{B}$ is the cohomology level map associated to $\mathbb{B}^\mathit{nu}$.
\end{proposition}
\begin{proof}
This is a direct consequence of Theorem \ref{theorem:cyclic-OC} and Proposition \ref{proposition:intert}.
\end{proof}

As a corollary, we get the following proof of Proposition \ref{proposition:geom-inter}.

\begin{corollary}\label{corollary:non-degenerate}
Let $M$ be a non-degenerate Liouville manifold, its wrapped Fukaya category $\mathcal{W}(M)$ is exact Calabi-Yau if and only if	there exists a cyclic dilation $\tilde{b}\in\mathit{SH}_{S^1}^1(M)$.
\end{corollary}
\begin{proof}
Since $M$ is non-degenerate, it follows from Theorem \ref{theorem:iso} that both of the maps $[\mathit{OC}]$ and $[\widetilde{\mathit{OC}}]$ in the commutative diagram (\ref{eq:comm1}) are isomorphisms. The corollary would then follow from the fact that $[\eta]\in\mathit{HH}_{-n}(\mathcal{W}(M))$ is non-degenerate if and only if its image under the open-closed map $[\mathit{OC}]$ is an invertible element $h\in\mathit{SH}^0(M)^\times$. 

The proof of $\cite{sg1}$, Theorem 3 shows that the geometrically defined smooth Calabi-Yau structure $[\tilde{\eta}_\mathit{std}]\in\mathit{HC}_{-n}^-(\mathcal{W}(M))$ induces a non-degenerate class $[\eta_\mathit{std}]\in\mathit{HH}_{-n}(\mathcal{W}(M))$, which is mapped under $[\mathit{OC}]$ to the identity $1\in\mathit{SH}^0(M)$. According to the definition of Calabi-Yau structures, any two such structures differ by an automorphism of the diagonal bimodule. In our case, any non-degenerate class $[\eta]\in\mathit{HH}_{-n}(\mathcal{W}(M))$ differs from $[\eta_\mathit{std}]$ by an automorphism of the diagonal bimodule, which is an invertible element of $\mathit{HH}^0(\mathcal{W}(M))$. By $\cite{sg2}$, Theorem 1.1, its image under the inverse of the closed-open map $[\mathit{CO}]^{-1}:\mathit{HH}^\ast(\mathcal{W}(M))\rightarrow\mathit{SH}^\ast(M)$ is an element $h\in\mathit{SH}^0(M)^\times$, so is $[\mathit{OC}]([\eta])$.

On the other hand, if $[\eta]\in\mathit{HH}_{-n}(\mathcal{W}(M))$ is mapped to a class $h\in\mathit{SH}^0(M)^\times$. By applying the closed-open map we get a class $[\mathit{CO}](h)\in\mathit{HH}^0(\mathcal{W}(M))^\times$, which induces an automorphism of the diagonal bimodule. Composing this with the isomorphism $\mathcal{W}(M)^\vee[n]\cong\mathcal{W}(M)$ between the diagonal bimodule and its shifted dual induced by $[\eta_\mathit{std}]$, we get another isomorphism $\mathcal{W}(M)^\vee[n]\xrightarrow{\cong}\mathcal{W}(M)$, which induced by $[\eta]$. This shows that $[\eta]$ is non-degenerate.
\end{proof}

\section{Lagrangian submanifolds}\label{section:Lag}

Let $M$ be a Liouville manifold with $c_1(M)=0$, and fix a trivialization of its canonical bundle $K_M$. We consider in this chapter the open string implications of the existence of a cyclic dilation. To be precise, we shall consider Lagrangian submanifolds in $M$ which are objects of the compact Fukaya category $\mathcal{F}(M)$, namely they satisfy the following:

\begin{assumption}\label{assumption:Lag}
$L\subset M$ is closed, connected, exact, graded, and $\mathit{Spin}$.
\end{assumption}

We shall actually fix the choice of a grading on $L$, so that the Lagrangian Floer cohomology $\mathit{HF}^\ast(L_0,L_1)$ of two Lagrangian submanifolds $L_0,L_1\subset M$ is well-defined as a $\mathbb{Z}$-graded algebra over $\mathbb{K}$.

\subsection{The Cieliebak-Latschev map}\label{section:CL}

Let $L\subset M$ be an exact Lagrangian submanifold satisfying Assumption \ref{assumption:Lag}. As a consequence of the Viterbo functoriality $\cite{ma2,cv}$, we have a map
\begin{equation}\label{eq:Viterbo}
\mathit{SH}^\ast(M)\rightarrow\mathit{SH}^\ast(T^\ast L)\cong H_{n-\ast}(\mathcal{L}L;\nu),
\end{equation}
where the latter isomorphism is established in $\cite{abs}$ in the case when $L$ is $\mathit{Spin}$, and in $\cite{ma2}$ in the general case. Since we have required in Assumption \ref{assumption:Lag} that $L$ is $\mathit{Spin}$, the local system $\nu:\pi_1(\mathcal{L}L)\rightarrow\mathbb{K}$ can be dropped out from our notations.

There is an $S^1$-equivariant analogue of (\ref{eq:Viterbo}) constructed by Cohen-Ganatra $\cite{cg}$ (see also $\cite{jz}$, Section 4.4.1 for a detailed exposition), which is an infinite sum
\begin{equation}
\widetilde{\mathit{CL}}:=\sum_{k=0}^\infty\mathit{CL}_ku^k:\mathit{SC}^\ast(M)\otimes_\mathbb{K}\mathbb{K}((u))/u\mathbb{K}[[u]]\rightarrow C_{n-\ast}^\lozenge\left(\mathcal{L}L;\mathbb{K}((u))/u\mathbb{K}[[u]]\right)
\end{equation}
whose degree 0 piece arises from relevant considerations by Cieliebak-Latschev in $\cite{cl}$. In the above, $C_{n-\ast}^\lozenge(\mathcal{L}L;\mathbb{K})$ is a quotient of the dg algebra $C_{n-\ast}(\mathcal{L}L;\mathbb{K})$ constructed by Cohen-Ganatra in $\cite{cg}$, Appendix A.1. See also \cite{jz}, Appendix B. It has the property that the projection $C_{n-\ast}(\mathcal{L}L;\mathbb{K})\rightarrow C_{n-\ast}^\lozenge(\mathcal{L}L;\mathbb{K})$ is a quasi-isomorphism, and $C_{n-\ast}^\lozenge(\mathcal{L}L;\mathbb{K})$ carries the structure of a strict $S^1$-complex. $\widetilde{\mathit{CL}}$ defines an $S^1$-complex morphism, so it descends to the map (\ref{eq:CL}) on the cohomology level. We shall give a brief account of Cohen-Ganatra's construction in this section, and explain its implications for Lagrangian submanifolds in Liouville manifolds with cyclic dilations.
\bigskip

The construction of the maps $\mathit{CL}_k:\mathit{SC}^\ast(M)\rightarrow C_{n-\ast+2k}^\lozenge(\mathcal{L}L;\mathbb{K})$ is in some sense parallel to the construction of the maps $\{\delta_k\}$ in Section \ref{section:equi-symp}, but we now consider half-cylinders instead of cylinders as our domains. A $k$-\textit{point angle decorated half-cylinder} is a (positive) half-cylinder $Z^+\subset Z$ together with a collection of auxiliary interior marked points $p_1,\cdots,p_k\in Z^+$ satisfying (\ref{eq:radial1}). Denote by $\mathcal{M}_{k,+}$ the moduli space of such half-cylinders. Every element of $\mathcal{M}_{k,+}$ is equipped with a positive cylindrical end
\begin{equation}\label{eq:pce}
\varepsilon^+:[0,\infty)\times S^1\rightarrow Z^+, (s,t)\mapsto(s+(p_k)_s+\xi,t),
\end{equation}
for some fixed $\xi>0$. Note that unlike the case of $\mathcal{M}_k$, there is no free $\mathbb{R}$-action on the moduli space $\mathcal{M}_{k,+}$.

$\mathcal{M}_{k,+}$ can be compactified to a manifold with corners $\overline{\mathcal{M}}_{k,+}$ by including broken trajectories in the moduli space. The codimension 1 boundary strata of $\overline{\mathcal{M}}_{k,+}$ is covered by the images of the natural inclusions
\begin{equation}\label{eq:boundary1}
\overline{\mathcal{M}}_{j,+}\times\overline{\mathcal{M}}_{k-j}\hookrightarrow\partial\overline{\mathcal{M}}_{k,+}, 0\leq j\leq k,
\end{equation}
\begin{equation}\label{eq:boundary2}
\overline{\mathcal{M}}_{k,+}^{i,i+1}\hookrightarrow\partial\overline{\mathcal{M}}_{k,+}, 1\leq i<k,
\end{equation}
\begin{equation}\label{eq:boundary3}
\overline{\mathcal{M}}_{k-1,+}^{S^1}\hookrightarrow\partial\overline{\mathcal{M}}_{k,+},
\end{equation}
where $\mathcal{M}_{k,+}^{i,i+1}$ is the locus where the $i$-th and $(i+1)$-th height coordinates coincide, and $\mathcal{M}_{k-1,+}^{S^1}$ is the locus where $h_1=0$. There exist forgetful maps
\begin{equation}
\pi_i:\mathcal{M}_{k,+}^{i,i+1}\rightarrow\mathcal{M}_{k-1,+}, 1\leq i\leq k-1,
\end{equation}
\begin{equation}
\pi_{S^1}:\mathcal{M}_{k-1,+}^{S^1}\rightarrow\mathcal{M}_{k-1,+},
\end{equation}
where the first map has been considered in (\ref{eq:forget1}), which forgets the point $p_{i+1}$, while the second map forgets $p_1$. Note that the maps $\pi_i$ for $i\geq1$ extend as maps $\bar{\pi}_i:\overline{\mathcal{M}}_{k,+}^{i,i+1}\rightarrow\overline{\mathcal{M}}_{k-1,+}$ on the compactifications. The same holds for $\pi_{S^1}$, and we denote its extension by $\bar{\pi}_{S^1}:\overline{\mathcal{M}}_{k-1,+}^{S^1}\rightarrow\overline{\mathcal{M}}_{k-1,+}$.
\bigskip

The definition of a Floer datum for a $k$-point angle-decorated half-cylinder $(Z^+,p_1,\cdots,p_k)$ is completely analogous to that of Definition \ref{definition:data-cylinder}, and will therefore be omitted. Inductively, there exist universal and consistent choices of Floer data for each $k\geq0$ and each $k$-point angle decorated half-cylinder in the sense that:
\begin{itemize}
	\item In a sufficiently small neighborhood of $L\subset M$, the Hamiltonian $H_{Z^+}=0$ near $s=0$.
	\item Near the boundary stratum (\ref{eq:boundary1}), the Floer datum coincides with the product of the Floer data chosen on lower dimensional strata up to conformal equivalence. The Floer data vary smoothly with respect to the gluing charts for the product Floer data.
	\item Near the boundary strata (\ref{eq:boundary2}), the Floer data are conformally equivalent to the ones obtained by pulling back from $\overline{\mathcal{M}}_{k-1,+}$ via the forgetful maps $\bar{\pi}_i$ for $i=1,\cdots,k-1$.
\end{itemize}

Fixing a universal and consistent choice of Floer data, for each $y\in\mathcal{O}_M$, define $\mathcal{M}_{k,+}(y,L)$ to be the moduli space of pairs $\left((Z^+,p_1,\cdots,p_k),u\right)$, where $(Z^+,p_1,\cdots,p_k)\in\mathcal{M}_{k,+}$, and $u:Z^+\rightarrow M$ is a map satisfying the parametrized Floer equation
\begin{equation}\label{eq:FE}
\left(du-X_{H_{Z^+}}\otimes dt\right)^{0,1}=0,
\end{equation}
where the $(0,1)$-part is taken with respect to $J_{Z^+}$, together with asymptotic and boundary conditions
\begin{equation}
\lim_{s\rightarrow\infty}(\varepsilon^+)^\ast u(s,\cdot)=y,
\end{equation}
\begin{equation}
u(0,t)=\gamma\textrm{ for some }\gamma\in\mathcal{L}L.
\end{equation}
For generic choices of $H_{Z^+}$, $\mathcal{M}_{k,+}(y,L)$ is a smooth manifold of dimension
\begin{equation}
n-\deg(y)+2k,
\end{equation}
which admits a well-defined Gromov bordification $\overline{\mathcal{M}}_{k,+}(y,L)$, whose codimension 1 boundary is covered by the inclusions
\begin{equation}\label{eq:st1}
\overline{\mathcal{M}}_{k-j,+}(y',L)\times\overline{\mathcal{M}}_j(y,y')\hookrightarrow\partial\overline{\mathcal{M}}_{k,+}(y,L),
\end{equation}
\begin{equation}\label{eq:st2}
\overline{\mathcal{M}}_{k,+}^{i,i+1}(y,L)\hookrightarrow\partial\overline{\mathcal{M}}_{k,+}(y,L),
\end{equation}
\begin{equation}\label{eq:st3}
\overline{\mathcal{M}}_{k-1,+}^{S^1}(y,L)\hookrightarrow\partial\overline{\mathcal{M}}_{k,+}(y,L).
\end{equation}
Choose some Riemannian metric $g$ on $L$. There is an evaluation map $\mathit{ev}:\mathcal{M}_{k,+}(y,L)\rightarrow\mathcal{L}L$ is defined by restricting $u\in\mathcal{M}_{k,+}(y,L)$ to $\{0\}\times S^1$ and taking the arc length parametrization of the boundary of $u$ with respect to $g$ and a base point determined by the position of $p_1$. The map $\mathit{ev}$ admits an extension $\overline{\mathit{ev}}:\overline{\mathcal{M}}_{k,+}(y,L)\rightarrow\mathcal{L}L$ to the boundary strata, and the $k$-th order Cieliebak-Latschev map
\begin{equation}
\mathit{CL}_k:\mathit{SC}^\ast(M)\rightarrow C_{n-\ast+2k}^\lozenge(\mathcal{L}L;\mathbb{K}).
\end{equation}
is defined as
\begin{equation}
\mathit{CL}_k(|o_y|_\mathbb{K})=(-1)^{\deg(y)}\overline{\mathit{ev}}_\ast\left(\left[\overline{\mathcal{M}}_{k,+}(y,L)\right]\right),
\end{equation}
where $\left[\overline{\mathcal{M}}_{k,+}(y,L)\right]$ denotes the fundamental chain.

\begin{proposition}[$\cite{jz}$, Proposition 4.4.12]\label{proposition:CL}
$\widetilde{\mathit{CL}}=\sum_{k=0}^\infty\mathit{CL}_ku^k$ defines a morphism of $S^1$-complexes, and therefore it is an $S^1$-equivariant enhancement of $\mathit{CL}_0$.
\end{proposition}

The proof follows from an analysis of the boundary strata of $\overline{\mathcal{M}}_{k,+}(y,L)$. In particular, our choice of Floer data ensures that the elements in the strata $\overline{\mathcal{M}}_{k,+}^{i,i+1}(y,L)$ will never contribute, since they are not rigid. For details, see the proof of \cite{jz}, Proposition 4.4.12. We will encounter a similar situation in the proof of Proposition \ref{proposition:prod}. On the other hand, the contribution from the stratum $\overline{\mathcal{M}}_{k-1,+}^{S^1}(y,L)$ is non-trivial, and can actually be identified with $\delta_1^\mathit{top}\circ\mathit{CL}_{k-1}$, where $\delta_1^\mathit{top}$ denotes the chain level BV operator on $C_{n-\ast}(\mathcal{L}L;\mathbb{K})$ defined by rotating the loops, which descends to a BV operator on the quotient dg algebra $C_{n-\ast}^\lozenge(\mathcal{L}L;\mathbb{K})$.

On the cohomology level, $\widetilde{\mathit{CL}}$ induces a map
\begin{equation}
[\widetilde{\mathit{CL}}]:\mathit{SH}_{S^1}^\ast(M)\rightarrow H_{n-\ast}^{S^1}(\mathcal{L}L;\mathbb{K}).
\end{equation}
This enables us to interpret a result of Davison ($\cite{bd}$, Corollary 6.4.4) as providing obstructions to Lagrangian embeddings in Liouville manifolds with cyclic dilations, see Proposition \ref{proposition:Davison}.

\begin{proof}[Proof of Proposition \ref{proposition:Davison}]
Let $M$ be a Liouville manifold with a cyclic dilation, and assume that there is an exact Lagrangian submanifold $L\subset M$ which is hyperbolic. It follows from Propositions \ref{proposition:intert} and \ref{proposition:CL} that there is a commutative diagram
\begin{equation}\label{eq:diagram1}
\begin{tikzcd}
&\mathit{SH}_{S^1}^{\ast}(M) \arrow[d, "{[\widetilde{\mathit{CL}}]}"'] \arrow[r,"\mathbf{B}"] &\mathit{SH}^{\ast-1}(M) \arrow[d] \\
&H_{n-\ast}^{S^1}(\mathcal{L}L;\mathbb{K}) \arrow[r, "\mathbf{B}"] &H_{n-\ast+1}(\mathcal{L}L;\mathbb{K})
\end{tikzcd}
\end{equation}
where the vertical arrow on the right is the usual Viterbo map (\ref{eq:Viterbo}). By our assumption, there is a class $\tilde{b}\in\mathit{SH}^1_{S^1}(M)$ whose image under the Connes' map $\mathbf{B}$ is an invertible element $h\in\mathit{SH}^0(M)^\times$. By the commutativity of (\ref{eq:diagram1}), and our assumption that $L$ is a $K(\pi,1)$ space, such a class induces an exact Calabi-Yau structure on the fundamental group algebra $\mathbb{K}[\pi_1(L)]$, which contradicts the main result of $\cite{bd}$.

This completes the proof of Proposition \ref{proposition:Davison}, under the additional assumption that $L$ is $\mathit{Spin}$. In general, we can use ($S^1$-equivariant) symplectic cohomologies with local coefficients and argue as above. Since Davison's result holds for any closed, orientable $L$, this enables us to remove the $\mathit{Spin}$ assumption on $L$.
\end{proof}

It would also be interesting to take a look at the special case when $h=1$ in the definition of a cyclic dilation, which leads to the following generalization of $\cite{ss}$, Corollary 6.3.

\begin{corollary}\label{corollary:h=1}
Suppose that the marking map $\mathbf{B}:\mathit{SH}^1_{S^1}(M)\rightarrow\mathit{SH}^0(M)$ hits the identity $1\in\mathit{SH}^0(M)$, then $M$ cannot contain a closed exact Lagrangian submanifold $L$ which is a $K(\pi,1)$ space.
\end{corollary}
\begin{proof}
Let $L\subset M$ be an exact Lagrangian submanifold which is topologically a $K(\pi,1)$ space. Since $T^\ast L$ is a Weinstein manifold, so Corollary \ref{corollary:non-degenerate} applies. It follows from the proof of $\cite{bd}$, Theorem 6.1.3 that the marking map $\mathbf{B}:\mathit{SH}_{S^1}^1(T^\ast L)\rightarrow\mathit{SH}^0(T^\ast L)$ cannot hit the identity. Suppose $M$ admits a cyclic dilation with $h=1$, one can then use the commutative diagram (\ref{eq:diagram1}) to get a contradiction.
\end{proof}

Let $\mathbb{K}=\mathbb{Q}$. It is proved in $\cite{zz2}$, Theorem A that the Milnor fibers $M_{a,\cdots,a}\subset\mathbb{C}^{n+1}$ associated to the Brieskorn singularities
\begin{equation}
z_1^a+\cdots+z_{n+1}^a=0\textit{ with }n\geq a
\end{equation}
admit cyclic dilations over $\mathbb{Q}$ with $h=1$. By Corollary \ref{corollary:h=1}, it implies the following:

\begin{corollary}
The Milnor fibers $M_{a,\cdots,a}$ do not contain exact Lagrangian tori.
\end{corollary}

The non-existence of exact Lagrangian tori has been proved for many Milnor fibers. For a recent account, see \cite{yl3}.

Corollary \ref{corollary:h=1} can also be applied to deduce non-existence results concerning cyclic dilations. For example, consider the Weinstein 4-manifold $T_{1,1,0}\subset\mathbb{C}^3$ defined by the equation
\begin{equation}
x+y+xyz=1,
\end{equation}
which is the complement of a nodal elliptic curve $\Sigma\subset\mathbb{CP}^2$. This manifold is studied in $\cite{cm}$, Section 4.1, and it follows from the computation loc. cit that
\begin{equation}
\mathit{SH}^0(T_{1,1,0})\cong\mathbb{K}[x,y,z]/(x+y+xyz-1)
\end{equation}
as $\mathbb{K}$-algebras. Since the polynomial $x+y+xyz-1$ is irreducible over $\mathbb{K}$, the only invertible element in $\mathit{SH}^0(T_{1,1,0})$ is the identity. If $T_{1,1,0}$ admits a cyclic dilation, then $\mathbf{B}:\mathit{SH}_{S^1}^1(T_{1,1,0})\rightarrow\mathit{SH}^0(T_{1,1,0})$ hits the identity. On the other hand, from the perspective of Legendrian surgery, $T_{1,1,0}$ can be constructed by attaching two 2-handles to the disc cotangent bundle $D^\ast T^2$, so there is an exact Lagranian torus $L\subset T_{1,1,0}$. Now Corollary \ref{corollary:h=1} shows that $T_{1,1,0}$ does not admit a cyclic dilation.

One can attach one more 2-handle to $D^\ast T^2$ to get the Liouville domain associated to the affine surface $T_{1,1,1}\subset\mathbb{C}^3$ defined by the equation
\begin{equation}
x+y+z+xyz=1.
\end{equation}
Since $\overline{T}_{1,1,0}$ embeds in $\overline{T}_{1,1,1}$ as a Liouville subdomain, we conclude that $T_{1,1,1}$ does not admit a cyclic dilation. In general, let $\overline{M}_0\subset\overline{M}_1$ be a Liouville subdomain, there is a commutative diagram
\begin{equation}\label{eq:diagram2}
\begin{tikzcd}
&\mathit{SH}_{S^1}^{\ast}(M_1) \arrow[d,"{[\tilde{v}^!]}"'] \arrow[r,"\mathbf{B}"] &\mathit{SH}^{\ast-1}(M_1) \arrow[d,"{[v^!]}"] \\
&\mathit{SH}_{S^1}^\ast(M_0) \arrow[r, "\mathbf{B}"] &\mathit{SH}^{\ast-1}(M_0)
\end{tikzcd}
\end{equation}
generalizing (\ref{eq:diagram1}), from which we see that if $M_1$ admits a cyclic dilation, then so is $M_0$. In (\ref{eq:diagram2}), the map
\begin{equation}\label{eq:equi-Vit}
\tilde{v}^!=\sum_{i=0}^\infty v^!_ku^k:\mathit{SC}^\ast(M_1)\otimes_\mathbb{K}\mathbb{K}((u))/u\mathbb{K}[[u]]\rightarrow\mathit{SC}^\ast(M_0)\otimes_\mathbb{K}\mathbb{K}((u))/u\mathbb{K}[[u]]
\end{equation}
is the $S^1$-equivariant enhancement of Viterbo functoriality $v^!=v^!_0$. We refer the reader to $\cite{jz}$, Appendix C for its detailed construction.

More generally, attaching 2-handles to $D^\ast T^2$ yields a sequence of Weinstein 4-manifolds $T_{p,q,r}$, with $p\geq q\geq r\geq0$. When $\frac{1}{p}+\frac{1}{q}+\frac{1}{r}\leq1$, these are the Milnor fibers of parabolic and hyperbolic unimodal singularities studied by Keating in $\cite{ak1,ak2}$. Our discussions above imply the following:

\begin{proposition}\label{proposition:pqr}
The Weinstein manifold $T_{p,q,r}$ admits a cyclic dilation if and only if $q=r=0$.
\end{proposition}
\begin{proof}
Note that $T_{0,0,0}$ is symplectomorphic to $T^\ast T^2$, so it admits a quasi-dilation. $T_{1,0,0}$ is symplectomorphic to $\mathbb{C}^2\setminus\{xy=1\}$, it follows from $\cite{ps5}$, Corollary 19.8 that there is a quasi-dilation in $\mathit{SH}^1(T_{1,0,0})$. Alternatively, one can compute its wrapped Fukaya category explicitly using the techniques developed in $\cite{bee}$, whose endomorphism algebra turns out to be formal, and is quasi-isomorphic to the associative algebra $\mathbb{K}[x,y][(xy-1)^{-1}]$, see Example \ref{example:110} for its superpotential description. The case when $p>1$ can be argued similarly, since there are Lefschetz fibrations $T_{p,0,0}\rightarrow\mathbb{C}^\ast$ whose smooth fibers are $T^\ast S^1$. In fact, $T_{p,0,0}$ is symplectomorphic to the $\widetilde{A}_p$ plumbing of $T^\ast S^2$'s. On the other hand, we have seen in the above that $T_{1,1,0}$ and $T_{1,1,1}$ do not admit cyclic dilations. Since any Weinstein manifold $T_{p,q,r}$ with $p\geq1$ and $q\geq1$ contains $\overline{T}_{1,1,0}$ as its Liouville subdomain, the non-existence of cyclic dilations follows from the commutative diagram (\ref{eq:diagram2}).
\end{proof}

Observe that among the examples $T_{p,q,r}$ considered above, the existence of a cyclic dilation is in fact equivalent to the existence of a quasi-dilation. This is not surprising in view of Proposition \ref{proposition:log-CY}. More interestingly, Proposition \ref{proposition:pqr} implies the following:

\begin{corollary}\label{corollary:non-simple}
Let $M$ be any 4-dimensional Milnor fiber associated to a non-simple singularity, then $M$ does not admit a cyclic dilation.
\end{corollary}
\begin{proof}
This follows from the commutative diagram (\ref{eq:diagram2}) and the adjacency of singularities. The latter implies the existence of some triple $(p,q,r)$ with $\frac{1}{p}+\frac{1}{q}+\frac{1}{r}=1$, so that $M$ contains $\overline{T}_{p,q,r}$ as a Liouville subdomain. A detailed explanation of this fact can be found in \cite{ak1}, Section 2.2. However, it follows from Proposition \ref{proposition:pqr} that any such $T_{p,q,r}$ cannot admit a cyclic dilation.
\end{proof}

\subsection{Parametrized closed-open maps}\label{section:CO}

Let $M$ be a $2n$-dimensional Liouville manifold. As is observed by Seidel in $\cite{ps6}$, if one considers only closed exact Lagrangian submanifolds $L\subset M$ satisfying Assumption \ref{assumption:Lag}, then one can define a chain map
\begin{equation}
\mathit{CO}_\mathit{cpt}:\mathit{SC}^\ast(M)\rightarrow\mathit{CH}^\ast(\mathcal{F}(M)).
\end{equation}
This is usually referred to as the \textit{closed-open string map} in literature. On the cohomology level, it is obtained by composing the closed-open map $\mathit{CO}:\mathit{SC}^\ast(M)\rightarrow\mathit{CH}^\ast(\mathcal{W}(M))$ considered in $\cite{sg2}$ with the restriction morphism induced by the cohomologically full and faithful embedding $\mathcal{F}(M)\hookrightarrow\mathcal{W}(M)$. Ganatra shows in $\cite{sg1}$, Proposition 14 that this map also admits an $S^1$-equivariant enhancement, which can be written as a homomorphism between $S^1$-complexes
\begin{equation}\label{eq:cyclic-CO}
\widetilde{\mathit{CO}}_\mathit{cpt}^\vee\in R\mathrm{Hom}_{S^1}\left(\mathit{CH}_\ast^\mathit{nu}(\mathcal{F}(M))\otimes\mathit{SC}^\ast(M),\mathbb{K}[-n]\right),
\end{equation}
where the complex $\mathit{CH}_\ast^\mathit{nu}(\mathcal{F}(M))\otimes\mathit{SC}^\ast(M)$ is equipped with the diagonal $S^1$-action (\ref{eq:diagonal}), while $\mathbb{K}[-n]$ is regarded as a trivial $S^1$-complex. The construction of $\widetilde{\mathit{CO}}_\mathit{cpt}^\vee$ is completely parallel to the cyclic open-closed string map $\widetilde{\mathit{OC}}$ recalled in Section \ref{section:OC}. In particular, it consists of an infinite sequence of maps
\begin{equation}
\mathit{CO}_\mathit{cpt}^{k,\vee}:=\widecheck{\mathit{CO}}_\mathit{cpt}^{k,\vee}\oplus\widehat{\mathit{CO}}_\mathit{cpt}^{k,\vee}
\end{equation}
for each $k\geq0$, where $\widecheck{\mathit{CO}}_\mathit{cpt}^{k,\vee}$ and $\widehat{\mathit{CO}}_\mathit{cpt}^{k,\vee}$ are maps acting on the check factor $\mathit{CH}_\ast(\mathcal{F}(M))\otimes\mathit{SC}^\ast(M)$ and the hat factor $\mathit{CH}_\ast(\mathcal{F}(M))[1]\otimes\mathit{SC}^\ast(M)$ respectively.
\bigskip

Here, instead of using the full construction of Ganatra, only the check components $\left\{\widecheck{\mathit{CO}}_\mathit{cpt}^{k,\vee}\right\}_{k\geq0}$ of the cyclic refinement of $\mathit{CO}_\mathit{cpt}$ will be relevant for our purposes. For the sake of readability, it seems to be appropriate to recall here the definitions as well as the basic properties of the first few of these maps. Since the reader should already be familiar with the operations associated to parametrized moduli spaces from our discussions in Section \ref{section:PF}, our exposition here will be very sketchy, focusing mainly on the underlying TCFT (Topological Conformal Field Theory) structures and describing primarily the domains defining various operations. We follow the general framework of $\cite{ps4}$ and $\cite{ss}$, while our set up is slightly more complicated as it involves one additional piece of data: an ordered set of interior auxiliary marked points, whose flexibility constitutes the parameter spaces of our families of Riemann surfaces.
\bigskip

Let $S=\overline{S}\setminus\Sigma$, where $\overline{S}$ is a bordered Riemann surface, and $\Sigma\subset\overline{S}$ is a finite set of points. Write $\Sigma=\Sigma^\mathit{op}\cup\Sigma^\mathit{cl}$, where $\Sigma^\mathit{op}\subset\partial\overline{S}$ is the set of boundary marked points, and $\Sigma^\mathit{cl}\subset S$ is the set of interior marked points. There is also a strictly ordered set of auxiliary marked points $\Sigma^\mathit{aux}\subset S$ lying in small neighborhoods of the interior marked points $\Sigma^\mathit{cl}$. For simplicity, we assume that all of the points in $\Sigma^\mathit{aux}$ lie in a small disc centered a particular point $\zeta_\bullet\in\Sigma^\mathit{cl}$ of radius $\varepsilon>0$. The points in $\Sigma^\mathit{aux}$ are ordered according to their distances to $\zeta_\bullet$. With respect to a fixed choice of complex coordinate near $\zeta_\bullet$, they should be strictly radially ordered, i.e. satisfying
\begin{equation}\label{eq:radial2}
0<|p_k|<\cdots<|p_1|<\varepsilon
\end{equation}
if $\zeta_\bullet$ is an input, and
\begin{equation}\label{eq:radial3}
0<|p_1|<\cdots<|p_k|<\varepsilon
\end{equation}
if $\zeta_\bullet$ is an output. As a convention, when $\overline{S}$ is the closed unit disc, we will take $\varepsilon=\frac{1}{2}$. At each point of $\Sigma^\mathit{cl}$, there a preferred tangent direction $\ell_\zeta$, which is fixed if $\zeta\in\Sigma^\mathit{cl}\setminus\{\zeta_\bullet\}$, and points towards the point in $\Sigma^\mathit{aux}$ which is closest from $\zeta$ if $\zeta=\zeta_\bullet$. Furthermore, we can divide the sets $\Sigma^\mathit{op}$ and $\Sigma^\mathit{cl}$ into inputs and outputs, namely $\Sigma^\mathit{op}=\Sigma^{\mathit{op},\mathit{in}}\cup\Sigma^{\mathit{op},\mathit{out}}$ and $\Sigma^\mathit{cl}=\Sigma^{\mathit{cl},\mathit{in}}\cup\Sigma^{\mathit{cl},\mathit{out}}$.

Additionally, $\overline{S}$ comes with a sub-closed 1-form $\nu_S\in\Omega^1(S)$ which satisfies $d\nu_S=0$ near $\Sigma$, and $\nu_S|_{\partial S}=0$ near $\Sigma^\mathit{op}$. One can associate a real number $\lambda_\zeta$ to every point $\zeta\in\Sigma$, by integrating $\nu_S$ along a small loop around $\zeta$ if $\zeta\in\Sigma^\mathit{cl}$, or by integrating $\nu_S$ along a small path connecting one component of $\partial S$ to the other one, if $\zeta\in\Sigma^\mathit{op}$. For $\zeta\in\Sigma^\mathit{cl}$, we require that $\lambda_\zeta\notin\mathcal{P}_M$.

For each boundary component $C\subset\partial S$, we want to have a label $L_C$, which is a Lagrangian submanifold in $M$ satisfying Assumption \ref{assumption:Lag}. For any point $\zeta\in\Sigma^\mathit{op}$, this determines a pair of Lagrangian submanifolds $(L_{\zeta,0},L_{\zeta,1})$. The convention is that if $\zeta\in\Sigma^{\mathit{op},\mathit{in}}$, then $L_{\zeta,0}$ is the Lagrangian submanifold associated to the boundary component preceding $\zeta$ with respect to the orientation of $\partial\overline{S}$, and $L_{\zeta,1}$ is the successive one. When $\zeta\in\Sigma^{\mathit{op},\mathit{out}}$, then one uses the opposite convention. The operation associated to the marked bordered Riemann surfaces $\left(\overline{S};\Sigma;\Sigma^\mathit{aux}\right)$ is a map
\begin{equation}\label{eq:TQFT}
\bigotimes_{\zeta\in\Sigma^{\mathit{cl},\mathit{in}}}\mathit{CF}^\ast(\lambda_\zeta)\otimes\bigotimes_{\zeta\in\Sigma^{\mathit{op},\mathit{in}}}\mathit{CF}^\ast(L_{\zeta,0},L_{\zeta,1})\rightarrow\bigotimes_{\zeta\in\Sigma^{\mathit{cl},\mathit{out}}}\mathit{CF}^\ast(\lambda_\zeta)\otimes\bigotimes_{\zeta\in\Sigma^{\mathit{op},\mathit{out}}}\mathit{CF}^\ast(L_{\zeta,0},L_{\zeta,1})
\end{equation}
of degree
\begin{equation}
n\left(-\chi(\overline{S})+2|\Sigma^{\mathit{cl},\mathit{out}}|+|\Sigma^{\mathit{op},\mathit{out}}|\right)-2|\Sigma^\mathit{aux}|.
\end{equation}
We remark that due to the presence of $\Sigma^\mathit{aux}$, (\ref{eq:TQFT}) is in general \textit{not} a chain map.
\bigskip

We give a quick sketch of the definition of (\ref{eq:TQFT}). Denote by ${}_k\mathcal{R}_\Sigma$ the moduli space of domains $\left(\overline{S};\Sigma;\Sigma^\mathit{aux}\right)$, where $\Sigma^\mathit{aux}=\{p_1,\cdots,p_k\}$. It parametrizes the (potential) variations of the conformal structure on $\overline{S}$. the positions of the points in $\Sigma$ and $\Sigma^\mathit{aux}$, up to automorphism. For any representative of an element of ${}_k\mathcal{R}_\Sigma$, we can assign it with a Floer datum, which consists of
\begin{itemize}
	\item a cylindrical end $\varepsilon_\zeta^\pm$ for each point $\zeta\in\Sigma^\mathit{cl}$, which is positive if $\zeta\in\Sigma^{\mathit{cl},\mathit{in}}$, and is negative if $\zeta\in\Sigma^{\mathit{cl},\mathit{out}}$. The cylindrical ends $\varepsilon_\zeta^\pm$ are required to be compatible with the tangent directions $\ell_\zeta$ specified above at $\zeta\in\Sigma^\mathit{cl}$;
	\item a strip-like end $\tau_\zeta^\pm:\mathbb{R}_{\pm}\times[0,1]\rightarrow S$ for each point $\zeta\in\Sigma^\mathit{op}$, which is positive if $\zeta\in\Sigma^{\mathit{op},\mathit{in}}$, and is negative if $\zeta\in\Sigma^{\mathit{op},\mathit{out}}$;
	\item a sub-closed 1-form $\nu_S\in\Omega^1(S)$ such that
	\begin{equation}
	\left(\varepsilon_\zeta^\pm\right)^\ast\nu_S=\left(\tau_\zeta^\pm\right)^\ast\nu_S=dt;
	\end{equation}
	\item a domain-dependent Hamiltonian $H_S:S\rightarrow\mathcal{H}_\ell(M)$ which satisfies
	\begin{equation}
	\left(\varepsilon_\zeta^\pm\right)^\ast H_S=H_{\lambda_\zeta,t},\left(\tau_\zeta^\pm\right)^\ast H_S=H_{L_{\zeta,0},L_{\zeta,1},t},
	\end{equation}
	for some $H_{\lambda_\zeta,t}\in\mathcal{H}_{\lambda_\zeta}(M)$, and $H_{L_{\zeta,0},L_{\zeta,1},t}$ is a time-dependent Hamiltonian which takes the form $\lambda_\zeta r+C$ on the conical end of $M$, so that all the time-1 chords of $X_{H_{L_{\zeta,0},L_{\zeta,1}},t}$ between $L_{\zeta,0}$ and $L_{\zeta,1}$ are non-degenerate, but now the parameter $t\in[0,1]$. Here, the $\lambda_\zeta$'s are real numbers associated to cylindrical and strip-like ends, as mentioned above.
	\item a domain-dependent almost complex structure $J_S:S\rightarrow\mathcal{J}(M)$ such that
	\begin{equation}
	\left(\varepsilon_\zeta^\pm\right)^\ast J_S=\left(\tau_\zeta^\pm\right)^\ast J_S=J_t,
	\end{equation}
	for some $J_t\in\mathcal{J}(M)$.
\end{itemize}

The moduli space ${}_k\mathcal{R}_\Sigma$ admits a well-defined compactification ${}_k\overline{\mathcal{R}}_\Sigma$, which is usually a real blow-up at $\zeta_\bullet$ of the corresponding Deligne-Mumford compactification. The codimension 1 boundary components of ${}_k\overline{\mathcal{R}}_\Sigma$ are covered by the degenerations of domains (as we will see in the case when $\overline{S}$ is an annulus), the real blow-up loci, the loci where two of the marked points $p_i$ and $p_{i+1}$ in $\Sigma^\mathit{aux}$ share the same modulus, and the locus where $p_k$ (when $\zeta_\bullet$ in the case of an output) or $p_1$ (when $\zeta_\bullet$ is an input) goes to the boundary of the disc centered at $\zeta_\bullet$ (the last three strata correspond to (\ref{eq:boundary1}), (\ref{eq:boundary2}) and (\ref{eq:boundary3}) respectively when $S$ is a half-cylinder). A universal and consistent choice of Floer data is an inductive choice of Floer data for each $k\geq1$ and each element of ${}_k\overline{\mathcal{R}}_\Sigma$ so that it varies smoothly over ${}_k\overline{\mathcal{R}}_\Sigma$ and has specified behaviours along the boundary strata of ${}_k\overline{\mathcal{R}}_\Sigma$. In order to construct the moduli spaces defining the operations (\ref{eq:TQFT}), we need to fix such a choice.

Fix a set of Lagrangian labellings $\left(L_C\right)$ for the boundary components in $\partial S$, and asymptotics $\left(\vec{x},\vec{y}\right)$, where $\vec{x}=(x_\zeta)$ is a set of time-1 chords of $X_{H_{L_{\zeta,0},L_{\zeta,1}},t}$ between $L_{\zeta,0}$ and $L_{\zeta,1}$, one for each $\zeta\in\Sigma^\mathit{op}$; and $\vec{y}=(y_\zeta)$ is a set of time-1 Hamiltonian orbits $y_\zeta\in\mathcal{O}_{M,\lambda_\zeta}$, one for each $\zeta\in\Sigma^\mathit{cl}$. Define the moduli space
\begin{equation}\label{eq:moduli}
{}_k\mathcal{R}_\Sigma(\vec{y};\vec{x})
\end{equation}
to be the space of pairs $\left(\left(\overline{S};\Sigma;\Sigma^\mathit{aux}\right),u\right)$, where $\left(\overline{S};\Sigma;\Sigma^\mathit{aux}\right)\in{}_k\overline{\mathcal{R}}_\Sigma$, and $u:S\rightarrow M$ is a solution of
\begin{equation}
\left\{\begin{array}{l}
\left(du-X_{H_S}\otimes\nu_S\right)^{0,1}=0, \\
u(C)\subset L_C \textrm{ for each } C\subset\partial S, \\
\lim_{s\rightarrow\pm\infty}u\left(\varepsilon_\zeta^\pm(s,\cdot)\right)=y_\zeta, \\
\lim_{s\rightarrow\pm\infty}u\left(\tau_\zeta^\pm(s,\cdot)\right)=x_\zeta.
\end{array}\right.
\end{equation}
For generic choices of perturbation data, the Gromov bordification ${}_k\overline{\mathcal{R}}_\Sigma(\vec{y};\vec{x})$ has boundary components coming from semi-stable breakings, together with maps from the boundary strata of ${}_k\overline{\mathcal{R}}_\Sigma$. A signed count of rigid elements of ${}_k\overline{\mathcal{R}}_\Sigma(\vec{y};\vec{x})$ for varying asymptotics $\left(\vec{x},\vec{y}\right)$ defines the operation (\ref{eq:TQFT}).
\bigskip

The simplest example of interest for us is a closed disc $\overline{S}$ with $\Sigma^\mathit{op}=\Sigma^\mathit{aux}=\emptyset$, and an interior marked point $\zeta_\mathit{out}\in\overline{S}$ which is an output. The boundary $\partial S$ is labelled with a Lagrangian submanifold $L\subset M$ satisfying Assumption \ref{assumption:Lag}. By counting solutions $u:S\rightarrow M$ of the corresponding Floer equation (with boundary and asymptotic conditions), this defines a Floer cocycle
\begin{equation}\label{eq:cocycle}
\phi_L^{1,0}\in\mathit{CF}^n(\lambda)
\end{equation}
for any prescribed real number $\lambda\notin\mathcal{P}_M$, whose cohomology class will be denoted by $[\![L]\!]\in\mathit{HF}^n(\lambda)$. A noteworthy fact here is that the only difference between the domains defining the operation $\phi_L^{1,0}$ and the zeroth Cieliebak-Latschev map $\mathit{CL}_0$ appeared in Section \ref{section:CL} is that the interior puncture of $S$ is now treated as an output instead of an input. 
\bigskip

For the next example, let $\overline{S}$ be a closed disc with two interior marked points, i.e. $\Sigma^\mathit{cl}=\{\zeta_\mathit{in},\zeta_\mathit{out}\}$, where $\zeta_\mathit{in}$ is an input, and $\zeta_\mathit{out}$ is an output. There is no boundary marked point, and $\partial S$ is labelled by a single Lagrangian submanifold $L$. The marked points $\zeta_\mathit{in}$ and $\zeta_\mathit{out}$ are equipped with asymptotic markers $\ell_\mathit{in}$ and $\ell_\mathit{out}$, so that $\ell_\mathit{out}$ points away from $\zeta_\mathit{in}$. Furthermore, there are $k$ additional auxiliary marked points $p_1,\cdots,p_k$ lying in a small neighborhood of $\zeta_\mathit{in}$, and they are strictly radially ordered as in (\ref{eq:radial2}), with respect to the local complex coordinate near $\zeta_\mathit{in}$. When $k=0$, the asymptotic marker $\ell_{\mathit{in}}$ is required to point towards $\zeta_\mathit{out}$, but when $k\geq1$, we require that $\ell_{\mathit{in}}$ is pointing towards $p_k$. The associated cochain level operation is a map
\begin{equation}
\phi_L^{2,0;k}:\mathit{CF}^{\ast+2k}(\lambda_1)\rightarrow\mathit{CF}^{\ast+n-1}(\lambda_0),
\end{equation}
where $\lambda_0,\lambda_1\notin\mathcal{P}_M$. See Figure \ref{fig:family1} for the domain defining the operation $\phi_L^{2,0;2}$.

\begin{figure}
	\centering
	\begin{tikzpicture}[scale=1.5]
	\filldraw[draw=black,color={black!15},opacity=0.5] (0,0) circle (1);
	\draw (0,0) circle [radius=1];
	\draw (0,0.5) node {$\times$};
	\draw (0,-0.5) node {$\times$};
	\draw [teal] [->] (0,0.5) to (-0.25,0.5);
	\draw [teal] [->] (0,-0.5) to (0,-0.75);
	\node at (0,0.7) {$\zeta_\mathit{in}$};
	\node at (0,-0.3) {$\zeta_\mathit{out}$};
	\draw [orange] (-0.25,0.5) node[circle,fill,inner sep=1pt] {};
	\draw [orange] (0.3,0.5) node[circle,fill,inner sep=1pt] {};
	\draw [orange,dashed] (0,0.5) circle [radius=0.4];
	\node [orange] at (-0.25,0.35) {\small $p_2$};
	\node [orange] at (0.3,0.35) {\small $p_1$};
	\end{tikzpicture}
	\caption{Domain of the map $\phi_L^{2,0;2}$}
	\label{fig:family1}
\end{figure}
\bigskip

We now recall the definitions of the check components of the cyclic closed-open map (cf. $\cite{sg1}$, Section 5.6.2). Let $\overline{S}$ be a closed disc with $d$ boundary marked points $\Sigma^\mathit{op}=\{\zeta_1,\cdots,\zeta_d\}$ ordered anticlockwisely, among which $\zeta_d$ is the only output. There is a unique interior marked point $\zeta_\mathit{in}$ at the origin, which is an input. Moreover, there are $k$ auxiliary interior marked points, $\Sigma^\mathit{aux}=\{p_1,\cdots,p_k\}$, which are ordered so that (\ref{eq:radial2}) is satisfied. This gives rise to a moduli space of domains, which will be denoted by ${}_k\widecheck{\mathcal{R}}_{d,\mathit{cpt}}^1$. Label the components of $\partial S$ with the Lagrangian submanifolds $L_1,\cdots,L_d$, and fix an ordered set of Hamiltonian chords $\vec{x}=(x_1,\cdots,x_d)$, where $x_i$ is a time-1 chord of $X_{H_{L_i,L_{i+1\textrm{ mod }d},t}}$, together with a Hamiltonian orbit $y_\mathit{in}\in\mathcal{O}_{M,\lambda}$. As in (\ref{eq:moduli}), one can build a moduli space
\begin{equation}\label{eq:moduli-CO}
{}_k\widecheck{\mathcal{R}}_{d,\mathit{cpt}}^1(y_\mathit{in};\vec{x}).
\end{equation}
A signed count of rigid elements in the moduli space (\ref{eq:moduli-CO}) for varying asymptotics $\vec{x}$ and $y_\mathit{in}$ defines a map
\begin{equation}\label{eq:cyclic-OC}
\phi_{L_1,\cdots,L_d}^{1,d;k}:\mathit{CF}^\ast(\lambda)\otimes\mathit{CF}^\ast(L_{d-1},L_d)\otimes\cdots\otimes\mathit{CF}^\ast(L_1,L_2)\rightarrow\mathit{CF}^\ast(L_1,L_d)[1-d-2k].
\end{equation}
Note that when $k=0$, these operations reduce to the usual closed-open string maps $\phi_{L_1,\cdots,L_d}^{1,d}$ considered in $\cite{ps4}$ and $\cite{ss}$.
\bigskip

Let us write down explicitly the first few of the maps (\ref{eq:cyclic-OC}), and try to understand their basic properties. The first one of these maps is
\begin{equation}\label{eq:1,1}
\phi_L^{1,1;k}:\mathit{CF}^{\ast+2k}(\lambda)\rightarrow\mathit{CF}^{\ast}(L,L),
\end{equation}
which is defined using a closed disc $\overline{S}$ with one interior marked point $\zeta_\mathit{in}$, which is an input and carries an asymptotic marker $\ell_{\mathit{in}}$, and one boundary marked point $\zeta_1$, which is an output. Parametrization is given by $k$ interior auxiliary marked points $p_1,\cdots,p_k$, whose positions satisfy the constraint (\ref{eq:radial2}). The asymptotic marker $\ell_{\mathit{in}}$ is required to point towards $p_k$. One can also treat the boundary marked point $\zeta_1$ of $\overline{S}$ as an input, which results in a map
\begin{equation}
\left(\phi_L^{1,1;k}\right)^\vee:\mathit{CF}^{\ast+2k}(\lambda)\otimes\mathit{CF}^{n-\ast}(L,L)\rightarrow\mathbb{K}.
\end{equation}
Going one step further, adding a marked point $\zeta_2\in\partial\overline{S}$ which serves as an input to the previously considered marked bordered Riemann surfaces defining $\phi_L^{1,1;k}$ gives rise to a map
\begin{equation}
\phi_{L_1,L_2}^{1,2;k}:\mathit{CF}^{\ast}(\lambda)\otimes\mathit{CF}^\ast(L_1,L_2)\rightarrow\mathit{CF}^\ast(L_1,L_2)[-2k-1].
\end{equation}
Consider the family of Riemann surfaces parametrized by $\mathbb{R}$ as shown in Figure \ref{fig:1,2}, with the marked points $p_1,\cdots,p_k$ being fixed, so that (\ref{eq:radial2}) holds with $\varepsilon=\frac{1}{2}$. When $k=0$, this yields a homotopy $\phi_{L_1,L_2}^{1,2}$ between $\mu^2\left(\phi_{L_2}^{1,1}(y_\mathit{in}),x\right)$ and $(-1)^{|x||y_\mathit{in}|}\mu^2\left(x,\phi_{L_1}^{1,1}(y_\mathit{in})\right)$, see \cite{ss}, (2.12). When the points $p_1,\cdots,p_k$ are allowed to vary in the disc centered at $\zeta_{\mathit{in}}$ with radius $\frac{1}{2}$, one obtains a $(2k+1)$-dimensional family $\mathcal{Y}_k$ of marked Riemann surfaces fibering over $\mathbb{R}$, whose $2k$-dimensional fibers come from the freedom of moving the points $p_1,\cdots,p_k$. The fiberwise compactification $\overline{\mathcal{Y}}_k$ of this family involves the strata coming from real blow-ups, which appear when $|p_{k-j+1}|\rightarrow0$ for some $1\leq j\leq k$. Note that this forces $|p_l|\rightarrow0$ for every $k-j+1\leq l\leq k$, so there will be cylinder bubbles at the origin $\zeta_\mathit{in}$ of $S$ containing the $j$ marked points $p_{k-j+1},\cdots,p_k$, which define elements of the moduli spaces $\mathcal{M}_j$, see Figure \ref{fig:real-blp}. The two ends of the 1-parameter family in Figure \ref{fig:1,2} still define two ends in the fiberwise compactification $\overline{\mathcal{Y}}_k$, and they contribute to the first two terms on the right-hand side of (\ref{eq:homotopy1}). Moreover, in the fiber direction, there are two additional boundary strata corresponding respectively to the loci where $|p_i|=|p_{i+1}|$ for some $1\leq i\leq k-1$ and $|p_1|=\frac{1}{2}$. Using a similar argument as in the proof of Proposition \ref{proposition:prod}, one can show that the boundary strata corresponding to $|p_i|=|p_{i+1}|$ do not contribute. On the other hand, the stratum $|p_1|=\frac{1}{2}$ gives rise to an operation, which we denote by
\begin{equation}
\hat{\phi}^{1,2;k-1}_{L_1,L_2}:\mathit{CF}^{\ast}(\lambda)\otimes\mathit{CF}^\ast(L_1,L_2)\rightarrow\mathit{CF}^\ast(L_1,L_2)[-2k-2].
\end{equation}
As a consequence, we have
\begin{equation}\label{eq:homotopy1}
\begin{split}
&\mu^1\left(\phi_{L_1,L_2}^{1,2;k}(y_\mathit{in},x)\right)+\sum_{j=0}^k\phi_{L_1,L_2}^{1,2;k-j}(\delta_j(y_\mathit{in}),x)+(-1)^{|y_\mathit{in}|}\phi_{L_1,L_2}^{1,2;k}\left(y_\mathit{in},\mu^1(x)\right) \\
&=\mu^2\left(\phi_{L_2}^{1,1;k}(y_\mathit{in}),x\right)-(-1)^{|x||y_\mathit{in}|}\mu^2\left(x,\phi_{L_1}^{1,1;k}(y_\mathit{in})\right)+\hat{\phi}^{1,2;k-1}_{L_1,L_2}(y_\mathit{in},x).
\end{split}
\end{equation}
As we will see later, the appearance of the additional term $\hat{\phi}^{1,2;k-1}_{L_1,L_2}(y_\mathit{in},x)$ in the above identity is the reason for the additional complexity that rises in our situation compared to Seidel-Solomon's original construction \cite{ss}, see Section \ref{section:q}.

\begin{figure}
	\centering
	\begin{tikzpicture}
	\filldraw[draw=black,color={black!15},opacity=0.5] (-1,-2) circle (1);
	\filldraw[draw=black,color={black!15},opacity=0.5] (1,-2) circle (1);
	\filldraw[draw=black,color={black!15},opacity=0.5] (4,-2) circle (1);
	\filldraw[draw=black,color={black!15},opacity=0.5] (7,-2) circle (1);
	\filldraw[draw=black,color={black!15},opacity=0.5] (9,-2) circle (1);
	
	\draw (0,0) to (8,0);
	\node at (0,0.2) {$-\infty$};
	\node at (8,0.2) {$\infty$};
	
	\draw (-1,-2) circle [radius=1];
	\draw (1,-2) circle [radius=1];
	\draw (4,-2) circle [radius=1];
	\draw (7,-2) circle [radius=1];
	\draw (9,-2) circle [radius=1];
	\node at (4,-2.2) {\small $\zeta_\mathit{in}$};
	\node at (-1,-2.2) {\small $\zeta_\mathit{in}$};
	\node at (9,-2.2) {\small $\zeta_\mathit{in}$};
	\node at (4,-0.8) {\small $\zeta_2$};
	\node at (4,-3.2) {\small $\zeta_1$};
	\node at (1,-0.8) {\small $\zeta_2$};
	\node at (1,-3.2) {\small $\zeta_1$};
	\node at (7,-0.8) {\small $\zeta_2$};
	\node at (7,-3.2) {\small $\zeta_1$};
	\node [orange] at (-1,-1.4) {\small $p_1$};
	\node [orange] at (4,-1.4) {\small $p_1$};
	\node [orange] at (9,-1.4) {\small $p_1$};
	\node [orange] at (-1.3,-1.8) {\small $p_2$};
	\node [orange] at (3.7,-1.8) {\small $p_2$};
	\node [orange] at (8.7,-1.8) {\small $p_2$};
	\draw [orange,dashed] (-1,-2) circle [radius=0.5];
	\draw [orange,dashed] (4,-2) circle [radius=0.5];
	\draw [orange,dashed] (9,-2) circle [radius=0.5];
	
	\draw (4,-2) node {$\times$};
	\draw (4,-1) node[circle,fill,inner sep=1pt] {};
	\draw (4,-3) node[circle,fill,inner sep=1pt] {};
	\draw (1,-1) node[circle,fill,inner sep=1pt] {};
	\draw (1,-3) node[circle,fill,inner sep=1pt] {};
	\draw (-1,-2) node {$\times$};
	\draw (7,-1) node[circle,fill,inner sep=1pt] {};
	\draw (7,-3) node[circle,fill,inner sep=1pt] {};
	\draw (9,-2) node {$\times$};
	
	\draw [orange] (4,-1.6) node[circle,fill,inner sep=1pt] {};
	\draw [orange] (3.7,-2) node[circle,fill,inner sep=1pt] {};
	\draw [orange] (-1,-1.6) node[circle,fill,inner sep=1pt] {};
	\draw [orange] (-1.3,-2) node[circle,fill,inner sep=1pt] {};
	\draw [orange] (9,-1.6) node[circle,fill,inner sep=1pt] {};
	\draw [orange] (8.7,-2) node[circle,fill,inner sep=1pt] {};
	
	\draw [teal] [->] (-1,-2) to (-1.3,-2);
	\draw [teal] [->] (4,-2) to (3.7,-2);
	\draw [teal] [->] (9,-2) to (8.7,-2);
	\end{tikzpicture}
	\caption{A section of the family $\overline{\mathcal{Y}}_2$ obtained by fixing the positions of $p_1$ and $p_2$}
	\label{fig:1,2}
\end{figure}

\begin{figure}
	\centering
	\begin{tikzpicture}
	\filldraw[draw=black,color={black!15},opacity=0.5] (0,0) circle (1.5);
	\draw (0,0) circle [radius=1.5];
	\draw [orange] [dashed] (0,0) circle [radius=0.75];
	\draw (0,0) node {$\times$};
	\draw (-1.06,1.06) node[circle,fill,inner sep=1pt] {};
	\draw (1.06,-1.06) node[circle,fill,inner sep=1pt] {};
	\node at (0,0.25) {$\zeta_\mathit{in}$};
	\node at (-1.25,1.25) {$\zeta_2$};
	\node at (1.25,-1.25) {$\zeta_1$};
	\draw [orange] (0,-0.3) node[circle,fill,inner sep=1pt] {};
	\draw [orange] (-0.5,0) node[circle,fill,inner sep=1pt] {};
	\draw [orange] (0,0.65) node[circle,fill,inner sep=1pt] {};
	\draw [teal] [->] (0,0) to (0,-0.3);
	\node [orange] at (0,-0.5) {\small $p_3$};
	\node [orange] at (-0.5,-0.2) {\small $p_2$};
	\node [orange] at (0,0.85) {\small $p_1$};
	
	\filldraw[draw=black,color={black!15},opacity=0.5] (6,0) circle (1.5);
	\draw (6,0) circle [radius=1.5];
	\draw [orange] [dashed] (6,0) circle [radius=0.75];
	\draw (6,0) node {$\times$};
	\draw (4.94,1.06) node[circle,fill,inner sep=1pt] {};
	\draw (7.06,-1.06) node[circle,fill,inner sep=1pt] {};
	\node at (6,-0.25) {input};
	\node at (4.75,1.25) {$\zeta_2$};
	\node at (7.25,-1.25) {$\zeta_1$};
	\draw [orange] (6,0.65) node[circle,fill,inner sep=1pt] {};
	\node [orange] at (6,0.85) {\small $p_1$};
	\draw [teal] [->] (6,0) to (6,0.4);
	
	\filldraw[draw=black,color={black!15},opacity=0.5] (6,3.75) circle (1.5);
	\draw (6,3.75) circle [radius=1.5];
	\draw (6,3.3) node {$\times$};
	\draw (6,4.2) node {$\times$};
	\draw [blue] [dashed] (6,3.3) to (6,0);
	\draw [orange] (6,2.55) node[circle,fill,inner sep=1pt] {};
	\draw [orange] (5.5,2.85) node[circle,fill,inner sep=1pt] {};
	\node [orange] at (6.25,2.55) {\small $p_3$};
	\node [orange] at (5.5,3.05) {\small $p_2$};
	\draw [dashed] (6,3.75) ellipse (1.5 and 0.45);
	\node at (6,4.45) {$\zeta_\mathit{in}$};
	\node at (6,3.55) {output};
	
	\draw [->] (2,0) to (4,0);
	\node at (3,0.25) {$|p_2|\rightarrow0$};
	\end{tikzpicture}
	\caption{A representative of an element of the moduli space ${}_3\widecheck{\mathcal{R}}_{2,\mathit{cpt}}^1$ and the sphere bubble at the origin when $|p_2|\rightarrow0$}
	\label{fig:real-blp}
\end{figure}

For completeness, we end this subsection by recalling some known operations constructed in $\cite{ps4}$, which are defined by considering a 1-parameter family of annuli $\overline{S}$. The outer boundary of $\overline{S}$ is labelled by the Lagrangian submanifold $L_0$, and the inner boundary of $\overline{S}$ is labelled by the Lagrangian submanifold $L_1$. There is a unique boundary marked point $\zeta_1$ on $\partial\overline{S}$, which is an input. See Figure \ref{fig:annulus}. Depending on whether $\zeta_1$ lies on the boundary component labelled by $L_1$ or $L_0$, there are operations
\begin{equation}
\psi_{L_0,L_1}^{0,1}:\mathit{CF}^1(L_1,L_1)\rightarrow\mathbb{K};
\end{equation}
\begin{equation}
\left(\psi_{L_0,L_1}^{0,1}\right)^\vee:\mathit{CF}^1(L_0,L_0)\rightarrow\mathbb{K},
\end{equation}
and they satisfy
\begin{equation}\label{eq:psi1}
\psi_{L_0,L_1}^{0,1}\left(\mu^1(x)\right)=(-1)^{n(n+1)/2}\mathrm{Str}\left(\mu^2(x,\cdot)\right)-(-1)^n\left\langle\phi_{L_0}^{1,0},\left(\phi_{L_1}^{1,1}\right)^\vee(x)\right\rangle,
\end{equation}
\begin{equation}\label{eq:psi2}
\left(\psi_{L_0,L_1}^{0,1}\right)^\vee\left(\mu^1(x)\right)=(-1)^{n(n+1)/2}\mathrm{Str}\left(\mu^2(\cdot,x)\right)-\left\langle\phi_{L_1}^{1,0},\left(\phi_{L_0}^{1,1}\right)^\vee(x)\right\rangle,
\end{equation}
where by $\mathrm{Str}$ we mean the supertrace. In the second term on the right hand side of (\ref{eq:psi1}), we take $\phi_{L_0}^{1,0}\in\mathit{CF}^n(-\lambda)$ and pair it with the result of $\left(\phi_{L_1}^{1,1}\right)^\vee:\mathit{CF}^\ast(L_1,L_1)\rightarrow\mathit{CF}^{\ast+n}(\lambda)$. This is a version of the \textit{Cardy relation}, see $\cite{ps4}$, Section (4c) for details.

\begin{figure}
	\centering
	\begin{tikzpicture}
		\filldraw[draw=black,color={black!15},opacity=0.5] (0,0) circle (1.6);
		\filldraw[draw=black,color={black!0}] (0,0) circle (0.8);
		\draw (0,0) circle [radius=1.6];
		\draw (0,0) circle [radius=0.8];
		\draw (0,0.8) node[circle,fill,inner sep=1pt] {};
		\node at (0,0.55) {\small $\zeta_1$};
		\node at (0,-0.55) {\small $L_0$};
		\node at (0,-1.85) {\small $L_1$};
		
		\filldraw[draw=black,color={black!15},opacity=0.5] (5,0) circle (1.6);
		\filldraw[draw=black,color={black!0}] (5,0) circle (0.8);
		\draw (5,0) circle [radius=1.6];
		\draw (5,0) circle [radius=0.8];
		\draw (5,1.6) node[circle,fill,inner sep=1pt] {};
		\node at (5,1.85) {\small $\zeta_1$};
		\node at (5,-0.55) {\small $L_0$};
		\node at (5,-1.85) {\small $L_1$};
	\end{tikzpicture}
	\caption{Domains defining the operations $\psi_{L_0,L_1}^{0,1}$ and $\left(\psi_{L_0,L_1}^{0,1}\right)^\vee$}
	\label{fig:annulus}
\end{figure}

\subsection{Seidel-Solomon's construction}\label{section:q}

Let $M$ be a $2n$-dimensional Liouville manifold, and let $\lambda\notin\mathcal{P}_M$ be a real number. Fix an arbitrary cohomology class $\tilde{b}\in\mathit{HF}_{S^1}^1(\lambda)$, with its cochain level representative given by
\begin{equation}\label{eq:coboundary}
\tilde{\beta}:=\sum_{k=0}^\infty\beta_k\otimes u^{-k}.
\end{equation}
Note that the sum on the right-hand side of (\ref{eq:coboundary}) is actually finite since only finitely many cochains $\beta_k\in\mathit{CF}^{2k+1}(\lambda)$ are non-zero. When $\beta_k=0$ for all $k\geq1$, Seidel and Solomon constructed in $\cite{ss}$ a derivation on Floer cohomologies $\mathit{HF}^\ast(L_0,L_1)$ of any simply-connected Lagrangian submanifolds. In this subsection, we study a higher order analogue of Seidel-Solomon's construction. This is in general not well-defined, and only works in very special cases, say, when $L_0=L_1=L$ is a Lagrangian sphere of some odd dimension $n\geq3$, which turns out to be enough for the purpose of proving Theorem \ref{theorem:main}.
\bigskip

For any Lagrangian submanifold $L\subset M$ satisfying Assumption \ref{assumption:Lag}, consider the Floer cochain
\begin{equation}\label{eq:cochain}
\sum_{k=0}^\infty\phi_{L}^{1,1;k}(\beta_k)\in\mathit{CF}^1(L,L)
\end{equation}
obtained by applying the closed-open maps (\ref{eq:1,1}). This is in general not a cocycle, due to the fact that the boundary of ${}_k\overline{\widecheck{\mathcal{R}}}_{1,\mathit{cpt}}^1(y_\mathit{in};x)$ contains the stratum ${}_{k-1}\overline{\mathcal{R}}_{1,\mathit{cpt}}^{S^1}(y_\mathit{in},x)$, whose contribution cannot be ignored in general.

From now on, let $L$ be a Lagrangian sphere of some odd dimension $n\geq3$. In this case, one can arrange so that $\mathit{CF}^2(L,L)=0$, which means (\ref{eq:cochain}) defines a cocycle. However, for our purposes, we will need to show that when $\tilde{b}$ is a cyclic dilation, this is the case regardless of the Floer data used in the definition of $\mathit{HF}^\ast(L,L)$.

\begin{lemma}\label{lemma:unique}
The image of the cyclic dilation $\tilde{b}\in\mathit{SH}_{S^1}^1(M)$ under the $S^1$-equivariant Viterbo functoriality $[\tilde{v}^!]:\mathit{SH}_{S^1}^\ast(M)\rightarrow\mathit{SH}_{S^1}^\ast(T^\ast L)$ (cf. (\ref{eq:equi-Vit})) is of the form
\begin{equation}\label{eq:mod}
\alpha_L\cdot\mathbf{I}(b)\textrm{ mod }u^{-1}, 
\end{equation}
where $\alpha_L\in\mathbb{K}^\times$ and the class $b\in\mathit{SH}^1(T^\ast S^n)$ is a dilation.
\end{lemma}
\begin{proof}
It follows from the diagram (\ref{eq:diagram2}) that $[\tilde{v}^!](\tilde{b})\in\mathit{SH}_{S^1}^1(T^\ast L)$ is a cyclic dilation. On the other hand, according to the computation of $\cite{cjy}$, Theorem 2 (2), we have the isomorphism
\begin{equation}
H_{n-\ast}(\mathcal{L}S^n;\mathbb{K})\cong\frac{\mathbb{K}[x]}{(x^2)}\otimes\mathbb{K}[y],
\end{equation}
where $|x|=n$, $|y|=1-n$, and $n$ is odd. This implies that $\mathit{SH}^0(T^\ast S^n)\cong\mathbb{K}$, so the class $\tilde{b}$ satisfies $\mathbf{B}\circ[\tilde{v}^!](\tilde{b})=\alpha_L$ for some $\alpha_L\in\mathbb{K}^\times$. It follows from the computations in \cite{ft}, Section 4.1 that
\begin{equation}
H_{n-1}^{S^1}(\mathcal{L}S^n;\mathbb{K})\cong\mathbb{K}\langle\mathbf{I}(x\otimes y)\rangle\oplus\mathbb{K}\left\langle\frac{\gamma^{\frac{n-1}{2}}}{(\frac{n-1}{2})!}\right\rangle,
\end{equation}
where $\gamma$ is the degree 2 generator dual to $x\in H^2(\mathbb{CP}^\infty;\mathbb{Z})$, which gives rise to the map $\mathbf{S}:H_{\ast+1}^{S^1}(\mathcal{L}S^n;\mathbb{K})\rightarrow H_{\ast-1}^{S^1}(\mathcal{L}S^n;\mathbb{K})$ in the Gysin sequence (\ref{eq:string}). Moreover, the marking map $\mathbf{B}:H_\ast^{S^1}(\mathcal{L}S^n;\mathbb{K})\rightarrow H_{\ast+1}(\mathcal{L}S^n;\mathbb{K})$ sends $\frac{\gamma^{\frac{n-1}{2}}}{(\frac{n-1}{2})}!$ to 0, and
\begin{equation}
\mathbf{B}\circ\mathbf{I}(x\otimes y)=\Delta(x\otimes k)=1.
\end{equation}
Since $x\otimes y$ defines a dilation under the BV algebra isomorphism $H_{n-\ast}(\mathcal{L}S^n;\mathbb{K})\cong\mathit{SH}^\ast(T^\ast S^n)$ established in $\cite{ma2}$, and the $u^0$ part of $\mathit{SH}_{S^1}^\ast(T^\ast L)$ is given by the image of the erasing map, we have proved (\ref{eq:mod}).
\end{proof}

\begin{proposition}\label{proposition:cocycle}
Let $\tilde{b}\in\mathit{HF}_{S^1}^1(\lambda)$ be a cyclic dilation, then (\ref{eq:cochain}) defines a Floer cocycle for any choice of Floer data which defines $\mathit{CF}^\ast(L,L)$.
\end{proposition}
\begin{proof}
The idea of the proof is to regard the operation $\phi_L^{1,1;k}$ as (up to homotopy) the composition of the (parametrized version) of the Viterbo functoriality relating the symplectic cochain complex $\mathit{SC}^\ast(M)$ of the ambient space to that of a Weinstein neighborhood $\mathit{SC}^\ast(T^\ast L)$, and the ordinary closed-open map $\phi_L^{1,1}$ defined in the Weinstein neighborhood of $L$. Since the image of the cyclic dilation under the Viterbo map has been determined by Lemma \ref{lemma:unique}, this allows us to conclude that $\sum_{k=0}^\infty\phi_L^{1,1;k}(\beta_k)$ is a cocycle. The proof is divided into four steps.

\begin{paragraph}{Step 1: Realizing $\mathit{SH}^\ast(T^\ast L)$ as the Floer cohomology of a Hamiltonian on $M$.}
Consider a family of marked bordered Riemann surfaces
\begin{equation}\label{eq:Sq}
\left(\overline{S}_q;\zeta_{\mathit{in}},\zeta_1;\frac{p_1}{q},\cdots,\frac{p_k}{q}\right)
\end{equation}
parametrized by $q\in[1,\infty)$, where $\overline{S}_q$ is the closed unit disc centered at the origin, whose center $\zeta_{\mathit{in}}$ is an interior puncture serving as an input, and $\zeta_1$ is a boundary puncture serving as an output. The marked points $p_1,\cdots,p_k$ are ordered so that (\ref{eq:radial2}) holds. As part of our Floer data, denote by 
\begin{equation}
\varepsilon_q^+:[-\log q,\infty)\times S^1\rightarrow S_q
\end{equation}
the family of positive cylindrical ends fixed at $\zeta_{\mathit{in}}$. Note that for $q=1$, (\ref{eq:Sq}) is just the domain defining $\phi_L^{1,1,;k}$ considered in Section \ref{section:CO}. We equip these domains with a family of Hamiltonians $(H_{S_q})$ so that $H_{S_1}:S_1\rightarrow\mathcal{H}_\ell(M)$ is part of the Floer data defining the operation $\phi_L^{1,1;k}$. In particular, $(\varepsilon_1^+)^\ast H_{S_1}=H_{\lambda,t}\in\mathcal{H}_\lambda(M)$ for some $\lambda\notin\mathcal{P}_M$. For the purpose of our argument, we need to work with the Floer cochain complex $\mathit{CF}^\ast_{(a,b]}(\lambda)$ truncated in the finite action window $(a,b]$, where $-\infty<a<b<\infty$. This is defined as the quotient complex $\mathit{CF}^\ast_{>a}(\lambda)/\mathit{CF}^\ast_{>b}(\lambda)$. Assume that there are sequences of real numbers $\{\nu_i\}$, $\{r_i\}$, $\{\eta_i\}$ and $\{\lambda_i\}$ so that
\begin{itemize}
	\item $\nu_i>0$ is smaller than the distance from $\eta_i$ to $\mathcal{P}_{T^\ast L}$, and $\nu_i\rightarrow0$ as $i\rightarrow\infty$;
	\item $r_i>\max\left\{1,\frac{\eta_i-a}{\nu_i}\right\}$;
	\item $\frac{\eta_i}{4}<\lambda_i<\frac{\eta_i}{2}$ and $\lambda_i\notin\mathcal{P}_M$.
\end{itemize}
Let $H_{\eta_i,\lambda_i,t}$ be a small time-dependent perturbation of a one step Hamiltonian $H_{\eta_i,\lambda_i}$ which vanishes on the Liouville subdomain $D^\ast L\subset\overline{M}$; linear and has slope $\eta_i$ on $[1,r_i]\times\partial D^\ast L$ for some $r_i>0$; constant and equal to $\eta_i(r_i-1)$ on $\psi^{\log r_i}\left(\overline{M}\setminus D^\ast L\right)$, where $\psi^{\log r_i}$ is the Liouville flow at time $\log r_i$; and equals $\eta_i(r_i-1)+\lambda_i(r-r_i)$ on $[r_i,\infty)\times\partial\overline{M}$. Consider the Floer cochain complexes $\mathit{CF}^\ast_{(a,b]}(H_{\eta_i,\lambda_i,t})$, it follows from $\cite{co}$, Lemma 5.1 that there is an identification 
\begin{equation}\label{eq:id1}
\mathit{CF}^\ast_{(a,b]}(H_{\eta_i,\lambda_i,t})\cong\mathit{CF}^\ast_{(a,b]}(T^\ast L,\eta_i)
\end{equation}
between the Floer cochain complexes defined on $M$ and those on its Liouville subdomain $D^\ast L$. In particular,
\begin{equation}
{\varinjlim}_i\mathit{HF}^\ast_{(a,b]}(H_{\eta_i,\lambda_i,t})\cong\mathit{SH}_{(a,b]}^\ast(T^\ast L):={\varinjlim}_i\mathit{HF}_{(a,b]}^\ast(T^\ast L,\eta_i).
\end{equation}
\end{paragraph}

\begin{paragraph}{Step 2: Interpreting Viterbo functoriality as continuation maps.}
For each Hamiltonian of the form $H_{\eta,\lambda,t}$, there exists a Hamiltonian $H_{\lambda,t}\in\mathcal{H}_\lambda(M)$ such that $H_{\lambda,t}\leq H_{\eta,\lambda,t}$. Choose a monotone homotopy $h_s$ between the Hamiltonians $H_{\lambda,t}$ and $H_{\eta,\lambda,t}$ so that the continuation maps
\begin{equation}
\kappa_k:\mathit{CF}_{(a,b]}^{\ast+2k}(\lambda)\rightarrow\mathit{CF}^\ast_{(a,b]}(H_{\eta,\lambda,t}),
\end{equation}
which are variants of the components $\kappa_k^{\lambda_1,\lambda_2}$ in (\ref{eq:continuation}), are well-defined for all $k\geq0$, and altogether they form the (truncated) $S^1$-equivariant Viterbo functoriality
\begin{equation}
[\tilde{v}^!]:\mathit{SH}_{S^1,(a,b]}^\ast(M)\rightarrow\mathit{SH}_{S^1,(a,b]}^\ast(T^\ast L)
\end{equation}
after passing to direct limits. We will denote by
\begin{equation}
\tilde{v}_k^!:\mathit{SC}^{\ast+2k}(M)\rightarrow\mathit{SC}^\ast(T^\ast L)
\end{equation}
the components of the $S^1$-complex morphism $\tilde{v}^!$. Let $\rho$ be a non-negative, monotone, non-decreasing cut-off function such that
\begin{equation}
\rho(s)=\left\{\begin{array}{ll}0 & s\ll 0, \\ 1 & s\gg 0. \end{array}\right.
\end{equation}
We require that the family of Hamiltonians $(H_{S_q})_{q\geq1}$ to satisfy
\begin{equation}\label{eq:homo}
(\varepsilon_q^+)^\ast H_{S_q}=\rho(s+q)h_s.
\end{equation}
Note that this is compatible with our requirement that $(\varepsilon_1^+)^\ast H_{S_1}=H_{\lambda,t}$.
\end{paragraph}

\begin{paragraph}{Step 3: Degeneration of domain.}
Assume that the input $\zeta_{\mathit{in}}$ receives generators from the truncated complex $\mathit{CF}^\ast_{(a,b]}(\lambda)$. When $q\rightarrow\infty$, the points $p_1,\cdots,p_k$ move to the origin and a cylinder with all the auxiliary marked points bubbles off at $\zeta_{\mathit{in}}$. According to (\ref{eq:homo}), this $k$-point angle decorated cylinder is equipped with the Floer data defining the continuation map $\kappa_k$. On the other hand, the component carrying the boundary marked point defines an operation
\begin{equation}
\tilde{\phi}_L^{1,1}:\mathit{CF}^\ast_{(a,b]}(H_{\eta,\lambda,t})\rightarrow\mathit{CF}^\ast(L,L),
\end{equation}
which is a slight variation of the usual closed-open string map $\phi_L^{1,1}:\mathit{CF}^\ast(\lambda)\rightarrow\mathit{CF}^\ast(L,L)$. There are also boundary strata when $q\in(1,\infty)$, which correspond to the cases when a cylinder, which may contain several of the marked points $p_1,\cdots,p_k$ breaks off at $\zeta_{\mathit{in}}$, and when the coordinates of the auxiliary marked points satisfy $|p_i|=|p_{i+1}|$ for some $1\leq i\leq k-1$. We know from the previous discussions that the latter strata are non-rigid by appropriate choices of Floer data, which can clearly be achieved here. For the former ones, we can arrange the choices of Floer data so that $H_{S_q}$ pulls back to $H_{\lambda,t}$ under the negative cylindrical end of the marked cylinder bubbled off at $\zeta_{\mathit{in}}$. As a consequence, their contributions can be identified as
\begin{equation}
\sum_{j=0}^k t_L^{1,1;j}\left(\delta_{k-j}(y)\right),
\end{equation}
where
\begin{equation}
t_L^{1,1;j}:\mathit{CF}^{\ast+2j+1}(\lambda)\rightarrow\mathit{CF}^\ast(L,L)
\end{equation}
is the operation defined by counting rigid elements in the moduli space of triples
\begin{equation}
\left(q,\left(S_q,\frac{p_1}{q},\cdots,\frac{p_j}{q}\right),u\right),
\end{equation}
where $u:S_q\rightarrow M$ is a solution to the equation $\left(du-X_{H_{S_q}}\otimes\nu_{S_q}\right)^{0,1}=0$ with boundary on $L$, and asymptotics at $\zeta_\mathit{in}$ and $\zeta_1$ being generators of $\mathit{CF}^\ast_{(a,b]}(\lambda)$ and $\mathit{CF}^\ast(L,L)$, respectively.

To sum up, we have proved the cochain level identity
\begin{equation}\label{eq:degeneration}
\phi_L^{1,1;k}(y)-\tilde{\phi}_L^{1,1}\circ\kappa_k(y)=\mu^1\left(t_L^{1,1;k}(y)\right)+\sum_{j=0}^k t_L^{1,1;j}\left(\delta_{k-j}(y)\right)
\end{equation}
for any $y\in\mathit{CF}^\ast_{(a,b]}(\lambda)$.
\end{paragraph}

\begin{paragraph}{Step 4: Applying to the cochains $\beta_k$.}
In view of (\ref{eq:id1}) and the identification (obtained by rescaling the Hamiltonian perturbation) of $\mathit{CF}^\ast(L,L)$ with a Floer cochain complex $\mathit{CF}_{T^\ast L}^\ast(L,L)$ computed in the Liouville subdomain $D^\ast L\subset\overline{M}$, the map $\tilde{\phi}_L^{1,1}$ can be identified with the usual closed-open map
\begin{equation}\label{eq:local}
\phi^{1,1}_{L\subset T^\ast L}:\mathit{CF}^\ast_{(a,b]}(T^\ast L,\eta)\rightarrow\mathit{CF}_{T^\ast L}^\ast(L,L)
\end{equation}
for the Liouville manifold $T^\ast L$ by an appropriate version of the maximum principle. See $\cite{as}$, Lemma 7.4. Choosing the action window $(a,b]$ so that $\beta_k\in\mathit{CF}^{2k+1}_{(a,b]}(\lambda)$ for all $k\geq1$ (which is possible since $\beta_k\neq0$ for only finitely many $k$), it follows from (\ref{eq:degeneration}) and (\ref{eq:local}) that
\begin{equation}\label{eq:sum}
\sum_{k=0}^\infty\phi_L^{1,1;k}(\beta_k)-\phi^{1,1}_{L\subset T^\ast L}\left(\sum_{k=0}^\infty\kappa_k(\beta_k)\right)=\mu^1\left(\sum_{k=0}^\infty t_L^{1,1;k}(\beta_k)\right)+\sum_{k=0}^\infty\sum_{j=0}^k t_L^{1,1;j}\left(\delta_{k-j}(\beta_k)\right).
\end{equation}
The double sum on the right-hand side of (\ref{eq:sum}) vanishes by the equivariant cocycle condition satisfied by $\tilde{\beta}=\sum_{k=0}^\infty\beta_k\otimes u^{-k}$. For the reader's convenience, we will spell out the details. The equivariant differential $\delta_\mathit{eq}=\sum_{k=0}^\infty\delta_ku^k$ applied to $\tilde{\beta}$ gives
\begin{equation}
\delta_\mathit{eq}(\tilde{\beta})=\sum_{k=0}^\infty\sum_{j=0}^k\delta_{k-j}(\beta_k)\otimes u^{-j}=0,
\end{equation}
which implies that for any fixed $j$,
\begin{equation}
\sum_{k=j}^\infty t_L^{1,1;j}(\delta_{k-j}(\beta_k))=t_L^{1,1;j}\left(\sum_{k=j}^\infty\delta_{k-j}(\beta_k)\right)=0.
\end{equation}
After passing to the inverse limit $a\rightarrow-\infty$ (which is finite due to the existence of the lower bound of the action), and the direct limits $b\rightarrow\infty$, $\lambda\rightarrow\infty$ and $\eta\rightarrow\infty$, we see that $\sum_{k=0}^\infty\phi_L^{1,1;k}(\beta_k)$ is cohomologous to $\phi^{1,1}_{L\subset T^\ast L}\left(\sum_{k=0}^\infty v^!_k(\beta_k)\right)$. By Lemma \ref{lemma:unique}, the image of the cyclic dilation $\tilde{\beta}$ under the $S^1$-equivariant Viterbo functoriality is $\alpha_L\cdot\beta\otimes1$ mod $u^{-1}$, where $\alpha_L\in\mathbb{K}^\times$ and $\beta\in\mathit{SC}^1(T^\ast L)$ is a dilation. It follows that
\begin{equation}
\tilde{v}^!(\tilde{\beta})=\sum_{k=0}^\infty v^!_k(\beta_k)\in\mathit{SC}^1(T^\ast L)\subset\mathit{SC}^1(T^\ast L)\otimes_\mathbb{K}\mathbb{K}((u))/u\mathbb{K}[[u]], 
\end{equation}
which is equal to $\alpha_L\cdot\beta$. Since $\phi^{1,1}_{L\subset T^\ast L}$ is a chain map, and $\phi^{1,1}_{L\subset T^\ast L}(\alpha_L\cdot\beta)$ is a Floer cocycle, so is $\sum_{k=0}^\infty\phi_L^{1,1;k}(\beta_k)$.
\end{paragraph}
\end{proof}

Since $\mathit{HF}^1(L,L)=0$, it follows that there is a cochain $\tilde{\gamma}_L\in\mathit{CF}^0(L,L)$, which satisfies
\begin{equation}\label{eq:def}
\mu^1(\tilde{\gamma}_L)=\sum_{k=0}^\infty\phi_{L}^{1,1;k}(\beta_k).
\end{equation}
Two $\tilde{\gamma}_L$'s are considered to be equivalent if their difference is a degree 0 coboundary. Note that the set of equivalence classes of choices of $\tilde{\gamma}_L$ is an affine space over $H^0(L;\mathbb{K})\cong\mathbb{K}$. Following $\cite{ss}$, the pair $\widetilde{L}=(L,\tilde{\gamma}_L)$ will be called a \textit{$\tilde{b}$-equivariant} Lagrangian sphere.
\bigskip

Let $\widetilde{L}_0=(L_0,\tilde{\gamma}_{L_0})$, $\widetilde{L}_1=(L_1,\tilde{\gamma}_{L_1})$ be two odd-dimensional $\tilde{b}$-equivariant Lagrangian spheres. One can define a cochain level endomorphism
\begin{equation}
\phi_{\widetilde{L}_0,\widetilde{L}_1}:\mathit{CF}^\ast(L_0,L_1)\rightarrow\mathit{CF}^\ast(L_0,L_1)
\end{equation}
by
\begin{equation}\label{eq:endo}
\phi_{\widetilde{L}_0,\widetilde{L}_1}(x):=\sum_{k=0}^\infty\phi_{L_0,L_1}^{1,2;k}(\beta_k,x)-\mu^2(\tilde{\gamma}_{L_1},x)+\mu^2(x,\tilde{\gamma}_{L_0}).
\end{equation}
Note that $\phi_{\widetilde{L}_0,\widetilde{L}_1}$ is in general not a chain map. In fact, it follows from (\ref{eq:homotopy1}) and the fact that $\sum_{k=0}^\infty\beta_k\otimes u^{-k}$ defines a cocycle in $\mathit{CF}^1_{S^1}(\lambda)$ such that
\begin{equation}\label{eq:chain}
\mu^1\left(\phi_{\widetilde{L}_0,\widetilde{L}_1}(x)\right)=\phi_{\widetilde{L}_0,\widetilde{L}_1}\left(\mu^1(x)\right)+\sum_{k=1}^\infty\hat{\phi}_{L_0,L_1}^{1,2;k-1}(\beta_k,x).
\end{equation}

\begin{proposition}
Assume that $\tilde{b}\in\mathit{HF}_{S^1}^1(\lambda)$ is a cyclic dilation. Let $L_0=L_1=L$, and $\tilde{\gamma}_{L_0}=\tilde{\gamma}_{L_1}$ in the above, then $\phi_{\widetilde{L},\widetilde{L}}$ is a chain map.
\end{proposition}
\begin{proof}
The idea of the proof is similar to that of	Proposition \ref{proposition:cocycle}. Using a degeneration of domain argument, the operation $\phi_{L,L}^{1,2;k}$ can be shown to coincide up to homotopy with the composition of the (parametrized) Viterbo functoriality and the usual closed-open map $\phi_{L,L}^{1,2}$ defined in the Weinstein neighborhood $D^\ast L\subset M$. Lemma \ref{lemma:unique} then enables us to understand the map $\sum_{k=0}^\infty\phi_{\widetilde{L},\widetilde{L}}(\beta_k)$.
	
Denote by
\begin{equation}
t_{L,L}^{1,2;k}:\mathit{CF}^{\ast+2k+2}(\lambda)\otimes\mathit{CF}^\ast(L,L)\rightarrow\mathit{CF}^\ast(L,L)
\end{equation}
the operation defined by a family of domains
\begin{equation}
\left(\overline{S}_q;\zeta_{\mathit{in}},\zeta_1,\zeta_2;\frac{p_1}{q},\cdots,\frac{p_k}{q}\right)
\end{equation}
parametrized by the auxiliary marked points $p_1,\cdots,p_k$ and $q\in[1,\infty)$, which is basically the family of domains appeared in the proof of Proposition \ref{proposition:cocycle}, except that now there is an additional boundary marked point $\zeta_2$, which is an input. A similar degeneration of domain argument as in the proof of Proposition \ref{proposition:cocycle} implies that
\begin{equation}
\begin{split}
\phi_{L,L}^{1,2;k}(y,x)&=\phi_{L,L\subset T^\ast L}^{1,2}\left(v_k^!(y),x\right)+\mu^1\left(t_{L,L}^{1,2;k}(y,x)\right)+t_{L,L}^{1,2;k}\left(y,\mu^1(x)\right)\\
&+\sum_{j=0}^kt_{L,L}^{1,2;j}\left(\delta_{k-j}(y),x\right),
\end{split}
\end{equation}
where
\begin{equation}
\phi_{L,L\subset T^\ast L}^{1,2}:\mathit{SC}^\ast(T^\ast L)\otimes\mathit{CF}^\ast(L,L)\rightarrow\mathit{CF}^\ast(L,L)[-2k-1]
\end{equation}
is a closed-open map defined in the Liouville subdomain $D^\ast L\subset\overline{M}$. When applied to the cyclic dilation $\tilde{\beta}\in\mathit{SC}_{S^1}^1(M)$, we obtain by Lemma \ref{lemma:unique} that
\begin{equation}\label{eq:t}
\begin{split}
\sum_{k=0}^\infty\phi_{L,L}^{1,2;k}(\beta_k,x)&=\phi_{L,L\subset T^\ast L}^{1,2}\left(\sum_{k=0}^\infty v_k^!(\beta_k),x\right)+\mu^1\left(\sum_{k=0}^\infty t_{L,L}^{1,2;k}(\beta_k,x)\right) \\
&+\sum_{k=0}^\infty t_{L,L}^{1,2;k}\left(\beta_k,\mu^1(x)\right)+\sum_{k=0}^\infty\sum_{j=0}^kt_{L,L}^{1,2;j}(\delta_{k-j}(\beta_k),x) \\
&=\alpha_L\cdot\phi_{L,L\subset T^\ast L}^{1,2}(\beta,x)+\mu^1\left(\sum_{k=0}^\infty t_{L,L}^{1,2;k}(\beta_k,x)\right)+\sum_{k=0}^\infty t_{L,L}^{1,2;k}\left(\beta_k,\mu^1(x)\right),
\end{split}
\end{equation}
where $\alpha_L\in\mathbb{K}^\times$ and the cocycle $\beta\in\mathit{SC}^1(T^\ast L)$ represents the dilation. Note that in the computations above, the term $\sum_{k=0}^\infty\sum_{j=0}^kt_{L,L}^{1,2;j}\left(\delta_{k-j}(\beta_k),x\right)$ vanishes because $\tilde{\beta}$ defines an $S^1$-equivariant cocycle, see the proof of Proposition \ref{proposition:cocycle}. We then have
\begin{equation}\label{eq:t1}
\begin{split}
\phi_{\widetilde{L},\widetilde{L}}(x)&=\alpha_L\cdot\phi_{L,L\subset T^\ast L}^{1,2}(\beta,x)+\mu^1\left(\sum_{k=0}^\infty t_{L,L}^{1,2;k}(\beta_k,x)\right)+\sum_{k=0}^\infty t_{L,L}^{1,2;k}\left(\beta_k,\mu^1(x)\right) \\
&-\mu^2(\tilde{\gamma}_L,x)+\mu^2(x,\tilde{\gamma}_L) \\
&=\alpha_L\cdot\phi_{L,L\subset T^\ast L}^{1,2}(\beta,x)+\mu^2\left(\sum_{k=0}^\infty t_L^{1,1;k}(\beta_k),x\right)-\mu^2\left(x,\sum_{k=0}^\infty t_L^{1,1;k}(\beta_k)\right) \\
&-\mu^2(\tilde{\gamma}_L,x)+\mu^2(x,\tilde{\gamma}_L),
\end{split}
\end{equation}
where the second equality holds because $t_{L,L}^{1,2;k}$ gives the homotopy between $\mu^2\left(t_L^{1,1;k}(\beta_k),x\right)$ and $\mu^2\left(x,t_L^{1,1;k}(\beta_k)\right)$. The proposition now follows from the fact that
\begin{equation}\label{eq:chain-map}
\phi_{\widetilde{L},\widetilde{L}\subset T^\ast L}:=\alpha_L\cdot\phi_{L,L\subset T^\ast L}^{1,2}(\beta,x)-\mu^2(\gamma_{L},x)+\mu^2(x,\gamma_{L}),
\end{equation}
where
\begin{equation}
\gamma_L:=\tilde{\gamma}_L-\sum_{k=0}^\infty t_L^{1,1;k}(\beta_k),
\end{equation}
is a chain map, which is proved in $\cite{ss}$, Section 4.
\end{proof}

In this way, we obtain an endomorphism
\begin{equation}
\Phi_{\widetilde{L},\widetilde{L}}=\left[\phi_{\widetilde{L},\widetilde{L}}\right]:\mathit{HF}^\ast(L,L)\rightarrow\mathit{HF}^\ast(L,L),
\end{equation}
which is independent of our choice of $\tilde{\gamma}_L$. To understand the endomorphism $\Phi_{\widetilde{L},\widetilde{L}}$ is equivalent to understanding its action on $\mathit{HF}^0(L,L)$ and $\mathit{HF}^n(L,L)$.

\begin{proposition}\label{proposition:deg-0}
$\Phi_{\widetilde{L},\widetilde{L}}$ acts trivially on $\mathit{HF}^0(L,L)$.
\end{proposition}
\begin{proof}
If we use the fact that $\tilde{b}\in\mathit{HF}_{S^1}^1(\lambda)$ is a cyclic dilation, then by (\ref{eq:t1}), the action of $\Phi_{\widetilde{L},\widetilde{L}}$ can be identified with the action of $\left[\phi_{\widetilde{L},\widetilde{L}\subset T^\ast L}\right]$ on $\mathit{HF}^\ast(L,L)$ computed in the Weinstein neighborhood $D^\ast L\subset M$. It is proved in $\cite{ss}$, Section 4 that $\left[\phi_{\widetilde{L},\widetilde{L}\subset T^\ast L}\right]$ is a derivation. Since $\mathit{HF}^0(L,L)\cong\mathbb{K}$, the action of $\left[\phi_{\widetilde{L},\widetilde{L}\subset T^\ast L}\right]$ on $\mathit{HF}^0(L,L)$ must have weight 0. Alternatively, one can make the argument independent of the fact that $\tilde{b}$ is a cyclic dilation. This is inspired by $\cite{ss}$, Remark 4.4. Consider the moduli space $\mathcal{M}_{k,+}(y,L)$ of pairs $\left((Z^+,p_1,\cdots,p_k),u\right)$ studied in Section \ref{section:CL}, there is a double evaluation map $\mathit{Ev}:S^1\times\overline{\mathcal{M}}_{k,+}(y,L)\rightarrow L\times L$ defined by
\begin{equation}
(t,u)\mapsto\left(u(0,0),u(0,t)\right).
\end{equation}
The formal sum of $\mathit{Ev}_\ast\left[S^1\times\overline{\mathcal{M}}_{k,+}(\beta_k,L)\right]$ for each $k\geq0$ represents a topological cycle in $L\times L$. If we use the singular chain model of Lagrangian Floer cohomology $\mathit{HF}^\ast(L,L)\cong H^\ast(L;\mathbb{K})$, then the action of $\Phi_{\widetilde{L},\widetilde{L}}$ on $H^\ast(L;\mathbb{K})$ can be equivalently interpreted as taking a cycle on the first sector $L$, intersecting with the cycle $\sum_{k=0}^\infty\mathit{Ev}_\ast\left[S^1\times\overline{\mathcal{M}}_{k,+}(\beta_k,L)\right]$, and then projecting to the second factor. In particular, the class $\Phi_{\widetilde{L},\widetilde{L}}\left([e_L]\right)$ is represented by the evaluation maps $(t,u)\mapsto u(0,0)$. Since the evaluation maps factor through the projections $S^1\times\overline{\mathcal{M}}_{k,+}(y,L)\rightarrow\overline{\mathcal{M}}_{k,+}(y,L)$, it represents the zero homology class.
\end{proof}

To see that the endomorphism $\Phi_{\widetilde{L},\widetilde{L}}$ is non-trivial, the fact that $\tilde{b}$ is a cyclic dilation plays a crucial role.

\begin{proposition}\label{proposition:top-deg}
Let $\tilde{b}\in\mathit{HF}_{S^1}^1(\lambda)$ be a cyclic dilation, and let $\widetilde{L}$ be a $\tilde{b}$-equivariant Lagrangian sphere, then the action of $\Phi_{\widetilde{L},\widetilde{L}}$ on $\mathit{HF}^n(L,L)$ is the multiplication by a non-zero scalar $\alpha_L$. Moreover, if $h=1$ in the definition of a cyclic dilation, then one can take $\alpha_L=1$ for all $\widetilde{L}$.
\end{proposition}
\begin{proof}
As noticed in the proof of Proposition \ref{proposition:deg-0}, the action of $\Phi_{\widetilde{L},\widetilde{L}}$ on $\mathit{HF}^n(L,L)$ coincides with the action of $\left[\phi_{\widetilde{L},\widetilde{L}\subset T^\ast L}\right]$ on $\mathit{HF}^n(L,L)$ computed in the Weinstein neighborhood $D^\ast L\subset M$. If $h=1$, then $\alpha_L=1$ in Lemma \ref{lemma:unique}, and $\beta\in\mathit{SC}^1(T^\ast L)$ in (\ref{eq:chain-map}) represents a dilation, so it follows from \cite{ps5}, Lemma 18.1 that the action of $\left[\phi_{\widetilde{L},\widetilde{L}\subset T^\ast L}\right]$ on $\mathit{HF}^n(L,L)$ has weight 1. The general case follows from a rescaling by $\alpha_L$.
\end{proof}

For convenience, we introduce the following notation:
\begin{equation}
\widetilde{L}\bullet\widetilde{L}:=\mathrm{Str}\left(\Phi_{\widetilde{L},\widetilde{L}}\right)\in\mathbb{K},
\end{equation}
It follows from Propositions \ref{proposition:deg-0} and \ref{proposition:top-deg} that:

\begin{lemma}\label{lemma:sphere}
Let $L\subset M$ be a Lagrangian sphere with dimension $n\geq3$ and $n$ is odd, then $\widetilde{L}\bullet\widetilde{L}=-\alpha_L$.
\end{lemma}

This fact will be crucial to the proof of Theorem \ref{theorem:main}.

\subsection{Odd-dimensional spheres}\label{section:disjoint}

In subsection we prove Theorem \ref{theorem:main}. The argument is a modification of that of Seidel in $\cite{ps4}$, Sections 3 and 4, and the main new ingredient that enters into our proof is the operation $\ast_k$ introduced in Appendix \ref{section:product}. First, we want to make precise the meaning of property $(\widetilde{\textrm{H}})$ in the statement of Theorem \ref{theorem:main}.

\begin{definition}\label{definition:H}
Let $M$ be a Liouville manifold with $c_1(M)=0$. We say that $M$ has property $(\widetilde{\textrm{H}})$ if there is a real number $\lambda>0$, such that
\begin{itemize}
	\item[(i)] $2\lambda\notin\mathcal{P}_M$;
	\item[(ii)] $\lambda<\min\mathcal{P}_M$. In particular, the PSS map $H^\ast(M;\mathbb{K})\rightarrow\mathit{HF}^\ast(\lambda)$ is an isomorphism;
	\item[(iii)] there is a cyclic dilation $\tilde{b}\in\mathit{HF}_{S^1}^1(2\lambda)$.
\end{itemize}
\end{definition}

The above definition is motivated by $\cite{ps4}$, Definition 2.11, where the corresponding notion deals with the special case of a dilation. As has been mentioned in Section \ref{section:dynamics}, a lot of known examples of Liouville manifolds admitting cyclic dilations satisfy property $(\widetilde{\textrm{H}})$. On the other hand, it is currently unknown whether there are examples of Liouville manifolds with cyclic dilations for which the property $(\widetilde{\textrm{H}})$ is violated.
\bigskip

We take a brief look at the situation when $h=1$ and explain what does property $(\widetilde{\textrm{H}})$ mean in this special case. Recall that the \textit{Gutt-Hutchings capacities} $\cite{gh}$ of the Liouville domain $\overline{M}$ are a sequence of symplectic capacities $\left\{c_k^\mathit{GH}(M)\right\}_{k\geq1}$ defined to be
\begin{equation}\label{eq:capacity}
c_k^\mathit{GH}(M):=\inf\left\{a|\delta_\mathit{eq}(x)=u^{-k+1}e\textrm{ for some }x\in F^{\leq a}\mathit{SC}_{S^1}^{-2k+1}(M)\right\},
\end{equation}
where $F^\bullet$ is induced by the action filtration on $\mathit{SC}^\ast(M)$ (cf. (\ref{eq:filtration})). It is clear from the definition and our discussions in Section \ref{section:Gysin}, in particular Remark \ref{remark:zig-zag} that $M$ admits a cyclic dilation with $h=1$ if and only if $c_1^\mathit{GH}(M)<\infty$.

Suppose that there is a sufficiently small $a>0$ so that there exists an $x\in F^{\leq a}\mathit{SC}_{S^1}^{-1}(M)$ with $\delta_\mathit{eq}(x)=e$, then according to the definition of the action functional (cf. (\ref{eq:action})), we conclude that there exists a small enough $\lambda>0$ so that $x\in\mathit{CF}^{-1}_{S^1}(2\lambda)$, which in turn implies that property $(\widetilde{\textrm{H}})$ holds for $M$. On the other hand, knowing the geometry of $M$ would enable us to find the largest possible $a$ for a fixed slope $\lambda$. Denote this number by $a_\lambda$, then $M$ satisfies property $(\widetilde{\textrm{H}})$ with $h=1$ as long as $c_1^\mathit{GH}(M)\leq a_\lambda$. Define
\begin{equation}
a_M:=\lim_{\lambda\rightarrow\min\mathcal{P}_M}a_\lambda.
\end{equation} 
Thus as a special case of Theorem \ref{theorem:main}, we have the following:

\begin{corollary}\label{corollary:GH}
Let $\overline{M}$ be a Weinstein domain with dimension $2n\geq6$, where $n$ is odd, and whose first Gutt-Hutchings capacity satisfies $c_1^\mathit{GH}(M)<a_M$. Then for any Lagrangian sphere $L\subset M$, its homology class $[L]\in H_n(M;\mathbb{Q})$ is nonzero.
\end{corollary}

From now on, let $M$ be a $2n$-dimensional Weinstein manifold satisfying property $(\widetilde{\textrm{H}})$. As before, we denote by $\tilde{\beta}=\sum_{k=0}^\infty\beta_k\otimes u^{-k}\in\mathit{CF}_{S^1}^1(2\lambda)$ the cochain level representative of the cyclic dilation $\tilde{b}$. 
\bigskip

Let $L\subset M$ be any Lagrangian submanifold satisfying Assumption \ref{assumption:Lag}. Following \cite{ps4}, Section (4d), for every $k\geq0$, we define a map
\begin{equation}\label{eq:map}
\chi_k:\mathit{CF}^{2k+1}(2\lambda)\rightarrow\mathbb{K},
\end{equation}
which is the sum of the following six expressions:
\begin{itemize}
	\item [$\mathrm{(i)}_k$] $\mathit{CF}^{2k+1}(2\lambda)\xrightarrow{\phi_{L}^{2,0;k}}\mathit{CF}^n(\lambda)\xrightarrow{\left\langle\phi_{L}^{1,0},\cdot\right\rangle}\mathbb{K}$,
	\item [$\mathrm{(ii)}_k$] $\mathit{CF}^{2k+1}(2\lambda)\xrightarrow{\phi_{L}^{1,1;k}}\mathit{CF}^1(L,L)\xrightarrow{\psi_{L,L}^{0,1}}\mathbb{K}$,
	\item [$\mathrm{(iii)}_k$] $x\mapsto(-1)^{n(n+1)/2+1}\mathrm{Str}\left(\phi_{L,L}^{1,2;k}(x,\cdot)\right)$,
	\item [$\mathrm{(iv)}_k$] $\mathit{CF}^{2k+1}(2\lambda)\xrightarrow{\phi_{L}^{1,1;k}}\mathit{CF}^1(L,L)\xrightarrow{-\left(\psi_{L,L}^{0,1}\right)^\vee}\mathbb{K}$,
	\item [$\mathrm{(v)}_k$] $\mathit{CF}^{2k+1}(2\lambda)\xrightarrow{\phi_{L}^{2,0;k}}\mathit{CF}^n(\lambda)\xrightarrow{(-1)^{n+1}\left\langle\phi_{L}^{1,0},\cdot\right\rangle}\mathbb{K}$,
	\item [$\mathrm{(vi)}_k$] $(-1)^n\left\langle\cdot,\phi_{L}^{1,0}\ast_k\phi_{L}^{1,0}\right\rangle$,
\end{itemize}
where nondegenerate pairing
\begin{equation}
\langle\cdot,\cdot\rangle:\mathit{CF}^\ast(\lambda)\otimes\mathit{CF}^{2n-\ast}(-\lambda)\rightarrow\mathbb{K},
\end{equation}
gives the chain level realizations of the Poincar\'{e} duality isomorphism on Floer cohomologies. When $k=0$, $\chi_0$ is a chain map, and it is proved in $\cite{ps4}$, Proposition 4.11 that $\chi_0$ is nullhomotopic, since the expressions $\mathrm{(i)}_0$-$\mathrm{(vi)}_0$ correspond to boundary components of a 2-dimensional family of Riemann surfaces. See also \cite{ps4}, Section (5b). In general, we have the following:

\begin{proposition}\label{proposition:null}
When $n$ is odd, there exist choices of Floer data so that the following identity holds:
\begin{equation}\label{eq:null}
\sum_{k=0}^\infty\chi_k(\beta_k)=0.
\end{equation}
\end{proposition}
\begin{proof}
Let $\overline{S}$ be an annulus with a unique interior puncture $\zeta_\mathit{in}$, which is an input, and no boundary punctures. Denote bt $\partial_\mathit{in}S$ the inner boundary of $S$, and by $\partial_\mathit{out}S$ the outer boundary of $S$, both of them are labelled by $L$. Moreover, there are $k$ auxiliary marked points $p_1,\cdots,p_k$ lying in a small neighborhood of $\zeta_\mathit{in}$, and they are ordered so that (\ref{eq:radial2}) is satisfied with respect to the local complex coordinate near $\zeta_\mathit{in}$. The asymptotic marker $\ell_{\mathit{in}}$ at $\zeta_{\mathit{in}}$ is required to point towards $p_k$. See the central picture of Figure \ref{fig:hexagon} for a description of the domain. The associated operation will be denoted as
\begin{equation}
\psi^{1,0;k}_{L,L}:\mathit{CF}^{2k+2}(2\lambda)\rightarrow\mathbb{K}.
\end{equation}

\begin{figure}
	\centering
	\begin{tikzpicture}
	\newdimen\R
	\R=3cm
	\draw (0:\R) \foreach \x in {60,120,...,360} {  -- (\x:\R) };
	
	\node at (0,2.4) {(i)};
	\node at (-2,1.2) {(ii)};
	\node at (-2,-1.2) {(iii)};
	\node at (0,-2.4) {(iv)};
	\node at (2,-1.2) {(v)};
	\node at (2,1.2) {(vi)};
	
	\filldraw[draw=black,color={black!15},opacity=0.5] (0,0) circle (1.6);
	\filldraw[draw=black,color={black!0}] (0,0.3) circle (0.8);
	\draw (0,0) circle [radius=1.6];
	\draw (0,0.3) circle [radius=0.8];
	\draw (0,-1) node {$\times$};
	\node at (0,0.85) {$L$};
	\node at (1.85,0) {$L$};
	\draw (0,-1.2) [orange] node[circle,fill,inner sep=1pt] {};
	\draw (-0.3,-1) [orange] node[circle,fill,inner sep=1pt] {};
	\draw (0.25,-1) [orange] node[circle,fill,inner sep=1pt] {};
	\draw [orange, dashed] (0,-1) circle [radius=0.4];
	\draw [teal] [->] (0,-1) to (0,-1.3);
	
	\filldraw[draw=black,color={black!15},opacity=0.5] (0,4) circle (1);
	\filldraw[draw=black,color={black!15},opacity=0.5] (0,6.5) circle (1);
	\draw (0,4) circle [radius=1];
	\draw (0,6.5) circle [radius=1];
	\draw (0,4) node {$\times$};
	\draw (-0.35,6.5) node {$\times$};
	\draw (0.35,6.5) node {$\times$};
	\draw [orange] (-0.65,6.5) node[circle,fill,inner sep=1pt] {};
	\draw [orange] (-0.15,6.35) node[circle,fill,inner sep=1pt] {};
	\draw [orange] (-0.35,6.75) node[circle,fill,inner sep=1pt] {};
	\draw [orange,dashed] (-0.35,6.5) circle [radius=0.4];
	\node at (1.25,4) {$L$};
	\node at (1.25,6.5) {$L$};
	\draw [teal] [->] (-0.35,6.5) to (-0.35,6.8);
	\draw [teal] [->] (0.35,6.5) to (0.7,6.5);
	\draw [teal] [->] (0,4) to (-0.35,4);
	\draw [blue,dashed] (0.35,6.5) to (0,4);
	
	\filldraw[draw=black,color={black!15},opacity=0.5] (-4.2,2.4) circle (1.6);
	\filldraw[draw=black,color={black!0}] (-4.2,2.4) circle (1);
	\filldraw[draw=black,color={black!15},opacity=0.5] (-4.2,2) circle (0.6);
	\draw (-4.2,2.4) circle [radius=1.6];
	\draw (-4.2,2.4) circle [radius=1];
	\draw (-4.2,2) circle [radius=0.6];
	\draw (-4.2,2) node {$\times$};
	\draw [orange] (-4.5,2) node[circle,fill,inner sep=1pt] {};
	\draw [orange] (-4,2) node[circle,fill,inner sep=1pt] {};
	\draw [orange] (-4.2,2.25) node[circle,fill,inner sep=1pt] {};
	\draw [orange,dashed] (-4.2,2) circle [radius=0.4];
	\draw [teal] [->] (-4.2,2) to (-3.9,2);
	\node at (-4.2,4.2) {$L$};
	\node at (-4.2,3.15) {$L$};
	
	\filldraw[draw=black,color={black!15},opacity=0.5] (0,-4.6) circle (1.6);
	\filldraw[draw=black,color={black!0}] (0,-4.6) circle (0.8);
	\filldraw[draw=black,color={black!15},opacity=0.5] (0,-6.8) circle (0.6);
	\draw (0,-4.6) circle [radius=1.6];
	\draw (0,-4.6) circle [radius=0.8];
	\draw (0,-6.8) circle [radius=0.6];
	\draw (0,-6.8) node {$\times$};
	\draw [orange] (-0.3,-6.8) node[circle,fill,inner sep=1pt] {};
	\draw [orange] (0.2,-6.8) node[circle,fill,inner sep=1pt] {};
	\draw [orange] (0,-7.05) node[circle,fill,inner sep=1pt] {};
	\draw [orange,dashed] (0,-6.8) circle [radius=0.4];
	\draw [teal] [->] (0,-6.8) to (0.3,-6.8);
	\node at (1.85,-4.6) {$L$};
	\node at (0,-5.15) {$L$};
	
	\filldraw[draw=black,color={black!15},opacity=0.5] (4.2,1.2) circle (1);
	\filldraw[draw=black,color={black!15},opacity=0.5] (3.2,3.2) circle (0.6);
	\filldraw[draw=black,color={black!15},opacity=0.5] (5.2,3.2) circle (0.6);
	\draw (4.2,1.2) circle [radius=1];
	\draw (3.2,3.2) circle [radius=0.6];
	\draw (5.2,3.2) circle [radius=0.6];
	\draw [dashed] (4.2,1.2) ellipse (1 and 0.3);
	\draw (3.2,3.2) node {$\times$};
	\draw (5.2,3.2) node {$\times$};
	\draw (4.2,1.5) node {$\times$};
	\draw (3.7,0.95) node {$\times$};
	\draw (4.7,0.95) node {$\times$};
	\draw [blue,dashed] (3.7,0.95) to (3.2,3.2);
	\draw [blue,dashed] (4.7,0.95) to (5.2,3.2);
	\draw [teal] [->] (3.2,3.2) to (3.2,3.5);
	\draw [teal] [->] (5.2,3.2) to (5.2,3.5);
	\draw [teal] [->] (3.7,0.95) to (3.7,1.25);
	\draw [teal] [->] (4.7,0.95) to (4.7,1.25);
	\node at (2.35,3.2) {$L$};
	\node at (6.05,3.2) {$L$};
	\draw [dashed,orange] (4.2,1.5) circle [radius=0.4];
	\draw [orange] (4,1.5) node[circle,fill,inner sep=1pt] {};
	\draw [orange] (4.5,1.5) node[circle,fill,inner sep=1pt] {};
	\draw [orange] (4.2,1.75) node[circle,fill,inner sep=1pt] {};
	\draw [teal] [->] (4.2,1.5) to (3.9,1.5);
	
	\filldraw[draw=black,color={black!15},opacity=0.5] (3.6,-1.8) circle (1);
	\filldraw[draw=black,color={black!15},opacity=0.5] (6.3,-1.8) circle (1);
	\draw (3.6,-1.8) circle [radius=1];
	\draw (6.3,-1.8) circle [radius=1];
	\draw (6.3,-1.8) node {$\times$};
	\draw (3.6,-2.15) node {$\times$};
	\draw (3.6,-1.45) node {$\times$};
	\draw [orange] (3.3,-1.45) node[circle,fill,inner sep=1pt] {};
	\draw [orange] (3.8,-1.45) node[circle,fill,inner sep=1pt] {};
	\draw [orange] (3.6,-1.2) node[circle,fill,inner sep=1pt] {};
	\draw [orange,dashed] (3.6,-1.45) circle [radius=0.4];
	\draw [teal] [->] (6.3,-1.8) to (6.3,-1.45);
	\draw [teal] [->] (3.6,-2.15) to (3.6,-2.5);
	\draw [teal] [->] (3.6,-1.45) to (3.9,-1.45);
	\draw [blue,dashed] (3.6,-2.15) to (6.3,-1.8);
	\node at (3.6,-3) {$L$};
	\node at (6.3,-3) {$L$};
	
	\begin{scope}[shift={(-4.1,-2.2)},rotate=0]
	\filldraw[draw=black,color={black!15},opacity=0.5] (-240:0.8) arc (-240:60:0.8) -- (75:1.55) arc (75:-255:1.55) -- cycle;
	\end{scope}
	\begin{scope}[shift={(-4.5,-1.1)},rotate=0]
	\filldraw[draw=black,color={black!15},opacity=0.5] (0,0) --  (90:0.4) arc (90:-90:0.4) -- cycle;
	\end{scope}
	\begin{scope}[shift={(-3.7,-1.1)},rotate=0]
	\filldraw[draw=black,color={black!15},opacity=0.5] (0,0) --  (270:0.4) arc (270:90:0.4) -- cycle;
	\end{scope}
	\draw (-4.5,-1.5) arc (-90:90:0.4);
	\draw (-3.7,-0.7) arc (90:270:0.4);
	\draw (-4.5,-1.5) arc (-240:60:0.8);
	\draw (-4.5,-0.7) arc (-255:75:1.55);
	\draw (-4.1,-3.35) node {$\times$};
	\draw [orange] (-4.3,-3.35) node[circle,fill,inner sep=1pt] {};
	\draw [orange] (-3.95,-3.35) node[circle,fill,inner sep=1pt] {};
	\draw [orange] (-4.1,-3.5) node[circle,fill,inner sep=1pt] {};
	\draw [teal] [->] (-4.1,-3.35) to (-4.4,-3.35);
	\draw [orange, dashed] (-4.1,-3.35) circle [radius=0.3];
	\node at (-3.7,-0.5) {$L$};
	\node at (-4.1,-2.7) {$L$};
	\end{tikzpicture}
	\caption{A section of the family $\overline{\mathcal{T}}_3$}
	\label{fig:hexagon}
\end{figure}

When there is no auxiliary marked point, the Riemann surfaces $\left(\overline{S};\zeta_\mathit{in},\ell_\mathit{in}\right)$, where $\ell_\mathit{in}$ is allowed to vary in a specific way, form a 2-dimensional family $\mathcal{T}$, which compactifies to a hexagon $\overline{\mathcal{T}}$, see Figure \ref{fig:hexagon}. The construction of this 2-dimensional family, which arises as the real blow up of the KSV (Kimura-Stasheff-Voronov) compactification, is explained in detail in $\cite{ps4}$, Section (5b). Note that the orientations of the boundary components (iii), (iv) and (v) reverse that of the usual boundary orientation of a hexagon, which leads to an additional $-1$ in the definitions of the operations $\mathrm{(iii)}_k$, $\mathrm{(iv)}_k$ and $\mathrm{(v)}_k$ appeared in the expression of $\chi_k$. After the points $p_1,\cdots,p_k$ are added, we get a (compactified) $(2k+2)$-dimensional family $\overline{\mathcal{T}}_k$ of domains $\left(\overline{S};\zeta_\mathit{in},\ell_\mathit{in};p_1,\cdots,p_k\right)$ which fibers over the hexagon, whose fibers can be identified with the compactifcations of the moduli space of $k$-point angle decorated half-cylinders $\overline{\mathcal{M}}_{k,+}$ considered in Section \ref{section:CL}. Figure \ref{fig:hexagon} depicts a section of the moduli space $\overline{\mathcal{T}}_3$ obtained by fixing the positions of $p_1,p_2,p_3$. There are six codimension 1 boundary strata in $\overline{\mathcal{T}}_k$ which correspond respectively to the boundaries of the hexagon labelled by (i)-(vi) in Figure \ref{fig:hexagon}, and they give rise to the operations $\mathrm{(i)}_k$-$\mathrm{(vi)}_k$ defined above. When the points $p_1,\cdots,p_k$ are allowed to vary, there are additional strata in $\partial\overline{\mathcal{T}}_k$ corresponding to the loci where $|p_i|=|p_{i+1}|$ and $|p_1|=\varepsilon$. 

For the boundary strata $|p_i|=|p_{i+1}|$, where $1\leq i\leq k-1$, it is clear from our previous discussions that one can choose Floer data so that they are non-rigid.

For the boundary stratum $|p_1|=\varepsilon$, denote its contribution by
\begin{equation}
\psi_{L,L;S^1}^{1,0;k-1}:\mathit{CF}^{2k+1}(2\lambda)\rightarrow\mathbb{K}.
\end{equation}
Consider the automorphism of the annulus $\overline{S}$ which switches the boundary circles $\partial_\mathit{in}S$ and $\partial_\mathit{out}S$, under which we obtain an operation which, up to homotopy, can be identified with $\psi_{L,L;S^1}^{1,0;k-1}$. Since exchanging the order of the two boundary circles leads to a Koszul sign of $(-1)^n$ ($\cite{ps4}$, Remark 5.3), we have
\begin{equation}\label{eq:relation}
\psi_{L,L;S^1}^{1,0;k-1}-(-1)^n\psi_{L,L;S^1}^{1,0;k-1}=\textrm{null homotopy}.
\end{equation}
This shows that the stratum $|p_1|=\frac{1}{2}$ of $\partial\overline{\mathcal{T}}_k$ does not contribute when $n$ is odd.

Finally, there are strata coming from real blow-ups in the fiber direction of the compactification $\overline{\mathcal{T}}_k$, which correspond to sphere bubbles at the puncture $\zeta_{\mathit{in}}$, containing at least one auxiliary marked points. These strata, together with the sphere bubble at $\zeta_{\mathit{in}}$ without auxiliary marked point, contribute
\begin{equation}\label{eq:blp}
\sum_{j=0}^k\psi^{1,0;k-j}_{L,L}\left(\delta_j(\beta_k)\right)
\end{equation}
for each $k\geq0$. Using the fact that $\tilde{\beta}\in\mathit{CF}_{S^1}^1(2\lambda)$ is an equivariant cocycle, we see that the sum of (\ref{eq:blp}) over $k\geq0$ vanishes.

Combining the above analysis we get (\ref{eq:null}) when $n$ is odd.
\end{proof}

\begin{remark}\label{remark:characteristic}
Note that when $\mathrm{char}(\mathbb{K})=2$, the relation (\ref{eq:relation}) holds for trivial reasons, so we cannot use it to deduce that $\psi_{L,L;S^1}^{1,0;k-1}=0$, and the argument above fails.
\end{remark}

Denote by $(\widetilde{C}^\ast,\tilde{d})$ the direct sum of two Floer cochain complexes
\begin{equation}
\begin{split}
&\widetilde{C}^\ast:=\mathit{CF}^\ast(-\lambda)\oplus\mathit{CF}^\ast(\lambda), \\
&\tilde{d}(\xi,x):=\left(d\xi,dx\right),
\end{split}
\end{equation}
where $d$ is the ordinary Floer differential, and $\lambda>0$ is chosen as in Definition \ref{definition:H}. The cohomology of $(\widetilde{C}^\ast,\tilde{d})$ will be denoted by $\widetilde{H}^\ast$. Note that by item (ii) of Definition \ref{definition:H}, we can choose the Hamiltonian $H_{\lambda,t}$ in the definition of the Floer cochain complex $\mathit{CF}^\ast(\lambda)$ so that it is isomorphic to the Morse complex of some function $\phi:M\rightarrow\mathbb{R}$, which computes the ordinary cohomology $H^\ast(M;\mathbb{K})$. Since $M$ is a Weinstein manifold, we may assume that the $\phi$ is $J$-convex, so that all the Morse critical points have index $\leq n$, and therefore $\mathit{CF}^\ast(\lambda)$ is supported in degrees $\ast\leq n$. For similar reasons, we can arrange so that $\mathit{CF}^\ast(-\lambda)$ is supported in degrees $\ast\geq n$.

We define a pairing on the above chain complex
\begin{equation}\label{eq:pairing}
\begin{split}
&\tilde{\iota}:\widetilde{C}^\ast\otimes\widetilde{C}^{2n-\ast}\rightarrow\mathbb{K}, \\
&\tilde{\iota}\left((\xi_0,x_0),(\xi_1,x_1)\right)=\langle x_0,\xi_1\rangle-(-1)^{|\xi_0|}\langle x_1,\xi_0\rangle+\sum_{k=0}^\infty\langle\beta_k,\xi_0\ast_k\xi_1\rangle.
\end{split}
\end{equation}
In the above, the operations
\begin{equation}
\ast_k:\mathit{CF}^\ast(-\lambda)\otimes\mathit{CF}^{2n-\ast}(-\lambda)\rightarrow\mathit{CF}^{2n-2k-1}(-2\lambda)
\end{equation}
are parametrized analogues of the usual star product (\ref{eq:star}) on the Floer cochain complex. Their detailed constructions are recorded in Appendix \ref{section:product}.

\begin{lemma}
$\tilde{\iota}$ is a chain map.	
\end{lemma}
\begin{proof}
Pick any cochains $(\xi_0,x_0)$ and $(\xi_1,x_1)$ whose degrees add up to $2n-1$, we have
\begin{equation}
\begin{split}
&\tilde{\iota}\left(\left(d\xi_0,dx_0\right),(\xi_1,x_1)\right)+(-1)^{|\xi_0|}\tilde{\iota}\left((\xi_0,x_0),\left(d\xi_1,dx_1\right)\right) \\
&=\langle dx_0,\xi_1\rangle+(-1)^{|\xi_0|}\langle x_0,d\xi_1\rangle-(-1)^{|\xi_1|}\langle x_1,d\xi_0\rangle-\langle dx_1,\xi_0\rangle \\
&+\sum_{k=0}^\infty\left\langle\beta_k,d\xi_0\ast_k\xi_1+(-1)^{|\xi_0|}\xi_0\ast_kd\xi_1\right\rangle \\
&=\sum_{k=0}^\infty\left\langle\beta_k,d\xi_0\ast_k\xi_1+(-1)^{|\xi_0|}\xi_0\ast_kd\xi_1\right\rangle. \\
\end{split}
\end{equation}
Since $\mathit{CF}^\ast(-\lambda)$ is supported in degrees $\ast\geq n$, and $|\xi_0|+|\xi_1|=2n-1$ by our assumption, either $\xi_0$ or $\xi_1$ is zero, which in turn implies that $d\xi_0\ast_k\xi_1+(-1)^{|\xi_0|}\xi_0\ast_kd\xi_1=0$.
\end{proof}

The induced pairing on cohomology will be denoted by
\begin{equation}
\widetilde{I}:\widetilde{H}^\ast\otimes\widetilde{H}^{2n-\ast}\rightarrow\mathbb{K}.
\end{equation}

To each $\tilde{b}$-equivariant Lagrangian sphere $\widetilde{L}=(L,\tilde{\gamma}_L)$, we can associate a cocycle
\begin{equation}\label{eq:Lag-class}
(\xi_L,x_L):=\left(\phi_L^{1,0},(-1)^{n+1}\sum_{k=0}^\infty\phi_L^{2,0;k}(\beta_k)+\left(\phi_L^{1,1}\right)^\vee(\tilde{\gamma}_L)\right)\in\widetilde{C}^n.
\end{equation}
To see that $(\xi_L,x_L)$ is indeed a cocycle, notice first that $\phi_L^{1,0}\in\mathit{CF}^n(-\lambda)$ is by definition a cocycle. On the other hand, our assumption implies that the differential of the second entry of (\ref{eq:Lag-class}) vanishes for degree reasons. Denote by $[\![\widetilde{L}]\!]\in\widetilde{H}^n$ the cohomology class of $(\xi_L,x_L)$. As a simple observation, under the natural projection $\widetilde{H}^n\rightarrow\mathit{HF}^n(-\lambda)$, the class $[\![\widetilde{L}]\!]$ goes to the Floer cohomology class $[\![L]\!]$ defined by (\ref{eq:cocycle}).

\begin{theorem}\label{theorem:pairing}
Let $L\subset M$ be a Lagrangian sphere of dimension $n\geq3$ and $n$ is odd, then
\begin{equation}\label{eq:co-pair}
\widetilde{I}\left([\![\widetilde{L}]\!],[\![\widetilde{L}]\!]\right)=(-1)^{n(n+1)/2}\widetilde{L}\bullet\widetilde{L}.
\end{equation}
\end{theorem}
\begin{proof}
We have
\begin{equation}
\begin{split}
\textrm{ }\textrm{ }\textrm{ }&(-1)^{n(n+1)/2}\mathrm{Str}(\phi_{\widetilde{L},\widetilde{L}}) \\
&=(-1)^{n(n+1)/2}\left(\sum_{k=0}^\infty\mathrm{Str}\left(\phi_{L,L}^{1,2;k}(\beta_k,\cdot)\right)-\mathrm{Str}\left(\mu^2(\tilde{\gamma}_{L},\cdot)\right)+\mathrm{Str}\left(\mu^2(\cdot,\tilde{\gamma}_{L})\right)\right) \\
&=(-1)^{n(n+1)/2}\sum_{k=0}^\infty\mathrm{Str}\left(\phi_{L,L}^{1,2;k}(\beta_k,\cdot)\right)-(-1)^n\sum_{k=0}^\infty\psi_{L,L}^{0,1}\left(\phi_{L}^{1,1;k}(\beta_k)\right) \\
&-\left\langle\phi_{L}^{1,0},\left(\phi_{L}^{1,1}\right)^\vee(\tilde{\gamma}_{L})\right\rangle+(-1)^n\sum_{k=0}^\infty\left(\psi_{L,L}^{0,1}\right)^\vee\left(\phi_{L}^{1,1;k}(\beta_k)\right)+(-1)^n\left\langle\phi_{L}^{1,0},\left(\phi_{L}^{1,1}\right)^\vee(\tilde{\gamma}_{L})\right\rangle,
\end{split}
\end{equation}
where the first line follows from (\ref{eq:endo}), and the second line follows from (\ref{eq:psi1}), (\ref{eq:psi2}) and (\ref{eq:def}). Proposition \ref{proposition:null} applied to the components of $\tilde{\beta}\in\mathit{CF}^1(2\lambda)$ yields
\begin{equation}
\begin{split}
&(-1)^{n(n+1)/2}\sum_{k=0}^\infty\mathrm{Str}\left(\phi_{L,L}^{1,2;k}(\beta_k,\cdot)\right)-(-1)^n\sum_{k=0}^\infty\psi_{L,L}^{0,1}\left(\phi_{L}^{1,1;k}(\beta_k)\right) \\
&+(-1)^n\sum_{k=0}^\infty\left(\psi_{L,L}^{0,1}\right)^\vee\left(\phi_{L}^{1,1;k}(\beta_k)\right) \\
&=-\sum_{k=0}^\infty\left\langle\phi_{L}^{1,0},\phi_{L}^{2,0;k}(\beta_k)\right\rangle+\sum_{k=0}^\infty\left\langle\beta_k,\phi_{L}^{1,0}\ast_k\phi_{L}^{1,0}\right\rangle+(-1)^n\sum_{k=0}^\infty\left\langle\phi_{L}^{1,0},\phi_{L}^{2,0;k}(\beta_k)\right\rangle,
\end{split}
\end{equation}
so one can rewrite
\begin{equation}
\begin{split}
&(-1)^{n(n+1)/2}\mathrm{Str}(\phi_{\widetilde{L},\widetilde{L}})=\left\langle(-1)^{n+1}\sum_{k=0}^\infty\phi_{L}^{2,0;k}(\beta_k)+\left(\phi_{L}^{1,1}\right)^\vee(\tilde{\gamma}_{L}),\phi_{L}^{1,0}\right\rangle \\
&-(-1)^n\left\langle(-1)^{n+1}\sum_{k=0}^\infty\phi_{L}^{2,0;k}(\beta_k)+\left(\phi_{L}^{1,1}\right)^\vee(\tilde{\gamma}_{L}),\phi_{L}^{1,0}\right\rangle+\sum_{k=0}^\infty\left\langle\beta_k,\phi_{L}^{1,0}\ast_k\phi_{L}^{1,0}\right\rangle \\
&=\tilde{\iota}\left((\xi_{L},x_{L}),(\xi_{L},x_{L})\right).
\end{split}
\end{equation}
Our assumption ensures that $\phi_{\widetilde{L},\widetilde{L}}$ is a well-defined chain map, so on the cohomology level we get (\ref{eq:co-pair}).
\end{proof}

\begin{proof}[Proof of Theorem \ref{theorem:main}]
Let $L\subset M$ be any Lagrangian sphere of some odd dimension $n\geq3$, it follows from Lemma \ref{lemma:sphere} that $\widetilde{L}\bullet\widetilde{L}=-\alpha_L$ for some $\alpha_L\neq0$. By Theorem \ref{theorem:pairing}, this implies that $\widetilde{I}\left([\![\widetilde{L}]\!],[\![\widetilde{L}]\!]\right)\neq0$.

By definition, $\mathit{HF}^n(\lambda)\subset\widetilde{H}^n$ is a half-dimensional subspace, which is isotropic for the pairing $\widetilde{I}$. This implies that by projecting to the first factor $\mathit{HF}^n(-\lambda)$, the class $[\![L]\!]$, which is the image of $[\![\widetilde{L}]\!]$, is non-zero. By Definition \ref{definition:H}, (ii), the dual of the PSS map $\mathit{HF}^n(-\lambda)\rightarrow H_\mathit{cpt}^n(M;\mathbb{K})$ is an isomorphism, and the Floer cohomology class $[\![L]\!]$ is therefore mapped to the Poincar\'{e} dual of the ordinary homology class $[L]\in H_n(M;\mathbb{K})$.
\end{proof}

\bigskip

We end this section with a short remark on the signs appeared in various formulas in Sections \ref{section:q} and \ref{section:disjoint}. Basically, we follow the convention in Seidel's original argument, see $\cite{ps4}$, Section (5c) for details. The only difference here is that we are dealing with operations defined using an additional parametrization by $k$ auxiliary marked points, so the orientations of the relevant moduli spaces defining the operations relating open and closed string invariants are fixed by following the original convention of Seidel, and choosing a preferred orientation of the moduli space $\overline{\mathcal{M}}_{k,+}$ of $k$ point angle decorated half-cylinders. The (relative) orientation of $\overline{\mathcal{M}}_{k,+}$ is determined inductively in $\cite{jz}$, Lemmas 4.4.7 and 4.4.11. This allows us to arrange the signs involved in the above computations so that they coincide with that of $\cite{ps4}$ and $\cite{ss}$.

\section{Existence of cyclic dilations}\label{section:existence}

Let $M$ be a $2n$-dimensional Liouville manifold with $c_1(M)=0$. In this section, we consider the existence questions of cyclic dilations. In Section \ref{section:Koszul}, we use Koszul duality to show that the manifold $M_{3,3,3,3}$ admits a cyclic dilation. This example is non-trivial as $M_{3,3,3,3}$ does not admit a quasi-dilation. With the help of Lefschetz fibrations, one can produce more examples of Liouville manifolds which carry cyclic dilations starting from the known ones. This is done in Section \ref{section:Lefschetz}. Section \ref{section:general type} proves the uniqueness of smooth Calabi-Yau structures on the wrapped Fukaya categories of log general type affine varieties containing exact Lagrangian $K(\pi,1)$'s, from which Theorem \ref{theorem:unique} follows as a corollary. The discussions in Section \ref{section:conjecture} are mostly speculative, they are included here merely as supplements to Section \ref{section:trichotomy}.

\subsection{Koszul duality}\label{section:Koszul}

Although for the most part of this paper we have been taking a geometric viewpoint, dealing with cyclic dilations in equivariant symplectic cohomologies instead of exact Calabi-Yau structures on wrapped Fukaya categories, this subsection is an exception. Here we shall return to the original notion of an exact Calabi-Yau structure (Definition \ref{definition:key}) which motivates the whole paper, and study it essentially from the algebraic perspective, based on a result of Van den Bergh (Theorem \ref{theorem:Koszul}).
\bigskip

Before we proceed, recall that smooth Calabi-Yau structures are Morita invariant, so it makes no difference to study smooth Calabi-Yau structures on an $A_\infty$-algebra $\mathcal{A}$ over some semisimple ring $\Bbbk$, or to consider them as Calabi-Yau structures on the $A_\infty$-category $\mathcal{A}^\mathit{perf}$. See $\cite{cg}$, Theorem 3.1 for an explanation of this fact.
\bigskip

One of the main ingredients of our proof of Theorem \ref{theorem:Fano} is the following theorem due to Van den Bergh $\cite{mv}$, which enables us to characterize a large class of exact Calabi-Yau $A_\infty$-algebras in terms of its Koszul dual. Recall that a \textit{cyclic $A_\infty$-algebra} $\mathcal{B}$ over $\Bbbk$ is an $A_\infty$-algebra equipped with a chain level perfect pairing
\begin{equation}
\langle\bullet,\bullet\rangle:\mathcal{B}\otimes\mathcal{B}\rightarrow\Bbbk[-n]
\end{equation}
such that the induced correlation functions
\begin{equation}
\left\langle\mu_\mathcal{B}^k(\bullet,\cdots,\bullet),\bullet\right\rangle
\end{equation}
are strictly (graded) cyclically symmetric for each $k\geq1$.

\begin{theorem}[$\cite{mv}$, Theorem 11.1]\label{theorem:Koszul}
Let $\mathcal{A}$ be a homologically smooth, complete, augmented dg algebra over $\Bbbk$, so that $H^\ast(\mathcal{A})$ is concentrated in degrees $\leq0$. Denote by $\mathcal{A}^!:=R\mathrm{Hom}_\mathcal{A}(\Bbbk,\Bbbk)$ the Koszul dual of $\mathcal{A}$. Then the following statements are equivalent:
\begin{itemize}
	\item $\mathcal{A}^!$ is a proper $A_\infty$-algebra which, up to quasi-isomorphism, carries a minimal cyclic $A_\infty$-structure of degree $n$.
	\item $\mathcal{A}$ is exact $n$-Calabi-Yau.
\end{itemize}
\end{theorem}

Here, by \textit{complete} we mean the underlying associative algebra of $\mathcal{A}$ is a quotient of the path algebra of some quiver completed at path length. 

Theorem \ref{theorem:Koszul} should be understood in the more general framework of Koszul duality between Calabi-Yau structures, which we now describe. Recall that over a field $\mathbb{K}$ of characteristic 0, cyclic $A_\infty$-structures provide explicit models for the more general notion of a proper Calabi-Yau structure. Precisely, a \textit{proper $n$-Calabi-Yau structure} on a proper $A_\infty$-algebra $\mathcal{B}$ over $\Bbbk$ is defined as a degree $n$ chain map
\begin{equation}
\widetilde{\mathit{tr}}:\mathit{CC}_\ast(\mathcal{B})\rightarrow\mathbb{K}[-n]
\end{equation}
whose projection to the Hochschild complex, $\mathit{tr}:\mathit{CH}_\ast(\mathcal{B})\rightarrow\mathbb{K}[-n]$, defines a \textit{weak proper $n$-Calabi-Yau structure}, i.e. it induces a perfect pairing
\begin{equation}
\begin{split}
&H^\ast\left(\hom_{\mathcal{B}^\mathit{perf}}(\mathcal{P},\mathcal{Q})\right)\otimes H^{n-\ast}\left(\hom_{\mathcal{B}^\mathit{perf}}(\mathcal{Q},\mathcal{P})\right)\xrightarrow{\left[\mu_\mathcal{B}^2\right]}H^n\left(\hom_{\mathcal{B}^\mathit{perf}}(\mathcal{Q},\mathcal{Q})\right) \\
&\rightarrow\mathit{HH}_n(\mathcal{B})\xrightarrow{[\mathit{tr}]}\mathbb{K}.
\end{split}
\end{equation}
When $\mathrm{char}(\mathbb{K})=0$, any cyclic $A_\infty$-algebra $\mathcal{B}$ over $\Bbbk$ has a canonically defined proper Calabi-Yau structure, and any proper Calabi-Yau structure on an $A_\infty$-algebra $\mathcal{B}$ determines a quasi-isomorphism between $\mathcal{B}$ and a cyclic $A_\infty$-algebra, the latter fact is due to Kontsevich-Soibelman $\cite{ks1}$. It is proved by Ganatra ($\cite{sg1}$, Theorem 2) that any full subcategory of the compact Fukaya category $\mathcal{F}(M)$ admits a geometrically defined proper Calabi-Yau structure (under certain technical assumptions which ensure that the Fukaya category $\mathcal{F}(M)$ is well-defined). As a consequence, we have the following:

\begin{proposition}[$\cite{sg1}$, Corollary 2]\label{proposition:proper-CY}
Let $M$ be a Liouville manifold with $c_1(M)=0$. If $\mathrm{char}(\mathbb{K})=0$, then any full $A_\infty$-subcategory of the compact Fukaya category $\mathcal{F}(M)$ admits a minimal cyclic $A_\infty$-structure.
\end{proposition}

If $\mathcal{A}$ and $\mathcal{B}$ are Koszul dual $A_\infty$-algebras, then there is a duality (cf. \cite{jm}, Section 4)
\begin{equation}
\mathit{CH}_{\ast-n}(\mathcal{A})\cong\hom\left(\mathit{CH}_{\ast+n}(\mathcal{B}),\mathbb{K}\right)
\end{equation}
between their Hochschild chains, which suggests that under Koszul duality, non-degenerate cycles in $\mathit{CH}_{-n}(\mathcal{A})$ should correspond to maps $\mathit{CH}_{\ast+n}(\mathcal{B})\rightarrow\mathbb{K}$ which induce proper Calabi-Yau structures on $\mathcal{B}$.

\begin{theorem}[Cohen-Ganatra]\label{theorem:duality-CY}
Let $\mathcal{A}$ be a homologically smooth dg algebra over $\Bbbk$, and let $\mathcal{A}^!$ be a proper $A_\infty$-algebra so that $\mathcal{A}$ and $\mathcal{A}^!$ are Koszul dual. Then $\mathcal{A}$ carries a smooth Calabi-Yau structure if and only if $\mathcal{A}^!$ is a proper Calabi-Yau $A_\infty$-algebra.
\end{theorem}
\begin{proof}
This is a slight variant of $\cite{cg}$, Theorem 25, where $\mathcal{A}$ is required to be \textit{strongly smooth} (cf. $\cite{cg}$, Definition 3), which means $\mathcal{A}$ is homologically smooth and $\Bbbk$ is perfect as a module over $\mathcal{A}$. This latter property is needed to ensure that $\mathcal{A}^!$ is proper, which we have included in the assumption.
\end{proof}

From this perspective, the content of Theorem \ref{theorem:Koszul} can be understood as saying that if we further impose the assumptions that $\mathcal{A}$ is complete and supported in non-positive degrees, then the Calabi-Yau structure on $\mathcal{A}$ induced by the proper Calabi-Yau structure on $\mathcal{A}^!$ is not only smooth, but also exact.
\bigskip

Geometrically, the $A_\infty$-Koszul duality between the endomorphism algebras of a set of generators in $\mathcal{F}(M)$ and $\mathcal{W}(M)$ is first studied by Etg\"{u}-Lekili in $\cite{etl1}$ when $M$ is a plumbing of $T^\ast S^2$'s according to a Dynkin tree and later generalized in $\cite{ekl}$ and $\cite{yl}$ to many interesting examples of Liouville manifolds in higher dimensions. More precisely, denote by $\mathcal{F}_M$ and $\mathcal{W}_M$ the $A_\infty$-algebras of some fixed sets of split-generators of $\mathcal{F}(M)$ and $\mathcal{W}(M)$ respectively, and assume in addition that both of these sets are indexed by the same finite set $\Gamma$. By saying that the Fukaya categories $\mathcal{F}(M)$ and $\mathcal{W}(M)$ are Koszul dual, we mean that there are quasi-isomorphisms between augmented $A_\infty$-algebras
\begin{equation}\label{eq:Koszul}
R\mathrm{Hom}_{\mathcal{F}_M}(\Bbbk,\Bbbk)\cong\mathcal{W}_M,R\mathrm{Hom}_{\mathcal{W}_M}(\Bbbk,\Bbbk)\cong\mathcal{F}_M,
\end{equation}
where $\Bbbk:=\bigoplus_{v\in\Gamma}\mathbb{K}e_v$ is the semisimple ring consisting of $|\Gamma|$ copies of the ground field $\mathbb{K}$, and it is regarded as a left $\mathcal{F}_M$-module in the first quasi-isomorphism, and a right $\mathcal{W}_M$-module in the second quasi-isomorphism above.
\bigskip

We will need several results from $\cite{ekl}$, which enables us to verify the completeness of the $A_\infty$-algebra $\mathcal{W}_M$ required in Theorem \ref{theorem:Koszul}.

Fix a finite set $\Gamma$. Let $\overline{M}_{-\Lambda}$ be a $2n$-dimensional Weinstein domain, with its Liouville form denoted by $\theta_M$. For each $v\in\Gamma$, let $\overline{L}_v\subset\overline{M}_{-\Lambda}$ be an oriented, connected, $\mathit{Spin}$ Lagrangian submanifold with vanishing Maslov class, such that its boundary $\partial\overline{L}_v\subset\partial\overline{M}_{-\Lambda}$ defines a Legendrian \textit{sphere} $\Lambda_v$ with respect to the contact structure on $\partial\overline{M}_{-\Lambda}$ defined by the restriction of $\theta_M$. Moreover, we assume that different $\overline{L}_v$'s intersect with each other transversely, and the intersections happen only in the interior of $\overline{M}_{-\Lambda}$. In particular, the Legendrian spheres $\Lambda_v$'s are disjoint from each other in $\partial\overline{M}_{-\Lambda}$, together they form a link $\Lambda:=\bigsqcup_{v\in\Gamma}\Lambda_v$. Attaching $n$-handles to $\overline{M}_{-\Lambda}$ along the Legendrian link $\Lambda$ gives rise to a new Weinstein domain $\overline{M}$. Note that $\overline{M}$ contains a set of closed Lagrangian submanifolds $\{L_v\}_{v\in\Gamma}$, which are unions of $\overline{L}_v$ with the Lagrangian core discs of the Weinstein handles attached along $\Lambda$. Define
\begin{equation}
\mathcal{V}_M:=\bigoplus_{v,w\in\Gamma}\mathit{CF}^\ast(L_v,L_w),
\end{equation}
to be the Fukaya $A_\infty$-algebra of these Lagrangian submanifolds, which is well-defined and $\mathbb{Z}$-graded with our assumptions on the $\overline{L}_v$'s. This is an $A_\infty$-algebra over $\Bbbk$. For simplicity, we shall assume that $\mathcal{V}_M$ is strictly unital. Otherwise there is always a standard algebraic procedure which replaces it with a quasi-isomorphic $A_\infty$-algebra which is strictly unital ($\cite{ps1}$, Lemma 2.1). On the other hand, the Legendrian link $\Lambda\subset\partial\overline{M}_{-\Lambda}$ also has an associated dg algebra, the Chekanov-Eliashberg algebra $\mathit{CE}^\ast(\Lambda)$. Denote by $\mathcal{R}$ the set of Reeb chords ending on $\Lambda$, we have
\begin{equation}
\mathit{CE}^\ast(\Lambda):=\bigoplus_{i=0}^\infty\mathbb{K}\langle\mathcal{R}\rangle^{\otimes i},
\end{equation}
with the differential defined by counting anchored holomorphic discs with boundary punctures in the symplectization $\mathbb{R}\times\partial\overline{M}_{-\Lambda}$, whose boundary components lie in the Lagrangian submanifold $\mathbb{R}\times\Lambda$, and whose punctures are asymptotic to the Reeb chords in $\mathcal{R}$, see $\cite{bee}$, Section 4.1. Note that $\mathit{CE}^\ast(\Lambda)$ can also be realized as a dg algebra over $\Bbbk$, by declaring $e_w\mathcal{R}e_v$ to be the set of Reeb chords from $\Lambda_w$ to $\Lambda_v$. Since the union of the Lagrangian submanifolds $\bigcup_{v\in\Gamma}\overline{L}_v\subset\overline{M}_{-\Lambda}$ gives a filling of the Legendrian link $\Lambda$, there is an induced augmentation
\begin{equation}\label{eq:aug}
\varepsilon_L:\mathit{CE}^\ast(\Lambda)\rightarrow\Bbbk,
\end{equation}
which equips $\mathit{CE}^\ast(\Lambda)$ with the structure of an augmented dg algebra over $\Bbbk$.

By $\cite{ekl}$, Theorem 4, we have the following quasi-isomorphism, which should be understood as a general version of the Eilenberg-Moore equivalence:
\begin{equation}\label{eq:EM}
R\mathrm{Hom}_{\mathit{CE}^\ast(\Lambda)}(\Bbbk,\Bbbk)\cong\mathcal{V}_M.
\end{equation}

Using the augmentation (\ref{eq:aug}), we can write
\begin{equation}\label{eq:CE}
\mathit{CE}^\ast(\Lambda)=\Omega\mathit{LC}_\ast(\Lambda),
\end{equation}
where $\Omega$ is the Adams cobar construction, and $\mathit{LC}_\ast(\Lambda)$ is an $A_\infty$-coalgebra over $\Bbbk$ whose underlying $\Bbbk$-bimodule is generated by $\mathcal{R}$, which is called the \textit{Legendrian $A_\infty$-coalgebra} in $\cite{ekl}$, whose linear dual $\mathit{LC}_\ast(\Lambda)^\#$ is quasi-isomorphic to the Fukaya $A_\infty$-algebra $\mathcal{V}_M$ defined above. On the other hand, the \textit{completed Chekanov-Eliashberg dg algebra} is defined to be
\begin{equation}\label{eq:completed-CE}
\widehat{\mathit{CE}}^\ast(\Lambda):=(\mathrm{B}\mathcal{V}_M)^\#,
\end{equation}
where on the right-hand side the bar construction is taken with respect to the trivial augmentation $\varepsilon:\mathcal{V}_M\rightarrow\Bbbk$ defined by projecting to the idemponents in the degree 0 part.
\bigskip

By (\ref{eq:CE}) we have
\begin{equation}\label{eq:LC}
\mathit{CE}^\ast(\Lambda)=\Bbbk\oplus\bigoplus_{i=1}^\infty\overline{\mathit{LC}}_\ast(\Lambda)[-1]^{\otimes_\Bbbk i},
\end{equation}
where $\overline{\mathit{LC}}_\ast(\Lambda)\subset\mathit{LC}_\ast(\Lambda)$ is the submodule obtained by quotienting out the idempotents $e_v$ of $\Bbbk$ . It follows from the definition (\ref{eq:completed-CE}) that as a graded algebra over $\Bbbk$,
\begin{equation}
\widehat{\mathit{CE}}^\ast(\Lambda)=\Bbbk\langle\langle\overline{\mathit{LC}}_\ast(\Lambda)[-1]\rangle\rangle,
\end{equation}
which is the completed tensor algebra of $\mathbb{K}\langle\overline{\mathit{LC}}_\ast(\Lambda)[-1]\rangle$, regarded as a module over $\Bbbk$. In particular, there is a completion map
\begin{equation}\label{eq:completion-map}
\phi:\mathit{CE}^\ast(\Lambda)\rightarrow\widehat{\mathit{CE}}^\ast(\Lambda).
\end{equation}
Define a quiver $Q_\Lambda$ so that its vertices correspond to elements of $\Gamma$ , and its arrows are in correspondence with the set of Reeb chords $\mathcal{R}$. More precisely, for $v,w\in\Gamma$, there is an arrow from $v$ to $w$ for every Reeb chord in $\mathcal{R}$ from $\Lambda_v$ to $\Lambda_w$. In this way, the underlying $\Bbbk$-algebra of $\mathit{CE}^\ast(\Lambda)$ is the path algebra of $Q_\Lambda$, while the underlying $\Bbbk$-algebra of $\widehat{\mathit{CE}}^\ast(\Lambda)$ is the completed path algebra $\widehat{\mathbb{K}Q_\Lambda}$. In particular, $\widehat{\mathit{CE}}^\ast(\Lambda)$ is a complete augmented dg algebra in the sense of Theorem \ref{theorem:Koszul}.
\bigskip

We now apply Theorem \ref{theorem:Koszul} to concrete geometric situations. As a quick application, let $T$ be a tree with vertex set $T_0$. For each $v\in T_0$ we associate a simply-connected closed manifold $L_v$ of dimension $n\geq3$. For simplicity, we also assume that $L_v$ is $\mathit{Spin}$. Denote by $M_T$ the result of plumbing the cotangent bundles $T^\ast L_v$ according to the tree $T$. As a Weinstein manifold, $M_T$ admits a handlebody decomposition, whose associated subcritical Weinstein manifold is the plumbing of $T^\ast\overline{L}_v$ according to $T$, with $\overline{L}_v$ being the manifold with boundary obtained by carving out an open disc from $L_v$. Denote by $\Lambda_T=\bigsqcup_v\Lambda_v$ the union of the boundaries of the manifolds $\{L_v\}_{v\in T_0}$.

\begin{proposition}\label{proposition:plumbing}
The wrapped Fukaya category $\mathcal{W}(M_T)$ carries an exact Calabi-Yau structure.
\end{proposition}
\begin{proof}
According to \cite{ekl}, Theorem 83 (see also \cite{bee,te}), there is a surgery map
\begin{equation}\label{eq:surgery}
\Theta:\mathcal{W}_{M_T}\rightarrow\mathit{CE}^\ast(\Lambda_T),
\end{equation}
which induces an isomorphism on homologies. In particular, the dg algebra $\mathit{CE}^\ast(\Lambda_T)$ is homologically smooth. It suffices to show that the Chekanov-Eliashberg dg algebra $\mathit{CE}^\ast(\Lambda_T)$ is an exact Calabi-Yau algebra.

It follows from the proof of $\cite{ekl}$, Theorem 68 that $\mathit{CE}^\ast(\Lambda_T)$ is quasi-isomorphic to a dg algebra concentrated in degrees $\leq0$, and it is Koszul dual to the Fukaya $A_\infty$-algebra $\mathcal{F}_{M_T}$ of the compact cores $\{L_v\}_{v\in T_0}$. Combined with (\ref{eq:completed-CE}), we have $\mathit{CE}^\ast(\Lambda_T)\cong\widehat{\mathit{CE}}^\ast(\Lambda_T)$, so the completion map $\phi$ is a quasi-isomorphism. $\mathit{CE}^\ast(\Lambda_T)$ is therefore a complete dg algebra in the sense of Theorem \ref{theorem:Koszul}. Since $\mathrm{char}(\mathbb{K})=0$, Proposition \ref{proposition:proper-CY} implies that up to quasi-isomorphism, $\mathcal{F}_{M_T}$ carries a minimal cyclic $A_\infty$-structure. Now the conclusion follows from Theorem \ref{theorem:Koszul}.
\end{proof}

Note that this gives an alternative way of seeing that the cotangent bundle $T^\ast Q$ of a simply-connected manifold $Q$ admits a cyclic dilation, compare with our discussions at the end of Section \ref{section:Gysin}. It is an interesting question whether the Weinstein manifolds $M_T$ admit higher dilations.
\bigskip

Our second application deals with the specific case of the affine hypersurface $M_{3,3,3,3}\subset\mathbb{C}^4$. Recall that the Liouville 6-manifold $M_{3,3,3,3}$ arises as the Milnor fiber associated to the isolated singularity at the origin
\begin{equation}
x^3+y^3+z^3+w^3=0,
\end{equation}
which is known as a \textit{3-fold triple point}. The smoothing of this singularity has been studied by Smith-Thomas $\cite{st}$, according which we know that there is a basis of vanishing cycles in $M_{3,3,3,3}$ which consists of a configuration of 16 Lagrangian spheres, whose intersection pattern is indicated in Figure \ref{fig:vs}, where each arc represents a Lagrangian sphere.

\begin{figure}
	\centering
	\begin{tikzpicture}
	\draw [blue] (0,0) to (0.9,0);
	\draw [blue] (1.1,0) to (1.9,0);
	\draw [blue] (2.1,0) to (2.9,0);
	\draw [blue] (0,1) to (0.9,1);
	\draw [blue] (1.1,1) to (1.9,1);
	\draw [blue] (2.1,1) to (2.9,1);
	\draw [blue] (0,2) to (0.9,2);
	\draw [blue] (1.1,2) to (1.9,2);
	\draw [blue] (2.1,2) to (2.9,2);
	\draw [blue] (1,3) to (1,0);
	\draw [blue] (2,3) to (2,0);
	\draw [blue] (3,3) to (3,0);
	\draw [blue] (3.1,2) to [in=60,out=0] (5,-2);
	\draw [blue] (3.1,1) to [in=75,out=0] (4.7,-1.7);
	\draw [blue] (3.1,0) to [in=90,out=0] (4.4,-1.4);
	\draw [blue] (1,0) to [in=-150,out=-90] (5.1,-2.1);
	\draw [blue] (2,0) to [in=-175,out=-90] (4.8,-1.8);
	\draw [blue] (3,0) to [in=180,out=-90] (4.5,-1.5);
	\draw [green] (4,-1) to (5.5,-2.5);
	\draw [red] (0.8,1.5) to (1.5,2.2);
	\draw [red] (1.8,1.5) to (2.5,2.2);
	\draw [red] (2.8,1.5) to (3.5,2.2);
	\draw [red] (0.8,0.5) to (1.5,1.2);
	\draw [red] (0.8,-0.5) to (1.5,0.2);
	\draw [red] (1.8,0.5) to (2.5,1.2);
	\draw [red] (1.8,-0.5) to (2.5,0.2);
	\draw [red] (2.8,0.5) to (3.5,1.2);
	\draw [red] (2.8,-0.5) to (3.5,0.2);
	
	\node [blue] at (1,3.3) {$V_{\gamma1}$};
	\node [blue] at (2,3.3) {$V_{\gamma2}$};
	\node [blue] at (3,3.3) {$V_{\gamma3}$};
	\node [blue] at (-0.3,2) {$V_{1\gamma}$};
	\node [blue] at (-0.3,1) {$V_{2\gamma}$};
	\node [blue] at (-0.3,0) {$V_{3\gamma}$};
	\node [green] at (5.75,-2.75) {$V_{\gamma\gamma}$};
	\node [red] at (1.5,2.4) {$V_{11}$};
	\node [red] at (2.5,2.4) {$V_{12}$};
	\node [red] at (3.5,2.4) {$V_{13}$};
	\node [red] at (1.5,1.4) {$V_{21}$};
	\node [red] at (2.5,1.4) {$V_{22}$};
	\node [red] at (3.5,1.4) {$V_{23}$};
	\node [red] at (1.5,0.4) {$V_{31}$};
	\node [red] at (2.5,0.4) {$V_{32}$};
	\node [red] at (3.5,0.4) {$V_{33}$};
	\end{tikzpicture}
	\caption{Configuration of vanishing cycles in $M_{3,3,3,3}$, note that the spheres coloured in blue are mutually disjoint}
	\label{fig:vs}
\end{figure}

One obtains from this the Legendrian surgery description of $M_{3,3,3,3}$:

\begin{lemma}
The Milnor fiber $M_{3,3,3,3}$ is the result of attaching Weinstein 3-handles to $D^6$ along a Legendrian surface $\Lambda_{3,3,3,3}\subset(S^5,\xi_\mathit{std})$, which is a disjoint union of 16 standard unknotted Legendrian $S^2$'s. Up to Legendrian isotopy, the Legendrian fronts of 10 of the components in link $\Lambda_{3,3,3,3}$ are depicted in Figure \ref{fig:front}, where each Legendrian unknot in the picture should be understood as a 2-sphere obtained by spinning the 1-dimensional unknot around along the vertical axis of symmetry of its front projection. The remaining 6 components $\Lambda_{12},\Lambda_{13},\Lambda_{21},\Lambda_{23},\Lambda_{31},\Lambda_{32}$ are unknots linking $\Lambda_{1\gamma}$ and $\Lambda_{\gamma2}$, $\Lambda_{1\gamma}$ and $\Lambda_{\gamma3}$, $\Lambda_{2\gamma}$ and $\Lambda_{\gamma1}$, $\Lambda_{2\gamma}$ and $\Lambda_{\gamma3}$, $\Lambda_{3\gamma}$ and $\Lambda_{\gamma1}$, and $\Lambda_{3\gamma}$ and $\Lambda_{\gamma2}$ respectively, with all the linking numbers being $\pm1$. They are pairwise disjoint and disjoint from $\Lambda_{\gamma\gamma}$.
\end{lemma}
\begin{proof}
Consider the Lefschetz fibration $t:\mathbb{C}^3\rightarrow\mathbb{C}$ obtained as the Morsification of the polynomial $x^3+y^3+z^3$. The smooth fiber of $t$ is symplectomotphic to the Milnor fiber $T_{3,3,3}$ associated to the singularity $x^3+y^3+z^3=0$, and its total monodromy is the composition of Dehn twists along a basis of 8 vanishing cycles in $T_{3,3,3}$, see $\cite{ak1}$, Section 4.2 for a detailed description of this Lefschetz fibration. By $\cite{vk}$, Theorem 4.4, this describes $D^6$ as the result of attaching 8 Weinstein 3-handles to $T_{3,3,3}\times D^2$ along a link of 8 Legendrian 2-spheres in $T_{3,3,3}\times S^1$, which restricts to the basis of vanishing cycles in $T_{3,3,3}$. Moreover, $M_{3,3,3,3}$ also carries a Lefschetz fibration $\pi:M_{3,3,3,3}\rightarrow\mathbb{C}$, with $T_{3,3,3}$ as its smooth fiber, under which the vanishing cycles of $M_{3,3,3,3}$ described in Figure \ref{fig:vs} can be realized as Lagrangian matching spheres. Figure \ref{fig:base} gives a description of the base of $\pi$, where the 24 crosses are critical values, which are divided into three groups, and $\pi^{-1}(\star)$ is a smooth fiber. See $\cite{ak1}$, Section 2.5, where the detailed construction of such a Lefschetz fibration is explained. The blue arc in Figure \ref{fig:base} which ends at two different critical values of $\pi$ is the projection of a matching sphere $V\subset M_{3,3,3,3}$. This shows that $M_{3,3,3,3}$ can be constructed by attaching 24 Weinstein 3-handles to $T_{3,3,3}\times D^2$ along a link of 24 Legendrian $S^2$'s, whose restrictions in $\pi^{-1}(\star)$ are vanishing cycles of $\pi$. Note that it can be arranged so that the basis of vanishing cycles of $\pi$ contains the aforementioned basis of vanishing cycles of $t$ as a subset. Comparing with the handlebody decomposition of $D^6$ described above, this realizes the Weinstein domain $\overline{M}_{3,3,3,3}$ as $D^6$ with 16 Weinstein 3-handles attached along a link of 16 unknotted Legendrian $S^2$'s in $(S^5,\xi_\mathit{std})$. When restricting to the smooth fiber $\pi^{-1}(\star)$ of $\pi$, these Legendrian spheres form a subset of the basis of vanishing cycles of $\pi$, and every one of them lies in a matching sphere, in the fiber above $\star\in\mathbb{C}$. More precisely, consider the associated Lefschetz fibration $\bar{\pi}:\overline{M}_{3,3,3,3}\rightarrow D^2$ on the Liouville domain (with corners) $\overline{M}_{3,3,3,3}$ obtained by cutting off the cylindrical ends of the fibers, and removing the preimage of the part outside of the dashed circle in Figure \ref{fig:base}, where the fibration is locally trivial (this is the original set up of $\cite{ps1}$). If we cut the base of $\pi$ along the orange dashed arc in Figure \ref{fig:base}, the preimage under $\bar{\pi}$ of the lower left half of the disc, which we denote by $D_-$, with the corners rounded-off, is deformation equivalent to $D^6$, and the restrictions of the 16 matching spheres to $\bar{\pi}^{-1}(D_-)$ become exact Lagrangian fillings of the corresponding vanishing cycles, which are considered as Legendrian spheres in the contact boundary $\partial\bar{\pi}^{-1}(D_-)$. For example, for the matching sphere $V$ in the figure, its restriction $\overline{V}:=V\cap\bar{\pi}^{-1}(D_-)$ is a Lagrangian disc, which fills its boundary $\partial\overline{V}\subset\partial\bar{\pi}^{-1}(D_-)$, which is a Legendrian 2-sphere. In this way, the linking pattern of the 16 Legendrian $S^2$'s in $\partial\bar{\pi}^{-1}(D_-)$ is determined by the intersection pattern of the Lagrangian matching spheres in $M_{3,3,3,3}$, which is shown in Figure \ref{fig:vs}.
	
\begin{figure}
	\centering
	\begin{tikzpicture}
	\node at (-1.4,3) {$\times$};
	\node at (-1,3) {$\times$};
	\node at (-0.6,3) {$\times$};
	\node at (-0.2,3) {$\times$};
	\node at (0.2,3) {$\times$};
	\node at (0.6,3) {$\times$};
	\node at (1,3) {$\times$};
	\node at (1.4,3) {$\times$};
		
	\node at (-1.898,-2.711) {$\times$};
	\node at (-2.098,-2.365) {$\times$};
	\node at (-2.298,-2.019) {$\times$};
	\node at (-2.498,-1.673) {$\times$};
	\node at (-2.698,-1.327) {$\times$};
	\node at (-2.898,-0.981) {$\times$};
	\node at (-3.098,-0.635) {$\times$};
	\node at (-3.298,-0.289) {$\times$};
		
	\node at (1.898,-2.711) {$\times$};
	\node at (2.098,-2.365) {$\times$};
	\node at (2.298,-2.019) {$\times$};
	\node at (2.498,-1.673) {$\times$};
	\node at (2.698,-1.327) {$\times$};
	\node at (2.898,-0.981) {$\times$};
	\node at (3.098,-0.635) {$\times$};
	\node at (3.298,-0.289) {$\times$};
		
	\node at (0,0) {$\star$};
	\draw [orange,dashed] (-2.8,2.8) to [in=120,out=-60] (0,0);
	\draw [orange,dashed] (0,0) to [in=120,out=-60] (1.6,-3.6);
	\draw [dashed] (0,0) circle [radius=4];
	\draw [blue] (1,3) to [in=0,out=-90] (0,0);
	\draw [blue] (0,0) to [in=0,out=180] (-3.098,-0.635);
	\node [blue] at (1.2,1) {$V$};
	\node [blue] at (-1.5,-0.6) {$\overline{V}$};
	\end{tikzpicture}
	\caption{Base of the Lefschetz fibration $\pi:M_{3,3,3,3}\rightarrow\mathbb{C}$}
	\label{fig:base}
	\end{figure}
\end{proof}

In Figure \ref{fig:front}, we have arranged so that the labellings of the components of $\Lambda_{3,3,3,3}$ coincide with that of the vanishing cycles in Figure \ref{fig:vs}, which means that for any vanishing cycle $V_{\bullet\bullet}$, the Lagrangian 3-disc $V_{\bullet\bullet}\cap D^6$ is a filling of the component $\Lambda_{\bullet\bullet}$ with the same labelling. The set of Lagrangian cocores $\{L_{\bullet\bullet}\}$ will be labelled in the same way, with $L_{\bullet\bullet}$ being the cocore of the 3-handle attached along $\Lambda_{\bullet\bullet}$. As a consequence, $V_{\bullet\bullet}$ intersects $L_{\bullet\bullet}$ non-trivially and transversely at a unique point if and only if they have the same labelling. We denote by $\mathcal{W}_{M_{3,3,3,3}}$ the Fukaya $A_\infty$-algebra of the cocores $\{L_{\bullet\bullet}\}$.

Regarding $\Lambda_{3,3,3,3}$ as a Legendrian surface in $J^1(\mathbb{R}^2)$, one can consider its image under the base projection $p_x:J^1(\mathbb{R}^2)\rightarrow\mathbb{R}^2$. Under suitable Legendrian isotopies, the image $p_x(\Lambda_{3,3,3,3})\subset\mathbb{R}^2$ consists of 31 circles, see Figure \ref{fig:bp}. The largest solid circle is the projection of the cusp edges of the spheres $\Lambda_{11},\cdots,\Lambda_{33},\Lambda_{\gamma\gamma}$, the 6 small solid circles are projections of the cusp edges of the remaining components $\Lambda_{1\gamma},\Lambda_{\gamma1},\Lambda_{2\gamma},\Lambda_{\gamma2},\Lambda_{3\gamma},\Lambda_{\gamma3}$, which are coloured blue in Figure \ref{fig:front}. For every one of these solid circles, there is a slightly smaller dashed circle, which is the projection of a crossing arc between a blue sphere and $\Lambda_{\gamma\gamma}$. Finally, in each of the 6 dashed circles there are 3 small dashed circles, which are projections of the crossing arcs formed by a blue sphere and a red one. Note that for each blue sphere, say $\Lambda_{1\gamma}$, its front intersects with three red spheres $\Lambda_{11}$, $\Lambda_{12}$ and $\Lambda_{13}$.

\begin{figure}
	\centering
	\begin{tikzpicture}
	\draw [green] (-2,0) to [in=180,out=0] (-0.33,-2);
	\draw [green] (-0.33,-2) to [in=180,out=0] (0.5,-1);
	\draw [green] (0.5,-1) to [in=180,out=0] (1.33,-2);
	\draw [green] (1.33,-2) to [in=180,out=0] (2.16,-1);
	\draw [green] (2.16,-1) to [in=180,out=0] (3,-2);
	\draw [green] (3,-2) to [in=180,out=0] (3.83,-1);
	\draw [green] (3.83,-1) to [in=180,out=0] (4.67,-2);
	\draw [green] (4.67,-2) to [in=180,out=0] (5.5,-1);
	\draw [green] (5.5,-1) to [in=180,out=0] (6.33,-2);
	\draw [green] (6.33,-2) to [in=180,out=0] (7.16,-1);
	\draw [green] (7.16,-1) to [in=180,out=0] (8,-2);
	\draw [green] (8,-2) to [in=180,out=0] (9.67,0);
	\draw [green] (-2,0) to [in=180,out=0] (3.83,2);
	\draw [green] (3.83,2) to [in=180,out=0] (9.67,0);
	\node [green] at (10,0) {\small $\Lambda_{\gamma\gamma}$};
	
	\draw [blue] (-1,-2.1) to [in=180,out=0] (-0.33,-1.7);
	\draw [blue] (-0.33,-1.7) to [in=180,out=0] (0.34,-2.1);
	\draw [blue] (-1,-2.1) to [in=180,out=0] (-0.33,-2.5);
	\draw [blue] (-0.33,-2.5) to [in=180,out=0] (0.34,-2.1);
	\node [blue] at (-0.33,-1.5) {\small $\Lambda_{1\gamma}$};
	
	\draw [blue] (0.66,-2.1) to [in=180,out=0] (1.33,-1.7);
	\draw [blue] (1.33,-1.7) to [in=180,out=0] (2,-2.1);
	\draw [blue] (0.66,-2.1) to [in=180,out=0] (1.33,-2.5);
	\draw [blue] (1.33,-2.5) to [in=180,out=0] (2,-2.1);
	\node [blue] at (1.33,-1.5) {\small $\Lambda_{\gamma1}$};
	
	\draw [blue] (2.33,-2.1) to [in=180,out=0] (3,-1.7);
	\draw [blue] (3,-1.7) to [in=180,out=0] (3.67,-2.1);
	\draw [blue] (2.33,-2.1) to [in=180,out=0] (3,-2.5);
	\draw [blue] (3,-2.5) to [in=180,out=0] (3.67,-2.1);
	\node [blue] at (3,-1.5) {\small $\Lambda_{2\gamma}$};
	
	\draw [blue] (4,-2.1) to [in=180,out=0] (4.67,-1.7);
	\draw [blue] (4.67,-1.7) to [in=180,out=0] (5.34,-2.1);
	\draw [blue] (4,-2.1) to [in=180,out=0] (4.67,-2.5);
	\draw [blue] (4.67,-2.5) to [in=180,out=0] (5.34,-2.1);
	\node [blue] at (4.67,-1.5) {\small $\Lambda_{\gamma2}$};
	
	\draw [blue] (5.66,-2.1) to [in=180,out=0] (6.33,-1.7);
	\draw [blue] (6.33,-1.7) to [in=180,out=0] (7,-2.1);
	\draw [blue] (5.66,-2.1) to [in=180,out=0] (6.33,-2.5);
	\draw [blue] (6.33,-2.5) to [in=180,out=0] (7,-2.1);
	\node [blue] at (6.33,-1.5) {\small $\Lambda_{3\gamma}$};
	
	\draw [blue] (7.33,-2.1) to [in=180,out=0] (8,-1.7);
	\draw [blue] (8,-1.7) to [in=180,out=0] (8.67,-2.1);
	\draw [blue] (7.33,-2.1) to [in=180,out=0] (8,-2.5);
	\draw [blue] (8,-2.5) to [in=180,out=0] (8.67,-2.1);
	\node [blue] at (8,-1.5) {\small $\Lambda_{\gamma3}$};
	
	\draw [red] (-1,-2.6) to [in=180,out=0] (-0.33,-2.2);
	\draw [red] (-0.33,-2.2) to [in=180,out=0] (0.5,-2.6);
	\draw [red] (0.5,-2.6) to [in=180,out=0] (1.33,-2.2);
	\draw [red] (1.33,-2.2) to [in=180,out=0] (2,-2.6);
	\draw [red] (-1,-2.6) to [in=180,out=0] (0.5,-3.6);
	\draw [red] (0.5,-3.6) to [in=180,out=0] (2,-2.6);
	\node [red] at (0.5,-3.8) {\small $\Lambda_{11}$};
	
	\draw [red] (2.33,-2.6) to [in=180,out=0] (3,-2.2);
	\draw [red] (3,-2.2) to [in=180,out=0] (3.835,-2.6);
	\draw [red] (3.835,-2.6) to [in=180,out=0] (4.67,-2.2);
	\draw [red] (4.67,-2.2) to [in=180,out=0] (5.34,-2.6);
	\draw [red] (2.33,-2.6) to [in=180,out=0] (3.835,-3.6);
	\draw [red] (3.835,-3.6) to [in=180,out=0] (5.34,-2.6);
	\node [red] at (3.835,-3.8) {\small $\Lambda_{22}$};
	
	\draw [red] (5.66,-2.6) to [in=180,out=0] (6.33,-2.2);
	\draw [red] (6.33,-2.2) to [in=180,out=0] (7.165,-2.6);
	\draw [red] (7.165,-2.6) to [in=180,out=0] (8,-2.2);
	\draw [red] (8,-2.2) to [in=180,out=0] (8.67,-2.6);
	\draw [red] (5.66,-2.6) to [in=180,out=0] (7.165,-3.6);
	\draw [red] (7.165,-3.6) to [in=180,out=0] (8.67,-2.6);
	\node [red] at (7.165,-3.8) {\small $\Lambda_{33}$};
    \end{tikzpicture}
	\caption{Front view of the Legendrian front of $\Lambda_{3,3,3,3}$, where the components $\Lambda_{12},\Lambda_{13},\Lambda_{21},\Lambda_{23},\Lambda_{31},\Lambda_{32}$ are omitted since they are covered by the other components}
	\label{fig:front}
\end{figure}

\begin{figure}
	\centering
	\begin{tikzpicture}
	\draw (-1.5,2.598) circle [radius=1.3];
	\draw (1.5,2.598) circle [radius=1.3];
	\draw (-3,0) circle [radius=1.3];
	\draw (3,0) circle [radius=1.3];
	\draw (-1.5,-2.598) circle [radius=1.3];
	\draw (1.5,-2.598) circle [radius=1.3];
	\draw [dashed] (-1.5,2.598) circle [radius=1];
	\draw [dashed] (1.5,2.598) circle [radius=1];
	\draw [dashed] (-3,0) circle [radius=1];
	\draw [dashed] (3,0) circle [radius=1];
	\draw [dashed] (-1.5,-2.598) circle [radius=1];
	\draw [dashed] (1.5,-2.598) circle [radius=1];
	\draw (0,0) circle [radius=5];
	\draw [dashed] (-1.5,3.098) circle [radius=0.3];
	\draw [dashed] (-1.933,2.348) circle [radius=0.3];
	\draw [dashed] (-1.067,2.348) circle [radius=0.3];
	\draw [dashed] (1.5,2.098) circle [radius=0.3];
	\draw [dashed] (1.933,2.848) circle [radius=0.3];
	\draw [dashed] (1.067,2.848) circle [radius=0.3];
	\draw [dashed] (-3,-0.5) circle [radius=0.3];
	\draw [dashed] (-3.433,0.25) circle [radius=0.3];
	\draw [dashed] (-2.567,0.25) circle [radius=0.3];
	\draw [dashed] (3,0.5) circle [radius=0.3];
	\draw [dashed] (2.567,-0.25) circle [radius=0.3];
	\draw [dashed] (3.433,-0.25) circle [radius=0.3];
	\draw [dashed] (-1.5,-2.098) circle [radius=0.3];
	\draw [dashed] (-1.933,-2.848) circle [radius=0.3];
	\draw [dashed] (-1.067,-2.848) circle [radius=0.3];
	\draw [dashed] (1.5,-3.098) circle [radius=0.3];
	\draw [dashed] (1.067,-2.348) circle [radius=0.3];
	\draw [dashed] (1.933,-2.348) circle [radius=0.3];
	
	\draw [orange] (-5,0) to (-4,0);
	\draw [orange, rotate around={60:(0,0)}] (-5,0) to (-4,0);
	\draw [orange, rotate around={120:(0,0)}] (-5,0) to (-4,0);
	\draw [orange, rotate around={180:(0,0)}] (-5,0) to (-4,0);
	\draw [orange, rotate around={240:(0,0)}] (-5,0) to (-4,0);
	\draw [orange, rotate around={300:(0,0)}] (-5,0) to (-4,0);
	
	\draw [orange, rotate around={60:(-3,0)}] (-3.693,-0.4) to (-3.866,-0.5);
	\draw [orange, rotate around={180:(-3,0)}] (-3.693,-0.4) to (-3.866,-0.5);
	\draw [orange, rotate around={300:(-3,0)}] (-3.693,-0.4) to (-3.866,-0.5);
	
	\draw [orange] (-2.193,2.198) to (-2.366,2.098);
	\draw [orange, rotate around={120:(-1.5,2.598)}] (-2.193,2.198) to (-2.366,2.098);
	\draw [orange, rotate around={240:(-1.5,2.598)}] (-2.193,2.198) to (-2.366,2.098);
	
	\draw [orange] (3.693,-0.4) to (3.866,-0.5);
	\draw [orange, rotate around={120:(3,0)}] (3.693,-0.4) to (3.866,-0.5);
	\draw [orange, rotate around={240:(3,0)}] (3.693,-0.4) to (3.866,-0.5);
	
	\draw [orange] (2.193,-2.198) to (2.366,-2.098);
	\draw [orange, rotate around={120:(1.5,-2.598)}] (2.193,-2.198) to (2.366,-2.098);
	\draw [orange, rotate around={240:(1.5,-2.598)}] (2.193,-2.198) to (2.366,-2.098);
	
	\draw [orange, rotate around={60:(1.5,2.598)}] (2.193,2.198) to (2.366,2.098);
	\draw [orange, rotate around={180:(1.5,2.598)}] (2.193,2.198) to (2.366,2.098);
	\draw [orange, rotate around={300:(1.5,2.598)}] (2.193,2.198) to (2.366,2.098);
	
	\draw [orange, rotate around={60:(-1.5,-2.598)}] (-2.193,-2.198) to (-2.366,-2.098);
	\draw [orange, rotate around={180:(-1.5,-2.598)}] (-2.193,-2.198) to (-2.366,-2.098);
	\draw [orange, rotate around={300:(-1.5,-2.598)}] (-2.193,-2.198) to (-2.366,-2.098);
	
	\node [orange] at (-5.2,-0.2) {\small $e_1^0$};
	\node [orange] at (-4.5,-0.2) {\small $e_2^0$};
	\node [orange] at (-3.8,-0.2) {\small $e_3^0$};
	\node [orange] at (-4.7,0.2) {\small $e_2^1$};
	\node [orange] at (-4.2,0.2) {\small $e_3^1$};
	
	\end{tikzpicture}
	\caption{Base projection of $\Lambda_{3,3,3,3}$, where the orange arcs are auxiliary edges added in order to obtain a cellular decomposition}
	\label{fig:bp}
\end{figure}

\begin{lemma}\label{lemma:grading}
Up to quasi-isomorphism, the wrapped Fukaya $A_\infty$-algebra $\mathcal{W}_{M_{3,3,3,3}}$ is concentrated in degrees $\leq0$.
\end{lemma}
\begin{proof}
By the surgery quasi-isomorphism (\ref{eq:surgery}), it suffices to check that the Chekanov-Eliashberg dg algebra $\mathit{CE}^\ast(\Lambda_{3,3,3,3})$ is concentrated in non-positive degrees up to quasi-isomorphism.
	
Equip the Legendrian link $\Lambda_{3,3,3,3}$ with a Maslov potential $\mu:\Lambda_{3,3,3,3}\rightarrow\mathbb{Z}$ as follows. For the green component $\Lambda_{\gamma\gamma}$, let $\mu=1$ on the upper strand, and $\mu=0$ on the lower strand. For each blue component, put $\mu=0$ on the upper strand, and $\mu=-1$ on the lower strand. Finally, for each red component, set $\mu=-1$ on the upper strand, and $\mu=-2$ on the lower strand.
	
To show that $\mathit{CE}^\ast(\Lambda_{3,3,3,3})$ is quasi-isomorphic to a dg algebra with all the generators concentrated in non-positive degrees, we make use of its cellular model introduced by Rutherford-Sullivan $\cite{rs}$. Recall that Rutherford-Sullivan's combinatorial model of $\mathit{CE}^\ast(\Lambda_{3,3,3,3})$ starts with a $\Lambda_{3,3,3,3}$-compatible polygonal decomposition associated to the base projection of $\Lambda_{3,3,3,3}$, which can be obtained by subdividing the base projection of $\Lambda_{3,3,3,3}$. For example, for the base projection shown in Figure \ref{fig:bp}, one way to obtain its $\Lambda_{3,3,3,3}$-compatible polygonal decomposition is to add the orange arcs. For each $i$-cell $e_\alpha^i$, where $0\leq i\leq 2$, the set of sheets of $\Lambda_{3,3,3,3}$ above $e_\alpha^i$ can be equipped with a partial ordering $\prec$ based on their heights. In particular, if the sheet $S_m$ has larger height than $S_n$, then $S_m\prec S_n$, and we label the sheets so that $m<n$. There is a generator of $\mathit{CE}^\ast(\Lambda_{3,3,3,3})$ associated to each pair of sheets $(S_m,S_n)$ above $e_\alpha^i$ with $S_m\prec S_n$. The generators associated to 0-, 1-, and 2-cells labelled by $\alpha$ are denoted respectively by $a_\alpha^{m,n}$, $b_\alpha^{m,n}$ and $c_\alpha^{m,n}$, and they can be assembled in strictly upper triangular matrices $A_\alpha$, $B_\alpha$ and $C_\alpha$. The gradings of these generators are given as follows:
\begin{equation}
|a_\alpha^{m,n}|=\mu(S_n)-\mu(S_m)+1,|b_\alpha^{m,n}|=\mu(S_n)-\mu(S_m),|c_\alpha^{m,n}|=\mu(S_n)-\mu(S_m)-1.
\end{equation}
	
Observe that by our choice of the Maslov potential $\mu$, for any sheet $S_m$ with larger height than $S_n$, we necessarily have $\mu(S_m)\geq\mu(S_n)$, which implies that all the generators of the form $b_\alpha^{m,n}$ and $c_\alpha^{m,n}$ have non-positive gradings. However, there may still be generators $a_\alpha^{m,n}$ with $|a_\alpha^{m,n}|=1$. For each such generator, consider any 1-cell $e_\beta^1$ in the $\Lambda_{3,3,3,3}$-compatible polygonal decomposition which ends at $e_\alpha^0$, with the other end point being $e_\gamma^0$ for some $\gamma$. The differential $d_C$ on cellular dg algebra is defined in completely combinatorial manners. In particular, by $\cite{rs}$, Section 3.6.2 we have (over $\mathbb{K}=\mathbb{Z}/2$)
\begin{equation}\label{eq:diff}
d_Cb_\beta^{m,n}=a_\alpha^{m,n}+a_\gamma^{m,n}+\sum_{m<k<n}a_\alpha^{m,k}b_\beta^{k,n}+\sum_{m<k<n}b_\beta^{m,k}a_\gamma^{k,n}.
\end{equation}
By a lemma of Chekanov ($\cite{rs}$, Theorem 2.1), there is a quasi-isomorphism
\begin{equation}
\mathit{CE}^\ast(\Lambda_{3,3,3,3})\cong\mathit{CE}^\ast(\Lambda_{3,3,3,3})/\left\langle d_Cb_\beta^{m,n},b_\beta^{m,n}\right\rangle
\end{equation}
between $\mathit{CE}^\ast(\Lambda_{3,3,3,3})$ and its quotient dg algebra, where $\left\langle d_Cb_\beta^{m,n},b_\beta^{m,n}\right\rangle$ is the ideal generated by $d_Cb_\beta^{m,n}$ and $b_\beta^{m,n}$. By \cite{yl}, Lemma 6.1, if one can always achieve that $a_\gamma^{m,n}=0$ or $1$ in (\ref{eq:diff}), then by replacing the cellular model of $\mathit{CE}^\ast(\Lambda_{3,3,3,3})$ with a quasi-isomorphic (actually, stable tame isomorphic) dg algebra if necessary, all the generators with positive degrees can be cancelled out. Here, we check this explicitly for the generator $a_3^{2,3}$ associated to the 0-cell $e_3^0$ in Figure \ref{fig:bp} explicitly, the verifications for the other generators $a_\alpha^{m,n}$ with $|a_\alpha^{m,n}|=1$ are similar. Since the Legendrian front of $\Lambda_{3,3,3,3}$ above the largest solid circle consists of cusp edges, it follows that $a_2^{2,3}=0$. Applying (\ref{eq:diff}) to $b_2^{2,3}$, which is a generator associated to the 1-cell $e_2^1$, we have $d_Cb_2^{2,3}=a_2^{2,3}$, so after passing to the quotient dg algebra $\mathit{CE}^\ast(\Lambda_{3,3,3,3})/\left\langle b_2^{2,3},a_2^{2,3}\right\rangle$, we have by (\ref{eq:diff}) applied to the generator $b_3^{2,3}$ associated to $e_3^1$ that $d_Cb_3^{2,3}=a_3^{2,3}$. The dg algebra $\mathit{CE}^\ast(\Lambda_{3,3,3,3})/\left\langle b_2^{2,3},a_2^{2,3},b_3^{2,3},a_3^{2,3}\right\rangle$ is quasi-isomorphic to $\mathit{CE}^\ast(\Lambda_{3,3,3,3})$ and has one less generator of positive degree.

Since the Legendrian front of $\Lambda_{3,3,3,3}$ does not involve any swallowtail singularity, as explained in the proof of $\cite{yl}$, Lemma 8.1, the cancellation of the generators does not depend on the ground field $\mathbb{K}$, so we conclude that $\mathit{CE}^\ast(\Lambda_{3,3,3,3})$ is quasi-isomorphic to a dg algebra with all the generators concentrated in degrees $\leq0$ over any field $\mathbb{K}$.
\end{proof}

\bigskip

Denote by $\tau_V$ the Dehn twist along the Lagrangian sphere $V\subset M_{3,3,3,3}$, it follows from $\cite{ps9}$, Lemmas 4.15 and 4.16 that
\begin{equation}\label{eq:Dehn}
\left(\tau_{V_{11}}\circ\cdots\circ\tau_{V_{33}}\circ\tau_{V_{\gamma1}}\circ\tau_{V_{\gamma2}}\circ\tau_{V_{\gamma3}}\circ\tau_{V_{1\gamma}}\circ\tau_{V_{2\gamma}}\circ\tau_{V_{3\gamma}}\circ\tau_{V_{\gamma\gamma}}\right)^3=[-2].
\end{equation}
Since the right hand side of (\ref{eq:Dehn}) is a non-trivial degree shift, by Seidel's long exact sequence $\cite{ps10}$, the compact Fukaya category $\mathcal{F}(M_{3,3,3,3})$ is split-generated by the Lagrangian spheres
\begin{equation}\label{eq:vs}
V_{11},\cdots,V_{33},V_{\gamma1},V_{\gamma2},V_{\gamma3},V_{1\gamma},V_{2\gamma},V_{3\gamma},V_{\gamma\gamma}.
\end{equation}
Denote by $\mathcal{F}_{M_{3,3,3,3}}$ the Fukaya $A_\infty$-algebra of these vanishing cycles. As a Corollary to Lemma \ref{lemma:grading}, we have the following:

\begin{corollary}\label{corollary:grading}
The Lagrangian spheres $\{V_{\bullet\bullet}\}$ admit gradings for which the $A_\infty$-algebra $\mathcal{F}_{M_{3,3,3,3}}$ is concentrated in degrees $\geq0$, and its degree 0 part is isomorphic to $\Bbbk:=\bigoplus_{i=1}^{16}\mathbb{K}e_i$.
\end{corollary}
\begin{proof}
It follows from the Eilenberg-Moore equivalence (\ref{eq:EM}) that as a $\Bbbk$-bimodule
\begin{equation}
\mathcal{F}_{M_{3,3,3,3}}\cong\Bbbk\oplus\bigoplus_{i=1}^\infty\overline{\mathit{CE}}_\ast(\Lambda_{3,3,3,3})[-1]^{\otimes_\Bbbk i}.
\end{equation}
It follows that any non-idempotent generator of $\mathcal{F}_{M_{3,3,3,3}}$ is of the form $a_1^\vee[-1]\cdots a_m^\vee[-1]$ for some $m\geq1$, where $a_i$ is a generator of $\mathit{CE}^\ast(\Lambda_{3,3,3,3})$ for any $1\leq i\leq m$. By Lemma \ref{lemma:grading}, $\mathit{CE}^\ast(\Lambda_{3,3,3,3})$ is non-positively graded up to quasi-isomorphism, therefore we may assume that $|a_i|\leq0$ for each $i$, which implies that
\begin{equation}
|a_1^\vee[-1]\cdots a_m^\vee[-1]|=(-|a_1|+1)\cdots(-|a_m|+1)\geq1.
\end{equation}
\end{proof}

\begin{remark}
In fact, a more careful study of the dg algebra $\mathit{CE}^\ast(\Lambda_{3,3,3,3})$ implies that it is quasi-isomorphic to a dg algebra freely generated by Reeb chords supported in degrees $-2\leq\ast\leq0$, which means that one can arrange the gradings so that $\mathcal{F}_{M_{3,3,3,3}}$ is concentrated in degrees $0\leq\ast\leq3$. A similar result is proved for the WKB algebra of Lagrangian 3-spheres in $\cite{is}$, Lemma 4.5. We expect that the same grading property holds for the Fukaya $A_\infty$-algebra of a basis of vanishing cycles in any Milnor fiber $M_{a_1,\cdots,a_{n+1}}$ with $\sum_{i=1}^{n+1}\frac{1}{a_i}>1$.
\end{remark}

The final ingredient needed for the proof of Theorem \ref{theorem:Fano} is the following Koszul duality result, which is essentially due to Lekili-Ueda \cite{lu}.

\begin{proposition}\label{proposition:KD0}
There are quasi-isomorphisms between $A_\infty$-algebras over $\Bbbk$:
\begin{equation}\label{eq:Koszulw}
R\mathrm{Hom}_{\mathcal{F}_{M_{3,3,3,3}}}(\Bbbk,\Bbbk)\cong\mathcal{W}_{M_{3,3,3,3}},R\mathrm{Hom}_{\mathcal{W}_{M_{3,3,3,3}}}(\Bbbk,\Bbbk)\cong\mathcal{F}_{M_{3,3,3,3}}.
\end{equation}
\end{proposition}
\begin{proof}
The second quasi-isomorphism follows from the Eilenberg-Moore equivalence (\ref{eq:EM}) and the surgery quasi-isomorphism (\ref{eq:surgery}). For the first quasi-isomorphism, we use (\ref{eq:Dehn}), which implies by \cite{lu}, Lemma 6.6 that the wrapped Floer cochain complex $\mathit{CW}^\ast(L,K)$ is bounded above for any two objects $L,K$ of $\mathcal{W}(M_{3,3,3,3})$. By \cite{lu}, Proposition 6.5, for any object $K$ of $\mathcal{W}(M_{3,3,3,3})$, there exists a sequence of objects $(K_i)_{i\in\mathbb{N}}$ of $\mathcal{F}(M_{3,3,3,3})^\mathit{perf}$, the $A_\infty$-category of perfect modules over $\mathcal{F}(M_{3,3,3,3})$, such that for any fixed $j\in\mathbb{Z}$, there exists an integer $i\gg1$ such that
\begin{equation}
\mathrm{Hom}_{\mathcal{W}(M_{3,3,3,3})^\mathit{perf}}^j(L,K)\cong\mathrm{Hom}_{\mathcal{W}(M_{3,3,3,3})^\mathit{perf}}^j(L,K_i)
\end{equation}
as $\mathbb{K}$-vector spaces for any object $L$ of $\mathcal{W}(M_{3,3,3,3})$. Note that in the above isomorphism, we have identified $K_i$ with objects of $\mathcal{W}(M_{3,3,3,3})^\mathit{perf}$ via the fully faithful embedding $\mathcal{F}(M_{3,3,3,3})^\mathit{perf}\hookrightarrow\mathcal{W}(M_{3,3,3,3})^\mathit{perf}$. In particular, $\mathit{CW}^j(L,K)$ of any two objects $L,K$ of $\mathcal{W}(M_{3,3,3,3})$ is finite-dimensional for any fixed $j$. Under the surgery quasi-isomorphism (\ref{eq:surgery}), this translates into the fact that the Chekanov-Eliashberg dg algebra $\mathit{CE}^\ast(\Lambda_{3,3,3,3})$ is locally finite (finite-dimensional in each degree) as a module over $\Bbbk$. It follows that the filtration on $\mathit{CE}^\ast(\Lambda_{3,3,3,3})$ by word length is complete and Hausdorff, so the completion map (\ref{eq:completion-map}) is a quasi-isomorphism for the Legendrian link $\Lambda_{3,3,3,3}$. By (\ref{eq:completed-CE}), we get the desired quasi-isomorphism
\begin{equation}
\mathcal{W}_{M_{3,3,3,3}}\cong\mathit{CE}^\ast(\Lambda_{3,3,3,3})\cong(\mathrm{B}\mathcal{F}_{M_{3,3,3,3}})^\#\cong R\mathrm{Hom}_{\mathcal{F}_{M_{3,3,3,3}}}(\Bbbk,\Bbbk).
\end{equation}
\end{proof}

\begin{remark}
\cite{lu}, Theorem 6.11 proves a general version of Koszul duality between compact and wrapped Fukaya $A_\infty$-algebras for Milnor fibers associated to weighted homogeneous singularities $\{w(z_1,\cdots,z_{n+1})=0\}\subset\mathbb{C}^{n+1}$. However, they assumed that the Fukaya-Seidel category $\mathcal{F}(w)$ admits a strong full exceptional collection, which isn't satisfied for the polynomial $w(z_1,z_2,z_3,z_4)=z_1^3+z_2^3+z_3^3+z_4^3$.
\end{remark}

\begin{proof}[Proof of Theorem \ref{theorem:Fano}]
It follows from Proposition \ref{proposition:proper-CY} that $\mathcal{F}_{M_{3,3,3,3}}$ is quasi-isomorphic to a minimal cyclic $A_\infty$-algebra. Combining (\ref{eq:completed-CE}) with Proposition \ref{proposition:KD0} we see that $\mathcal{W}_{M_{3,3,3,3}}$ is quasi-isomorphic to a complete dg algebra, whose Koszul dual is $\mathcal{F}_{M_{3,3,3,3}}$. Moreover, Lemma \ref{lemma:grading} implies that $H^\ast\left(\mathcal{W}_{M_{3,3,3,3}}\right)$ is supported in non-positive degrees. Applying Theorem \ref{theorem:Koszul} completes the proof.
\end{proof}

As a by-product, we have the following non-formality result. A similar result is proved in $\cite{lu}$, Theorem 7.3.

\begin{corollary}\label{corollary:non-formality}
The Fukaya $A_\infty$-algebra $\mathcal{F}_{M_{3,3,3,3}}$ is not formal over $\Bbbk$.
\end{corollary}
\begin{proof}
Suppose that $\mathcal{F}_{M_{3,3,3,3}}$ is formal, then there is a quasi-isomorphism
\begin{equation}
\mathcal{F}_{M_{3,3,3,3}}\cong F_{M_{3,3,3,3}}:=H^\ast(\mathcal{F}_{M_{3,3,3,3}}).
\end{equation}
On the cohomology level, the proper Calabi-Yau structure on $\mathcal{F}_{M_{3,3,3,3}}$ induces a non-degenerate pairing
\begin{equation}
\langle\cdot,\cdot\rangle_\mathit{CY}:F_{M_{3,3,3,3}}\otimes F_{M_{3,3,3,3}}\rightarrow\Bbbk[-3],
\end{equation}
which makes $\left(F_{M_{3,3,3,3}},\langle\cdot,\cdot\rangle_\mathit{CY}\right)$ a Frobenius algebra. Denote by $\Delta_F$ the BV operator on the Hochschild cohomology $\mathit{HH}^\ast(F_{M_{3,3,3,3}})$. For any class $c\in\mathit{HH}^{1,0}(F_{M_{3,3,3,3}})$ and $a\in F_{M_{3,3,3,3}}$, we have
\begin{equation}
\langle\Delta_F(c),a\rangle_\mathit{CY}=\langle c(a),1_F\rangle_\mathit{CY},
\end{equation}
where $1_F\in F_{M_{3,3,3,3}}$ is the identity. When applied to the Euler vector field $\mathit{eu}_F\in\mathit{HH}^{1,0}(F(M_{3,3,3,3}))$, we obtain
\begin{equation}
\left\langle\Delta_F\left(\frac{1}{3}\mathit{eu}_F\right),a\right\rangle_\mathit{CY}=\langle a,1_F\rangle_\mathit{CY}=\langle1_F,a\rangle_\mathit{CY}
\end{equation}
for any $a\in F_{M_{3,3,3,3}}$ of degree 3. It follows that
\begin{equation}
\Delta_F\left(\frac{1}{3}\mathit{eu}_F\right)=1.
\end{equation}
Denote by $\mathit{eu}_\mathcal{F}$ the image of the class $\mathit{eu}_F$ under the BV algebra isomorphism $\mathit{HH}^\ast(F_{M_{3,3,3,3}})\cong\mathit{HH}^\ast(\mathcal{F}(M_{3,3,3,3}))$. We have $\Delta_\mathcal{F}(\mathit{eu}_\mathcal{F})=1$, where $\Delta_\mathcal{F}$ is the BV operator on $\mathit{HH}^\ast(\mathcal{F}(M_{3,3,3,3}))$.

By Proposition \ref{proposition:KD0}, there is an isomorphism $\mathit{HH}^\ast(\mathcal{F}(M_{3,3,3,3}))\cong\mathit{HH}^\ast(\mathcal{W}(M_{3,3,3,3}))$ as Gerstenhaber algebras, under which $\mathit{eu}_\mathcal{F}$ goes to a class $b\in\mathit{HH}^1(\mathcal{W}(M_{3,3,3,3}))$. Since changing the Calabi-Yau structure on $\mathcal{W}(M_{3,3,3,3})$ amounts to applying the conjugation action of an invertible element $h_\mathcal{W}\in\mathit{HH}^0(\mathcal{W}(M_{3,3,3,3}))^\times$ to the BV operator
\begin{equation} \Delta_\mathcal{W}:\mathit{HH}^\ast(\mathcal{W}(M_{3,3,3,3}))\rightarrow\mathit{HH}^{\ast-1}(\mathcal{W}(M_{3,3,3,3})),
\end{equation}
the class $b$ satisfies $\Delta_\mathcal{W}(h_\mathcal{W}b)=h_\mathcal{W}$. Under the BV algebra isomorphism
\begin{equation}
\mathit{HH}^\ast(\mathcal{W}(M_{3,3,3,3}))\cong\mathit{SH}^\ast(M_{3,3,3,3})
\end{equation}
established by Ganatra (cf. \cite{sg2}, Theorem 1.1), the image of $b$ defines a quasi-dilation in $\mathit{SH}^1(M_{3,3,3,3})$.

To complete the proof, it suffices to show that $M_{3,3,3,3}$ does not admit a quasi-dilation. we follow the argument of $\cite{ps4}$, Example 2.7 to show that $M_{3,3,3,3}$ does not admit a quasi-dilation. Consider the Milnor fiber $M$ of a 5-fold triple point, which is the affine hypersurface in $\mathbb{C}^6$ given by the equation
\begin{equation}
z_1^3+z_2^3+z_3^3+z_4^3+z_5^3+z_6^3=1.
\end{equation}
Since it is the complement of a smooth divisor in the Fermat projective cubic 5-fold, there is a Morse-Bott spectral sequence $\cite{ps2}$ which converges to $\mathit{SH}^\ast(M)$. Using this one can deduce that $\mathit{SH}^1(M)=0$, which in particular implies that $M$ does not admit a quasi-dilation. On the other hand, there is a Lefschetz fibration $M\rightarrow\mathbb{C}$ on $M$ whose smooth fiber $F$ is symplectomorphic to the Milnor fiber of a 4-fold triple point, i.e. the affine hypersurface in $\mathbb{C}^5$ defined by the equation
\begin{equation}
z_1^3+z_2^3+z_3^3+z_4^3+z_5^3=1.
\end{equation}
Similarly, $F$ also admits a Lefschetz fibration $F\rightarrow\mathbb{C}$ with the Milnor fiber $M_{3,3,3,3}$ as its fiber. If $M_{3,3,3,3}$ admits a quasi-dilation, then by \cite{ps5}, Lemma 19.5 it lifts to a quasi-dilation in $\mathit{SH}^1(F)$. Applying the same lemma again to the Lefschetz fibration $M\rightarrow\mathbb{C}$, it follows that $M$ also admits a quasi-dilation, which gives the desired contradiction.
\end{proof}

\subsection{Lefschetz fibrations}\label{section:Lefschetz}

This subsection is devoted to the proof of Theorem \ref{theorem:Lefschetz}, which allows us to get new examples Liouville manifolds which admit cyclic dilations in terms of the known ones. The argument here is a slight variation of those in $\cite{ss}$, Section 7 and $\cite{jz}$, Section 5.3.
\bigskip

We use the general set up of $\cite{ss}$, Section 7. Let $\pi:M\rightarrow\mathbb{C}$ be an exact symplectic Lefschetz fibration, which means that its smooth fibers $F$ are completions of Liouville domains $\overline{F}$. More explicitly, we require that
\begin{itemize}
	\item For some almost complex structure $J\in\mathcal{J}(M)$, the map $\pi$ is $(J,j)$-holomorphic, where $j$ is the standard complex structure on $\mathbb{C}$.
	\item $\pi$ has finitely many isolated critical points, so that each singular fiber contains at most one critical point, and the almost complex structure $J$ is locally integrable near each of these critical points.
	\item There is a relatively open compact subset $\overline{M}\subset M$, so that its complement $M\setminus\overline{M}$ is identified with
	\begin{equation}
	\widetilde{M}:=(\mathbb{R}_+\times T)\cup_{\mathbb{R}_+\times S^1\times\mathbb{R}_+\times\partial\overline{F}}(\mathbb{C}\times\mathbb{R}_+\times\partial\overline{F}),
	\end{equation}
	where $T=(\mathbb{R}\times F)/(t,x)\sim(t-1,\mu(x))$ is the mapping torus, with $\mu$ being the total monodromy of $\pi$. By construction, $\overline{M}$ is a manifold with corners, which coincides with the Liouville domain associated to $M$ up to deformation once the corners are rounded off.
	\item Fix the choice of a trivialization of the canonical bundle $K_M$, which induces a trivialization of $K_F$, the canonical bundle of the fiber.
\end{itemize}

Given such a Lefschetz fibration, consider the autonomous Hamiltonian $H_M:M\rightarrow\mathbb{R}$ defined by
\begin{equation}\label{eq:Ham-Lef}
H_M=H_F+\pi^\ast H_\mathbb{C},
\end{equation}
where $H_\mathbb{C}(z)=\varepsilon|z-c|^2/2$ for some $\varepsilon>0$ is a function on the base, and $H_F:F\rightarrow\mathbb{R}$ is a Hamiltonian on the fiber which is linear on the cylindrical end $[1,\infty)\times\partial\overline{F}$ with slope $\lambda>0$, where $\lambda\notin\mathcal{P}_F$. By $\cite{ss}$, Lemma 7.2, for sufficiently small $\varepsilon$, there is a short exact sequence
\begin{equation}\label{eq:SES}
0\rightarrow\mathbb{K}^{\mathrm{Crit}(\pi)}[-n]\rightarrow\mathit{CF}^\ast_{\mathit{vert}}(M,\lambda)\rightarrow\mathit{CF}^\ast(F,\lambda)\rightarrow 0,
\end{equation}
where the Floer complexes $\mathit{CF}^\ast_{\mathit{vert}}(M,\lambda)$ and $\mathit{CF}^\ast(F,\lambda)$ are defined by choosing time-dependent perturbations of the autonomous Hamiltonians $H_M$ and $H_F$, and the notation $\mathbb{K}^{\mathrm{Crit}(\pi)}[-n]$ means the complex with trivial differential so that there is a copy of $\mathbb{K}$ in degree $n$ for every critical point of $\pi$. Here, we use the notation $\mathit{CF}^\ast_{\mathit{vert}}(M,\lambda)$ to indicate that when $\lambda\rightarrow\infty$, the slope of our Hamiltonian $H_M$ only increases in the vertical direction. As a consequence, the cohomology level direct limit $\mathit{SH}_\mathit{vert}^\ast(M)$ is in general not isomorphic to the symplectic cohomology $\mathit{SH}^\ast(M)$. The same notational convention will be used later on for equivariant Floer cohomologies.
\bigskip

Let $M$ be a Liouville manifold. Recall that the action functional $\mathcal{A}_{H_t}:\mathcal{L}M\rightarrow\mathbb{R}$ of a time-dependent perturbation $H_t:S^1\times M\rightarrow\mathbb{R}$ of some autonomous Hamiltonian $H:M\rightarrow\mathbb{R}$ is defined to be
\begin{equation}\label{eq:action}
\mathcal{A}_{H_t}(x)=-\int_{S^1}x^\ast\theta_M+\int_0^1H_t(x(t))dt.
\end{equation}
The period spectrum $\mathcal{P}_M$ is a strictly ordered set with elements $0<\eta_1<\eta_2<\cdots$, where $\eta_1$ is the minimum period of a Reeb orbit on the contact boundary $\partial\overline{M}$, and we set $\eta_0=0$. Let $\lambda_j=\frac{\eta_j+\eta_{j+1}}{2}$, so in particular $\lambda_j\notin\mathcal{P}_M$ for any $j$, and introduce the real numbers
\begin{equation}
a_{\lambda_j}:=-\frac{\lambda_j^2}{2}-\lambda_j,\textrm{ }j\geq0.
\end{equation}
Consider a Hamiltonian $H_{\lambda,t}\in\mathcal{H}_\lambda(M)$ so that $\lambda\notin\mathcal{P}_M$, let $\mathcal{O}_{M,\lambda}$ be the set of 1-periodic orbits of $X_{H_{\lambda,t}}$, there is an action filtration $F^\bullet$ on the Floer complex $\mathit{CF}^\ast(\lambda)$ of $H_{\lambda,t}$ given by
\begin{equation}\label{eq:filtration}
F^j\mathit{CF}^\ast(\lambda):=\bigoplus_{x\in\mathcal{O}_{M,\lambda},\mathcal{A}_{H_{\lambda,t}}(x)\geq a_{\lambda_j}}|o_x|_\mathbb{K}
\end{equation}

\bigskip

In order to analyze the compatibility between the $S^1$-complex structure maps $\{\delta_i\}_{i\geq0}$ and the filtration $F^\bullet$ on $\mathit{CF}^\ast(\lambda)$, we study a specific autonomous Hamiltonian $\widetilde{H}_\lambda:M\rightarrow\mathbb{R}$, which has the form
\begin{equation}
\widetilde{H}_\lambda(x)=\left\{\begin{array}{ll}
\textrm{some negative }C^2\textrm{-small Morse function} & x\in M^\mathit{in}, \\
\frac{(r-1)^2}{2} & x\in[1,\lambda+1]\times\partial\overline{M}, \\
\lambda(r-1)-\frac{\lambda^2}{2} & x\in[\lambda+1,\infty)\times\partial\overline{M}.
\end{array}\right.
\end{equation}
Following $\cite{jz}$, Section 3.2.1, we define a carefully-chosen small time-dependent perturbation $\widetilde{H}_{\lambda,t}$ of $\widetilde{H}_\lambda$. For any 1-periodic orbit $x\in\widetilde{\mathcal{O}}_{M,\lambda}$ of $X_{\widetilde{H}_\lambda}$, fix an isolating neighborhood $U_x\subset M$. If $x$ corresponds to a Reeb orbit of multiplicity $k\in\mathbb{N}$, one considers a Morse function $f_x:S^1\rightarrow\left[-1,-\frac{1}{2}\right]$ that has a unique minimum $f_x(0)=-1$ and a unique maximum $f_x(t_0)=-1/2$ for some small enough $t_0\in S^1$. Define $h_x:\overline{U}_x\rightarrow[-1,0]$ by
\begin{equation}
h_x(t,x(s))=f_x(ks-kt)
\end{equation}
on the image of $x$ and extend it smoothly to $U_x$ so that $h_x=0$ on $\partial\overline{U}_x$. The time-dependent perturbation of $\widetilde{H}_\lambda$ is defined to be
\begin{equation}
\widetilde{H}_{\lambda,t}=\widetilde{H}_\lambda+\varepsilon'\sum_{x\in\widetilde{\mathcal{O}}_{M,\lambda}}h_x(t),
\end{equation}
where $h_x(t)=f_x\circ\tilde{\phi}^{-t}$, with $\tilde{\phi}^{-t}$ being the time $-t$ flow of $X_{\widetilde{H}_\lambda}$, and $\varepsilon'>0$ is a small positive number. 

With our choice of $\widetilde{H}_{\lambda,t}$, an energy estimate for Floer trajectories in $\mathcal{M}_i(y^+;y^-)$ with $|o_{y^+}|_\mathbb{K}\in F^j\mathit{CF}^\ast(\lambda)$ implies the following:

\begin{lemma}[$\cite{jz}$, Lemma 3.2.4]\label{lemma:filtration}
For any fixed $\lambda\notin\mathcal{P}_M$, there is an $\varepsilon'>0$ depending on $\lambda$ such that
\begin{equation}
\delta_i\left(F^j\mathit{CF}^\ast(\lambda)\right)\subset F^j\mathit{CF}^\ast(\lambda)
\end{equation}
for all $i,j\geq0$.
\end{lemma}

\begin{proof}[Proof of Theorem \ref{theorem:Lefschetz}]
We first show that (\ref{eq:SES}) is a short exact sequence of $S^1$-complexes. Basically, (\ref{eq:SES}) follows from the fact that the set of generators of $\mathit{CF}^\ast_{\mathit{vert}}(M,\lambda)$ consists of the following three kinds:
\begin{itemize}
	\item[(i)] Critical points of $H_F$. These generators have small negative action, if we perturb $H_F$ so that it is a $C^2$-small Morse function with negative values in the interior of $\overline{F}$.
	\item[(ii)] Non-constant 1-periodic orbits of $X_{H_F}$. Writing $H_F=h_F(r)$ on the cylindrical end $[1,\infty)\times\partial\overline{F}$, such an orbit $x$ has action
	\begin{equation}
	\mathcal{A}_{H_F}(x)=h_F(r_x)-r_xh_F'(r_x)<0,
	\end{equation}
	where $r_x\in[1,\infty)$ is the radial coordinate of $x$.
	\item[(iii)] Constant orbits near the critical points of $\pi^\ast H_\mathbb{C}$, which have Conley-Zehnder index $-n$. If we choose the $C^2$-small Morse function appearing in (i) to be sufficiently small, it can be achieved that
	\begin{equation}
	\mathcal{A}_{H_M}(x)=\varepsilon|\pi(x)-c|^2+H_F(x)>0.
	\end{equation}
\end{itemize}

By choosing the $t$-dependent perturbation of the Hamiltonian $H_F$ (and thus $H_M$) carefully, Lemma \ref{lemma:filtration} applies and shows that the operations $\{\delta_i\}_{i\geq0}$ preserve the action filtration on $\mathit{CF}^\ast_\mathit{vert}(M,\lambda)$. This implies that the generators of $\mathit{CF}^\ast_\mathit{vert}(M,\lambda)$ with positive actions form the (trivial) $S^1$-subcomplex $\mathbb{K}^{\mathrm{Crit}(\pi)}[-n]$. Denote by $\mathcal{H}_\pi(M)$ the space of Hamiltonians which are small $t$-dependent perturbations of the autonomous Hamiltonians of the form (\ref{eq:Ham-Lef}), and by $\mathcal{J}_\pi(M)$ the space of compatible almost complex structures which are of contact type when restricted to the fibers $F$, so that $\pi:M\rightarrow\mathbb{C}$ is $(J,j)$-holomorphic for each $J\in\mathcal{J}_\pi(M)$. The same energy estimate as in the proof of $\cite{ss}$, Lemma 7.2 shows that for any solution $u:Z\rightarrow M$ of the Floer equation $\left(du-X_{H_Z}\otimes\nu_Z\right)^{0,1}=0$ with asymptotics $y^\pm\in\mathcal{O}_{F,\lambda}$, where $\nu_Z\in\Omega^1(Z)$, $H_Z:Z\rightarrow\mathcal{H}_\pi(M)$, and the $(0,1)$-part is taken with respect to some $J_Z:Z\rightarrow\mathcal{J}_\pi(M)$, its image necessarily lies in the fiber $F$. Since universal and consistent Floer data for the operations $\{\delta_i\}_{i\geq0}$ can be chosen among $\mathcal{H}_\pi(M)$ and $\mathcal{J}_\pi(M)$, the fact that $u(Z)\subset F$ implies that the quotient complex $\mathit{CF}^\ast_\mathit{vert}(M,\lambda)/\mathbb{K}^{\mathrm{Crit}(\pi)}[-n]$ can be identified with $\mathit{CF}^\ast(F,\lambda)$ as an $S^1$-complex. This proves that (\ref{eq:SES}) is a short exact sequence of $S^1$-complexes, and it follows from Proposition \ref{proposition:intert} that we have a commutative diagram
\begin{equation}\label{eq:diagram}
\begin{tikzcd}
	\cdots \arrow[r] &\mathbb{K}((u))/u\mathbb{K}[[u]]^{\mathrm{Crit}(\pi)}[-n] \arrow[d, "\mathbf{B}"] \arrow[r] &\mathit{HF}_{S^1,\mathit{vert}}^\ast(M,\lambda) \arrow[d, "\mathbf{B}"] \arrow[r] &\mathit{HF}_{S^1}^\ast(F,\lambda) \arrow[d, "\mathbf{B}"] \arrow[r] &\cdots \\
	\cdots \arrow[r] &\mathbb{K}^{\mathrm{Crit}(\pi)}[-n] \arrow[r] &\mathit{HF}^{\ast-1}_\mathit{vert}(M,\lambda) \arrow[r] &\mathit{HF}^{\ast-1}(F,\lambda) \arrow[r] &\cdots
\end{tikzcd}
\end{equation}
whose rows are long exact sequences of $S^1$-equivariant and ordinary Floer cohomology groups, which are related through the marking map $\mathbf{B}$.

Now suppose that $F$ admits a cyclic dilation, so that for some $\lambda\gg0$ and $\lambda\notin\mathcal{P}_F$, there is a class $\tilde{b}_F\in\mathit{HF}_{S^1}^1(F,\lambda)$, whose image under the composition of the marking map $\mathbf{B}$ and the continuation map $\kappa^{\lambda,\infty}$ defines an invertible element $h_F\in\mathit{SH}^0(F)^\times$. Consider the boundary map $\mathit{HF}_{S^1}^1(F,\lambda)\rightarrow\mathbb{K}((u))/u\mathbb{K}[[u]]^{\mathrm{Crit}(\pi)}[-n]$ in the first row of (\ref{eq:diagram}). On the chain level, it consists of an infinite sequence of maps
\begin{equation}
\partial_k:\mathit{CF}^{2k+1}(F,\lambda)\rightarrow\mathbb{K}^{\mathrm{Crit}(\pi)}[-n]
\end{equation}
for every $k\geq0$, under which the cochain level representative $\tilde{\beta}_F=\sum_{k=0}^\infty\beta_{F,k}\otimes u^{-k}$ of $\tilde{b}_F$ goes to $\sum_{k=0}^\infty\partial_k(\beta_{F,k})$, which lies in degree 2. However, by our assumption that $n\geq3$, we necessarily have $\sum_{k=0}^\infty\partial_k(\beta_{F,k})=0$. This implies that the map $\mathit{HF}_{S^1,\mathit{vert}}^1(M,\lambda)\rightarrow\mathit{HF}_{S^1}^1(F,\lambda)$ in (\ref{eq:diagram}) is surjective, and applying similar argument to the second row of (\ref{eq:diagram}) shows that there is an isomorphism $\mathit{HF}_\mathit{vert}^0(M,\lambda)\cong\mathit{HF}^0(F,\lambda)$, which induces an isomorphism
\begin{equation}\label{eq:iso}
\mathit{SH}_\mathit{vert}^0(M)\cong\mathit{SH}^0(F)
\end{equation}
after passing to direct limits. This is in fact an isomorphism of $\mathbb{K}$-algebras, since applying the same argument as above to the pair-of-pants surface $S$ instead of $Z$ shows that the image of any Floer solution $u$ with asymptotics $y_0^-,y_1^-\in\mathcal{O}_{F,\lambda}$ and $y^+\in\mathcal{O}_{F,2\lambda}$ will be contained in $F$. By the surjectivity of $\mathit{HF}_{S^1,\mathit{vert}}^1(M,\lambda)\rightarrow\mathit{HF}_{S^1}^1(F,\lambda)$ and the commutative diagram (\ref{eq:diagram}), $\tilde{b}_F$ lifts to a class $\tilde{b}_M\in\mathit{HF}_{S^1,\mathit{vert}}^1(M,\lambda)$, whose image under $\mathbf{B}$, followed by the continuation map, is the lift of $h_F$ in $\mathit{SH}_\mathit{vert}^0(M)$. In view of the isomorphism (\ref{eq:iso}), this defines an invertible element of $\mathit{SH}_\mathit{vert}^0(M)$. Finally, there is an equivariant continuation map
\begin{equation}
\mathit{SC}^\ast_\mathit{vert}(M)\otimes_\mathbb{K}\mathbb{K}((u))/u\mathbb{K}[[u]]\rightarrow\mathit{SC}^\ast(M)\otimes_\mathbb{K}\mathbb{K}((u))/u\mathbb{K}[[u]],
\end{equation}
under which the class $\tilde{b}_M$ goes to a cyclic dilation in $\mathit{SH}_{S^1}^1(M)$. For its construction, one modifies the construction of the continuation map $\mathit{SC}^\ast_\mathit{vert}(M)\rightarrow\mathit{SC}^\ast(M)$ (cf. \cite{ps5} (18.25)) in the non-equivariant case by replacing the (1-parameter family of) domain cylinders with $k$-point angle decorated cylinders, as in the construction of (\ref{eq:continuation}).
\end{proof}

\begin{remark}
It is an easy observation in the above argument that if the cyclic dilation $\tilde{b}_F$ of the fiber $F$ satisfies $h=1$, then so is the total space $M$. In particular, according to our remarks in Section \ref{section:dynamics}, any Milnor fiber $M$ of a singularity of the form (\ref{eq:cubic}) has a cyclic dilation $\tilde{b}$ with $\mathbf{B}(\tilde{b})=1$.
\end{remark}

\subsection{Varieties of log general type}\label{section:general type}

We prove Theorem \ref{theorem:unique} in this section. Our argument is based on the work of McLean $\cite{mm}$ on the symplectic invariance of the log Kodaira dimension, and the techniques in $\cite{ab,dh}$, which allow us to produce $J$-holomorphic curves starting from Floer trajectories. For completeness, we shall start by recalling some of the important notions and results from $\cite{mm}$.
\bigskip

Let $(\overline{M},\theta_M)$ be any Liouville domain and let $J$ be an almost complex structure on $\overline{M}$ which is compatible with the symplectic form $d\theta_M$. It is \textit{convex} if there is some function $\phi:\overline{M}\rightarrow\mathbb{R}$ such that
\begin{itemize}
	\item $\partial\overline{M}$ is a regular level set of $\phi$ and $\phi$ attains its maximum on $\partial\overline{M}$;
	\item $\theta_M\circ J=d\phi$ near $\partial\overline{M}$.
\end{itemize}
Every Liouville domain $\overline{M}$ has a convex almost complex structure since one can take $\phi=r$ to be the radial coordinate function in a collar neighborhood of $\partial\overline{M}$, and then extend it smoothly to the interior. The following notion plays a pivotal role in McLean's theory.

\begin{definition}[$\cite{mm}$, Definition 2.2]\label{definition:uniruled}
Let $k\in\mathbb{Z}_{>0}$, and let $\mu\in\mathbb{R}_{>0}$. A Liouville domain $\overline{M}$ is $(k,\mu)$-uniruled if for every convex almost complex structure $J$ and every point $p\in M^\mathit{in}$ so that $J$ is integrable in a neighborhood of $p$, there is a proper $J$-holomorphic map $u:S\rightarrow M^\mathit{in}$ whose image passes through $p$, where $S$ is a genus 0 open Riemann surface with $\dim H_1(S;\mathbb{Q})\leq k-1$, and the energy of $u$ is at most $\mu$.
\end{definition}

It follows from $\cite{mm}$, Theorem 2.3 that $(k,\mu)$-uniruledness is a symplectic invariant of Liouville manifolds after forgetting about the energy bound $\mu$.

We now restrict ourselves to the special case when $M$ is an $n$-dimensional smooth affine variety. We say that $M$ is \textit{algebraically $k$-uniruled} if there is a polynomial map $S\rightarrow M$ passing through every generic point $p\in M$, where $S$ is $\mathbb{CP}^1$ with at most $k$ points removed. This notion of uniruledness is related to Definition \ref{definition:uniruled} in the following way.

\begin{theorem}[$\cite{mm}$, Theorem 2.5]\label{theorem:ML}
Let $M$ be a smooth affine variety. If the associated Liouville domain $\overline{M}$ is $(k,\mu)$-uniruled for some $\mu$, then $M$ is algebraically $k$-uniruled.
\end{theorem}

$k$-uniruledness of an affine variety is closely related to its log Kodaira dimension (\ref{eq:Kodaira}). In particular, we have the following:

\begin{lemma}[$\cite{mm}$, Lemma 7.1]\label{lemma:ML}
Let $M$ be a smooth affine variety which is algebraically $k$-uniruled. If $k=1$, then $\kappa(M)=-\infty$, and if $k=2$, then $\kappa(M)\leq n-1$.
\end{lemma}

\begin{proof}[Proof of Theorem \ref{theorem:unique}]
We start with a summary of the main idea of the proof. In order to show that $M$ does not admit a cyclic dilation, we argue by contradiction. Suppose $M$ has a cyclic dilation, we first notice that since $M$ contains an exact Lagrangian torus, by Corollary \ref{corollary:h=1}, the marking map $\mathbf{B}:\mathit{SH}^1_{S^1}(M)\rightarrow\mathit{SH}^0(M)$ cannot hit the identity. Thus in order for $M$ to have a cyclic dilation, there must be some non-trivial invertible element $h\in\mathit{SH}^0(M)^\times$. Using a limiting argument, we will show that the existence of such an element $h$ would imply that $M$ is uniruled by cylinders, which contradicts with our assumption that $M$ is long general type by Lemma \ref{lemma:ML}. The proof is divided into three steps.

\begin{paragraph}{Step 1: Existence of a Floer trajectory.}
Suppose that $M$ is an $n$-dimensional smooth affine variety so that $\mathit{SH}^0(M)^\times$ is not isomorphic to $\mathbb{K}^\times$, or equivalently, there is an $h\in\mathit{SH}^0(M)^\times$ which is not a multiple of the identity. Since $hh^{-1}=1$ holds in $\mathit{SH}^0(M)$, there must be some $\eta\in\mathit{SC}^0_+(M)$ so that $\alpha\cdot e+\eta$ is the cochain level representative of $h$, where $\alpha\in\mathbb{K}$ and $\mathit{SC}^0_+(M)\subset\mathit{SC}^0(M)$ is the submodule generated by non-constant Hamiltonian orbits. Assume further that $M$ contains an exact Lagrangian torus $L$, consider the Viterbo map
\begin{equation}\label{eq:Vit}
\mathit{SH}^0(M)\rightarrow H_n(\mathcal{L}L;\mathbb{K})\cong Z\left(\mathbb{K}[\pi_1(L)]\right)=\mathbb{K}[\pi_1(L)],
\end{equation}
where $Z\left(\mathbb{K}[\pi_1(L)]\right)$ is the center of $\mathbb{K}[\pi_1(L)]$, which is just $\mathbb{K}[\pi_1(L)]$ by the assumption that $L$ is a torus. In our case,
\begin{equation}\label{eq:Laurent}
\mathbb{K}[\pi_1(L)]\cong\mathbb{K}\left[x_1^{\pm1},\cdots,x_n^{\pm1}\right]
\end{equation}
is just the Laurent polynomial ring. Under the isomorphism (\ref{eq:Laurent}), (\ref{eq:Vit}) maps the cocycle $\alpha\cdot e+\eta$ to a non-trivial unit of $\mathbb{K}\left[x_1^{\pm1},\cdots,x_n^{\pm1}\right]$, which must be a non-zero multiple of some monomial $z_1^{a_1}\cdots z_n^{a_n}$, where $a_1,\cdots,a_n\in\mathbb{Z}$. In particular, we must have $\alpha=0$, because the map (\ref{eq:Vit}) maps the identity to $\iota_\ast[L]$, where $\iota:L\rightarrow\mathcal{L}L$ is the inclusion of constant loops, and $\iota_\ast[L]$ corresponds to $1$ under the isomorphism (\ref{eq:Laurent}). It follows that the pairing
\begin{equation}
\mathit{SC}^0_+(M)\otimes\mathit{SC}^0_+(M)\xrightarrow{\smile}\mathit{SC}^0(M)\xrightarrow{\mathit{pr}}C^0(M;\mathbb{K})\cong\mathbb{K}
\end{equation}
defined by composing the pair-of-pants product with the natural projection to the subcomplex $C^0(M;\mathbb{K})\subset\mathit{SC}^0(M)$ does not vanish, so there must be some $y_0^+,y_1^+\in\mathcal{O}_M$, such that $y_1^+\smile y_0^+=\alpha'\cdot e+\zeta$, where $\alpha'\in\mathbb{K}^\times$ is some non-zero scalar and $\zeta\in\mathit{SC}_+^0(M)$. Without loss of generality, we may assume that $y_1^+\smile y_0^+=e+\zeta$ for convenience.
 
This implies the existence of a map $u:S\rightarrow M$, with $S$ being a 3-punctured sphere, which satisfies the Floer equation $\left(du-X_{H_S}\otimes \nu_S\right)^{0,1}=0$, with asymptotic conditions specified by the non-constant periodic orbits $y_0^+,y_1^+\in\mathcal{O}_M$ at two positive cylindrical ends, and converges to the minimum $y^-$ of some $C^2$-small Morse function defined on $M^\mathit{in}$. Here $H_S:S\rightarrow\mathcal{H}(M)$ is a domain-dependent Hamiltonian-function so that its restriction to $M^\mathit{in}$ is a (domain-independent) $C^2$-small Morse function, and $\nu_S\in\Omega^1(S)$ is a closed 1-form, they are fixed as part of our Floer data defining the pair-of-pants product $\smile$, and the $(0,1)$-part in the Floer equation is taken with respect to some domain-dependent almost complex structure $J_S:S\rightarrow\mathcal{J}(M)$. 
\end{paragraph}

\begin{paragraph}{Step 2: Producing a pseudoholomorphic curve.}
Starting from the Floer trajectory $u$, one can apply a limiting argument of $\cite{ab,dh}$ to produce a $J$-holomorphic cylinder $\bar{u}_\infty:Z\rightarrow M_{1-\varepsilon}^\mathit{in}$ with finite energy which passes through $y^-$ for any convex almost complex structure $J$ on the slightly shrinked Liouville domain
\begin{equation}
\overline{M}_{1-\varepsilon}:=\overline{M}\setminus(1-\varepsilon,0]\times\partial\overline{M}
\end{equation}
containing $y^-$ in its interior, and whose completion is still deformation equivalent to $M$, where $\varepsilon>0$ is a sufficiently small constant.

To do this, we work with linear Hamiltonians instead, and introduce a particular 1-parameter family of domain-dependent Hamiltonians $H_{\lambda,S,\theta}:S\rightarrow\mathcal{H}_\lambda(M)$ which depend on a small parameter $\theta>0$, where as before $\lambda\notin\mathcal{P}_M$ and $\lambda\gg0$. Specifically, for each point $z\in S$, there is a Hamiltonian
\begin{equation}
H_{\lambda,z,\theta}=H_{\lambda,\theta}+F_{\lambda,z}\in\mathcal{H}_\lambda(M),
\end{equation}
where $F_{\lambda,z}:S\times M\rightarrow\mathbb{R}$ is independent of $s$ when restricted to the cylindrical ends, and since it is supported near non-constant orbits of $X_{H_{\lambda,\theta}}$, we can choose $\varepsilon>0$ small enough so that $F_{\lambda,z}$ vanishes on $\overline{M}_{1-\varepsilon}$. Set
\begin{equation}
H_{\lambda,\theta}(x)=\left\{\begin{array}{ll}
-\delta_{\lambda}+\theta f(x) & x\in M^\mathit{in} \\
h_{\lambda,\theta}(r) & x\in\overline{M}_{1+2\varepsilon}\setminus M^\mathit{in} \\
\lambda(r-1-\varepsilon) & x\in M\setminus\overline{M}_{1+2\varepsilon}
\end{array}\right.
\end{equation}
where $\delta_{\lambda}>0$ is a small scalar which satisfies $\lim_{\lambda\rightarrow\infty}\delta_{\lambda}=0$, $f$ is a $C^2$-small Morse function which satisfies $-1\leq f\leq0$ when restricted to $\overline{M}_{1-\varepsilon}$, has a relative minimum at $y^-\in M_{1-\varepsilon}^\mathit{in}$, and equals $r-1+\varepsilon$ on $[1-2\varepsilon,1]\times\partial\overline{M}$. $h_{\lambda,\theta}(r)$ is an arbitrary convex function on $[1,1+2\varepsilon]\times\partial\overline{M}$ which depends only on $r$, and whose slope varies from $\theta$ to $\lambda$ as $r$ goes from 1 to $1+2\varepsilon$, such that $h_{\lambda}(r):=\lim_{\theta\rightarrow0}h_{\lambda,\theta}(r)$ is a smooth function.

With our particular choice of the domain-dependent Hamiltonian $H_{\lambda,S,\theta}$ as above, we get a Floer trajectory $u_{\lambda,\theta}:S\rightarrow M$ which is asymptotic to a Morse critical point $y^-_\theta$ at its negative cylindrical end, and to $y_{0,\theta}^+,y_{1,\theta}^+\in\mathcal{O}_{M,\lambda}$ at two positive cylindrical ends. It follows from our definition of $H_{\lambda,\theta}$ that the non-constant orbits $y_{0,\theta}^+,y_{1,\theta}^+$ necessarily lie in the collar $[1,1+2\varepsilon]\times\partial\overline{M}$. To achieve the non-degeneracies of the orbits $y_{0,\theta}^+$ and $y_{1,\theta}^+$, the perturbation $F_{\lambda,z}$ can be taken to be supported near $y_{0,\theta}^+$ and $y_{1,\theta}^+$, so we may assume (by possibly rescaling $\varepsilon$) that $H_{\lambda,z,\theta}=H_{\lambda,\theta}$ is domain-independent in the shrinked Liouville domain $\overline{M}_{1-\varepsilon}$. Applying the maximum principle from $\cite{as}$, Section 7d to the map $u_{\lambda,\theta}$ shows that $u_{\lambda,\theta}(S)\subset\overline{M}_{1+2\varepsilon}$. To achieve transversality of the moduli space $\mathcal{P}\left(y_{0,\theta}^+,y_{1,\theta}^+;y_\theta^-\right)$ where the trajectory $u_{\lambda,\theta}$ lies in, one can start from any convex almost complex structure $J$ on $\overline{M}$ and perturb it slightly outside of $\overline{M}_{1-\varepsilon}$ to get a domain-dependent almost complex structure $J_{\lambda,S,\theta}:S\rightarrow\mathcal{J}(M)$. Note that we have arranged so that both of $H_{\lambda,S,\theta}$ and $J_{\lambda,S,\theta}$ are domain-independent on $\overline{M}_{1-\varepsilon}$, and we denote the restriction of $J_{\lambda,S,\theta}$ on $\overline{M}_{1-\varepsilon}$ as $J_{\lambda,\theta}$.

We want to pass to the limit $\theta\rightarrow0$. Notice that when restricted to the Liouville domain $\overline{M}_{1-\varepsilon}$, we have $\lim_{\theta\rightarrow0} H_{\lambda,\theta}=-\delta_\lambda$, and $\lim_{\theta\rightarrow0}J_{\lambda,\theta}=J$, for some fixed convex almost complex structure $J\in\mathcal{J}(M)$ which doesn't need to depend on $\lambda\gg0$. By $\cite{dh}$, Proposition 5.11 (which deals with the case when $S$ is a cylinder, but extends in a straightforward way to pair-of-pants), one can find a sequence $\{\theta_n\}$ which limits to 0 so that the corresponding Floer trajectories $\left\{u_{\lambda,\theta_n}\right\}$ converge to a limit $u_\lambda$ in $C_\mathrm{loc}^\infty(S,M)$, and the energy of the limiting trajectory
\begin{equation}
E\left(u_\lambda\right):=\frac{1}{2}\int_S\left|\left|du_\lambda-X_{H_{\lambda,z,0}}\otimes dt\right|\right|^2_{J_{\lambda,z,0}}
\end{equation}
is bounded above by some constant $\mu_M>0$, which is independent of $\lambda\gg0$. Denote by $\phi$ a biholomorphic map which identifies $S$ with $\mathbb{CP}^1\setminus\{0,1,\infty\}$ so that the negative puncture $\zeta_\mathit{out}$ is mapped to the origin. The composition $\tilde{u}_\lambda=u_\lambda\circ\phi:\mathbb{CP}^1\setminus\{0,1,\infty\}\rightarrow M$ is a map whose limit at the origin is $y^-:=\lim_{\theta\rightarrow0}y^-_\theta$ and whose image goes outside of $M^\mathit{in}$ when approaching the other two punctures. Note that by our choice of $H_{\lambda,\theta}$, the minimum $y^-_\theta$ of the $C^2$-small Morse function $f$ is independent of $\theta>0$, so we actually have $y^-\in M_{1-2\varepsilon}^\mathit{in}$. On the other hand, it also follows from our choice of $H_{\lambda,\theta}$ that $y_0^+:=\lim_{\lambda\rightarrow\infty}y^+_{0,\theta}$ and $y_1^+:=\lim_{\lambda\rightarrow\infty}y^+_{1,\theta}$ fall outside of $M^\mathit{in}$.

Pick any $R_\lambda\in(1-2\varepsilon,1-\varepsilon)$ so that $\lim_{\lambda\rightarrow\infty}R_\lambda=1-\varepsilon$, and consider the inverse image $\tilde{u}_\lambda^{-1}(\overline{M}_{R_\lambda})$. Since $y^-\in M_{1-2\varepsilon}^\mathit{in}$ and $y_0^+,y_1^+\notin M^\mathit{in}$, $\tilde{u}_\lambda^{-1}(M^\mathit{in}_{R_\lambda})\subset\mathbb{CP}^1\setminus\{0,1,\infty\}$ is an open punctured cylinder for some $\varepsilon>0$ which can be taken to be sufficiently small. We will denote it by $Z^\ast_\lambda\subset Z$. Since $H_\lambda\equiv-\delta_\lambda$ in $M^\mathit{in}$, it follows that the map $\tilde{u}_\lambda:Z_\lambda^\ast\rightarrow M_{R_\lambda}^\mathit{in}$ is $J$-holomorphic. Moreover, we have
\begin{equation}
\int_{\partial\overline{Z}_\lambda^\ast}\tilde{u}_\lambda^\ast\theta_M\leq E\left(u_\lambda\right)\leq\mu_M.
\end{equation}
In particular, the removable singularity theorem for pesudoholomorphic maps applies, which shows that $\tilde{u}_\lambda$ extends to a $J$-holomorphic map $\bar{u}_\lambda:Z_\lambda\rightarrow M_{R_\lambda}^\mathit{in}$. Letting $\lambda\rightarrow\infty$, we get a $J$-holomorphic map $\bar{u}_\infty:Z\rightarrow M_{1-\varepsilon}^\mathit{in}$ whose image passes through $y^-$.
\end{paragraph}

\begin{paragraph}{Step 3: Uniruledness.}
The uniruledness of the Liouville domain $\overline{M}_{1-\varepsilon}$ follows by noticing that $y^-$ can be taken to be any generic point in $M_{1-\varepsilon}^\mathit{in}$. Alternatively, one can argue as follows.

A slight variation of the construction of the moduli space $\mathcal{P}(y^+_0,y_1^+;y_-)$ enables us to define $\mathcal{P}(y_0^+,y_1^+;\overline{M})$, which parametrizes maps $u:S\rightarrow M$ satisfying Floer's equation, but are now asymptotic to $y_0^+,y_1^+\in\mathcal{O}_{M,\lambda}$ at two positive ends, and $\lim_{s\rightarrow-\infty}(\varepsilon^-)^\ast u(s,\cdot)$ belongs to the relative fundamental cycle in $C_{2n}(\overline{M},\partial\overline{M})$, where $\varepsilon^-$ is the negative cylindrical end. The Gromov bordification of $\mathcal{P}(y_0^+,y_1^+;\overline{M})$ carries an evaluation map
\begin{equation}
\overline{\mathit{ev}}:\overline{\mathcal{P}}\left(y_0^+,y_1^+;\overline{M}\right)\rightarrow\overline{M}
\end{equation}
defined to be the asymptote at the negative puncture for every $u\in\mathcal{P}\left(y_0^+,y_1^+;\overline{M}\right)$, and the coefficient before the identity $e\in\mathit{CF}^0(2\lambda)$ under the pair-of-pants product $y_1^+\smile y_0^+$ is defined by pushing forward the fundamental chain $\left[\overline{\mathcal{P}}\left(y_0^+,y_1^+;\overline{M}\right)\right]$ via $\overline{\mathit{ev}}$. Our assumption that $h$ defines a non-trivial unit in $\mathit{SH}^0(M)$ implies that for appropriate choices of Floer data, there is an identification between $\overline{\mathcal{P}}\left(y_0^+,y_1^+;\overline{M}\right)$ and $\overline{M}$ relative to the boundaries. Applying the same argument as above to every element $u$ of $\mathcal{P}\left(y_0^+,y_1^+;\overline{M}\right)$ proves that the Liouville domain $\overline{M}_{1-\varepsilon}$ is $(2,\mu_M)$-uniruled in the sense of Definition \ref{definition:uniruled}. It follows from Theorem \ref{theorem:ML} that $M$ is algebraically 2-uniruled.

By Lemma \ref{lemma:ML}, $M$ cannot be of log general type. In other words, for any smooth affine variety $M$ of log general type which contains an exact Lagrangian torus, we necessarily have $\mathit{SH}^0(M)^\times\cong\mathbb{K}^\times$. Appealing to Corollary \ref{corollary:h=1} completes our proof.
\end{paragraph}
\end{proof}

\begin{remark}
A key point in the above proof is that any central unit in the fundamental group algebra of a torus has vanishing constant coefficient. This is actually true for the group algebra of any torsion-free group, see \cite{yl3}, Theorem 4.1. Because of this, Theorem \ref{theorem:unique} can be generalized to log general type affine varieties containing an exact Lagrangian $K(\pi,1)$.
\end{remark}

Note that our theorem provides an alternative way to understand Corollary \ref{corollary:non-simple}. One can also try to prove a statement of similar flavour as Theorem \ref{theorem:unique} by making use of the \textit{logarithmic PSS map} introduced by Ganatra-Pomerleano in $\cite{gp1,gp2}$. Under the assumption that $M=X\setminus D$, where $(X,D)$ is a \textit{multiplicatively topological pair} in the sense of $\cite{gp2}$, $\mathit{SH}^0(M)$ is isomorphic to the \textit{logarithmic cohomology} $H_\mathit{log}^\ast(X,D)$ as a $\mathbb{K}$-algebra, while $H_\mathit{log}^0(X,D)$ does not contain any non-trivial unit.

\subsection{A conjectural picture}\label{section:conjecture}

Although the results obtained in this paper are far from providing a complete classification of Liouville manifolds admitting cyclic dilations, in view of our discussions in Section \ref{section:trichotomy}, it seems to be reasonable to expect the following (note that we consider here only the case when $\mathrm{char}(\mathbb{K})=0$):
\begin{conjecture}\label{conjecture:trichotomy}
Let $M$ be an $n$-dimensional smooth affine variety.
\begin{itemize}
	\item If $\kappa(M)=-\infty$, then $M$ admits a cyclic dilation with $h=1$.
	\item If $\kappa(M)=0$, then $M$ admits a cyclic dilation if and only if it admits a quasi-dilation with $h\neq1$.
	\item If $\kappa(M)=n$, then $M$ does not admit a cyclic dilation.
\end{itemize}
\end{conjecture}
Note that in order for our conjecture to make sense, we need to regard manifolds with $\mathit{SH}^\ast(M)=0$ as manifolds which carry cyclic dilations.

The expectation that cyclic dilations should exist for all affine varieties with $\kappa(M)=-\infty$ is probably too optimistic, it seems to be more reasonable to state the conjecture for all the Milnor fibers with $\kappa(M)=-\infty$. However, there are affine varieties with $\kappa(M)=-\infty$ which are not Milnor fibers, but which do admit cyclic dilations. As an example, consider the affine hypersurface $M\subset\mathbb{C}^4$ defined by the equation
\begin{equation}
x+y+xyz+w^2=1.
\end{equation}
Since $M$ carries a Lefschetz fibration $\pi:M\rightarrow\mathbb{C}$ with the smooth fiber being symplectomorphic to a 4-dimensional $D_4$ Milnor fiber (cf. $\cite{cm}$, Section 4.1), combining the argument in $\cite{yl}$, Section 4.2 with the Lefschetz fibration method due to Seidel-Solomon $\cite{ss}$ shows that $M$ admits a quasi-dilation. This example is also interesting in the sense that the existence of an exact Calabi-Yau structure on $\mathcal{W}(M)$ does not follow from Van den Bergh's Theorem \ref{theorem:Koszul}. Direct computations yield the quasi-isomorphism
\begin{equation}
\mathcal{W}_M\cong\mathbb{K}[x,y], |x|=1, |y|=-2,
\end{equation}
see for example $\cite{yl}$, Section 7.4. Since $\mathcal{W}_M$ is formal, and has generators in positive degrees, Theorem \ref{theorem:Koszul} is not applicable here.
\bigskip

The relation between the existence of a cyclic dilation and the finiteness of the first Gutt-Hutchings capacity was explained in Section \ref{section:disjoint}. In view of Lemma \ref{lemma:ML}, the first item of Conjecture \ref{conjecture:trichotomy} implies the following:
\begin{conjecture}
Let $M$ be a smooth affine variety which is algebraically $1$-uniruled, then as a Liouville manifold, we have $c_1^\mathit{GH}(M)<\infty$.
\end{conjecture}
For related studies in the case of closed symplectic manifolds, see $\cite{gl}$.

It seems likely that there is no exact Lagrangian tori in smooth affine varieties with $\kappa(M)=-\infty$. In view of Corollary \ref{corollary:h=1}, this provides evidences for the more precise expectation that the marking map $\mathbf{B}:\mathit{SH}_{S^1}^1(M)\rightarrow\mathit{SH}^0(M)$ should actually hit the identity. 
\bigskip

Since there should be an exact Lagrangian torus in every smooth log Calabi-Yau variety, one expects that $h\neq1$ in view of Corollary \ref{corollary:h=1}. In fact, this can be rigorously proved. By $\cite{zz1}$, Theorem L, if a smooth affine variety $M$ admits a dilation, then $\overline{M}$ is $(1,\mu)$-uniruled for some $\mu>0$ in the sense of Definition \ref{definition:uniruled}. In particular, $\kappa(M)=-\infty$. The same argument as in $\cite{zz1}$, Section 5 can be applied to prove the uniruledness of $M$ by affine lines when it admits a cyclic dilation with $h=1$.
\bigskip

Although this paper does not deal with affine varieties with $0<\kappa(M)<n$, it is not difficult to find affine surfaces of log Kodaira dimension 1 which admit cyclic dilations. For example, since $T^\ast S^1$ admits a quasi-dilation, so do $T^\ast S^1\times F_g$, where $F_g$ is a once punctured surface with genus $g\geq2$. Note that these affine surfaces can be partially compactified to contractible affine surfaces of log Kodaira dimension 1, whose classification can be found in $\cite{tp}$. It is unclear whether these contractible affine surfaces admit cyclic dilations, although we know that there are non-trivial invertible elements in $\mathit{SH}^0(M)$.

To prove the non-existence of cyclic dilations for affine varieties with $\kappa(M)=n$, one needs to exclude the possibility of having a cyclic dilation with $h\neq1$. It seems that the argument in the proof of Theorem \ref{theorem:unique} would still be useful, but it is in general not clear how to show that $h\in\mathit{SH}^0_+(M)$.

\appendix

\section{Construction of the operations $\ast_k$}\label{section:product}

The construction in this appendix is motivated by the \textit{equivariant pair-of-pants product} introduced by Seidel $\cite{ps8}$. Here we need a slight variant of his construction for $S^1$-equivariant Hamiltonian Floer cohomologies. For each $k\geq1$, we will introduce a chain level operation $\smile_k$ which decreases the degree by $2k$. Our real goal here is to construct a parametrized version $\ast_k$ of the star product (\ref{eq:star}) on Hamiltonian Floer cohomologies, which played a role in our proof of Theorem \ref{theorem:main}.
\bigskip

Let $k\geq1$ be an integer, we first define the operation $\smile_k$. Consider the 3-punctured sphere $S=S^2\setminus\{\zeta_{\mathit{in},0},\zeta_{\mathit{in},1},\zeta_\mathit{out}\}$, with two of the punctures $\zeta_{\mathit{in},0}$ and $\zeta_{\mathit{in},1}$ serving as inputs and the remaining one $\zeta_\mathit{out}$ is an output. As a convention, we shall take the representative of the punctured sphere so that $\zeta_{0,\mathit{in}}=e^{\frac{\pi i}{3}}$, $\zeta_{\mathit{in},1}=e^{\frac{2\pi i}{3}}$ and $\zeta_\mathit{out}=1$, so they are equidistributed along the equator. For the purpose of developing a parametrized theory, we also need to introduce the auxiliary marked points $p_1,\cdots,p_k\in S$. We require that the marked points $\{p_1,\cdots,p_k\}$ lie in a disc centered at $\zeta_\mathit{out}$ of radius $\varepsilon$, and they should be strictly radially ordered in the sense of (\ref{eq:radial3}) with respect to the standard complex coordinate near $1\in\mathbb{CP}^1$. Denote by $\mathcal{P}_k$ the moduli space of these punctured surfaces with $k$ marked points. 

For any representative $(S,p_1,\cdots,p_k)$ of an element of $\mathcal{P}_k$, we fix cylindrical ends
\begin{equation}\label{eq:cylend}
\varepsilon_0^+,\varepsilon_1^+:[0,\infty)\times S^1\rightarrow S,\textrm{ }\varepsilon^-:(-\infty,0]\times S^1\rightarrow S
\end{equation}
with coordinates $(s,t)\in\mathbb{R}_\pm\times S^1$, where $\varepsilon_0^+$ and $\varepsilon_1^+$ are positive cylindrical ends at $\zeta_{\mathit{in},0}$ and $\zeta_{\mathit{in},1}$ respectively, and $\varepsilon^-$ is a negative cylindrical end at $\zeta_\mathit{out}$. The choices are made here so that none of the cylindrical ends $\varepsilon_0^+$ and $\varepsilon_1^+$ contain any of the auxiliary marked points $\{p_i\}_{1\leq i\leq k}$, and $\varepsilon^-$ is chosen so that the negative $s$-direction is given by $\theta_1=\arg(p_1)$, where again the argument is taken with respect to the local complex coordinate near the origin. In other words, the corresponding asymptotic markers $\ell_{\mathit{in},0}$ and $\ell_{\mathit{in},1}$ at $\zeta_{\mathit{in},0}$ and $\zeta_{\mathit{in},1}$ are fixed, pointing respectively along the arcs $\left\{\varepsilon_0^+(s,0)\right\}$ and $\left\{\varepsilon_1^+(s,0)\right\}$, while the asymptotic marker $\ell_{\mathit{out}}$ at $\zeta_{\mathit{out}}$ is allowed to vary freely, since it is required to point towards $p_1$. To further fix conventions, we shall require that $\ell_{\mathit{in},0}$ points towards $\zeta_\mathit{out}$, and $\ell_{\mathit{in},1}$ points towards $\zeta_{\mathit{in},0}$, or equivalently, these two arrows are arranged so that they point clockwisely along the equator. We say that the choices of cylindrical ends $\varepsilon_0^+$, $\varepsilon_1^+$ and $\varepsilon^-$ \textit{are compatible with} the asymptotic markers in the sense that the positive (resp. negative) $s$-directions of the cylindrical ends coincide with the directions of $\ell_{\mathit{in},0}$ and $\ell_{\mathit{in},1}$ (resp. $\ell_\mathit{out}$), i.e.
\begin{equation}
\lim_{s\rightarrow\infty}\varepsilon_0^+(s,1)=\ell_{\mathit{in},0}, \lim_{s\rightarrow\infty}\varepsilon_1^+(s,1)=\ell_{\mathit{in},1},
\end{equation}
\begin{equation}
	\lim_{s\rightarrow-\infty}\varepsilon^-(s,1)=\ell_\mathit{out}.
\end{equation}
The codimension 1 boundary strata of the Deligne-Mumford compactification $\overline{\mathcal{P}}_k$ is covered by the images of the natural inclusions of the following strata:
\begin{equation}\label{eq:bdy1}
\overline{\mathcal{P}}_j\times\overline{\mathcal{M}}_{k-j},0<j\leq k,
\end{equation}
\begin{equation}\label{eq:bdy3}
\overline{\mathcal{P}}_k^{i,i+1},1\leq i\leq k-1,
\end{equation}
\begin{equation}\label{eq:bdy4}
\overline{\mathcal{P}}^{S^1}_{k-1},
\end{equation}
where the strata $\mathcal{P}_k^{i,i+1}$ are the loci where $|p_i|=|p_{i+1}|$ for some $i$, and the stratum $\mathcal{P}^{S^1}_{k-1}$ is the locus where $|p_k|=\frac{1}{2}$. Abstractly, the moduli space $\mathcal{P}^{S^1}_{k-1}$ can be identified with $S^1\times\mathcal{P}_{k-1}$, so its compactification is given by $S^1\times\overline{\mathcal{P}}_{k-1}$. However, this identification is not compatible with the choice of cylindrical ends and holds only on the topological level.

Analogous to (\ref{eq:forget1}), there is a forgetful map
\begin{equation}\label{eq:f1}
\pi^i:\mathcal{P}_k^{i,i+1}\rightarrow\mathcal{P}_{k-1}
\end{equation}
for each $1\leq i\leq k-1$, which forgets the auxiliary marked point $p_{i+1}$. Since $\pi^i$ is compatible with our choices of the cylindrical ends, it extends as a map $\bar{\pi}^i:\overline{\mathcal{P}}_k^{i,i+1}\rightarrow\overline{\mathcal{P}}_{k-1}$ on the compactifications.

One can also consider the map
\begin{equation}\label{eq:f2}
\pi^{S^1}:\mathcal{P}^{S^1}_{k-1}\rightarrow\mathcal{P}_{k-1}
\end{equation}
which forgets the marked point $p_k$. Under the identification $\mathcal{P}^{S^1}_{k-1}\cong S^1\times\mathcal{P}_{k-1}$, $\pi^{S^1}$ is the natural projection to the second factor. However, since $\pi^{S^1}$ is not compatible with the cylindrical end $\varepsilon^-$ when $k=1$, the identification fails when taking the choices of Floer data into account.

In order to write down the appropriate Floer equations, we need to specify our choices of Floer data on the domains. For later purposes, we shall work here with Hamiltonians of the form (\ref{eq:linear-Ham}) on the cylindrical end instead of the quadratic ones.

\begin{definition}\label{definition:data-domain}
A Floer datum for a representative $(S,p_1,\cdots,p_k)$ of an element of $\mathcal{P}_k$ consists of the following choices:
\begin{itemize}
	\item cylindrical ends $\varepsilon_0^+$, $\varepsilon_1^+$ and $\varepsilon^-$ which are compatible with the asymptotic markers $\ell_{\mathit{in},0}$, $\ell_{\mathit{in},1}$ and $\ell_{\mathit{out}}$ specified above;
	\item a closed 1-form $\nu_S\in\Omega^1(S)$ which pulls back to $dt$ via the maps $\varepsilon_0^+$, $\varepsilon_1^+$ and $\varepsilon^-$;
	\item a surface-dependent Hamiltonian $H_S:S\rightarrow\mathcal{H}_\ell(M)$ which is compatible with the cylindrical ends, in the sense that
	\begin{equation}
	(\varepsilon_0^+)^\ast H_S=H_{\lambda_0,t}, (\varepsilon_1^+)^\ast H_S=H_{\lambda_1,t}, (\varepsilon^-)^\ast H_S=H_{\lambda_0+\lambda_1,t}
	\end{equation}
	for some fixed choices of Hamiltonians $H_{\lambda_0,t}\in\mathcal{H}_{\lambda_0}(M)$, $H_{\lambda_1,t}\in\mathcal{H}_{\lambda_1}(M)$, and $H_{\lambda_0+\lambda_1,t}\in\mathcal{H}_{\lambda_0+\lambda_1}(M)$, where $\lambda_0$ and $\lambda_1$ are real numbers so that $\lambda_0,\lambda_1,\lambda_0+\lambda_1\notin\mathcal{P}_M$;
	\item a surface-dependent almost complex structure $J_S:S\rightarrow\mathcal{J}(M)$ which is compatible with the cylindrical ends, meaning that
	\begin{equation}
	(\varepsilon_0^+)^\ast J_S=(\varepsilon_1^+)^\ast J_S=(\varepsilon^-)^\ast J_S=J_t
	\end{equation}
	for some fixed $J_t\in\mathcal{J}(M)$.
\end{itemize}
\end{definition}

\begin{definition}\label{definition:data}
	A universal and consistent choice of Floer data for the operations $\{\smile_k\}$ is an inductive choice of Floer datum for each $k\geq1$ and each marked surface $(S,p_1,\cdots,p_k)$ representing a point of $\overline{\mathcal{P}}_k$, varying smoothly in $(S,p_1,\cdots,p_k)$, such that the following conditions are satisfied:
	\begin{itemize}
		\item Along the boundary strata (\ref{eq:bdy1}), the Floer data should be chosen to agree with the product of Floer data chosen previously on $\overline{\mathcal{P}}_j$ and $\overline{\mathcal{M}}_{k-j}$ up to conformal equivalence. Moreover, the choices vary smoothly with respect to the gluing charts.
		\item Along the boundary strata (\ref{eq:bdy3}), the Floer data are conformally equivalent to the ones pulled back from $\overline{\mathcal{P}}_{k-1}$ via the forgetful map $\bar{\pi}^i$.
	\end{itemize}
\end{definition}

In the above, the conformal equivalence of Floer data is defined similarly as before. More precisely, given two Floer data $\left(\varepsilon_{0,i}^+,\varepsilon_{1,i}^+,\alpha_{S,i},H_{S,i},J_{S,i}\right)$, where $i=1,2$ for $(S,p_1,\cdots,p_k)$, we say that they are \textit{conformally equivalent} if the choices of cylindrical ends coincide, and there is a constant $c>0$ such that $H_{\lambda,t,1}=\frac{H_{\lambda,t,2}}{c}\circ\psi^c$ and $J_{t,1}=\left(\psi^c\right)^\ast J_{t,2}$ on the cylindrical ends, where the value $\lambda$ is determined by the corresponding cylindrical end, namely $\lambda=\lambda_0$ for $\varepsilon_0^+$, $\lambda=\lambda_1$ for $\varepsilon_1^+$, and $\lambda=\lambda_0+\lambda_1$ for $\varepsilon^-$.

\begin{remark}
Unlike the boundary strata (\ref{eq:bdy1}) and (\ref{eq:bdy3}), we didn't impose any requirements on the behavior of universal and consistent Floer data along the stratum (\ref{eq:bdy4}) in the above definition. This is mainly due to the fact that for the purposes of this paper, we don't need to analyze it and identify the contribution of (\ref{eq:bdry4}) below. The same can be said for the stratum (\ref{eq:b6}). Compare with the way we dealt with the boundary stratum (\ref{eq:st3}) in Section \ref{section:CL}.
\end{remark}

Inductively, since the space of choices of Floer data at each level is non-empty and contractible, universal and consistent choices of Floer data exist. From now on, fix such a choice. For every $k\geq1$, and Hamiltonian orbits $y_0^+\in\mathcal{O}_{M,\lambda_0}$, $y_1^+\in\mathcal{O}_{M,\lambda_1}$ and $y^-\in\mathcal{O}_{M,\lambda_0+\lambda_1}$, we can define the moduli space $\mathcal{P}_k(y_0^+,y_1^+;y^-)$ of pairs
\begin{equation}
\left((S,p_1,\cdots,p_k),u\right),
\end{equation}
where $(S,p_1,\cdots,p_k)\in\mathcal{P}_k$, and $u:S\rightarrow M$ is a map satisfying the Floer equation
\begin{equation}
\left(du-X_{H_S}\otimes \nu_S\right)^{0,1}=0
\end{equation}
with respect to the domain-dependent almost complex structure $J_S$, which has been fixed as part of the Floer datum for $(S,p_1,\cdots,p_k)$, together with asymptotic conditions
\begin{equation}
	\lim_{s\rightarrow+\infty}(\varepsilon_0^+)^\ast u(s,\cdot)=y_0^+, \lim_{s\rightarrow+\infty}(\varepsilon_1^+)^\ast u(s,\cdot)=y_1^+,
\end{equation}
\begin{equation}
	\lim_{s\rightarrow-\infty}(\varepsilon^-)^\ast u(s,\cdot)=y^-.
\end{equation}
The moduli space $\mathcal{P}_k(y_0^+,y_1^+;y^-)$ admits a well-defined Gromov bordification $\overline{\mathcal{P}}_k(y_0^+,y_1^+;y^-)$, whose codimension 1 boundary $\partial\overline{\mathcal{P}}_k(y_0^+,y_1^+;y^-)$ is covered by the inclusions of the following strata:
\begin{equation}\label{eq:bdry1}
\overline{\mathcal{P}}_{k-j}(y_0^+,y_1^+;y)\times\overline{\mathcal{M}}_j(y;y^-), 1\leq j\leq k,
\end{equation}
\begin{equation}\label{eq:bdry1.1}
\overline{\mathcal{M}}(y_1^+;y)\times\overline{\mathcal{P}}_k(y_0^+,y;y^-),
\end{equation}
\begin{equation}\label{eq:bdry1.2}
\overline{\mathcal{M}}(y_0^+;y)\times\overline{\mathcal{P}}_k(y,y_1^+;y^-),
\end{equation}
\begin{equation}\label{eq:bdry2}
\overline{\mathcal{P}}_k(y_0^+,y_1^+;y)\times\overline{\mathcal{M}}(y;y^-),
\end{equation}
\begin{equation}\label{eq:bdry3}
\overline{\mathcal{P}}_k^{i,i+1}(y_0^+,y_1^+;y^-),
\end{equation}
\begin{equation}\label{eq:bdry4}
\overline{\mathcal{P}}^{S^1}_{k-1}(y_0^+,y_1^+;y^-),
\end{equation}
where the boundary strata (\ref{eq:bdry1.1}), (\ref{eq:bdry1.2}) and (\ref{eq:bdry2}) come from the semi-stable breaking, and the strata (\ref{eq:bdry1}), (\ref{eq:bdry3}), and (\ref{eq:bdry4}) correspond to the boundary strata (\ref{eq:bdy1}), (\ref{eq:bdy3}) and (\ref{eq:bdy4}) of $\partial\overline{\mathcal{P}}_k$ respectively.

For generic choices of Floer data, the moduli space $\overline{\mathcal{P}}_k(y_0^+,y_1^+;y^-)$ is a compact manifold with corners of dimension
\begin{equation}
	\deg(y^-)-\deg(y_0^+)-\deg(y_1^+)+2k.
\end{equation}
Every rigid element of $\overline{\mathcal{P}}_k(y_0^+,y_1^+;y^-)$ gives rise to an isomorphism
\begin{equation}
	\mu_u:o_{y_1^+}\otimes o_{y_0^+}\rightarrow o_{y^-}
\end{equation}
of orientation lines. Define the operation
\begin{equation}
	\smile_k:\mathit{CF}^\ast(\lambda_1)\otimes\mathit{CF}^\ast(\lambda_0)\rightarrow\mathit{CF}^{\ast-2k}(\lambda_0+\lambda_1)
\end{equation}
by
\begin{equation}
y_1^+\smile_ky_0^+=\sum_{|y^-|=|y_0^+|+|y_1^+|-2k}\sum_{\left((S,p_1,\cdots,p_k),u\right)\in\overline{\mathcal{P}}_k(y_0^+,y_1^+;y^-)}(-1)^{|y_0^+|+|y_1^+|}\mu_u,
\end{equation}
where the notational convention of Remark \ref{remark:o-line} has been applied, and we have used the abbreviation $|\cdot|$ for $\deg(\cdot)$.
\bigskip

Similarly, by considering the same pair-of-pants surface, but now allowing the auxiliary marked points $p_1,\cdots,p_k\in S$ to vary in a small neighborhood of $\zeta_{\mathit{in},1}$, we can define an operation
\begin{equation}
\smile^k:\mathit{CF}^{\ast+2k}(\lambda_1)\otimes\mathit{CF}^\ast(\lambda_0)\rightarrow\mathit{CF}^\ast(\lambda_0+\lambda_1).
\end{equation}
Denote the moduli space of the corresponding domains by $\mathcal{P}^k$. When analysing the boundary strata of $\overline{\mathcal{P}}^k$, note that in order for a $(k-j)$-point angle decorated cylinder to break ``above" the surface $S$ at $\zeta_{\mathit{in},1}$, the points $p_1,\cdots,p_k$ should be ordered so that (\ref{eq:radial2}) holds with respect to the local complex coordinate near $\zeta_{\mathit{in},1}$. The codimenson 1 boundary $\partial\overline{\mathcal{P}}^k$ is covered by the strata corresponding to (\ref{eq:bdy1}), (\ref{eq:bdy3}), and (\ref{eq:bdy4}), and as a consequence, the codimension 1 boundary of the Gromov compactification $\overline{\mathcal{P}}^k(y_0^+,y_1^+;y^-)$ is covered by the inclusions of the strata
\begin{equation}\label{eq:b1}
\overline{\mathcal{M}}_j(y_1^+;y)\times\overline{\mathcal{P}}^{k-j}(y_0^+,y;y^-), 1\leq j\leq k,
\end{equation}
\begin{equation}\label{eq:b2}
\overline{\mathcal{M}}(y_1^+;y)\times\overline{\mathcal{P}}^k(y_0^+,y;y^-),
\end{equation}
\begin{equation}\label{eq:b3}
\overline{\mathcal{M}}(y_0^+;y)\times\overline{\mathcal{P}}^k(y,y_1^+;y^-),
\end{equation}
\begin{equation}\label{eq:b4}
\overline{\mathcal{P}}^k(y_0^+,y_1^+;y)\times\overline{\mathcal{M}}(y;y^-),
\end{equation}
\begin{equation}
\overline{\mathcal{P}}_{i,i+1}^k(y_0^+,y_1^+;y^-),
\end{equation}
\begin{equation}\label{eq:b6}
\overline{\mathcal{P}}_{S^1}^{k-1}(y_0^+,y_1^+;y^-),
\end{equation}
which are analogous to (\ref{eq:bdry1}) to (\ref{eq:bdry4}) respectively. A signed count of rigid elements in the moduli space $\overline{\mathcal{P}}_{S^1}^{k-1}(y_0^+,y_1^+;y^-)$ gives rise to an operation
\begin{equation}
\smile^{k-1}_{S^1}:\mathit{CF}^{\ast+2k}(\lambda_1)\otimes\mathit{CF}^\ast(\lambda_0)\rightarrow\mathit{CF}^{\ast+1}(\lambda_0+\lambda_1),
\end{equation}
and we have the following:

\begin{proposition}\label{proposition:prod}
As chain level operations,
\begin{equation}\label{eq:eq-prod}
\sum_{j=0}^k\delta_j(y_1^+)\smile^{k-j}y_0^+=d\left(y_1^+\smile^ky_0^+\right)-(-1)^{|y_1^+|}y_1^+\smile^kdy_0^++y_1^+\smile^{k-1}_{S^1}y_0^+.
\end{equation}
\end{proposition}
\begin{proof}
As in (\ref{eq:f1}), there is a forgetful map $\pi_i:\mathcal{P}_{i,i+1}^k\rightarrow\mathcal{P}^{k-1}$, which leaves out the point $p_{i+1}$ and extends to a map $\bar{\pi}_i$ defined on the compactifications. The universal and consistency of our choices of Floer data when defining the operations $\smile^k$ implies that the Floer data chosen for elements $(S,p_1,\cdots,p_k)\in\overline{\mathcal{P}}^k_{i,i+1}$ depend only on their images under $\bar{\pi}_i$. Since the forgetful map $\bar{\pi}_i$ has 1-dimensional fibers, given a representative $((S,p_1,\cdots,p_k),u)$ of an element of $\overline{\mathcal{P}}^k_{i,i+1}(y_0^+,y_1^+;y^-)$, any other point $(S',p_1',\cdots,p_k')\in\mathcal{P}^k_{i,i+1}$ in the same fiber as $(S,p_1,\cdots,p_k)$ defines another representative $((S',p_1',\cdots,p_k'),u)$ of an element of $\overline{\mathcal{P}}^k_{i,i+1}(y_0^+,y_1^+;y^-)$. This shows that the elements of the moduli space $\overline{\mathcal{P}}^k_{i,i+1}(y_0^+,y_1^+;y^-)$ are never rigid.
	
Thus we only need to consider the boundary strata (\ref{eq:b1}), (\ref{eq:b2}), (\ref{eq:b3}), (\ref{eq:b4}) and (\ref{eq:b6}). Standard breaking analysis then implies (\ref{eq:eq-prod}).
\end{proof}

Given a family of Riemann surfaces $(S_q)$ parametrized by $q\in[0,\pi]$, so that each $S_q$ is a sphere with three fixed punctures $\zeta_{\mathit{in},0}$, $\zeta_{\mathit{in},1}$ and $\zeta_\mathit{out}$ equidistributed along the equator. As above, these punctures are equipped with asymptotic markers $\ell_{\mathit{in},0}$, $\ell_{\mathit{in},1}$ and $\ell_{\mathit{out}}$. When $q=0$, $\ell_{\mathit{in},0}$ points towards $\zeta_\mathit{out}$, $\ell_{\mathit{in},1}$ points away from $\zeta_\mathit{out}$, and $\ell_{\mathit{out}}$ points towards $\zeta_{\mathit{in},1}$. As $q$ varies from 0 to $\pi$, all of the asymptotic markers perform a half-turn: anticlockwise for the output, and clockwise for the inputs. These are domains defining the operation
\begin{equation}\label{eq:star}	\ast:\mathit{CF}^\ast(\lambda_1)\otimes\mathit{CF}^\ast(\lambda_0)\rightarrow\mathit{CF}^{\ast-1}(\lambda_0+\lambda_1).
\end{equation}
By symmetrizing it, one gets the familiar Lie bracket $[\cdot,\cdot]$ on Hamiltonian Floer cohomologies, which, after passing to direct limit, equips $\mathit{SH}^\ast(M)$ with the structure of a Gerstenhaber algebra.

In view of the discussions above, one may expect to have higher order analogues $\ast_k,k\geq1$ of the operation $\ast$. To define them, one considers the domains $(S_q,p_1,\cdots,p_k)$, the family of punctured spheres $(S_q)$ together with $k$ auxiliary marked points $p_1,\cdots,p_k$ lying in a small neighborhood of $\zeta_\mathit{out}$. If we fix $\zeta_\mathit{out}$ at the origin, these points should be strictly radially ordered as in (\ref{eq:radial}). Moreover, when $q$ goes from 0 to $\pi$, the asymptotic markers $\ell_{\mathit{in},0}$ and $\ell_{\mathit{in},1}$ are required to rotate clockwisely by an angle of $\pi$, while $\ell_{\mathit{out}}$ is required to point towards $p_1$, so it is free to vary. Figure \ref{fig:eq-Lie} describes the family of domains defining $\ast_3$ with the positions of the marked points $p_1,p_2,p_3$ being fixed. If we allow the marked points $p_1,\cdots,p_k$ to vary under the constraint (\ref{eq:radial}), we will get a $(2k+1)$-dimensional family of marked surfaces, which fibers over $[0,\pi]$, and whose fibers are topologically $\mathbb{R}^k\times T^k$.

\begin{figure}
	\centering
	\begin{tikzpicture}
		\draw (-6,2.25) to (6,2.25);
		\node at (-6,2.5) {$q=0$};
		\node at (6,2.5) {$q=\pi$};
		
		\filldraw[draw=black,color={black!15},opacity=0.5] (0,0) circle (1.5);
		\draw (0,0) circle [radius=1.5];
		\draw [dashed] (0,0) ellipse (1.5 and 0.45);
		\draw (-0.75,-0.375) node {$\times$};
		\draw (0.75,-0.375) node {$\times$};
		\draw (0,0.45) node {$\times$};
		\draw [teal] [->] (-0.75,-0.375) to (-0.75,-0.075);
		\draw [teal] [->] (0.75,-0.375) to (0.75,-0.075);
		\draw [orange, dashed] (0,0.45) circle [radius=0.6];
		\draw [orange] (0,0.9) node[circle,fill,inner sep=1pt] {};
		\draw [orange] (-0.3,0.45) node[circle,fill,inner sep=1pt] {};
		\draw [orange] (0.375,0.45) node[circle,fill,inner sep=1pt] {};
		\draw [teal] [->] (0,0.45) to (-0.3,0.45);
		\node at (-0.75,-0.675) {$\zeta_{\mathit{in},0}$};
		\node at (0.75,-0.675) {$\zeta_{\mathit{in},1}$};
		\node [orange] at (-0.3,0.25) {\small $p_3$};
		\node [orange] at (0.375,0.25) {\small $p_2$};
		\node [orange] at (0,1.1) {\small $p_1$};
		
		\filldraw[draw=black,color={black!15},opacity=0.5] (-6,0) circle (1.5);
		\draw (-6,0) circle [radius=1.5];
		\draw [dashed] (-6,0) ellipse (1.5 and 0.45);
		\draw (-6.75,-0.375) node {$\times$};
		\draw (-5.25,-0.375) node {$\times$};
		\draw (-6,0.45) node {$\times$};
		\draw [teal] [->] (-6.75,-0.375) to (-7.125,-0.3);
		\draw [teal] [->] (-5.25,-0.375) to (-5.625,-0.45);
		\draw [orange, dashed] (-6,0.45) circle [radius=0.6];
		\draw [orange] (-6,0.9) node[circle,fill,inner sep=1pt] {};
		\draw [orange] (-6.3,0.45) node[circle,fill,inner sep=1pt] {};
		\draw [orange] (-5.625,0.45) node[circle,fill,inner sep=1pt] {};
		\draw [teal] [->] (-6,0.45) to (-6.3,0.45);
		\node at (-6.725,-0.675) {$\zeta_{\mathit{in},0}$};
		\node at (-5.25,-0.675) {$\zeta_{\mathit{in},1}$};
		\node [orange] at (-6.3,0.25) {\small $p_3$};
		\node [orange] at (-5.625,0.25) {\small $p_2$};
		\node [orange] at (-6,1.1) {\small $p_1$};
		
		\filldraw[draw=black,color={black!15},opacity=0.5] (6,0) circle (1.5);
		\draw (6,0) circle [radius=1.5];
		\draw [dashed] (6,0) ellipse (1.5 and 0.45);
		\draw (5.25,-0.375) node {$\times$};
		\draw (6.75,-0.375) node {$\times$};
		\draw (6,0.45) node {$\times$};
		\draw [teal] [->] (5.25,-0.375) to (5.625,-0.45);
		\draw [teal] [->] (6.75,-0.37) to (7.125,-0.3);
		\draw [orange, dashed] (6,0.45) circle [radius=0.6];
		\draw [orange] (6,0.9) node[circle,fill,inner sep=1pt] {};
		\draw [orange] (5.7,0.45) node[circle,fill,inner sep=1pt] {};
		\draw [orange] (6.375,0.45) node[circle,fill,inner sep=1pt] {};
		\draw [teal] [->] (6,0.45) to (5.7,0.45);
		\node at (5.25,-0.675) {$\zeta_{\mathit{in},0}$};
		\node at (6.725,-0.675) {$\zeta_{\mathit{in},1}$};
		\node [orange] at (5.7,0.25) {\small $p_3$};
		\node [orange] at (6.375,0.25) {\small $p_2$};
		\node [orange] at (6,1.1) {\small $p_1$};
	\end{tikzpicture}
	\caption{By fixing the points $p_1,p_2,p_3$, we get a slice of the moduli space $\mathcal{P}_{q,3}$}
	\label{fig:eq-Lie}
\end{figure}

For a fixed $q\in[0,\pi]$, the moduli space of marked surfaces $(S_q,p_1,\cdots,p_k)$ will be denoted by $\mathcal{P}_{q,k}$. The choice of Floer data on its Deligne-Mumford compactification $\overline{\mathcal{P}}_{q,k}$ will be essentially the same as that on $\overline{\mathcal{P}}_k$, with the only exception that the choices of the cylindrical ends $\varepsilon_{q,0}^+$ and $\varepsilon_{q,1}^+$ differ for different $q$. It is not hard to arrange the choices of the asymptotic markers at $\zeta_{\mathit{in},0}$ and $\zeta_{\mathit{in},1}$ so that the moduli space $\mathcal{P}_{0,k}$ is identical to $\mathcal{P}_k$ together with their Floer data.

Recall that as part of our Floer datum chosen for $(S_0,p_1,\cdots,p_k)$, there is a domain-dependent Hamiltonian function $H_{S_0}:S_0\rightarrow\mathcal{H}_\ell(M)$ and a domain-dependent almost complex structure $J_{S_0}:S_0\rightarrow\mathcal{J}(M)$. On the cylindrical ends, the Floer datum pulls back to
\begin{equation}
	(\varepsilon_{0,0}^+)^\ast H_{S_0}=H_{\lambda_0,t},(\varepsilon_{0,1}^+)^\ast H_{S_0}=H_{\lambda_1,t},(\varepsilon_0^-)^\ast H_{S_0}=H_{\lambda_0+\lambda_1,t},
\end{equation}
\begin{equation}
	(\varepsilon_{0,0}^+)^\ast J_{S_0}=(\varepsilon_{0,1}^+)^\ast J_{S_0}=(\varepsilon_0^-)^\ast J_{S_0}=J_t,
\end{equation}
where $H_{\lambda_0,t},H_{\lambda_1,t},H_{\lambda_0+\lambda_1,t}\in\mathcal{H}_\ell(M)$, and $J_t\in\mathcal{J}(M)$.

\begin{definition}\label{definition:data-r}
	In general, the Floer datum for a representative $(S_q,p_1,\cdots,p_k)$ of $\mathcal{P}_{q,k}$ consists of
	\begin{itemize}
		\item the choices of two positive cylindrical ends $\varepsilon_{0,q}^+$, $\varepsilon_{1,q}^+$ and a negative cylindrical end $\varepsilon_q^-$, which are compatible respectively with the asymptotic markers $\ell_{\mathit{in},0}$, $\ell_{\mathit{in},1}$, which depend on $q\in[0,\pi]$, and $\ell_{\mathit{out}}$, which points to $p_1$;
		\item a closed 1-form $\nu_{S_q}\in\Omega^1(S_q)$ which pulls back to $dt$ under $\varepsilon_{0,q}^+$, $\varepsilon_{1,q}^+$ and $\varepsilon_q^-$;
		\item a surface-dependent Hamiltonian $H_{S_q}:S_q\rightarrow\mathcal{H}_\ell(M)$ which satisfies
		\begin{equation}
			(\varepsilon_{0,q}^+)^\ast H_{S_q}=H_{\lambda_0,t+q},(\varepsilon_{1,q}^+)^\ast H_{S_q}=H_{\lambda_1,t+q},(\varepsilon_q^-)^\ast H_{S_q}=H_{\lambda_1+\lambda_2,t+q};
		\end{equation}
		\item a surface-dependent almost complex structure $J_{S_q}:S_q\rightarrow\mathcal{J}(M)$ so that
		\begin{equation}
			(\varepsilon_{0,q}^+)^\ast J_{S_q}=(\varepsilon_{1,q}^+)^\ast J_{S_q}=(\varepsilon_q^-)^\ast J_{S_q}=J_{t+q}.
		\end{equation}
	\end{itemize}
\end{definition}

Note that for general $q\in[0,\pi]$, there is an obvious identification $\mathcal{P}_{q,k}\cong\mathcal{P}_k$ given by rotating the asymptotic markers $\ell_{\mathit{in},0}$ and $\ell_{\mathit{in},1}$ by a certain angle. Although the identification does not preserve Floer data, this shows that the codimension 1 boundary strata of $\overline{\mathcal{P}}_{q,k}$ correspond exactly to those given in (\ref{eq:bdy1}) to (\ref{eq:bdy4}).

The moduli space of domains $\mathcal{S}_k$ defining the operation $\ast_k$ has an additional parameter $q$, namely
\begin{equation}
	\mathcal{S}_k:=\bigsqcup_{q\in[0,\pi]}\mathcal{P}_{q,k}.
\end{equation}
Abstractly, $\mathcal{S}_k$ can be identified with $[0,\pi]\times\mathcal{P}_k$, but as we have explained above, this identification is not compatible with Floer data. Denote by $\overline{\mathcal{S}}_k$ the Deligne-Mumford compactification of $\mathcal{S}_k$. The previous identification extends to one on the compactifications: $\overline{\mathcal{S}}_k\cong[0,\pi]\times\overline{\mathcal{P}}_k$.

\begin{definition}\label{definition:u-c}
	A universal and consistent choice of Floer data for the moduli spaces $\left\{\overline{\mathcal{S}}_k\right\}_{k\geq1}$ is an inductive choice of Floer data for each $k\geq1$, each $q\in[0,\pi]$, and each representative $(S_q,p_1,\cdots,p_k)$ of an element of $\overline{\mathcal{P}}_{q,k}$ in the sense of Definition \ref{definition:data-r}, so that
	\begin{itemize}
		\item when restricted to the slice $\overline{\mathcal{P}}_{q,k}\subset\overline{\mathcal{S}}_k$, the Floer data should be universal and consistent in the same sense as in Definition \ref{definition:data};
		\item the Floer data on $\overline{\mathcal{P}}_{q,k}$ should vary smoothly with respect to the parameter $q$;
		\item the Floer data on $\overline{\mathcal{P}}_{0,k}$ coincides with those on $\overline{\mathcal{P}}_k$ under the obvious identification $\overline{\mathcal{P}}_{0,k}\cong\overline{\mathcal{P}}_k$;
		\item the Floer data on the moduli spaces $\overline{\mathcal{P}}_{q,k}$ should be chosen so that for any representative of an element $(S_0,p_1,\cdots,p_k)\in\mathcal{P}_{0,k}$, the Floer datum coincides with the one obtained by pulling back of that on $(S_\pi,\mathit{sw}(p_1),\cdots,\mathit{sw}(p_k))\in\mathcal{P}_{\pi,k}$ via the automorphism $\mathit{sw}:S^2\rightarrow S^2$ of the sphere which swaps the two inputs $\zeta_{\mathit{in},0}$ and $\zeta_{\mathit{in},1}$ and preserves the output $\zeta_\mathit{out}$.
	\end{itemize}
\end{definition}

Pick a univeral and consistent choice of Floer data. For $y_0^+\in\mathcal{O}_{M,\lambda_0}$, $y_1^+\in\mathcal{O}_{M,\lambda_1}$ and $y^-\in\mathcal{O}_{M,\lambda_0+\lambda_1}$, we define for each $k\geq1$ the moduli space $\mathcal{S}_k(y_0^+,y_1^+;y^-)$, which parametrizes triples $(q,(S_q,p_1,\cdots,p_k),u)$, with $q\in[0,\pi]$, $(S_q,p_1,\cdots,p_k)\in\mathcal{P}_{q,k}$ and the map $u:S_q\rightarrow M$ is a solution of Floer's equation
\begin{equation}
	\left(du-X_{H_{S_q}}\otimes\nu_{S_q}\right)^{0,1}=0,
\end{equation}
whose behavior at infinity is controlled by the asymptotic conditions
\begin{equation}
	\lim_{s\rightarrow+\infty}(\varepsilon_{0,q}^+)^\ast u(s,\cdot)=y_0^+,\lim_{s\rightarrow+\infty}(\varepsilon_{1,q}^+)^\ast u(s,\cdot)=y_1^+,
\end{equation}
\begin{equation}
	\lim_{s\rightarrow-\infty}(\varepsilon_q^-)^\ast u(s,\cdot)=y^-.
\end{equation}

Denote by $\overline{\mathcal{S}}_k(y_0^+,y_1^+;y^-)$ the Gromov bordification of $\mathcal{S}_k(y_0^+,y_1^+;y^-)$. For generic choices of Floer data, $\overline{\mathcal{S}}_k(y_0^+,y_1^+;y^-)$ is a compact manifold with corners of dimension
\begin{equation}
	\deg(y^-)-\deg(y_0^+)-\deg(y_1^+)+2k+1.
\end{equation}
A signed count of rigid elements of $\overline{\mathcal{S}}_k(y_0^+,y_1^+;y^-)$ defines the map
\begin{equation}
	\ast_k:\mathit{CF}^\ast(\lambda_1)\otimes\mathit{CF}^\ast(\lambda_0)\rightarrow\mathit{CF}^{\ast-2k-1}(\lambda_0+\lambda_1).
\end{equation}

Just as the usual star product $\ast$ measures the homotopy commutativity of the usual pair-of-pants product $\smile$, with our choices of Floer data as in Definition \ref{definition:u-c}, there is a relation between the operations $\ast_k$ and $\smile_k$ as well, but we are not going to discuss it here.

\Addresses

\end{document}